\documentclass[11pt]{article} %
\usepackage{etoolbox}
\newtoggle{journal}
\togglefalse{journal}
\newcommand{\journal}[1]{\iftoggle{journal}{#1}{}}
\newcommand{\preprint}[1]{\iftoggle{journal}{}{#1}}
\newcommand{\loose}{\journal{\looseness=-1}}

\newtoggle{supp}
\togglefalse{supp}

\usepackage{inconsolata}

\newtoggle{draft}
\togglefalse{draft}             %
\newcommand{\draft}[1]{\iftoggle{draft}{#1}{}}
\newcommand{\notdraft}[1]{\iftoggle{draft}{}{#1}}

\usepackage{comment}

\usepackage{float}

\usepackage{transparent}

\usepackage{multirow}
\usepackage{booktabs}

\journal{\usepackage{type1cm}}

\preprint{
\usepackage[letterpaper, left=1in, right=1in, top=1in, bottom=1in]{geometry}
}

\usepackage[dvipsnames]{xcolor}
\preprint{
\usepackage[colorlinks=true, linkcolor=blue!70!black, citecolor=blue!70!black,urlcolor=black,breaklinks]{hyperref}
}
\journal{
\usepackage[colorlinks=true, linkcolor=blue!70!black, citecolor=blue!70!black,urlcolor=blue!70!black,breaklinks]{hyperref}
\usepackage{breakcites}
}
\usepackage{microtype}

\preprint{
\usepackage{natbib}
\bibliographystyle{abbrvnat}
\bibpunct{(}{)}{;}{a}{,}{,}
}
\journal{
\usepackage[authoryear]{natbib}
}

\usepackage{amsthm}
\usepackage{amsmath}
\usepackage{amsfonts}
\usepackage{amssymb}

\usepackage{mathtools}
\usepackage{bbm}

\usepackage{xpatch}

\preprint{
\usepackage[nameinlink,capitalize]{cleveref}
}
\journal{
\usepackage[capitalize]{cleveref}
}

\newcommand{\pref}[1]{\cref{#1}}
\newcommand{\pfref}[1]{Proof of \pref{#1}}
\newcommand{\savehyperref}[2]{\texorpdfstring{\hyperref[#1]{#2}}{#2}}

\crefformat{equation}{#2(#1)#3}
\Crefformat{equation}{#2(#1)#3}

\preprint{\Crefformat{figure}{#2Figure #1#3}}

\journal{\Crefformat{figure}{Figure #2#1#3}}

\crefname{framedmetaalgorithm}{Meta-Algorithm}{Meta-Algorithms}
\theoremstyle{theorem}  %

\newtheorem{lemma}{Lemma}

\newtheorem{corollary}{Corollary}
\newtheorem{proposition}{Proposition}

\newtheorem{assumption}{Assumption}
\Crefname{assumption}{Assumption}{Assumptions}

\newtheorem{theorem}{Theorem}
\theoremstyle{plain}
\newtheorem{remark}{Remark}
\newtheorem{example}{Example}
\newtheorem{definition}{Definition}

\preprint{
\xpatchcmd{\proof}{\itshape}{\normalfont\proofnameformat}{}{}
\newcommand{\proofnameformat}{\bfseries}
}

\DeclarePairedDelimiter{\abs}{\lvert}{\rvert} %
\DeclarePairedDelimiter{\brk}{[}{]}
\DeclarePairedDelimiter{\crl}{\{}{\}}
\DeclarePairedDelimiter{\prn}{(}{)}
\DeclarePairedDelimiter{\nrm}{\|}{\|}
\DeclarePairedDelimiter{\tri}{\langle}{\rangle}

\DeclarePairedDelimiter{\ceil}{\lceil}{\rceil}
\DeclarePairedDelimiter{\floor}{\lfloor}{\rfloor}

\let\Pr\undefined

\DeclareMathOperator{\En}{\mathbb{E}}

\DeclareMathOperator{\Pr}{Pr}

\DeclareMathOperator*{\argmin}{arg\,min} %

\newcommand{\ls}{\ell}
\newcommand{\indic}{\mathbbm{1}}    %

\newcommand{\eps}{\epsilon}
\newcommand{\veps}{\varepsilon}

\newcommand{\ldef}{\vcentcolon=}
\newcommand{\rdef}{=\vcentcolon}

\newcommand{\wt}[1]{\widetilde{#1}}
\newcommand{\wh}[1]{\widehat{#1}}

\def\ddefloop#1{\ifx\ddefloop#1\else\ddef{#1}\expandafter\ddefloop\fi}
\def\ddef#1{\expandafter\def\csname bb#1\endcsname{\ensuremath{\mathbb{#1}}}}
\ddefloop ABCDEFGHIJKLMNOPQRSTUVWXYZ\ddefloop
\def\ddefloop#1{\ifx\ddefloop#1\else\ddef{#1}\expandafter\ddefloop\fi}
\def\ddef#1{\expandafter\def\csname b#1\endcsname{\ensuremath{\mathbf{#1}}}}
\ddefloop ABCDEFGHIJKLMNOPQRSTUVWXYZ\ddefloop
\def\ddef#1{\expandafter\def\csname c#1\endcsname{\ensuremath{\mathcal{#1}}}}
\ddefloop ABCDEFGHIJKLMNOPQRSTUVWXYZ\ddefloop
\def\ddef#1{\expandafter\def\csname h#1\endcsname{\ensuremath{\widehat{#1}}}}
\ddefloop ABCDEFGHIJKLMNOPQRSTUVWXYZ\ddefloop
\def\ddef#1{\expandafter\def\csname hc#1\endcsname{\ensuremath{\widehat{\mathcal{#1}}}}}
\ddefloop ABCDEFGHIJKLMNOPQRSTUVWXYZ\ddefloop
\def\ddef#1{\expandafter\def\csname t#1\endcsname{\ensuremath{\widetilde{#1}}}}
\ddefloop ABCDEFGHIJKLMNOPQRSTUVWXYZ\ddefloop
\def\ddef#1{\expandafter\def\csname tc#1\endcsname{\ensuremath{\widetilde{\mathcal{#1}}}}}
\ddefloop ABCDEFGHIJKLMNOPQRSTUVWXYZ\ddefloop

\newcommand{\Holder}{H{\"o}lder}

\let\wt\undefined

\newcommand{\wt}[1]{\widetilde{#1}}
\newcommand{\mb}[1]{\boldsymbol{#1}}
\newcommand{\R}{\bbR}

\newcommand{\conv}{\textrm{conv}}

\newcommand{\xr}[1][n]{x_{1:#1}}

\newcommand{\zr}[1][n]{z_{1:#1}}

\newcommand{\grad}{\nabla}

\newcommand{\trn}{\top}

\newcommand{\yh}{\hat{y}}

\renewcommand{\trn}{\dagger}

\usepackage{xspace}

\usepackage{mdframed}
\newmdtheoremenv[innertopmargin=10pt,linewidth=.3mm]{framedmetaalgorithm}{Meta-Algorithm}
\newmdtheoremenv[innertopmargin=10pt,linewidth=.3mm]{framedalgorithm}{Algorithm}

\preprint{
\usepackage{inconsolata}
\usepackage[scaled=.88]{helvet}
}

\usepackage[suppress]{color-edits}
\addauthor{df}{red!80!black}
\addauthor{vs}{red!80!black}

\newtoggle{revone}
\toggletrue{revone}

\addauthor{revone}{red!80!black}
\newenvironment{revone}{\color{black}}{\ignorespacesafterend}
\newcommand{\revonecolor}[1]{\iftoggle{revone}{{\color{black}#1}}{}}

  \usepackage{nicefrac}

\usepackage{enumitem}

\newcommand{\thetabar}{\bar{\theta}}

\newcommand{\Ot}{\wt{O}}
\newcommand{\littleo}{o}
\newcommand{\bigoh}{O}
\newcommand{\bigoht}{\Ot}

\newcommand{\Tk}{T_{\cK}}

\newcommand{\unitrange}{\brk*{-1,+1}}

\newcommand{\vphi}{\varphi}

\newcommand{\Pn}{\ensuremath{\mathbb{P}}}
\newcommand{\sh}{\ensuremath{\mathrm{star}}}

\newcommand{\rbar}{|}

\newcommand{\Rad}{\mathfrak{R}_{n}}

\newcommand{\Rademp}{\mathfrak{R}_n}
\newcommand{\Radexp}{\mathcal{R}_n}

\newcommand{\midsem}{;\,}

\newcommand{\poprisk}{L_{\cD}}

\newcommand{\approxleq}{\lesssim}

\renewcommand{\yh}{\wh{y}}

\renewcommand{\conv}{\mathrm{conv}}

\newcommand{\starhull}{\mathrm{star}}

\newcommand{\var}{\mathrm{Var}}
\newcommand{\varn}{\mathrm{Var}_{n}}

\newcommand{\auxX}{\wt{X}}
\newcommand{\auxY}{\wt{Y}}
\newcommand{\auxS}{\wt{S}}

\newcommand{\auxL}{\wt{L}}

\newcommand{\varall}{z}
\newcommand{\varone}{w} %
\newcommand{\vartwo}{x} %
\newcommand{\Varall}{\cZ}
\newcommand{\Varone}{\cW} %
\newcommand{\Vartwo}{\cX} %

\newcommand{\estone}{\wh{g}}
\newcommand{\esttwo}{\wh{\theta}}

\newcommand{\esttwoinit}{\wh{\theta}_{\mathrm{init}}}

\newcommand{\Aclass}{\cA}
\newcommand{\esta}{\wh{a}}
\newcommand{\gta}{a_0}

\newcommand{\gtF}{\cF_{0}}
\newcommand{\estF}{\cF}

\newcommand{\pone}{p_1}
\newcommand{\ptwo}{p_2}

\newcommand{\done}{d_1}
\newcommand{\dtwo}{d_2}

\newcommand{\vecone}{\cV_1}
\newcommand{\vectwo}{\cV_2}

\newcommand{\nuisance}{\cG}
\newcommand{\target}{\Theta}
\newcommand{\targetemp}{\wh{\Theta}}

\newcommand{\normone}[1]{\nrm*{#1}_{\vecone}}
\newcommand{\normtwo}[1]{\nrm*{#1}_{\vectwo}}

\newcommand{\dimone}{K_1}
\newcommand{\dimtwo}{K_2}

\newcommand{\dist}{\cD}

\newcommand{\Rate}{\mathrm{Rate}_{\cD}}

\newcommand{\besttwo}{\theta^{\star}}

\newcommand{\gtone}{g_0}
\newcommand{\gttwo}{\theta_0}

\newcommand{\sampleone}{S_1}
\newcommand{\sampletwo}{S_2}
\newcommand{\samplethree}{S_3}
\newcommand{\samplefour}{S_4}

\newcommand{\algone}[1]{\mathrm{Alg}(\nuisance,#1)}

\newcommand{\momentconstant}{C_{2\to{}4}}

\renewcommand{\trn}{\top}

\newcommand{\ind}[1]{^{\scriptscriptstyle (#1)}}

\renewcommand{\hat}[1]{\wh{#1}}

\newcommand{\indep}{\perp}

\newcommand{\param}{parameter\xspace}
\newcommand{\params}{parameters\xspace}

\newcommand{\targeteps}{\target_0(\veps)}

\newcommand{\loglink}{\sigma_{\mathrm{log}}}

\newcommand{\fstar}{f^{\star}}

\journal{\renewcommand{\footnote}{\textcolor{red}{footnote}}}
\newcommand{\footnoteorinline}[1]{\preprint{\footnote{#1}}\journal{ #1}}
\newcommand{\footnoteorhide}[1]{\preprint{\footnote{#1}}\journal{}}

\newcommand{\si}{\mathsf{si}}
\newcommand{\Lsi}{L_{\si}}
\newcommand{\Rsi}{R_{\si}}
\newcommand{\Tsi}{T_{\si}}
\newcommand{\tausi}{\tau_{\si}}
\newcommand{\musi}{\mu_{\si}}
\newcommand{\gammasi}{\lambda_{\si}}
\newcommand{\lambdasi}{\gammasi}
\newcommand{\betasi}{\beta_{\si}}
\newcommand{\rsi}{r_{\si}}

\preprint{
\setcounter{tocdepth}{3}
}
\journal{
  \notdraft{\setcounter{tocdepth}{0}}
}

\newcommand{\alphamark}[1]{\textnormal{#1)}}

\preprint{
  \let\oldparagraph\paragraph
  \renewcommand{\paragraph}[1]{\oldparagraph{#1.}}
  }
\journal{
  \let\oldparagraph\paragraph
  \renewcommand{\paragraph}[1]{\oldparagraph*{#1}}
}

\preprint{
\usepackage{parskip}
}

\preprint{
\title{\huge{Orthogonal Statistical Learning}}
\author{Dylan J. Foster \\ Microsoft Research \\dylanfoster@microsoft.com\and Vasilis Syrgkanis \\ Stanford University \\vsyrgk@stanford.edu}
\date{}
}

\begin{document}
\preprint{
\maketitle
\begin{abstract}
We provide non-asymptotic excess risk guarantees for statistical learning in a setting where the population risk with respect to which we evaluate the target parameter depends on an unknown \emph{nuisance parameter} that must be estimated from data. We analyze a two-stage sample splitting meta-algorithm that takes as input arbitrary estimation algorithms for the target parameter and nuisance parameter. We show that if the population risk satisfies a condition called \emph{Neyman orthogonality}, the impact of the nuisance estimation error on the excess risk bound achieved by the meta-algorithm is of second order. Our theorem is agnostic to the particular algorithms used for the target and nuisance and only makes an assumption on their individual performance. This enables the use of a plethora of existing results from machine learning to give new guarantees for learning with a nuisance component. Moreover, by focusing on excess risk rather than parameter estimation, we can provide rates under weaker assumptions than in previous works and accommodate settings in which the target parameter belongs to a complex nonparametric class. We provide conditions on the metric entropy of the nuisance and target classes such that \emph{oracle rates} of the same order as if we knew the nuisance parameter are achieved. \end{abstract}
}

\journal{
\begin{frontmatter}
\title{Orthogonal Statistical Learning}
\runtitle{Orthogonal Statistical Learning}

\begin{aug}
\author[A]{\fnms{Dylan J.} \snm{Foster}\ead[label=e1]{dylanfoster@microsoft.com}}
\and
\author[B]{\fnms{Vasilis} \snm{Syrgkanis}\ead[label=e2]{vasy@microsoft.com}}
\address[A]{Microsoft Research, New England,
\printead{e1}}
\address[B]{Microsoft Research, New England,
\printead{e2}}
\end{aug}

\begin{abstract}
We provide non-asymptotic excess risk guarantees for statistical learning in a setting where the population risk with respect to which we evaluate the target parameter depends on an unknown \emph{nuisance parameter} that must be estimated from data. We analyze a two-stage sample splitting meta-algorithm that takes as input arbitrary estimation algorithms for the target parameter and nuisance parameter. We show that if the population risk satisfies a condition called \emph{Neyman orthogonality}, the impact of the nuisance estimation error on the excess risk bound achieved by the meta-algorithm is of second order. Our theorem is agnostic to the particular algorithms used for the target and nuisance and only makes an assumption on their individual performance. This enables the use of a plethora of existing results from machine learning to give new guarantees for learning with a nuisance component. Moreover, by focusing on excess risk rather than parameter estimation, we can provide rates under weaker assumptions than in previous works and accommodate settings in which the target parameter belongs to a complex nonparametric class. We provide conditions on the metric entropy of the nuisance and target classes such that \emph{oracle rates} of the same order as if we knew the nuisance parameter are achieved. \end{abstract}



\end{frontmatter}
}

\preprint{
{ %
\hypersetup{linkcolor=black}
\tableofcontents
}
\newpage
}

\draft{
{ %
\hypersetup{linkcolor=black}
\tableofcontents
}
\newpage
}

\section{Introduction}
\label{sec:intro}

Predictive models based on modern machine learning methods are becoming increasingly widespread in policy making, with applications in healthcare, education, law enforcement, and business decision making. Most problems that arise in policy making, such as attempting to predict counterfactual outcomes for different interventions or optimizing policies over such interventions, are not pure prediction problems, but rather are causal in nature. It is important to address the causal aspect of these problems and build models that have a causal interpretation.

A common paradigm in the search of causality is that to estimate a model with a causal interpretation from observational data---that is, data not collected via randomized trial or via a known treatment policy---one typically needs to estimate many other quantities that are not of primary interest, but that can be used to de-bias a purely predictive machine learning model by formulating an appropriate loss. One example of such a \emph{nuisance parameter} is the propensity for taking an action under the current policy, which can be used to form unbiased estimates for the reward for new policies, but is typically unknown in datasets that do not come from controlled experiments. 

To make matters more concrete, let us walk through an example for which certain variants have been well-studied in machine learning \citep{dudik2011doubly,swaminathan2015counterfactual,nie2017quasi,kallus2018policy}. Suppose a decision maker wants to estimate the causal effect of some treatment $T\in \{0,1\}$ on an outcome $Y$ as a function of a set of observable features $X$; the causal effect will be denoted as $\theta(X)$. Typically, the decision maker has access to data consisting of tuples $(X_i, T_i, Y_i)$, where $X_i$ is the observed feature for sample $i$, $T_i$ is the treatment taken, and $Y_i$ is the observed outcome. 
Due to the partially observed nature of the problem, one needs to create unbiased estimates of the unobserved outcome. A standard approach is to make an \emph{unconfoundedness} assumption \citep{rosenbaum1983central} and use the so-called \emph{doubly-robust} formula, which is a combination of direct regression and inverse propensity scoring. Let $Y_i(t)$ denote the potential outcome for treatment $t$ in sample $i$, and let $f_0(t,x_i)\ldef\En\brk[\big]{Y_i(t) \mid x_i}$ and $p_0(t,x_i) \ldef \En\brk*{1\{T=t\} \mid x_i}$. If $(Y_i(0),Y_i(1))\indep{}T_i\mid{}X_i$, then the following is an unbiased estimator for the \revoneedit{conditional mean potential outcome (given covariates):}
\begin{equation}
\hat{Y}_i{(t)} = f_0(t,x_i) + \frac{(Y_i - f_0(t,x_i))\, 1\{T_i=t\}}{p_0(t,x_i)}.
\end{equation}
Given such an estimator, we can estimate the treatment effect by running a regression between the unbiased estimates and the features, i.e. solve $\min_{\theta\in \target} \sum_i \prn{\hat{Y}(1) - \hat{Y}(0) - \theta(X_i)}^2$ over a target parameter class $\target$. In the population limit, with infinite samples, this corresponds to finding a \param $\theta(x)$ that minimizes the population risk $\En\brk[\big]{(\hat{Y}_i(1) - \hat{Y}_i(0) - \theta(X))^2}$. Similarly, if the decision maker is interested in policy optimization rather than estimating treatment effects, they can use these unbiased estimates to solve $\min_{\theta\in \target} \sum_{i} (\hat{Y}_i(0) - \hat{Y}_i(1))\cdot \theta(X_i)$ over a policy space $\target$ of functions mapping features to $\{0,1\}$. However, when dealing with observational data, the functions $f_0$ and $p_0$ are not known, and must be estimated if we wish to evaluate the proxy labels $\hat{Y}(t)$. Since these functions are only used as a means to learn the target parameter $\theta$, we may regard them as nuisance parameters. The goal of the learner is to estimate a target parameter that achieves low population risk when evaluated at the \emph{true} nuisance parameters as opposed to the estimated nuisance parameters, since only then does the model have a causal interpretation. 

This phenomenon is ubiquitous in causal inference and motivates us to formulate the abstract problem of \emph{statistical learning with a nuisance component}: Given $n$ i.i.d. examples from a distribution $\cD$, a learner is interested in finding a \emph{target \param} $\esttwo\in \target$ so as to minimize a population risk function $\poprisk: \target \times \nuisance \rightarrow \R$. The population risk depends not just on the target \param, but also on a  \emph{nuisance \param} whose true value $\gtone\in \nuisance$ is unknown to the learner. The goal of the learner is to produce an estimate that has small \emph{excess risk} evaluated at the unknown true nuisance \param:
\begin{equation}
\label{eq:excess_risk_intro}
\poprisk(\esttwo,g_0) - \inf_{\theta\in\target}\poprisk(\theta,g_0)\to_n0.
\end{equation}
Depending on the application, such an excess risk bound can take different interpretations. For many settings, such as treatment effect estimation, it is closely related to mean squared error, while in policy optimization it typically corresponds to regret. Following the tradition of statistical learning theory \citep{vapnik1995nature,bousquet2004introduction}, we make excess risk the primary focus of our work, independent of the interpretation. We develop algorithms and analysis tools that generically address \pref{eq:excess_risk_intro}, then apply these tools to a number of applications of interest.

The problem of statistical learning with a nuisance component is strongly connected to the well-studied semiparametric inference problem \citep{levit1976efficiency,ibragimov1981statistical,pfanzagl1982contributions,bickel1982adaptive,klaassen1987consistent,robinson1988root,bickel1993efficient,newey1994asymptotic,robins1995semiparametric, ai2003efficient,van2003unified,van2003unifiedb,ai2007estimation,tsiatis2007semiparametric,kosorok2008introduction,van2011targeted,ai2012semiparametric,chernozhukov2016locally,belloni2017program,chernozhukov2016double}, which focuses on providing so-called ``$\sqrt{n}$-consistent and asymptotically normal'' estimates for a low-dimensional target parameter $\theta_0$ (which may be expressed as a population risk minimizer or a solution to estimating equations) in the presence of a typically nonparametric nuisance parameter. Unlike the semiparametric inference problem, statistical learning with a nuisance component does not require a well-specified model, nor a unique minimizer of the population risk. Moreover, we do not ask for parameter recovery or asymptotic inference (e.g., asymptotically valid confidence intervals). Rather, we are content with an excess risk bound, regardless of whether there is an underlying true parameter to be identified. As a consequence, we provide guarantees even in the presence of misspecification, and when the target \param belongs to a large, potentially nonparametric class. For example, one line of previous work gives semiparametric inference guarantees when the nuisance \param is a neural network \citep{chen1999improved,farrell2018deep}; by focusing on excess risk we can give guarantees for the case where the \emph{target} \param is a neural network.

The case where the target parameter belongs to an arbitrary class has not been addressed at the level of generality we consider in the present work, but we mention some prior work that goes beyond the low-dimensional/parametric setup for special cases. \cite{athey2017efficient} and \cite{zhou2018offline} give guarantees based on metric entropy of the target class for the specific problem of treatment policy learning. For estimation of treatment effects, various nonparametric classes have been used for the target class on a case by case basis, including kernels \citep{nie2017quasi}, random forests \citep{athey2016generalized, oprescu2018orthogonal,friedberg2018}, and high-dimensional linear models \citep{chernozhukov2017orthogonal,chernozhukov2018plugin}. Other results allow for fairly general choices for the target parameter class in specific statistical models \citep{rubin2005general,rubin2007doubly,diaz2013targeted,van2014targeted,kennedy2017non,kennedy2019robust,kunzel2017meta}. Our work unifies these directions into a single framework, and our general tools lead to improved or refined results when specialized to many of these individual settings.

Our approach is to reduce the problem of statistical learning with a nuisance component to the standard formulation of statistical learning. We build on a recent thread of research on semiparametric inference known as ``double'' or ``debiased'' machine learning %
\citep{chernozhukov2016locally,chernozhukov2017orthogonal,chernozhukov2016double,chernozhukov2018double,chernozhukov2018plugin}, which leverages sample splitting to provide inference guarantees under weak assumptions on the estimator for the nuisance parameter. Rather than directly analyzing particular algorithms and models for the target parameter (e.g., regularized regression, gradient boosting, or neural network estimation), we assume a black-box guarantee for the excess risk in the case where a nuisance value $g\in\nuisance$ is fixed. Our main theorem asks only for the existence of an algorithm $\mathrm{Alg}(\target, S \midsem g)$ that, for any given nuisance \param $g$ and data set $S$, achieves low excess risk with respect to the population risk $L_{\cD}(\theta, g)$, i.e. with probability at least $1-\delta$,
\begin{equation}
\poprisk(\esttwo,g) - \inf_{\theta\in\target}\poprisk(\theta,g) \leq \Rate(\target,  S, \delta\midsem{}g).
\end{equation}
Likewise, we assume the existence of a black-box algorithm $\mathrm{Alg}(\nuisance,S)$ to estimate the nuisance component $\gtone$ from the data, with the required estimation guarantee varying from problem to problem.

\begin{figure}[tp]
\begin{framedmetaalgorithm}[Two-Stage Estimation with Sample Splitting]~\\
\textbf{Input}: Sample set $S=z_1,\ldots,z_n$.
\begin{itemize}
\item Split $S$ into subsets $\sampleone=z_1,\ldots,z_{\floor{n/2}}$ and $\sampletwo=S\setminus{}\sampleone$.
\item Let $\estone$ be the output of $\mathrm{Alg}(\nuisance,\sampleone)$.
\item Return $\esttwo$, the output of $\mathrm{Alg}(\target,\sampletwo\midsem\estone)$.
\end{itemize}
\label{alg:sample_splitting}
\end{framedmetaalgorithm}
\end{figure}
Given access to the two black-box algorithms, we analyze a simple sample splitting meta-algorithm for statistical learning with a nuisance component, presented as \savehyperref{alg:sample_splitting}{Meta-Algorithm \ref{alg:sample_splitting}}.
We can now state the main question addressed in this paper: 
\emph{When is the excess risk achieved by sample splitting robust to nuisance component estimation error?}

In more technical terms, we seek to understand when the two-stage sample splitting meta-algorithm achieves an excess risk bound with respect to $\gtone$, in spite of error in the estimator $\estone$ output by the first-stage algorithm. Robustness to nuisance estimation error allows the learner to use more complex models for nuisance estimation and---under certain conditions on the complexity of the target and nuisance \param classes---to learn target \params whose error is, up to lower order terms, as good as if the learner had known the true nuisance \param in advance. Such a guarantee is referred to as achieving an \emph{oracle rate} in semiparametric inference.

\paragraph{Overview of results}  We use \emph{Neyman orthogonality} \citep{neyman1959optimal,neyman1979c}, a key tool in inference in semiparametric models \citep{newey1994asymptotic,van2000asymptotic,robins2008higher,zheng2010asymptotic,belloni2017program,chernozhukov2016double}, to provide oracle rates for statistical learning with a nuisance component. We show that if the population risk satisfies a functional analogue of Neyman orthogonality, the estimation error of $\estone$ has a second order impact on the overall excess risk (relative to $g_0$) achieved by $\esttwo$. To gain some intuition, Neyman orthogonality is weaker condition than double-robustness, albeit similar in flavor, (see, e.g., \cite{chernozhukov2016locally}) and is satisfied by both the treatment effect loss and the policy learning loss described in the introduction. In more detail, our variant of the Neyman orthogonality condition asserts that a functional cross-derivative of the loss vanishes when evaluated at the optimal target and nuisance \params. Prior work provides a number of means through which to construct Neyman orthogonal losses whenever certain moment conditions are satisfied by the data generating process \citep{chernozhukov2016double,chernozhukov2016locally,chernozhukov2018plugin}. Indeed, orthogonal losses can be constructed in settings including treatment effect estimation, policy learning, missing and censored data problems, estimation of structural econometric models, and game-theoretic models.

We identify two regimes of excess risk behavior:
\begin{enumerate}
\item \emph{Fast rates.} When the population risk is strongly convex with respect to the prediction of the target \param (e.g., the treatment effect estimation loss), then typically so-called \emph{fast rates} (e.g., rates of order of $O(1/n)$ for parametric classes) are optimal if the true nuisance \param is known. Letting $R_{\nuisance}$ denote the estimation error of the nuisance component, in this setting we show that orthogonality implies that the first stage error has an impact on the excess risk of the order of $R_{\nuisance}^4$ (in particular, $n^{-1/4}$-RMSE rates for the nuisance suffice when the target is parametric).
\item \emph{Slow rates.} Absent any strong convexity of the population risk (e.g., for the treatment policy optimization loss), typically \emph{slow rates} (e.g. rates of order $O(1/\sqrt{n})$ for parametric classes) are optimal if the true nuisance \param is known.  For this setting, we show that the impact of nuisance estimation error is of the order $R_{\nuisance}^2$ so, once again, $n^{-1/4}$ RMSE rates for the nuisance suffice when the target is parametric.
\end{enumerate}

To make the conditions above concrete for arbitrary classes, we give conditions on the relative complexity of the target and nuisance classes---quantified via \emph{metric entropy}---under which the sample splitting meta-algorithm achieves oracle rates, assuming the two black-box estimation algorithms are instantiated appropriately. This allows us to extend several prior works beyond the parametric regime to complex nonparametric target classes. Our technical results extends the works of \cite{yang1999information,rakhlin2017empirical}, which provide minimax optimal rates without nuisance components and utilize the technique of \emph{aggregation} in designing optimal algorithms.

The flexibility of our approach allows us to instantiate our framework with any machine learning model and algorithm of interest for both nuisance and target \param estimation, and to utilize the vast literature on generalization bounds in machine learning to establish refined (e.g., data-dependent or dimension-independent) rates for several classes of interests. For instance, our approach allows us to leverage recent work on size-independent generalization error of neural networks.

Moving beyond black-box results, we use our main theorems as a starting point to provide sharp analyses for certain general-purpose statistical learning algorithms for target estimation in the presence of nuisance parameters. First, we provide a new analysis for empirical risk minimization with plug-in estimation of nuisance parameters, wherein we extend the classical local Rademacher complexity analysis of empirical risk minimization \citep{koltchinskii2000rademacher,bartlett2005local} to account for the impact of the nuisance error (leveraging orthogonality). Second, in the slow rate regime we give a new analysis of \emph{variance-penalized} empirical risk minimization with plug-in nuisance estimation, which allows us to recover and extend several prior results in the literature on policy learning. Our result improves upon the variance-penalized risk minimization approach of \cite{maurer2009empirical} by replacing the dependence on the metric entropy at a fixed approximation level with the \emph{critical radius}, which is related to the entropy integral.

As a consequence of focusing on excess risk, we obtain oracle rates under weaker assumptions on the data generating process than in previous works. Notably, we obtain guarantees even when the target \param is misspecified and the target parameters are not identifiable. For instance, for sparse high-dimensional linear classes, we obtain optimal prediction rates with no restricted eigenvalue assumptions. We highlight the applicability of our results to four settings of primary importance in the literature: 1) estimation of heterogeneous treatment effects from observational data, 2) offline policy optimization, 3) domain adaptation, 4) learning with missing data. For each of these applications, our general theorems allow for the use of arbitrary estimators for the nuisance and target \param classes and provide robustness to the nuisance estimation error. 
\subsection{Related work}
\paragraph{General frameworks for learning/inference with nuisance parameters} The work of \cite{van2003unified} and subsequent refinements and extensions \citep{van2006cross,van2007super} develops cross-validation methodology for a similar risk minimization setting in which the target risk parameter depends on an unknown nuisance parameter. \cite{van2003unified} analyze a cross-validation meta-algorithm in which the learner simultaneously forms a nuisance parameter estimator and a set of candidate target parameter estimators using a set of training samples, then selects a final estimate for the target parameter by minimizing an empirical loss over a validation set. The train and validation splits may be chosen in a general fashion that encompasses $K$-fold and Monte Carlo validation.
They provide finite-sample oracle rates for the excess risk in the case where the target parameter belongs to a finite class (in particular, rates of the type $\log\abs*{\target}/n$ for a class of square losses and $\sqrt{\log\abs*{\target}/n}$ for general losses), and also extend these guarantees to linear combinations of basis functions via pointwise $\veps$-nets (in our language, such classes are \emph{parametric}). Overall, our approach offers several new benefits:
\begin{itemize}
\item By completely splitting nuisance estimation and target estimation into separate stages and taking advantage of orthogonality, we can provide
  meta-theorems on robustness that are invariant to the choice of learning algorithm both for the first and second stage, which obviates the need to assume the target class is finite or admits a linear representation (\pref{sec:orthogonal}).
\item When we do specialize to algorithms such as ERM and variants, we can provide finite-sample guarantees for rich classes of target \params
  in terms of sharp learning-theoretic complexity measures such as local Rademacher complexity and empirical metric entropy (\pref{sec:erm}). In particular, we can provide conditions under which oracle rates are attained under very general complexity assumptions on the target and nuisance parameters (\pref{sec:oracle}).%
\end{itemize}
The methodology of \cite{van2003unified} can be used to directly estimate a target parameter or to select the best of many candidate nuisance estimators in a data-driven fashion. \cite{van2007super} refers to the use of this cross-validation methodology to perform data-adaptive estimation of nuisance parameters as the ``super learner'', and subsequent work has advocated for its use for nuisance estimation within a framework for semiparametric inference known as \emph{targeted maximum likelihood estimation} (TMLE). TMLE \citep{scharfstein1999theory,van2006targeted,zheng2010asymptotic,van2011targeted} and its more general variant, targeted minimum loss-based estimation, are general frameworks for semiparametric inference which---like our framework---employ empirical risk minimization in the presence of nuisance parameters. TMLE estimates the target parameter by repeatedly minimizing an empirical risk (typically the negative log-likelihood) in order to refine an initial estimate. This approach easily incorporates constraints, and can be used in tandem with the super learning technique. The analysis leverages orthogonality, and is also agnostic to how the nuisance estimates are obtained. However, the main focus of this framework is on the classical semiparametric inference objective; minimizing a population risk is not the end goal as it is here.

\paragraph{Specific instances of risk minimization with nuisance parameters}
A number of prior works employ empirical risk minimization with nuisance parameters for specific statistical models \citep{rubin2005general,rubin2007doubly,diaz2013targeted,van2014targeted,kennedy2017non,kennedy2019robust,kunzel2017meta}. These results allow for general choices for the target class and nuisance class (typically subject to Donsker conditions, or with guarantees in the vein of \cite{van2003unified}), and the main focus is semiparametric inference rather than excess risk guarantees.

\paragraph{Nonparametric target parameters}
Outside of the risk minimization-based approaches above and the examples in the prequel \citep{athey2016generalized,nie2017quasi,athey2017efficient,zhou2018offline,oprescu2018orthogonal,friedberg2018,chernozhukov2017orthogonal,chernozhukov2018plugin}, a number of other results also consider inference for nonparametric target parameters in the presence of nuisance parameters. In \cite{van2006estimating}, the target is a Lipschitz function over $[0,\infty)$ (the marginal survival function) and an estimation rate of $n^{-2/3}$ is given. \cite{wang2010nonparametric} consider estimation of smooth nonparametric target parameters in the presence of missing outcomes, and give algorithms based on kernel smoothing. \cite{robins2001comment,robins2008higher} consider settings where the target parameter is scalar, but the optimal rate is nonparametric due to the presence of complex nuisance parameters.

\paragraph{Sample splitting}
While our use of sample splitting is directly inspired by recent use of the technique in double/debiased machine learning \citep{chernozhukov2016locally,chernozhukov2016double}, the basic technique dates back to the early days of semiparametric inference and it has found use in many other works to remove Donsker conditions for estimation in the presence of nuisance parameters \citep{bickel1982adaptive,klaassen1987consistent,van2000asymptotic,robins2008higher,zheng2010asymptotic}.

\paragraph{Limitations}
Our results are quite general, but there are some applications that go beyond the scope of our framework. For example, while we consider only plug-in estimation for the nuisance \params, several works attain refined results by using specialized estimators \cite{van2006targeted,hirshberg2017augmented,chernozhukov2018double,ning2018robust}. While our focus is on methods based on loss minimization, some problems such as nonparametric instrumental variables \citep{newey2003instrumental,hall2005nonparametric,blundell2007semi,chen2009efficient,chen2012estimation,chen2015sieve,chen2018optimal} are more naturally posed in terms of conditional moment restrictions.\preprint{\footnote{In fact, nonparametric IV can be cast as a special case of the setup in \pref{eq:excess_risk}, but we do not know of any estimators for this problem that satisfy the conditions required to apply our main theorems.}}\journal{ In fact, nonparametric IV can be cast as a special case of the setup in \pref{eq:excess_risk}, but we do not know of any estimators for this problem that satisfy the conditions required to apply our main theorems.}

\begin{revone}
  Another direction where our results leave room for future improvement concerns the reliance on Neyman orthogonality. While Neyman orthogonality is a fairly general condition which allows one to handle many nuisance parameters simultaneously, many problems admit additional structure which can lead to more refined guarantees. For example, in the context of treatment effect estimation, subsequent work of \cite{kennedy2020optimal} uses the doubly robust structure of the problem to give guarantees that accommodate the case where different nuisance components (regression functions and propensity scores) are estimated at different rates.
\end{revone}

\subsection{Organization}

\begin{revone}
The first part of this paper presents our main results. \pref{sec:setup} contains technical preliminaries and definitions, and \pref{sec:orthogonal} presents our main theorems concerning the excess risk of \pref{alg:sample_splitting}. \pref{sec:orthogonal} also includes basic examples in which we apply these theorems to treatment effect estimation and policy learning.%

Our main results are stated at a high level of generality, and consider generic estimation algorithms for the target and nuisance parameters. In the second part of the paper, we make matters more concrete and focus on specific algorithms.  We leverage the main theorems to give explicit bounds based on the statistical capacity of the target and nuisance class. In particular:
\begin{itemize}
\item \pref{sec:erm} (\emph{Plug-in Empirical Risk Minimization}) provides explicit bounds for plug-in empirical risk minimization as the second stage of the meta-algorithm.
\item \pref{sec:oracle} (\emph{Sufficient Conditions for Oracle Rates}) considers aggregation based algorithms that go beyond empirical risk minimization, and gives sufficient conditions (as a function of the statistical capacity of the nuisance and target class) under which \pref{alg:sample_splitting} can be configured such that oracle excess risk bounds are achieved. 
\end{itemize}
We conclude with discussion in \pref{sec:discussion}. Additional results are deferred to the appendix, which is split into three parts. \pref{part:experiments} contains experiments, and \pref{part:additional} contains supplementary theoretical results, including sufficient conditions for Neyman orthogonality, applications of our main results to specific settings, and further guarantees for specific algorithms and function classes. \pref{part:proofs} contains proofs for our main results.
\end{revone}

\section{Framework: Statistical Learning with a Nuisance Component}
\label{sec:setup}

We work in a learning setting in which observations belong to an abstract set $\Varall$. We receive a sample set $S\ldef{}z_1,\ldots,z_n$ where each $z_t$ is drawn i.i.d. from an unknown distribution $\dist$ over $\cZ$. Define variable subsets $\Vartwo\subseteq{}\Varone\subset\Varall$; the restriction $\Vartwo\subseteq{}\Varone$ is not strictly necessary but simplifies notation. We focus on learning \params that come from a \emph{target \param class} $\target:\Vartwo\to\vectwo$ and \emph{nuisance \param class} $\nuisance:\Varone\to{}\vecone$, where $\vecone$ and $\vectwo$ are finite dimensional vector spaces of dimension $\dimone$ and $\dimtwo$ respectively, equipped with norms $\normone{\cdot}$ and $\normtwo{\cdot}$. Note that since our results are fully non-asymptotic, the classes $\target$ and $\nuisance$ may be taken to grow with $n$.%

Given an example $\varall_{t}\in{}\Varall$, we write $\varone_t\in\Varone$ and $\vartwo_t\in\Vartwo$ to denote the subsets of $\varall_t$ that act as arguments to the nuisance and target parameters respectively. For example, we may write $g(\varone_t)$ for $g\in\nuisance$ or $\theta(\vartwo_t)$ for $\theta\in\target$. We assume that the function spaces $\target$ and $\nuisance$ are equipped with \vsedit{pre-norms $\nrm*{\cdot}_{\target}$ and $\nrm*{\cdot}_{\nuisance}$ respectively, which need to satisfy non-negativity and $\|0\|=0$, but not necessarily the triangle inequality nor absolute homogeneity.} In our applications, both pre-norms take the form 
$\nrm*{f}_{L_p(\cV,\cD)}=\prn*{\En_{z\sim{}\cD}\nrm*{f(z)}_{\cV}^{p}}^{1/p}$
for functions $f:\cZ\to\cV$, where $\cV\in\crl*{\vecone,\vectwo}$.

We measure performance of the target predictor through the real-valued \emph{population loss functional} $\poprisk(\theta,g)$, which maps a target predictor $\theta$ and nuisance predictor $g$ to a loss. The subscript $\cD$ in $\poprisk$ denotes that the functional depends on the underlying distribution $\cD$. For all of our applications, $\poprisk$ has the following structure, in line the classical statistical learning setting: First define a pointwise loss function $\ls(\theta,g\midsem{}z)$, then define $\poprisk(\theta,g) \ldef \En_{z\sim{}\cD}\brk{\ls(\theta,g\midsem{}z)}$. Our general framework does not explicitly assume this structure, however.

Let $g_0\in\nuisance$ be the unknown true value for the nuisance parameter. Given the samples $S$, and without knowledge of $g_0$, we aim to produce a \emph{target predictor} $\esttwo$ that minimizes the \emph{excess risk} evaluated at $\gtone$
\begin{equation}
\label{eq:excess_risk}
\poprisk(\esttwo,g_0) - \inf_{\theta\in\target}\poprisk(\theta,g_0).
\end{equation}
As discussed in the introduction, we will always produce such a predictor via the sample splitting meta-algorithm (\pref{alg:sample_splitting}), which makes uses of a nuisance predictor $\estone$.

When the infimum in the excess risk is obtained, we use $\besttwo$ to denote the corresponding minimizer, in which case the excess risk can be written as
\[
\poprisk(\esttwo,\gtone) - \poprisk(\besttwo,\gtone).
\]
We occasionally use the notation $\gttwo$ to refer to a particular target parameter with respect to which the second stage satisfies a \emph{first-order condition}, e.g. $D_{\theta}L_{\cD}(\gttwo,\gtone)[\theta-\gttwo]=0\,\,\forall{}\theta\in\target$. If $\gttwo\in\target$ and the population risk is convex, then we can take $\besttwo=\gttwo$ without loss of generality, but we do not assume this, and in general we do not assume existence of a such a parameter $\gttwo$.
\paragraph{Notation} 
We let $\tri*{\cdot,\cdot}$ denote the standard inner product. $\nrm*{\cdot}_{p}$ will denote the $\ls_{p}$ norm over $\bbR^{d}$ and $\nrm*{\cdot}_{\sigma}$ will denote the spectral norm over $\bbR^{d_1\times{}d_2}$.

Unless otherwise stated, the expectation $\En\brk*{\cdot}$, probability $\bbP(\cdot)$, and variance $\var(\cdot)$ operators will be taken with respect to the underlying distribution $\cD$. We define empirical analogues $\En_{n}\brk{\cdot}$, $\bbP_n(\cdot)$, and $\varn(\cdot)$ with respect to a sample set $z_1,\ldots,z_n$, whose value will be clear from context. For a vector space $\cV$ with norm $\nrm*{\cdot}_{\cV}$ and function $f:\cZ\to\cV$, we define $\nrm*{f}_{L_p(\cV,\cD)}=\prn*{\En_{z\sim{}\cD}\nrm*{f(z)}_{\cV}^{p}}^{1/p}$ for $p\in(0,\infty)$, with $L_p(\ell_q,\cD)$ referring to the special case where $\nrm*{\cdot}_{\cV}=\nrm*{\cdot}_{q}$. For a sample set $S=z_{1:n}$, we define the empirical variant $\nrm*{f}_{L_p(\cV,S)}=\prn*{\frac{1}{n}\sum_{i=1}^{n}\nrm*{f(z_i)}_{\cV}^{p}}^{1/p}$. When $\cV=\bbR$, we drop the first argument and write $L_p(\cD)$ and $L_p(S)$. We extend these definitions to $p=\infty$ in the natural way.

For a subset $\cX$ of a vector space, $\conv(\cX)$ will denote the convex hull. For an element $x\in\cX$, we define the star hull via 
\begin{equation}
\starhull(\cX,x) = \crl*{t\cdot{}x+(1-t)\cdot{}x'\mid{}x'\in\cX,t\in\brk*{0,1}},
\end{equation}
and adopt the shorthand $\starhull(\cX) \ldef \starhull(\cX,0)$.

Given functions $f,g:\cX\to [0,\infty)$ where $\cX$ is any set,
we use non-asymptotic big-$O$ notation, writing $f=O(g)$ if there exists a numerical
constant $c < \infty$ such that $f(x)\leq{}c\cdot{}g(x)$ for all $x\in\cX$
and $f=\Omega(g)$ if there is a numerical constant $c>0$ such that
$f(x)\geq{}c\cdot{}g(x)$. We write $f = \Ot(g)$ as shorthand for $f=O(g
\max\{1,\mathrm{polylog}(g)\})$.

\section{Orthogonal Statistical Learning}
\label{sec:orthogonal}

In this section we present our main results on orthogonal statistical learning, which state that under certain conditions on the loss function, the error due to estimation of the nuisance component $\gtone$ has higher-order impact on the prediction error of the target component. The results in this section, which form the basis for all subsequent results, are \emph{algorithm-independent}, and only involve assumptions on properties of the population risk $\poprisk$. To emphasize the high level of generality, the results in this section invoke the learning algorithms in \pref{alg:sample_splitting} only through ``rate'' functions $\Rate(\nuisance,\ldots)$ and $\Rate(\target,\ldots)$ which respectively bound the estimation error of the first stage and the excess risk of the second stage.
\begin{definition}[Algorithms and Rates]
\label{def:algorithms}
The first and second stage algorithms and corresponding rate functions are defined as follows:
\journal{\begin{enumerate}}
  \preprint{\begin{enumerate}[label=\textnormal{\textbf{\alph*)}}]}
\item \emph{Nuisance algorithm and rate.} The first stage learning algorithm $\algone{S}$, when given a sample set $S$ from distribution $\cD$, outputs a predictor $\estone$ for which
\[
\nrm*{\estone-\gtone}_{\nuisance} \leq{} \Rate(\nuisance, S, \delta)
\]
with probability at least $1-\delta$.
\item \emph{Target algorithm and rate.}  The second stage learning algorithm $\mathrm{Alg}(\target,S\midsem{}g)$, when given sample set $S$ from distribution $\cD$ and any $g\in\cG$ outputs a predictor $\esttwo$ for which
\[
\poprisk(\esttwo,g) - L_{\cD}(\besttwo, g) \leq{} \Rate(\target,  S, \delta\midsem{}\esttwo,g)
\]
with probability at least $1-\delta$. 
\end{enumerate}
\revoneedit{We let $\targetemp$ denote any function class (fixed a-priori) for which $\esttwo,\besttwo\in\targetemp$ almost surely.}
We denote worst-case variants of the rates by $\Rate(\nuisance, n, \delta) \ldef \sup_{S:\abs*{S}=n}\Rate(\nuisance, S, \delta)$ and $\Rate(\target, n, \delta\midsem{}\esttwo,g) \ldef \sup_{S:\abs*{S}=n}\Rate(\target, S, \delta\midsem{}\esttwo,g)$.				
\end{definition}
Observe that if one naively applies the algorithm for the target class
using the nuisance predictor $\estone$ as a plug-in estimate for
$\gtone$, the rate stated in \pref{def:algorithms} will only yield a ``pseudo''-excess risk bound of the form
\begin{equation}
\label{eq:excess_risk_wrong}
\poprisk(\esttwo,\estone) - \poprisk(\besttwo,\estone) \leq{} \Rate(\target,  S, \delta\midsem{}\esttwo,\estone).
\end{equation}
This clearly does not match the desired bound \pref{eq:excess_risk},
which concerns the excess risk evaluated at $\gtone$ rather than the plug-in estimate $\estone$. The bulk of our work is to show that orthogonality can be used to correct this mismatch.

\pref{def:algorithms} and subsequent results are stated in terms of a class $\targetemp$ containing $\esttwo$, which in general may have $\targetemp\neq{}\target$. This extra level of generality serves two purposes. First, it allows for refined analysis in the case where $\targetemp\subset\target$, which is encountered when using algorithms based on regularization that do not impose hard constraints on, e.g., the norm of the estimator. Second, it permits the use of \emph{improper prediction}, i.e. $\targetemp\supset\target$, which in some cases is required to obtain optimal rates for misspecified models \citep{audibert2008progressive,foster2018logistic}.

Recall that for a sample set $S=z_1,\ldots,z_n$, the empirical loss is
defined via $L_{S}(\theta,g) =
\frac{1}{n}\sum_{t=1}^{n}\ls(\theta,g\midsem{}z_t)$. Many classical
results from statistical learning can be applied to the double machine
learning setting by minimizing the empirical loss with plug-in
estimates for $g_0$, and we can simply cite these results to provide
examples of $\Rate$ for the target class $\target$. Note however that
this structure is not assumed by \pref{def:algorithms}, and we indeed
consider algorithms that do not have this form
\revoneedit{(cf. \pref{sec:oracle})}.  \revoneedit{Let us highlight that we allow the function
  $\Rate(\target,S,\delta;\esttwo,\estone)$ to depend on both the
  target estimator $\esttwo$ and the nuisance estimator $\estone$;
  this extra level of generality is useful for deriving
  algorithm-specific guarantees (cf. \pref{sec:erm}).}

\paragraph{Fast rates and slow rates}
The rates presented in this section fall into two distinct categories, which we distinguish by referring to them as either \emph{fast rates} or \emph{slow rates}. The meaning of the word ``fast'' or ``slow'' here is two-fold: First, for fast rates, our assumptions on the loss imply that when the target class $\target$ is not too large (e.g. a parametric or VC-subgraph class) prediction error rates of order $O(1/n)$ are possible in the absence of nuisance parameters. For our slow rate results, the best prediction error rate that can be achieved is $O(1/\sqrt{n})$, even for small classes. This distinction is consistent with the usage of the term fast rate in statistical learning \citep{bousquet2004introduction,bartlett2005local,srebro2010smoothness}, and we will see concrete examples of such rates for specific classes in later sections (\pref{sec:erm}, \pref{sec:algsandrates}). 

The second meaning of ``fast'' versus ``slow'' refers to the first stage: When estimation error for the nuisance is of order $\veps$, the impact on the second stage in our fast rate results is of order $\veps^{4}$, while for our slow rate results the impact is of order $\veps^{2}$. The fast rate regime---particularly, the $\veps^{4}$-type dependence on the nuisance error---will be the more familiar of the two for readers accustomed to semiparametric inference. While fast rates might at first seem to strictly improve over slow rates, these results require stronger assumptions on the loss. Our results in \pref{sec:oracle} show that which setting is more favorable will in general depend on the precise relationship between the complexity of the target \param class and the nuisance parameter class.

\subsection{Fast Rates Under Strong Convexity}
\label{ssec:orthogonal_strongly_convex}

We first present general conditions under which the sample splitting meta-algorithm obtains so-called fast rates for prediction. Our assumptions are stated in terms of \emph{directional derivatives} with respect to the target and nuisance \params.
\begin{definition}[Directional Derivative]
Let $\cF$ be a vector space of functions. For a functional $F:\cF\to\bbR$, we define the derivative operator $D_{f}F(f)[h] = \frac{d}{dt}F(f+th)\Big|_{t=0}$
for a pair of functions $f,h\in\cF$. Likewise, we define
$D_{f}^{k}F(f)[h_1,\ldots,h_k] = \frac{\partial^{k}}{\partial{}t_1\ldots\partial{}t_k}
F(f + t_1{}h_1+\ldots+t_k{}h_k)\Big|_{t_1=\cdots=t_k=0}$. When considering a functional in two arguments, e.g. $\poprisk(\theta,g)$, we write $D_{\theta}\poprisk(\theta,g)$ and $D_{g}\poprisk(\theta,g)$ to make the argument with respect to which the derivative is taken explicit.
\end{definition}
\revoneedit{To present our results, we fix a representative $\besttwo\in\argmin_{\theta\in\target}\poprisk(\theta,\gtone)$. In general, the minimizer may not be unique---indeed, by focusing on excess risk, we can provide guarantees even when parameter recovery is impossible. Thus, we assume that a single fixed representative $\besttwo$ is used throughout all the assumptions stated in this subsection.} 

Our first assumption is the starting point for this work, and asserts that the population loss is \emph{orthogonal} in the sense that the certain pathwise derivatives vanish.
\begin{assumption}[Orthogonal Loss]
\label{ass:orthogonal}
The population risk $\poprisk$ is \emph{Neyman orthogonal}:
\begin{equation}
\label{eq:orthogonal}
D_{g}D_{\theta}\poprisk(\besttwo,\gtone)[\theta-\besttwo,g-\gtone] = 0\quad\forall{}\theta\in\targetemp,\forall{}g\in\nuisance.
\end{equation}
\end{assumption}
\revoneedit{Note that while \pref{ass:orthogonal} is stated in terms of the risk
$\poprisk$, it is typically satisfied by choosing a particular
point-wise loss function whose expectation equals the risk; examples
are given in the sequel. The construction of such a point-wise loss is typically achieved by adding a de-biasing correction term to some ``initial'' loss, whose minimizer is the target quantity (see Appendix~\ref{sec:orthogonal_loss} for details on automated orthogonal loss construction). The de-biasing correction reduces the impact of errors in the nuisance function estimates on the gradient of the loss, and is related to the notion of an efficient influence function in semi-parametric inference (however, our estimand is not necessarily pathwise differentiable, and hence violates the basic premise of most semi-parametric inference theory).}

Beyond orthogonality, our main theorem requires three additional
assumptions, all of which are fairly standard in the context of fast
rates for statistical learning. We require a first-order optimality
condition for the target class, and require that the population risk
is both smooth and strongly convex with respect to the target parameter.

\begin{assumption}[First Order Optimality]
\label{ass:well_specified}
The minimizer for the population risk satisfies the first-order optimality condition:
\begin{equation}
\label{eq:well_specified}
D_{\theta}\poprisk(\besttwo,\gtone)[\theta-\besttwo]\geq 0 \quad{}\forall{}\theta\in\starhull(\targetemp,\besttwo).
\end{equation}
\end{assumption}
\begin{remark}
The first-order condition is typically satisfied for models that are \emph{well-specified}, meaning that there is some variable in $\varall$ that identifies the target \param $\gttwo$. More generally, it suffices to ``almost'' satisfy the first-order condition, i.e. to replace \pref{eq:well_specified} by the condition
\begin{equation}
D_{\theta}\poprisk(\besttwo,\gtone)[\theta-\besttwo]
 \geq{} - o_n\prn*{\Rate(\target, n, \delta\midsem{}\esttwo,\estone)}.
\end{equation}
The first-order condition is also satisfied whenever $\targetemp$ is star-shaped around $\besttwo$, i.e. $\starhull(\targetemp,\besttwo)\subseteq\targetemp$.
\end{remark}
\begin{assumption}[Higher-Order Smoothness]
\label{ass:smooth_loss}
There exist constants $\beta_1$ and $\beta_2$ such that the following derivative bounds hold:
\journal{\begin{enumerate}}
\preprint{\begin{enumerate}[label=\textnormal{\textbf{\alph*)}}]}
\item \emph{Second-order smoothness with respect to target.} For all $\theta\in\targetemp$ and all $\bar{\theta}\in\starhull(\targetemp,\besttwo)$:
\begin{equation}
\label{eq:smooth1}
D^{2}_{\theta}\poprisk(\bar{\theta},g_0)[\theta-\besttwo,\theta-\besttwo] \leq{}\beta_1\cdot \nrm*{\theta-\besttwo}_{\target}^{2}.
\end{equation}
\item \emph{Higher-order smoothness.} \revoneedit{There exists $r\in[0,1)$ such
  that for all $\theta\in\starhull(\targetemp,\besttwo),\;g\in\nuisance$, and $\bar{g}\in\starhull(\nuisance,\gtone)$:
\begin{align}
\label{eq:smooth2}
\abs*{D^{2}_{g}D_{\theta}\poprisk(\besttwo,\bar{g})[\theta-\besttwo,g-\gtone,g-\gtone]} 
\leq{}~& \beta_2\cdot \nrm*{\theta-\besttwo}^{1-r}_{\target} \cdot \nrm*{g-\gtone}_{\nuisance}^2.
\end{align}
}
\end{enumerate}
\end{assumption}

\begin{assumption}[Strong Convexity]
\label{ass:strong_convex_loss}
The population risk is strongly convex with respect to the target parameter: \revoneedit{There exist constants $\lambda,\kappa>0$ such that for all $\theta\in\targetemp$ and $g\in\nuisance$,
\begin{equation}
\label{eq:strongly_convex}
D^{2}_{\theta}\poprisk(\bar{\theta}, g)[\theta-\besttwo,\theta-\besttwo]\geq{}\lambda\nrm*{\theta-\besttwo}_{\target}^{2}  - \kappa \nrm*{g - \gtone}_{\nuisance}^{\frac{4}{1+r}} \quad{}\forall{}\bar{\theta}\in\starhull(\targetemp,\besttwo),
\end{equation}
where $r\in[0,1)$ is as in \pref{ass:smooth_loss}.}
\end{assumption}

\pref{ass:smooth_loss} and \pref{ass:strong_convex_loss} are easily
satisfied whenever the loss is obtained by applying a square loss or
another smooth, strongly convex loss pointwise to the prediction of the
target class; concrete examples are given in
\pref{sec:sufficient}. \revoneedit{For most of our results, we apply
  these assumptions with $r=0$, but the case $r>0$ will prove useful
  for certain settings in which strong $L_{\infty}$-type estimation guarantees for the
  target parameter are available (cf. \pref{ex:r_learner}). In general,
  \pref{ass:orthogonal,ass:well_specified,ass:strong_convex_loss,ass:smooth_loss}
  do not imply that $\besttwo$ is uniquely identified unless
  $\nrm{\cdot}_{\target}$ is a norm. However, if the assumptions are
  satisfied by two parameters $\besttwo\neq\tilde{\theta}^{\star}$, we
  must have $\nrm{\besttwo-\tilde{\theta}^{\star}}_{\target}=0$,
  meaning convergence in the sense that $\nrm{\esttwo-\besttwo}_{\target}\to{}0$ is equivalent for both representatives.}\loose

We now state our main theorem concerning fast rates.
\begin{theorem}
\label{thm:generic_strongly_convex}
Suppose there exists $\besttwo\in\argmin_{\theta\in\target}\poprisk(\theta,\gtone)$ such that \pref{ass:orthogonal,ass:well_specified,ass:strong_convex_loss,ass:smooth_loss}, are satisfied. Then the sample splitting meta-algorithm (\pref{alg:sample_splitting}) produces a \param $\esttwo$ such that with probability at least $1-\delta$,
\begin{equation}
\label{eq:generic_strongly_convex_l2}
\nrm[\big]{\esttwo-\besttwo}_{\target}^{2} \leq{} C_1\cdot\Rate(\target,\sampletwo,\delta/2\midsem{}\esttwo,\estone) +
C_2
\cdot{}\prn*{
\Rate(\nuisance,\sampleone,\delta/2)
}^{\frac{4}{1+r}},
\end{equation}
where $C_1\leq\frac{4}{\lambda}$ and $C_2\leq2\prn[\Big]{\prn[\big]{\frac{\beta_2}{\lambda}}^{\frac{2}{1+r}} + \frac{\kappa}{\lambda}}$. In addition,
\begin{equation}
\label{eq:generic_strongly_convex_excess_risk}
\begin{aligned}
L_{\cD}(\esttwo,\gtone) - L_{\cD}(\besttwo, \gtone)\leq{} \frac{\beta_1\, C_1}{2}\cdot\Rate(\target,\sampletwo,\delta/2\midsem{}\esttwo,\estone) +
\frac{\beta_1\, C_2}{2}
\cdot\prn*{
\Rate(\nuisance,\sampleone,\delta/2)
}^{\frac{4}{1+r}}.
\end{aligned}
\end{equation}
\end{theorem}
\revoneedit{
  The majority of the results in this paper concern the special case
  in which $r=0$. In this case, since
$\Rate(\nuisance,\sampleone,\delta/2)\propto\nrm{\estone-\gtone}_{\nuisance}$, 
\pref{thm:generic_strongly_convex} shows that for
\pref{alg:sample_splitting}, the impact of the unknown nuisance
parameter on the prediction is of second-order,
i.e.
\[
\nrm[\big]{\esttwo-\besttwo}_{\target}^{2} \text{ and }  L_{\cD}(\esttwo,\gtone) - L_{\cD}(\besttwo, \gtone) \approxleq{}
\Rate(\target,\sampletwo,\delta/2\midsem{}\esttwo,\estone)
+ \nrm[\big]{\estone-\gtone}_{\nuisance}^{4}
\]
}This implies that if the optimal rate without nuisance
parameters is of order $O(n^{-1})$, it suffices to take
$\nrm{\estone-\gtone}_{\nuisance}^{2}=\littleo(n^{-1/2})$ to achieve
the oracle rate.
\begin{revone}
\begin{proof}[\pfref{thm:generic_strongly_convex}]
We prove \pref{thm:generic_strongly_convex} by performing a Taylor
expansion to relate the excess risk at $\estone$ to the excess risk at
$\gtone$, employing orthogonality and self-bounding arguments to control
cross terms. We abbreviate
$R_{\Theta}\ldef\Rate(\target,\sampletwo,\delta/2\midsem{}\esttwo,\estone)$ and
$R_{\cG}\ldef\Rate(\nuisance,\sampleone,\delta/2)$ to simplify notation.

By a second-order Taylor expansion on the risk at $\estone$, there exists $\bar{\theta}\in\starhull(\targetemp, \besttwo)$ such that
\[
\frac{1}{2}\cdot{}D^{2}_{\theta}\poprisk(\bar{\theta},\estone)[\esttwo-\besttwo,\esttwo-\besttwo]
= \poprisk(\esttwo,\estone) - \poprisk(\besttwo,\estone) - D_{\theta}L_{\cD}(\besttwo,\estone)[\esttwo-\besttwo].
\]
Next, using the strong convexity assumption (\pref{ass:strong_convex_loss}), we have
\begin{align*}
D^{2}_{\theta}\poprisk(\bar{\theta},\estone)[\esttwo-\besttwo,\esttwo-\besttwo] 
\geq~& \lambda \cdot{}\nrm[\big]{\esttwo-\besttwo}_{\target}^{2} - \kappa \cdot \nrm[\big]{\estone - \gtone}_{\nuisance}^\frac{4}{1+r} = \lambda \cdot{}\nrm[\big]{\esttwo-\besttwo}_{\target}^{2} - \kappa\cdot{}R_{\cG}^\frac{4}{1+r}.
\end{align*}
Combining these statements, we conclude that
\[
\frac{\lambda}{2}\cdot{}\nrm[\big]{\esttwo-\besttwo}_{\target}^{2} \leq \poprisk(\esttwo,\estone) - \poprisk(\besttwo,\estone)  + \frac{\kappa}{2}\cdot R_{\cG}^{\frac{4}{1+r}} - D_{\theta}L_{\cD}(\besttwo,\estone)[\esttwo-\besttwo].
\]
Using the assumed rate for $\esttwo$ (\pref{def:algorithms}), this implies the inequality
\begin{equation}
\label{eq:generic_strongly_convex_lower}
\frac{\lambda}{2}\cdot{}\nrm[\big]{\esttwo-\besttwo}_{\target}^{2}
\leq{} R_{\Theta} + \frac{\kappa}{2}\cdot R_{\cG}^{\frac{4}{1+r}} - D_{\theta}L_{\cD}(\besttwo,\estone)[\esttwo-\besttwo].
\end{equation}
We now apply a second-order Taylor expansion (using the assumed derivative continuity from \pref{ass:smooth_loss}), which implies that there exists $\bar{g}\in\starhull(\nuisance,\gtone)$ such that
\begin{align*}
- D_{\theta}\poprisk(\besttwo,\estone)[\esttwo-\besttwo] 
&=
- D_{\theta}\poprisk(\besttwo,\gtone)[\esttwo-\besttwo] - D_{g}D_{\theta}\poprisk(\besttwo,\gtone)[\esttwo-\besttwo,\estone-\gtone]
\journal{\\&~~~~}-
  \frac{1}{2}\cdot{}D^{2}_{g}D_{\theta}\poprisk(\besttwo,\bar{g})[\esttwo-\besttwo,\estone-\gtone,\estone-\gtone].
  \end{align*}
  Using orthogonality of the loss (\pref{ass:orthogonal}), this is equal to
  \begin{align*}
- D_{\theta}\poprisk(\besttwo,\gtone)[\esttwo-\besttwo] 
- \frac{1}{2}\cdot{}D^{2}_{g}D_{\theta}\poprisk(\besttwo,\bar{g})[\esttwo-\besttwo,\estone-\gtone,\estone-\gtone].
  \end{align*}
We use the second order smoothness assumed in \pref{ass:smooth_loss}
to upper bound by
\begin{align*}
\leq{}
- D_{\theta}\poprisk(\besttwo,\gtone)[\esttwo-\besttwo] 
+ \frac{\beta_2}{2}\cdot{}\nrm[\big]{\esttwo-\besttwo}_{\target}\cdot\nrm[\big]{\estone-\gtone}_{\nuisance}^{2}.
\end{align*}
Invoking Young's inequality and using that $r\in[0,1)$, we have that
for any constant $\eta>0$, this is at most
\begin{align*}
&\leq{}
D_{\theta}\poprisk(\besttwo,\gtone)[\esttwo-\besttwo]
                                                                                                                + \frac{\beta_2 \eta}{4}\cdot{}\nrm[\big]{\esttwo-\besttwo}_{\target}^{2}
+ \frac{\beta_2}{2\eta^{\frac{1-r}{1+r}}}\cdot{}\nrm[\big]{\estone-\gtone}_{\nuisance}^{\frac{4}{1+r}}.
\end{align*}
Lastly, we use the assumed rate for $\estone$
(\pref{def:algorithms}) to bound by
\begin{align*}
  \leq{}
 - D_{\theta}\poprisk(\besttwo,\gtone)[\esttwo-\besttwo] 
 + \frac{\beta_2
    \eta}{4}\cdot{}\nrm[\big]{\esttwo-\besttwo}_{\target}^{2} + \frac{\beta_2}{2\eta^{\frac{1-r}{1+r}}}\cdot{}
R_{\cG}^{\frac{4}{1+r}}.
\end{align*}
Choosing $\eta = \frac{\lambda}{\beta_2}$, combining this string of inequalities with \pref{eq:generic_strongly_convex_lower}, and rearranging, we have:
\begin{equation}
\label{eq:strongly_convex_quadratic}
\begin{aligned}
\nrm[\big]{\esttwo-\besttwo}_{\target}^{2} 
&\leq \frac{4}{\lambda}\prn*{- D_{\theta}\poprisk(\besttwo,\gtone)[\esttwo-\besttwo] 
+ R_{\Theta}} +
2\prn[\bigg]{\prn[\bigg]{\frac{\beta_2}{\lambda}}^{\frac{2}{1+r}} + \frac{\kappa}{\lambda}}\cdot R_{\cG}^{\frac{4}{1+r}}.
\end{aligned}
\end{equation}
\pref{ass:well_specified} implies that $D_{\theta}L_{\cD}(\besttwo,\gtone)[\esttwo-\besttwo]\geq{}0$, which establishes the inequality \pref{eq:generic_strongly_convex_l2}. 

To derive the inequality \pref{eq:generic_strongly_convex_excess_risk}, we use another Taylor expansion, which implies that there exists $\bar{\theta}\in\starhull(\targetemp, \besttwo)$ such that
\begin{align}
\poprisk(\esttwo,\gtone) - \poprisk(\besttwo,\gtone)
&= D_{\theta}\poprisk(\besttwo,\gtone)[\esttwo-\besttwo] 
+ \frac{1}{2}\cdot{}D^{2}_{\theta}\poprisk(\bar{\theta},\gtone)[\esttwo-\besttwo,\esttwo-\besttwo].\notag
\intertext{Using the smoothness bound from \pref{ass:smooth_loss},
                                                       we upper bound
                                                       this by}
&\leq{} D_{\theta}\poprisk(\besttwo,\gtone)[\esttwo-\besttwo] 
+ \frac{\beta_1}{2}\cdot\nrm[\big]{\esttwo-\besttwo}_{\target}^{2}.\label{eq:generic_strongly_convex_risk}
\end{align}
We combine \pref{eq:strongly_convex_quadratic} with \pref{eq:generic_strongly_convex_risk} to conclude that $\poprisk(\esttwo,\gtone) - \poprisk(\besttwo,\gtone)$ is bounded by
\begin{align*}
\frac{2\beta_1}{\lambda} \cdot R_{\Theta} +
  \beta_1\prn[\bigg]{\prn*{\frac{\beta_2}{\lambda}}^{\frac{2}{1+r}} + \frac{\kappa}{\lambda}}\cdot R_{\cG}^{\frac{4}{1+r}}
- \prn[\bigg]{\frac{2\beta_1}{\lambda}-1}\cdot{}D_{\theta}L_{\cD}(\besttwo,\gtone)[\esttwo-\besttwo].
\end{align*}
The result follows by again using that $D_{\theta}L_{\cD}(\besttwo,\gtone)[\esttwo-\besttwo]\geq{}0$, along with the fact that $\beta_1/\lambda\geq{}1$ without loss of generality.
\end{proof}
\end{revone}

There is one issue not addressed by \pref{thm:generic_strongly_convex}: If the nuisance parameter $\gtone$ were known, the rate for the target parameters would be $\Rate(\target,\ldots\midsem{}\esttwo,\gtone)$, but the bound in \pref{eq:generic_strongly_convex_excess_risk} scales instead with $\Rate(\target,\ldots\midsem{}\esttwo,\estone)$. This is addressed in \pref{sec:erm,sec:oracle}, where---building on \pref{thm:generic_strongly_convex}---we show that for many standard algorithms, the cost to relate these quantities grows as $\prn*{\Rate(\nuisance,\sampleone,\delta/2)}^{4}$, and can be absorbed into the second term in \pref{eq:generic_strongly_convex_l2} or \pref{eq:generic_strongly_convex_excess_risk}.\loose

\subsection{Beyond Strong Convexity: Slow Rates}
\label{ssec:orthogonal_slow_rate}
The strong convexity assumption used by
\pref{thm:generic_strongly_convex} requires curvature only in the
prediction space, not the parameter space. This is considerably weaker
than what is assumed in prior works on double machine learning (e.g.,
\cite{chernozhukov2018plugin}), and is a major advantage of analyzing
prediction error rather than parameter recovery. Nonetheless, in some
situations even assuming strong convexity on predictions may be
unrealistic. A second advantage of studying prediction is that, while
parameter recovery is not possible in this case, it is still possible
to achieve low prediction error, albeit with slower rates than in the
strongly convex case. We now give guarantees under which these
(slower) oracle rates for prediction error can be obtained in the
presence of nuisance \params using
\pref{alg:sample_splitting}. \revoneedit{As in the prequel, we fix a
  representative
  $\besttwo\in\argmin_{\theta\in\target}\poprisk(\theta,\gtone)$
  throughout this subsection.}

The key technical assumption for next result is \emph{universal orthogonality}, which informally states that the loss is not simply orthogonal around $\besttwo$, but rather is orthogonal for all $\theta\in\target$.

\begin{assumption}[Universal Orthogonality]
\label{ass:universal_orthogonality}
For all $\bar{\theta}\in\starhull(\targetemp,\besttwo) + \starhull(\targetemp-\besttwo,0)$,
\[
D_{g}D_{\theta}L_{\cD}(\bar{\theta},\gtone)[\theta-\besttwo,g-\gtone]=0\quad\forall{}g\in\nuisance,\;\;\revoneedit{\theta\in\targetemp}.
\]
\end{assumption}
\revoneedit{Universal orthogonality is a strengthening of \pref{ass:orthogonal}, which requires that the cross derivative at $\gtone$ vanishes for all $\bar{\theta}\in\starhull(\targetemp,\besttwo)$, rather than only at $\besttwo$.} It is satisfied for examples including treatment effect estimation (\pref{sec:treatment_effect}) and policy learning (\pref{sec:policy_learning_body}), and is used implicitly in previous work in these settings \citep{nie2017quasi,athey2017efficient}. Beyond orthogonality, we require a mild smoothness assumption for the nuisance class.\loose
\begin{assumption}
\label{ass:smooth_loss_slow}
The derivatives $D^{2}_{g}L_{\cD}(\theta,g)$ and $D_{\theta}^{2}D_{g}L_{\cD}(\theta,g)$ are continuous.
Furthermore, there exists a constant $\beta$ such that for all $\theta\in\starhull(\targetemp,\besttwo)$ and $\bar{g}\in\starhull(\nuisance,\gtone)$,
\begin{equation}
\abs*{D^{2}_{g}L_{\cD}(\theta,\bar{g})[g-\gtone,g-\gtone]}\leq{}\beta \cdot{}\nrm*{g-\gtone}_{\nuisance}^{2}\quad\forall{}g\in\nuisance.
\end{equation}
\end{assumption}
Our main theorem for slow rates is as follows.
\begin{theorem}
\label{thm:orthogonal_slow}
Suppose that there is $\besttwo\in\argmin_{\theta\in\target}\poprisk(\theta,\gtone)$ such that \pref{ass:universal_orthogonality}  and \pref{ass:smooth_loss_slow} are satisfied. Then with probability at least $1-\delta$, the target \param $\esttwo$ produced by \pref{alg:sample_splitting} enjoys the excess risk bound:
\[
L_{\cD}(\esttwo,\gtone) - L_{\cD}(\besttwo, \gtone) 
\leq{}  \Rate(\target, \sampletwo, \delta/2\midsem{}\esttwo,\estone) + \beta \cdot\prn*{\Rate(\nuisance,\sampleone,\delta/2)}^{2}.
\]
\end{theorem}

\begin{revone}
  For generic Lipschitz losses, the optimal rate for parametric
  classes---in the absence of nuisance parameters---scales with
  $n^{-1/2}$. Without orthogonality, one expects the dependence
  on nuisance estimation error to scale linearly with
  $\nrm{\estone-\gtone}_{\nuisance}$, which would require
  $\nrm{\estone-\gtone}_{\nuisance}^{2}=\littleo(n^{-1})$ to achieve
  the oracle rate. \pref{thm:orthogonal_slow} shows that under
  orthogonality, the impact of nuisance parameter estimation is of
  lower order, and it suffices that
  $\nrm{\estone-\gtone}_{\nuisance}^{2}=\littleo(n^{-1/2})$.
  The proof follows similar reasoning to that of
  \pref{thm:generic_strongly_convex}; see \pref{app:orthogonal}.
\end{revone}

\subsection{Example: Treatment Effect Estimation}
\label{sec:treatment_effect}

\newcommand{\lambdap}{\lambda_{p}}
\newcommand{\lambdare}{\lambda_{\mathsf{re}}}

To make matters concrete, we now walk through a detailed example in which we specialize our general framework to the well-studied problem of treatment effect estimation. We show how the setup falls in our framework, explain what statistical assumptions are required to apply our main theorems, and show how to interpret the resulting excess risk bounds.

Following, e.g., \cite{robinson1988root,nie2017quasi}, we receive examples $z=(X,W,Y,T)$ according to the following data generating process:
\begin{equation}
\label{eqn:treatment_control_dgp}
\begin{aligned}
&Y = T\cdot{}\theta_0(X) + f_0(W) + \veps_1,&&\En\brk*{\veps_1\mid{}X,W,T}=0,\\
&T=e_0(W) + \veps_2,&&\En\brk*{\veps_2\mid{}X,W}=0,
\end{aligned}
\end{equation}
where $X\in\cX$ and $W\in\cW$ are covariates, $T\in\crl*{0,1}$ is the treatment variable, and $Y\in\bbR$ is the target variable. The true target \param is $\theta_0:\cX\to\bbR$, but we do not necessarily assume that $\gttwo\in\target$. The functions $e_0:\cW\to\brk*{0,1}$ and $f_{0}:\cW\to\bbR$ are unknown; we define $m_0(x,w)=\En\brk*{Y\mid{}X=x,W=w}=\theta_0(x)e_0(w) + f_0(w)$ and take $g_0=\crl*{m_0,e_0}$ to be the true nuisance parameter. We set $\varone=(X,W,T)$ and $\vartwo=(X)$.

\begin{revone}
\subsubsection{Residualized Loss (R-Loss)}
Following \cite{robinson1988root,nie2017quasi}, we consider the residualized square loss
\begin{equation}
\label{eq:cate_control_loss}
\ls(\theta,\crl*{m,e};z) = \prn*{(Y-m(X,W) - (T-e(W))\theta(X)}^{2}.
\end{equation}
Let us take a moment to interpret the meaning of excess risk under this loss. %
It is simple to verify that if the true nuisance parameters $g_0=\crl*{m_0,e_0}$ are plugged in, then
\[
L_{\cD}(\theta,\gtone) - L_{\cD}(\gttwo,\gtone) = \En\prn*{(T-e_{0}(W))\cdot{}(\theta(X)-\theta_{0}(X))}^{2}.
\]
Thus, if a predictor $\theta$ has low risk it, must be good at predicting $\theta_{0}(X)$ whenever there is sufficient variation in the treatment $T$. In addition, If the model is not well-specified ($\gttwo\notin\target$) but $\target$ is convex, we can still deduce that 
\[
L_{\cD}(\theta,\gtone) - L_{\cD}(\besttwo,\gtone) \geq{} \En\prn*{(T-e_{0}(W))\cdot{}(\theta(X)-\besttwo(X))}^{2},
\]
so in this case low excess risk implies that we predict nearly as well as the best predictor in class (again, assuming sufficient variation in $T$).

\paragraph{Applying the main results: Fast rates}
We now apply \pref{thm:generic_strongly_convex} to derive oracle excess risk bounds for the residualized loss. Let us consider the seminorms $\nrm*{\theta}_{\target}^2\ldef\En\brk*{(T-e_0(W))^{2}\theta^2(X)}$ and $\nrm*{\cdot}_{\cG}\ldef\nrm*{\cdot}_{L_4(\ell_2,\cD)}$ and $r=0$. Establishing the basic orthogonality and first-order conditions required to apply \pref{thm:generic_strongly_convex} is a simple exercise (see \pref{app:orthogonal} for a full derivation). To establish the smoothness and strong convexity properties \pref{ass:smooth_loss} and \pref{ass:strong_convex_loss}, we require mild boundedness assumptions, and a lower bound on the coverage parameter
\begin{equation}
\label{eq:rsc_treatment}
\lambdare\ldef\inf_{\theta\in\target}\crl*{\frac{\En(T-e_0(W))^{2}\prn*{\theta(X)-\besttwo(X)}^{2}}{\En\prn*{\theta(X)-\besttwo(X)}^{2}}} > 0.
\end{equation}
In particular, we have the following result.
\end{revone}

\begin{revone}
  \begin{proposition}
    \label{prop:treatment_example}
    Consider the treatment effect estimation setting with the residualized loss and norms $\nrm*{\theta}_{\target}^2\ldef\En\brk*{(T-e_0(W))^{2}\theta^2(X)}$ and $\nrm*{\cdot}_{\cG}\ldef\nrm*{\cdot}_{L_4(\ell_2,\cD)}$. Suppose that $\gttwo\in\target$ and $\abs{\gttwo(x)}\leq{}1$. Then the assumptions of \pref{thm:generic_strongly_convex} are satisfied with constants $r=0$, $\lambda = \frac{1}{4}$,  $\kappa=4\lambdare^{-1}$, $\beta_1=1$, and $\beta_2= 4\lambdare^{-1/2}$. As a result, the sample splitting meta-algorithm (\pref{alg:sample_splitting}) with the residualized loss produces a \param $\esttwo$ such that with probability at least $1-\delta$,
\begin{align*}
L_{\cD}(\esttwo,\gtone) - L_{\cD}(\gttwo, \gtone)
\leq{} 8\cdot{}\Rate(\target,\sampletwo,\delta/2\midsem{}\esttwo,\estone) +
144\lambdare^{-1}\cdot{}\prn*{
\Rate(\nuisance,\sampleone,\delta/2)
}^{4}.
\end{align*}
More generally, whenever $\Theta$ is convex, the same conclusion holds with $\gttwo$ replaced by $\besttwo\in\argmin_{\theta\in\target}\poprisk(\theta,\gtone)$, regardless of whether $\gttwo\in\target$.
\end{proposition}
\pref{prop:treatment_example} implies that for any class, oracle rates for excess risk are achievable whenever  $\Rate(\nuisance,n,\delta)=\littleo\prn*{\Rate(\target,\cdots)^{1/4}}$. Interestingly, in the case where the target class $\target$ is convex, this holds even when the target parameter is arbitrarily misspecified.  In addition, the excess risk bound in \pref{prop:treatment_example} has the desirable property that the coverage parameter $\lambdare$ enters only through the higher-order nuisance error term. 

Let us interpret the coverage parameter $\lambdare$, which acts as a problem-dependent constant whose value reflects the interaction between the treatment policy and the treatment effect. In general, to lower bound $\lambdare$, it suffices to assume that $\var(T-e_0(W) \mid X)\geq \eta$ for some $\eta> 0$, with no further assumptions required on the data distribution or target parameter class. This condition is typically referred to as \emph{overlap}, since it requires that the treatment is not deterministic for any realization of the covariates, and implies that $\lambdare\geq{}\eta$. On the other hand, even if overlap is not satisfied, one can still lower bound the coverage parameter. To do so, we focus on a special case investigated in \cite{chernozhukov2017orthogonal} and \cite{chernozhukov2018plugin}, where $\target$ is a class of high-dimensional predictors of the form $\theta(x) = \tri*{w,\phi(x)}$, where $w\in\bbR^{p}$ and $\phi:\cX\to\bbR^{p}$ is a fixed featurization; in general, the dimension $p$ may grow with $n$, with  $p\gg{}n$. In this case, note that it suffices that the matrix $\En\brk*{\mathrm{Var}(T-e_0(W)\mid X)\phi(X)\phi(X)^{\trn}}$ satisfies a restricted minimum eigenvalue condition. Hence, a lower bound on $\lambdare$ generalizes assumptions used in \cite{chernozhukov2017orthogonal,chernozhukov2018plugin}. 

\end{revone}

\paragraph{Stronger guarantees for specific target classes}
The results in the prequel apply to arbitrary target classes, but require that the nuisance estimation algorithm is close in the $L_4$ norm (i.e., $\nrm*{\cdot}_{\cG}=\nrm*{\cdot}_{L_4(\ls_2,\cD)}$). For specific target classes (typically, classes with additional structure that facilitates estimation in parameter error), it is possible to provide improved guarantees that scale with weaker $L_2$ estimation error for the nuisance class. To illustrate the flexibility of \pref{thm:generic_strongly_convex} in accommodating such cases, we consider a constrained variant of the \emph{R-learner} of \citet{nie2017quasi} and recover the oracle rates from this work.
\begin{revone}
\begin{example}[Constrained R-Learner]
  \label{ex:r_learner}
  The R-learner of \cite{nie2017quasi} corresponds to a special case of the treatment effect estimation setup in \pref{eqn:treatment_control_dgp} in which the target parameter belongs to a kernel class, and is estimated by minimizing the orthogonal loss \pref{eq:cate_control_loss} with regularization. Specializing the sample splitting meta-algorithm (\pref{alg:sample_splitting}) to this setting, we obtain a constrained variant of their method. \preprint{\\}In more detail, consider the treatment effect estimation setting with $\nrm*{\theta}_{\target}^2\ldef\En\brk*{(T-e_0(W))^{2}\theta^2(X)}$ and $\nrm*{\cdot}_{\cG}\ldef\nrm*{\cdot}_{L_2(\ell_2,\cD)}$. Let $\cH$ be a reproducing kernel Hilbert space (RKHS) with norm $\nrm{\cdot}_{\cH}$ and kernel $\cK$. Assume that $\abs{Y}\leq{}1$ almost surely and that treatments satisfy overlap, and consider the constrained target parameter class
\[
\target = \crl*{\theta\in\cH\mid{}\nrm{\theta}_{\cH}\leq{}c, \nrm{\theta}_{L_{\infty}(\cD)}\leq{}1}, 
\]
where $c\geq{}1$ is a parameter. For the target estimation algorithm, consider the plug-in empirical risk minimizer
\[
  \esttwo = \arg\min_{\theta \in \target}\sum_{i=n/2}^n \ell(\theta(\vartwo_i), \estone(\varone_i); \varall_i),
\]
where $\estone$ is the nuisance estimator. Assume that the kernel $\cK$ has eigenvalue decay of the form $\sigma_j\sim j^{-1/p}$ for some parameter $p\in(0,1)$ and that a smoothed version of $\theta_0$ lies in the RKHS for smoothing parameter $\alpha\in (0,1/2)$ (refer to proof in \pref{app:example_proofs} for definitions), and choose $c\propto{}n^{\alpha/(p+(1-2\alpha))}$. If the nuisance estimation algorithm has $\nrm*{\estone-\gtone}_{L_2(\ls_2,\cD)}\leq\Rate(\nuisance,\sampleone,\delta)=\wt{o}(n^{-1/4})$, then with probability at least $1-\delta$, \pref{alg:sample_splitting} has
\begin{equation}
  \label{eq:r_learner}
  \poprisk(\esttwo,\gtone) - \poprisk(\besttwo,\gtone)
  \leq \bigoht\prn*{n^{-\frac{1-2\alpha}{p+(1-2\alpha)}} + \frac{\log(\delta^{-1})}{n}},
\end{equation}
where $\bigoht(\cdot)$ suppresses dependence on regularity parameters and $\log(n)$ factors. This matches the best known rate for the oracle learner \citep{nie2017quasi}. 
\end{example}
This example shows that an $\bigoh(n^{-1/4})$ rate in $L_2$-error for the nuisances suffices to achieve the optimal rate in the absence of nuisance parameter. The proof leverages a lemma of \cite{mendelson2010regularization}, which states that for all $\theta\in\cH$,
$\nrm{\theta}_{L_{\infty}(\cD)}\approxleq\nrm{\theta}_{\cH}^{p}  \nrm{\theta}_{L_{2}(\cD)}^{1-p}$, where $p$ is the eigenvalue decay parameter. This allows us to establish \pref{ass:smooth_loss} and \pref{ass:strong_convex_loss} with respect to $L_2$-error for the nuisance parameter, at the cost of incurring exponent $r=p$, rather than $r=0$ as in the generic result (\pref{prop:treatment_example}). Moreover, as a consequence of the norm comparison inequality of \cite{mendelson2010regularization}, the $L_2(\cD)$-error bound from our theorem also implies a bound on $L_{\infty}(\cD)$ error.
\end{revone}

\paragraph{Slow rates}
We mention in passing that some distributions may simply not satisfy the coverage condition in \pref{eq:rsc_treatment}. In this case, we can appeal to \pref{thm:orthogonal_slow} (we show in \pref{app:example_proofs} that the residualized loss satisfies the universal orthogonality property), which does not require any lower bounds in the vein of \pref{eq:rsc_treatment}, but leads to slower rates. In general, whether the fast rate (\pref{thm:generic_strongly_convex}) or slow rate (\pref{thm:orthogonal_slow}) will give better results given finite samples will depend on the behavior of the data distribution and target class.

\begin{revone}
\subsubsection{Doubly-Robust Loss (DR-Loss)}
As an alternative to the residualized loss, for the special case of  a binary treatment, we can use the doubly-robust approach described in \pref{sec:intro}. Consider the special case of \pref{eq:treatment_control_dgp} in which $X=W$. Recall that that $e_0(X)=\En\brk*{T\mid{}X}$ is the treatment propensity, and define $f_0(t,x)=\En\brk*{Y\mid{}T=t,X=x}$. Define
\[
\varphi(f,e;z) \ldef f(1,X) - f(0,X) - \frac{T-e(X)}{e(X)(1-e(X))}(Y-f(T,X)).\] 
We take $g=\crl*{f(1,\cdot),f(0,\cdot), e}$ as the nuisance parameter. Then the doubly robust loss takes the form
\begin{equation}
\label{eq:doubly_robust_treatment_loss}
\ls(\theta,\crl{f(0,\cdot),f(1,\cdot),e};z) = \prn*{\varphi(f,e;z)- \theta(X)}^{2}.
\end{equation}
One can verify that $\En\brk*{\vphi(f_0,e_0;z)\mid{}X}=\theta_0(X)$. As a result, whenever the true nuisance parameters $g_0=\crl{f_0(0,\cdot),f_0(1,\cdot),e_0}$ are plugged in, the oracle excess risk satisfies
\[
L_{\cD}(\theta,\gtone) - L_{\cD}(\gttwo,\gtone) = \En\brk*{\prn*{\theta(X)-\theta_{0}(X)}^{2}},
\]
and hence is equivalent to $L_2$-error.
It is straightforward to verify that the doubly-robust loss satisfies the preconditions of \pref{thm:generic_strongly_convex}, which leads to the following result.
\begin{proposition}
  \label{prop:doubly_robust}
  Consider the treatment effect estimation setting with the doubly-robust loss and norms $\nrm*{\cdot}_{\target}\ldef\nrm{\cdot}_{L_2(\cD)}$ and $\nrm*{\cdot}_{\cG}\ldef\nrm*{\cdot}_{L_4(\ell_2,\cD)}$. Suppose that $\gttwo\in\target$ and that $\abs{\theta(x)},\nrm*{g(x)}_{\infty}\leq{}1$ for all $\theta\in\target,g\in\nuisance$. In addition, assume that $\abs{Y}\leq{}1$ almost surely, and that $\eta{}\leq{}e(X)\leq{}1-\eta$ for all $g=\crl{f(0,\cdot),f(1,\cdot),e}\in\cG$. Then the assumptions of \pref{thm:generic_strongly_convex} are satisfied with constants $r=0$, $\lambda = \frac{1}{4}$,  $\kappa=0$, $\beta_1=1$, and $\beta_2= 24\eta^{-3}$. As a result, the sample splitting meta-algorithm (\pref{alg:sample_splitting}) with the doubly-robust loss produces a \param $\esttwo$ such that with probability at least $1-\delta$,
\begin{align*}
L_{\cD}(\esttwo,\gtone) - L_{\cD}(\gttwo, \gtone)
\leq{} 8\cdot{}\Rate(\target,\sampletwo,\delta/2\midsem{}\esttwo,\estone) +
  \bigoh(\eta^{-6})\cdot{}\prn*{
\Rate(\nuisance,\sampleone,\delta/2)
}^{4}.
\end{align*}
\end{proposition}

This approach was further developed in the subsequent work of \citet{kennedy2020optimal}, who termed it the \emph{DR-Learner}, and provided improved oracle estimation rates which are doubly-robust with respect to the estimation errors for the propensities and conditional means. A variant of our two-stage algorithm for the doubly robust loss was also explored in the prior work of \cite{oprescu2018orthogonal} for the special case where the target estimation algorithm is a Generalized Random Forest. 

\vsedit{We note that the explicit dependence on $\eta$ in the second order term in \cref{prop:doubly_robust} can be avoided if one instead re-defines the inverse propensity term $a_0(T,X)=\frac{T-e_0(X)}{e_0(X)\, (1-e_0(X))}$ as the nuisance function. In this case, the second part of \pref{ass:smooth_loss} is satisfied with respect to the product pre-norm: $\nrm*{g}_{\nuisance}=\nrm*{(f,a)}=\sqrt{\|f\|_{L_4(\ell_2,\cD)}\, \|a\|_{L_4(\ell_2,\cD)}}$, since we have that:
\begin{align*}
  \abs*{D^{2}_{g}D_{\theta}\poprisk(\besttwo,\bar{g})[\nu_\theta, \nu_g, \nu_g]} 
  ={}~& 4\En\brk*{\nu_\theta(X)\, \nu_a(T,X)\, \nu_f(T,X) }
  \leq{} 4 \nrm*{\nu_\theta}_{L_2(\cD)} \cdot \nrm*{\nu_g}_{\nuisance}^2
\end{align*}
Applying \pref{thm:generic_strongly_convex} with this definition of nuisance pre-norm yields a doubly robust version of \pref{prop:doubly_robust}, where only products of nuisance estimation rates arise.}
\end{revone}

\subsection{Example: Policy Learning}
\label{sec:policy_learning_body}
As a second example, we show how to apply our framework to the classical problem of policy learning. Compared to our treatment effect estimation example, losses for this setting do not typically satisfy the strong convexity property, meaning that \pref{thm:orthogonal_slow} is the relevant meta-theorem, and slow rates are to be expected.

In policy learning, we receive examples of the form $Z=(X,T,Y)$, where $Y\in \R$ is an incurred loss, $T \in \cT$ is a treatment vector and $X\in \cX$ is a vector of covariates. %
The treatment $T$ is chosen based on an unknown, potentially randomized policy which depends on $X$. Specifically, we assume the following data generating process:
\begin{equation}
\label{eq:treatment_control_dgp}
\begin{aligned}
&Y = f_0(T, X) + \veps_1,&&\En\brk*{\veps_1\mid{}X,T}=0,\\
&T=e_0(X) + \veps_2,&&\En\brk*{\veps_2\mid{}X}=0.
\end{aligned}
\end{equation}
The learner wishes to optimize over a set of treatment policies $\target\subseteq{}(\cX\to\cT)$ (i.e., policies take as input covariates $X$ and return a treatment). Their goal is to produce a policy $\esttwo$ that achieves small regret with respect to the population risk:
\begin{equation}
\label{eq:policy_risk}
\En\brk[\big]{f_0(\esttwo(X), X)} - \min_{\theta \in \target} \En\brk*{f_0(\theta(X), X)}.
\end{equation}
This formulation has been extensively studied in statistics \citep{qian2011performance,zhao2012estimating,zhou2017residual,athey2017efficient,zhou2018offline} and machine learning \citep{beygelzimer2009offset,dudik2011doubly,swaminathan2015counterfactual,kallus2018policy}\preprint{; in the latter, it is sometimes referred to as counterfactual risk minimization}.

The learner does not know the so-called counterfactual outcome function $f_0$, so it is treated as a nuisance parameter. Typically, orthogonalization of this nuisance parameter is possible by utilizing the secondary treatment equation in \pref{eq:treatment_control_dgp} and fitting a \param for the observational policy $e_0$, which is also treated as a nuisance parameter. We can then write the expected counterfactual reward as
\begin{equation}
f_0(t, X) = \En\brk*{\ell(t, f_0, e_0;Z) \mid X}
\end{equation}
for some known loss function $\ell$ that utilizes the treatment \param $e_0$. Letting $g_0=\crl*{f_0,e_0}$, the learner's goal can be phrased as minimizing the population risk,
\begin{equation}
\label{eq:policy_generic}
\En\brk*{f_0(\theta(X), X)} = \En\brk*{\En\brk*{\ell(\theta(X), f_0, e_0;Z) \mid X}} = \En\brk*{\ell(\theta(X), f_0, e_0;Z)} \rdef \poprisk(\theta, g_0),
\end{equation}
over $\theta \in \target$. This formulation clearly falls into our orthogonal statistical learning framework, where the target parameter is the policy $\theta$ and the counterfactual outcome $f_0$ and observed treatment policy $e_0$ together form the nuisance parameter $g_0\ldef\crl*{f_0,e_0}$. To facilitate the use of estimation for the nuisance components, one typically assumes access to function classes $\cE$ and $\cF$ with $e_0\in\cE$ and $f_0\in\cF$ (so that $\cG=\cF\times\cE$), and fits the nuisance parameters via regression over these classes.

We make this discussion concrete for the special case of binary treatments $T\in \{0,1\}$, with additional examples in \pref{sec:policy}. To simplify notation, define $p_0(t, x) = \bbP\brk*{T=t \mid X=x}$, so that $p_0(t, x) = e_0(x)$ if $t=1$ and $1-e_0(x)$ if $t=0$. Consider the nuisance parameter $g=\crl*{f(0,\cdot),f(1,\cdot),e}$. Then the loss\preprint{ function}
\begin{equation}
\ell(t, g;Z) = f(t, X) + 1\{T=t\} \frac{(Y - f(t, X))}{p(t, X)},
\end{equation}
has the structure in \pref{eq:policy_generic}: it evaluates to the true risk \pref{eq:policy_risk} whenever the true nuisance parameter is plugged in. This formulation leads to the well-known doubly-robust estimator for the counterfactual outcome \citep{cassel1976some,robins1994estimation,robins1995semiparametric,dudik2011doubly}. It is straightforward to verify that the resulting population risk is orthogonal with respect to $g$. We can also obtain an equivalent loss function by subtracting the loss incurred by choosing treatment $0$. Define
\begin{equation}
\label{eq:binary_treatment}
\beta(g;Z)= \prn*{f(1, X) - f(0, X) + T\frac{Y - f(1, X)}{e(X)} - (1-T)\frac{Y - f(0, X)}{1-e(X)}},
\end{equation}
and set $\ell(t, g;Z) = \beta(g;Z)\cdot{}t$. This formulation leads to a linear population risk:
\begin{equation}
\poprisk(\theta, g) = \En\brk*{\beta(g;Z)\cdot \theta(X)}.
\end{equation}
This population risk satisfies universal orthogonality, and \pref{thm:orthogonal_slow} can be applied with $\nrm{\cdot}_{\nuisance}=\nrm{\cdot}_{L_2(\ls_2,\cD)}$ whenever the nuisance parameters are bounded appropriately. In particular, we have the following corollary of \pref{thm:orthogonal_slow}.

\begin{revone}
  \begin{proposition}
    \label{prop:policy_learning}
    Consider the policy learning setting with binary treatments and norm $\nrm*{\cdot}_{\nuisance}=\nrm*{\cdot}_{L_2(\ls_2,\cD)}$. Suppose that $\abs{Y}\leq{}1$ almost surely, and that all $g=\crl*{f(0,\cdot),f(1,\cdot),e}\in\cG$ have $\abs{f(t,X)}\leq{}1$ and $e(X)\in\brk{\eta,1-\eta}$ for some $\eta\in(0,1/2]$. Then with probability at least $1-\delta$, the target \param $\esttwo$ produced by \pref{alg:sample_splitting} enjoys the excess risk bound:\loose
\[
L_{\cD}(\esttwo,\gtone) - L_{\cD}(\besttwo, \gtone) 
\leq{}  \Rate(\target, \sampletwo, \delta/2\midsem{}\esttwo,\estone) + \bigoh(\eta^{-3}) \cdot
\prn*{\Rate(\nuisance,\sampleone,\delta/2)}^{2}.
\]
\end{proposition}
Note that this bound depends on the overlap parameter $\eta$ only through the nuisance error term. We mention in passing that explicit dependence on this parameter can be avoided entirely by treating the inverse propensity term $a_0(T,X)=\frac{T-e_0(X)}{e_0(X)\, (1-e_0(X))}$ as nuisance parameter (see, e.g., \cite{chernozhukov2021automatic}). In this case, note that \pref{ass:smooth_loss_slow} is satisfied with respect to the product pre-norm: $\nrm*{g}_{\nuisance}=\nrm*{(f,a)}=\sqrt{\|f\|_{L_2(\ell_2,\cD)}\, \|a\|_{L_2(\ell_2,\cD)}}$, since we have that:
\begin{align*}
  \abs*{D^{2}_{g}\poprisk(\theta,\bar{g})[\nu_g, \nu_g]} 
  ={}~& 4\En\brk*{\theta(X)\, \nu_a(T,X)\, \nu_f(T,X) }
  \leq{} 4 \nrm*{\nu_g}_{\nuisance}^2
\end{align*}
Applying our \pref{thm:generic_strongly_convex} with this definition of nuisance pre-norm, yields a doubly robust version of \pref{prop:policy_learning}, where only products of nuisance estimation rates arise. Recent follow-up work of \cite{chernozhukov2021automatic} provides a statistical learning approach for estimating such nuisance functions with respect to mean-squared error, which is based on minimizing an empirical analogue of the risk function $\En\brk*{a(T, X)^2 - 2 (a(1, X) - a(0,X))}$.
\end{revone}

\subsection{Discussion}
\label{sec:orthogonal_discussion}
\begin{revone}
We close by discussing extensions that build on the
results presented in this section, as well as additional connections between our
results and existing techniques in the literature on semiparametric
inference and double machine learning.

\paragraph{Experiments, additional tools, and applications}
\pref{part:experiments} of the appendix contains an empirical evaluation of the
techniques presented in this section, with applications to treatment
effect estimation and policy learning.  \pref{part:additional} of the
appendix contains supplementary theoretical results that build on the development of this section,
including user-friendly variants of the main theorems
(\pref{sec:user_friendly}), construction of orthogonal losses
(\pref{sec:orthogonal_loss}), sufficient conditions to apply the main theorems
(\pref{sec:sufficient}), and further applications
(\pref{sec:applications}).

\paragraph{Construction of orthogonal losses}

While orthogonal losses are already known for many problem settings
and statistical models (treatment effect estimation, policy learning,
regression with missing or censored data, and so on), for new
problems one often begins with a loss that is not necessarily
orthogonal. In \pref{sec:orthogonal_loss}, we give a
generic approach to construct orthogonal losses, building on a
technique from \cite{chernozhukov2018plugin}.

\paragraph{One the use of cross-fitting}
\pref{alg:sample_splitting} relies on sample splitting. While this
strategy is quite general and results in rate-optimal estimates, it
can be inefficient, since the target parameter is only estimated using
a subset of the data. A more practical alternative is to employ the
well-known \emph{cross-fitting} approach (e.g., \cite{chernozhukov2016double}), in which we split the data
into $K$ folds, obtain estimators using complementary folds,
and combine the results. Cross-fitting variants of
\pref{alg:sample_splitting} are given in
\pref{app:additional_algorithms} as \pref{alg:cross_fitting,alg:cross_fitting2}.\loose

One can show that under fairly genreal assumptions, the analysis in
this section remains valid when cross-fitting is employed, and we recommend this in
practice. However, compared to the setting considered in
\cite{chernozhukov2016double}, in which cross-fitting provably enjoys improved
efficiency over basic sample splitting, there is no hope of
establishing that cross-fitting improves efficiency at the level of
generality considered in the present work. This is because our framework
permits the use of arbitrary, potentially nonparametric or
high-dimensional estimators
which may be biased due to the use of regularization or
constraints and---for example---may not be asymptotically linear. As a result, even in the absence of nuisance parameters,
there is no guarantee that averaging multiple target estimators obtained
from independent sample splits will lead to improved efficiency.

\paragraph{One the use of influence functions}

A special case of our framework can be phrased in the language of classical
semiparametric inference as follows: If the population risk functional is pathwise differentiable, and one estimates the
target by minimizing an estimator for the risk based on influence functions,
\vsedit{which will typically lead to a Neyman orthogonal loss and}
the resulting target estimator will have
favorable second-order errors dependence on the error of the nuisance
estimator; see \cite{curth2020estimating} for follow-up work which
takes this approach explicitly.
However, Neyman orthogonality goes beyond pathwise
differentiability (for instance, one can construct orthogonal losses assuming only existence of pathwise
derivatives locally at $(\besttwo,\gtone)$; see \pref{sec:orthogonal_loss}), and
our results apply in settings where influence functions may not
exist. Moreover, one can obtain Neyman orthogonal losses without invoking influence functions, hence 
orthogonal statistical learning is a more flexible framework. See \cite{van2003unifiedb,tsiatis2007semiparametric,kosorok2008introduction,kennedy2016semiparametric}
for a review of influence functions and semiparametric theory.

\end{revone}

\section{\revonecolor{Instantiating the Main Results: Plug-In Empirical Risk Minimization}}
\label{sec:erm}
\newcommand{\muhat}{\hat{\mu}}

\revoneedit{The results in \cref{sec:orthogonal} are stated at a high level of generality, and concern generic estimation algorithms for the target and nuisance parameters. In this section we shift our focus to specific algorithms, and instantiate our general tools to provide explicit bounds based on intrinsic properties of the function classes under consideration.} In particular, we develop algorithms and analysis for orthogonal statistical learning with $M$-estimation losses of the form
\begin{equation}
\poprisk(\theta, g) = \En\brk*{\ell(\theta(\vartwo), g(\varone); \varall)}.
\end{equation}
We analyze one of the most natural and widely used estimation algorithms for the target parameter: \emph{plug-in empirical risk minimization (plug-in ERM)}. Specifically, recalling that $S=S_1\cup{}S_2$, we define the empirical risk via
\begin{align}
L_{\sampletwo}(\theta, g) =~& \frac{1}{n} \sum_{i=1}^n \ell(\theta(\vartwo_i), g(\varone_i); \varall_i),
\end{align}
where we adopt the convention that $\abs*{S}=2n$ with $\sampletwo=\crl*{z_1,\ldots,z_n}$ to keep notation compact. The plug-in ERM algorithm returns the minimizer plug-in empirical loss obtained by plugging in the first-stage estimate of the nuisance component:
\begin{equation}
\esttwo = \arg\min_{\theta \in \target} L_{\sampletwo}(\theta, \estone).
\end{equation}
\revoneedit{We provide oracle excess risk bounds for the plug-in ERM algorithm (and variants) in terms of statistical standard complexity measures for the target class $\target$. The main results in this section show that the impact of $\estone$ on the oracle excess risk achieved ERM is of second order, and that classical excess risk bounds carry over up to lower order terms and constant factors. These results are derived by bounding the second-stage $\Rate(\target,\sampletwo,\delta\midsem{}\esttwo,\estone)$ using (localized) empirical process tools, then appealing to the main theorems (\pref{thm:generic_strongly_convex} and \pref{thm:orthogonal_slow}).}

In the fast rate regime (i.e., for strongly convex losses) we offer a generalization of the local Rademacher complexity analysis of \cite{bartlett2005local} in the presence of an estimated nuisance component and show that the notion of the \emph{critical radius} of the class $\target$ still governs rate $\Rate(\target,\sampletwo,\delta\midsem{}\esttwo,\estone)$. This leads to \preprint{several }applications of our theory to specific target classes, including sparse linear models, neural networks and kernels (\pref{sec:specific_classes}).\loose

In the slow rate regime (i.e., for generic Lipschitz losses), we offer a novel moment-penalized variant of the plug-in ERM algorithm that achieves a rate whose leading term is equal to the critical radius, multiplied by the \emph{variance} of the population loss evaluated at the optimal target parameter. This offers an improvement over prior variance-penalized ERM approaches \citep{maurer2009empirical}, whose leading term depends on the metric entropy of the target function class at single scale, and which typically is larger than the critical radius.%

\paragraph{Technical preliminaries}
To present our main results, we need to introduce additional tools from empirical process theory and statistical learning. For any real-valued function class $\cG$, define the localized Rademacher complexity:
\begin{equation}
\Radexp(\cG,\delta) \ldef \En_{\epsilon, \zr[n]}\brk*{\sup_{g\in \cG: \|g\|_{L_2(\cD)}\leq \delta} \left| \frac{1}{n}\sum_{i=1}^n \epsilon_i\, g(z_i) \right| },
\end{equation}
where $\eps_1,\ldots,\eps_n$ are independent Rademacher random variables. Let $\cR_n(\cG)$ denote the non-localized Rademacher complexity (that is, $\cR_n(\cG,\infty)$). We also make use of the \emph{metric entropy} of a function class (which is closely related to the Rademacher complexity). \dfedit{We make the mild assumption that $\target$ and $\nuisance$ are separable, so as to ensure that associated empirical processes are measurable (cf. \citet[pp 314-315]{boucheron2013concentration}).}
\begin{definition}[Metric Entropy]
\label{def:metric_entropy}
 For any real-valued function class $\cG$ and sample $\zr[n]$, the empirical metric entropy $\cH_p(\cG, \veps, \zr[n])$ is the logarithm of the size of the smallest function class $\cG'$, such that for any $g\in \cG$ there exists $g'\in \cG'$, with $\|g - g'\|_{L_p(\zr[n])}\leq \veps$. Moreover $\cH_p(\cG, \veps, n)$ will denote the maximal empirical entropy over all possible sample sets $z_1,\ldots,z_n$.
\end{definition}
Finally, for a vector-valued function class $\cF$, let $\cF\rbar_t = \{ f_t: (f_1, \ldots, f_t, \ldots, f_d) \in \cF\}$ denote the projection of the class onto the $t$-th coordinate.

\subsection{Fast Rates for Plug-In Empirical Risk Minimization}
Our first contribution is an extension of the foundational results of \cite{bartlett2005local,koltchinskii2000rademacher}---which bound the excess risk for empirical risk minimization in terms of local Rademacher complexities---to incorporate misspecification due to nuisance parameter estimation error. A crucial parameter in this approach is the \emph{critical radius} $\delta_n$ of a function class $\cG$, defined as the smallest solution to the inequality
\begin{equation}
\label{eq:fixed_point0}
\Radexp(\cG,\delta_n) \leq \delta_n^2.
\end{equation}
Classical work shows that in the absence of a nuisance component, if a loss $\ell(\theta(z);z)$ is Lipschitz in its first argument and satisfies standard assumptions required for fast rates (strong convexity in the first argument), then empirical risk minimization achieves an excess risk bound of order $\delta_n^2$. For the case of parametric classes, $\delta_n=\Ot(n^{-1/2})$, leading to the fast $\Ot(n^{-1})$ rates for strongly convex losses. For more general classes (cf. \cite{wainwright2019}) the critical radius is---up to constant factors---equal to the solution to an inequality on the metric entropy of the function class (cf. \pref{app:specific}):
\begin{equation}
\int_{\frac{\delta_n^{2}}{8}}^{\delta_n} \sqrt{\frac{\cH_2(\cG(\delta_n,\zr[n]),\veps,\zr[n])}{n}} d\veps \leq{} \frac{\delta_n^{2}}{20},
\end{equation}
where $\cG(\delta,\zr[n]) \ldef \crl{g\in\cG: \nrm*{g}_{L_2(\zr[n])}\leq{}\delta}$; see \pref{sec:specific_classes} for concrete examples.

Our first theorem in this section extends this result in the presence of a nuisance component and bounds the excess risk of the plug-in ERM algorithm by the critical radius of the target function class $\target$ (more precisely, the worst-case critical radius for each coordinate of the target class, since we deal with vector-valued function classes).
\begin{theorem}[Fast Rates for Plug-In ERM]\label{thm:fast_erm}
Consider a function class $\target:\cX \rightarrow \R^{\dimtwo}$ with $R\ldef{}\sup_{\theta\in \target} \|\theta\|_{L_{\infty}(\ls_2,\cD)}\vee{}1$. Let $\delta_n^2 = \Omega\left(\frac{R^{2}\dimtwo\log(\log(n))}{n}\right)$ be a solution to the equation:\loose
\begin{equation}\label{eqn:critical_radius}
\Radexp(\sh(\Theta\rbar_{t} - \theta_t^\star),\delta) \leq \frac{\delta^2}{R},\quad \forall t\in \{1,\ldots, d\},
\end{equation}
where $\besttwo_t$ is the projection of $\besttwo$ onto coordinate $t$. Suppose that $\ell(\cdot, \estone(\varone);\varall)$ is $L$-Lipschitz with respect to the $\ell_2$ norm and that the population risk $\poprisk$ satisfies \pref{ass:orthogonal,ass:well_specified,ass:strong_convex_loss,ass:smooth_loss} with $\nrm*{\cdot}_{\target}=\|\cdot\|_{L_2(\ell_2,\cD)}$ and $\nrm*{\cdot}_{\nuisance}$ arbitrary. \begin{revone}Define $B_1\ldef \frac{L^{2}\dimtwo^2}{\lambda^2}$ and $B_2\ldef \prn[\big]{\frac{\beta_2}{\lambda}}^{\frac{2}{1+r}}
    + \frac{\kappa}{\lambda}$, and let $\esttwo$ be the outcome of the plug-in ERM algorithm. Then, with probability at least $1-\delta$,
  \begin{align}
    \nrm[\big]{\esttwo-\besttwo}_{\target}^2 \leq
\bigoh\prn*{
  B_1\cdot{}\prn*{\frac{\delta_n^{2}}{R^2} + \frac{\log(\delta^{-1})}{n}} + B_2\cdot{}\nrm*{\estone-\gtone}_{\cG}^{\frac{4}{1+r}}
  },
  \end{align}
  and
\begin{align}
\poprisk(\esttwo, \gtone) - \poprisk(\besttwo, \gtone)=
\bigoh\prn*{
  \beta_1B_1\cdot{}\prn*{\frac{\delta_n^{2}}{R^2} + \frac{\log(\delta^{-1})}{n}} + \beta_1B_2\cdot{}\nrm*{\estone-\gtone}_{\cG}^{\frac{4}{1+r}}
  }.
\end{align}
\end{revone}
\end{theorem}
\journal{\vspace{-15pt}}
\revoneedit{Critically, when $r=0$, the dependence on the nuisance estimation error scales as $\|\estone - \gtone\|_{\nuisance}^4$ due to orthogonality, meaning that we can use a complex function class for nuisance estimation without spoiling the rate for the target class. This result is proven in two steps. First, we show that one can take $\Rate(\target,n,\delta;\esttwo,\estone)\approxleq{}\delta_n\cdot{}\nrm{\esttwo-\besttwo}_{\target}+\delta_n^2$; this result uses standard empirical process theory tools, and does not leverage orthogonality. Then, we invoke orthogonality through \pref{thm:generic_strongly_convex} to derive the final guarantee. See \pref{app:constrained} for details.
  }

%
%
%

\subsection{Slow Rates and Variance Penalization}
\label{sec:variance_penalized}
We now turn to the slow rate regime from \pref{ssec:orthogonal_slow_rate}, where the loss is not necessarily strongly convex in the prediction. We prove upper bounds on the generalization error of a variance penalized version of the plug-in ERM algorithm. Our main result gives a slow rate that scales with the variance of the loss rather than the range, and is robust to nuisance estimation error. The basic algorithm we analyze first estimates the nuisance parameter, then estimates the optimal loss value $\mu^{\star}\ldef{}\inf_{\theta\in\target}\poprisk(\theta,g_0)$ using auxiliary samples, and finally performs plug-in empirical risk minimization with an empirical variance penalty which is centered using the estimate for $\mu^{\star}$. \revoneedit{See \pref{alg:variance_penalization} in \pref{app:additional_algorithms} for a full description.} To simplify notation, we assume that $\abs*{S}=3n$ and is partitioned equal splits $S=S_1\cup{}S_2\cup{}S_3$. Define the variance of the loss at $(\besttwo,\gtone)$ via
\[
V^{\star} = \var(\ls(\besttwo(\cdot),\gtone(\cdot);\cdot)).
\]
\begin{theorem}[Plug-In ERM with Centered Second Moment Penalization]
\label{thm:variance_penalized}
Consider the centered second moment-penalized plugin empirical risk minimizer in \pref{alg:variance_penalization}:
\begin{equation}
\hat{\theta} = \arg\min_{\theta\in \target} \crl*{L_{\sampletwo}(\theta, \hat{g}) + 36\delta_nR^{-1}\|\ell(\theta(\cdot),\estone(\cdot);\cdot)-\muhat\|_{L_2(\sampletwo)}},
\end{equation}
where $\estone$ is the output of $\mathrm{Alg}(\nuisance,\sampleone)$ and $\muhat = \inf_{\theta\in \target} L_{S_3}(\theta, \hat{g})$. Consider the function class $\cF = \{\ell(\theta(\cdot), \estone(\cdot); \cdot): \theta \in \target\}$, and let $R\ldef\sup_{f\in \cF}\|f\|_{L_{\infty}(\cD)}\vee{}1$ and $\fstar\ldef\ell(\besttwo(\cdot), \estone(\cdot); \cdot)$. Let $\delta_n^2 \geq 0$ be any solution to the inequality
\begin{equation}
  \label{eq:variance_critical}
\Radexp(\sh(\cF - f^\star),\delta) \leq \frac{\delta^2}{R},
\end{equation}
\begin{revone}
  Suppose that \pref{ass:universal_orthogonality} holds, $\ls(\theta(x),\cdot;z)$ is $L$-Lipschitz, and  \pref{ass:smooth_loss_slow} holds with parameter $\beta$ and $\nrm*{\cdot}_{\cG}\ldef\nrm*{\cdot}_{L_2(\ls_2,\cD)}$. Let $C\ldef (L^2+\beta{}R)$. Then with probability at least $1-\delta$,\loose
\begin{align*}
  &L_{\cD}(\esttwo, \gtone) - L_{\cD}(\besttwo, \gtone) \\&= O\prn[\bigg]{ \sqrt{V^{\star}}\prn[\bi
  gg]{\frac{\delta_n}{R}  + \sqrt{\frac{\log(\delta^{-1})}{n}}} + \frac{1}{R}\prn*{\delta_n^{2} + \cR_n^{2}(\ls\circ\target) + C\cdot{}\nrm*{\estone-\gtone}_{\nuisance}^{2}} + R\frac{\log(\delta^{-1})}{n}
  }.
\end{align*}
\end{revone}
\end{theorem}
\journal{\vspace{-15pt}}
\revoneedit{As with the previous result, \pref{thm:variance_penalized} is proven by first upper bounding $\Rate(\target,n,\delta;\esttwo,\estone)$ using empirical process tools, then invoking orthogonality through one of the main theorems (in this case, \pref{thm:orthogonal_slow}). The only complication is that the result requires the additional step of relating $\Rate(\target,\cdots,\estone)$ to the function $\Rate(\target,\cdots,\gtone)$, which entails bounding the variance of the loss at $\estone$ in terms of the variance of the loss at $\gtone$ and nuisance estimation error.\looseness=-1
}

Our approach offers  an improvement over the rates for empirical variance penalization in \cite{maurer2009empirical}, which provides a generalization error bound whose leading term is of the form: $
\sqrt{\frac{\var_n(\ell(\besttwo(\cdot), \estone(\cdot),\cdot)) \cH_{\infty}(\ls\circ\target, n^{-1}, z_{1:n})}{n}}$. The drawback of such a bound is that it evaluates the metric entropy at a fixed approximation level of $1/n$, which can be suboptimal compared to the critical radius. %
\revoneedit{In \pref{app:vc}, we show that for classes $\Theta$ with bounded VC dimension, this guarantee can be further improved as a consequence of our general machinery\preprint{, and give a bound which scales with the so-called \emph{Alexander capacity function}}.}

\begin{revone}
  \paragraph{Discussion}
  Due to space constraints, applications to specific target classes (sparse linear models, neural networks, kernel classes) are deferred to \pref{sec:specific_classes}.
\end{revone}

\section{\revonecolor{Instantiating the Main Results: Sufficient
    Conditions for Oracle Rates}}
\label{sec:oracle}

\begin{revone}
  The previous section developed guarantees for orthogonal statistical
  learning with a specific algorithm, plug-in empirical risk
  minimization. While empirical risk minimization is a
  workhorse of statistical learning, in general it does not attain
  minimax excess risk for rich function classes, even in the absence of
  nuisance parameters. In this section we build on the development so
  far and, by appealing to \emph{aggregation} techniques, provide
  algorithms that always attain minimax excess risk up to second-order
  dependence on nuisance parameters. 
  Our main results provide \emph{sufficient
  conditions} under which oracle rates are achieved which explicitly
depend on intrinsic properties of both the target and nuisance parameter classes. In
particular, we give sufficient conditions based on the relationship between the \emph{metric
  entropy} for the nuisance and target classes.\loose

For any real-valued
  function class $\cF$, we say that the \emph{complexity} of
  $\cF$ is $p$ if for all $\veps>0$,
  \begin{equation}
    \textstyle{\cH_{2}(\cF,\veps,n) = O\left(\max\{ \veps^{-p}, \log(1/\veps)\}\right)},
  \end{equation}
  where we recall that $\cH_2$ is the metric entropy defined in
  \pref{def:metric_entropy}.  When $p=0$, this corresponds to the case of parametric
    functions (e.g., linear models and VC-subgraph classes), while
  for $p>0$, we recover nonparametric function classes, such as
  Lipschitz/smooth functions or kernel spaces. We let $p_1$ and $p_2$
  denote the maximum complexity of any output coordinate projection
  for the nuisance and target class, respectively. We provide
  sufficient conditions on the pair $(p_1,p_2)$ under which the sample splitting
  meta-algorithm (\pref{alg:sample_splitting})---with an appropriate
  choice for the target and nuisance estimator---can achieve oracle rates.

\end{revone}

\label{sec:algsandrates}
\label{sec:square_oracle}
We focus on\preprint{ the important special case of} square losses of the form
\begin{equation}%
\ls(\theta(\vartwo),g(\varone); z) = \bigl(
\tri{\Lambda(g(\varone), v), \theta(\vartwo)}- \Gamma(g(\varone),\varall)\bigr)^{2}, \quad\quad\poprisk(\theta,g)=\En\brk{\ls(\theta,g;z)},
\label{eq:square_oracle}
\end{equation}
where $\Lambda$ and $\Gamma$ are known functions, and where we recall from \pref{sec:setup} that $x$, $w$ are subsets of the data $z$, and $v\subseteq{}z$ is an arbitrary auxiliary subset of the data. We assume that the nuisance \params are defined in terms of regression problems, i.e., that $\gtone(w)=\En\brk*{u\mid{}w}$ for some known random vector $u\subseteq{}z$. This assumption is standard in semiparametric literature \citep{bickel1993efficient,kosorok2008introduction,van2011targeted}, and implies that each coordinate $t$ of $g_0$ may be expressed as the minimizer of a squared loss: $g_{0,t}=\argmin_{g_t\in\cG\rbar_t}\En\brk[\big]{\prn*{g_{t}(w)-u_t}^2}$. In this setting, a sufficient condition for orthogonality is that
\begin{equation}\label{eqn:oracle_orthogonality}
\En\brk*{\nabla_{\gamma}\nabla_{\zeta} \ell(\besttwo(\vartwo), \gtone(\varone), z) \mid \varone} = 0,
\end{equation}
where $\nabla_{\zeta}$ and $\nabla_{\gamma}$ denote the gradient of $\ls$ with respect to the first and second argument\preprint{, respectively}.

In the absence of nuisance parameters, minimax optimal rates for excess risk in square loss regression have been characterized for the \emph{well-specified} setting in which %
\begin{equation}
\label{eq:oracle_well_specified}
\En\brk*{\Gamma(\gtone(\varone),\varall)\mid{}\varone,v} = \tri{\Lambda(\gtone(\varone),v), \gttwo(\vartwo)},
\end{equation}
for some $\gttwo\in\target$, 
and for the misspecified setting where this assumption is removed. In the former setting, the minimax rates are of order $\Theta\prn{n^{-\frac{2}{2+p_2}}}$ \citep{yang1999information}, while in the latter setting the optimal rate is $\wt{\Theta}\prn{n^{-\frac{2}{2+\ptwo}\wedge{}\frac{1}{\ptwo}}}$ \citep{rakhlin2017empirical}. We show that under orthogonality, the optimal \preprint{well-specified and misspecified }rates can be achieved in the presence of nuisance parameters even when the nuisance class $\nuisance$ is larger than the target class $\target$, provided it is not \emph{too} much larger. This generalizes the large body of results on semiparametric inference\preprint{ \citep{levit1976efficiency,ibragimov1981statistical,pfanzagl1982contributions,bickel1982adaptive,klaassen1987consistent,robinson1988root,bickel1993efficient,newey1994asymptotic,robins1995semiparametric, ai2003efficient,van2003unified,van2003unifiedb,ai2007estimation,tsiatis2007semiparametric,kosorok2008introduction,van2011targeted,ai2012semiparametric,chernozhukov2016locally,belloni2017program,chernozhukov2016double}}, which show under various assumptions that if the target class is parametric, one can obtain a $\sqrt{n}$-consistent estimator for the target if the nuisance estimator converges at a $n^{-\frac{1}{4}}$ rate.

Our main workhorse for the results in this section is the ``Aggregation of $\veps$-Nets'' or ``Skeleton Aggregation'' algorithm described in \cite{yang1999information} and extended to random design in \cite{rakhlin2017empirical}. \revoneedit{The Skeleton Aggregation method operates by splitting the samples in two, building an empirical cover for the function class under consideration using the first split, and then aggregating the elements of the cover using the second split. See \pref{sec:skeleton} for a full description. This approach is related to sieve-based methods (e.g., \citep{semenova2021debiased}), which employ parametric methods to learn a linear combination of basis elements that approximate the target, but \preprint{an important difference is that }Skeleton Aggregation builds the basis in a data-dependent fashion. We use Skeleton Aggregation as-is to provide rates for the first stage, and provide an extension in the presence of nuisance parameters for the second stage, which entails relating $\Rate(\target,n,\delta;\esttwo,\estone)$ to $\Rate(\target,n,\delta;\esttwo,\gtone)$.
}

\revoneedit{
  We caution that the algorithms in this section are only designed to attain the minimax rates for \emph{generic} square losses of the type in \pref{eq:square_oracle} (e.g., vanilla square loss regression), and specific special cases may admit better rates. Deriving minimax lower bounds for specific losses of interest (as in \cite{kennedy2020optimal}) is an interesting direction for future research.
  }

\paragraph{Assumptions}
Since our aim is to provide sufficient conditions based on the metric entropy of the classes $\target$ and $\nuisance$, which is already quite technical, we assume that all other problem-dependent parameters are constant. This is only for expository purposes.
\begin{assumption}
  \label{ass:square_oracle}
  \preprint{The classes are bounded in the sense that for all}
  \journal{For all}
  $\theta\in\target+\starhull(\target-\target)$ and \vsedit{$g\in\nuisance+\starhull(\nuisance-\nuisance,0)$} %
  the following bounds \vsedit{hold a.s.}: \alphamark{a}
  $\tri{\Lambda(g(\varone),v), \theta(\vartwo)}\in\unitrange$,
  \alphamark{b} $\Gamma(g(\varone),z)\in{}\unitrange$, \alphamark{c}
  $\Lambda(g(\varone), v) \Lambda(g(\varone), v)^{\trn}\preceq{}I$,
  \alphamark{d} $\nrm*{g(w)}_{\infty}, \|\theta(x)\|_{\infty} \leq
  1$, \vsedit{\alphamark{e} $\dimone,\dimtwo=\bigoh(1)$, \alphamark{f}
  $\nrm*{u}_{\infty}\leq{}1$ almost surely, \alphamark{g}
  the functions $\crl*{\Lambda_t(\cdot, v)}_{t=1}^{\dimtwo}$ and $\Gamma(\cdot, z)$ have $\bigoh(1)$-Lipschitz gradients with respect to $\ls_2$,
  \alphamark{h} the strong convexity condition
  $\En\brk[\big]{\tri*{\Lambda(\gtone(\varone), v),
    \theta(x)-\besttwo(x)}^{2}}\geq \gamma
  \En\nrm*{\theta(x)-\besttwo(x)}_{2}^{2}$ is satisfied for all $\theta\in\target+\starhull(\target-\target)$ for some $\gamma=\Omega(1)$.}
\end{assumption}
\pref{ass:square_oracle} implies that \pref{ass:smooth_loss} and \pref{ass:strong_convex_loss} are satisfied with respect to the seminorms $\nrm*{\theta}_{\target} \ldef \prn[\big]{\En
\tri*{\Lambda(g_{0}(\varone),v), \theta(\vartwo)}^{2}}^{1/2}$ and $\nrm*{\cdot}_{\nuisance} = \nrm*{\cdot}_{L_4(\ls_2,\cD)}$, with $r=0$. Since typical results on minimax oracle rates provide rates for the nuisance $g$ with respect to $\|\cdot\|_{L_2(\ls_2,\cD)}$, we assume control on the ratio between these seminorms. 

\begin{assumption}[Moment Comparison]
\label{ass:moment}
\journal{$\nrm*{g-\gtone}_{L_{4}(\ls_2,\cD)}\leq{}\momentconstant\cdot \nrm*{g-\gtone}_{L_{2}(\ls_2,\cD)},$ for all $g\in\nuisance$ for a constant $\momentconstant$.
}
\preprint{
  There is a constant $\momentconstant$ such that\[
\frac{\nrm*{g-\gtone}_{L_{4}(\ls_2,\cD)}}{\nrm*{g-\gtone}_{L_{2}(\ls_2,\cD)}}\leq{}\momentconstant,\quad{}\forall{}g\in\nuisance.
\]}
\end{assumption}
The moment comparison condition has been used in statistics as a minimal assumption for learning without boundedness \citep{lecue2013learning,mendelson2014learning,liang2015learning}. For example, suppose that each $g\in\nuisance$ has the form $x\mapsto{}\tri*{w,x}$ for $w,x\in\bbR^{d}$. Then $\momentconstant\leq{}3^{1/4}$ if $x$ is mean-zero gaussian and $\momentconstant\leq{}\sqrt{8}$ if $x$ follows any distribution that is independent across all coordinates and symmetric (via the Khintchine inequality). Moment comparison is also implied by the ``subgaussian class'' assumption used in \cite{mendelson2011discrepancy,lecue2013learning}.\preprint{\footnote{Suppose $\cG$ is scalar-valued and let $\nrm*{g}_{\psi_{2}}=\inf\crl*{c>0\mid{}\En\exp\prn*{g^{2}(w)/c^{2}}\leq{}2}$. Then the subgaussian class assumption for our setting asserts that $\nrm*{g-\gtone}_{\psi_{2}}\leq{}C\nrm*{g-\gtone}_{L_2(\cD)}$ for all $g\in\nuisance$.}} We emphasize that the moment constant $\momentconstant$ does not enter the leading term in any of our bounds---only the $\Rate(\nuisance,\ldots,)$ term in \pref{thm:generic_strongly_convex}---and so it does not affect the asymptotic rates under conditions on metric entropy growth of $\nuisance$ that we prescribe in the sequel. We also note that  this condition is not required for many classes of interest, where direct $L_{4}$ estimation rates are available (see discussion in \pref{sec:sufficient}). We adopt the condition here because it allows us to develop guarantees for arbitrary classes at the highest possible level of generality.

%

%
%
%
%
%
%
%
%
%
%
%
%
%
%
%
%
%
%
%
%
%
%
%
%
%
%
%
%
%
%
%
%
%
%
%
%
%
%
%
%
%
%
%
%
%
%
%
%
%
%
%
%
%

%

%
%
\paragraph{Main Result}
\revoneedit{The main theorem for this section provides sufficient conditions for oracle rates in the well-specified setting \pref{eq:oracle_well_specified}. For extensions of this result to misspecified models, as well as non-strongly convex losses, see \pref{sec:oracle_lipschitz}.
}

\begin{theorem}[Oracle Rates, Well-Specified Case]
  \label{thm:oracle_well_specified}
Suppose that we are in the well-specified setting \pref{eq:oracle_well_specified}, and that \pref{ass:orthogonal,ass:well_specified} and \pref{ass:square_oracle,ass:moment} are satisfied for the class $\targetemp$ defined below. Suppose that the following relationship holds:
\begin{equation}
\pone < 2\ptwo+2.
\end{equation}
Then for appropriate choice of sub-algorithms, the sample splitting meta-algorithm \pref{alg:sample_splitting} produces a predictor $\esttwo$ that guarantees that with probability at least $1-\delta$,
\begin{equation}
  \label{eq:oracle_well_specified2}
\poprisk(\esttwo,\gtone) - \poprisk(\gttwo,\gtone)
\leq{}  \Ot\prn[\big]{n^{-\frac{2}{2+\ptwo}}},
\end{equation}
\revoneedit{where $\Ot(\cdot)$ hides problem-dependent parameters and $\log(\delta^{-1})$ terms. This result matches the minimax rate in the absence of nuisance parameters. In particular, when $\ptwo\leq{}2$ it suffices to take $\targetemp=\target$ and use plug-in ERM for stage two, and when $\ptwo>2$ it suffices to take $\targetemp=\target + \starhull(\target-\target,0)$ and use Skeleton Aggregation for stage two; in both cases, it suffices to use Skeleton Aggregation for stage one.
  }
\end{theorem}
\begin{figure}[t]
  \begin{center}
  \includegraphics[width=.35\textwidth]{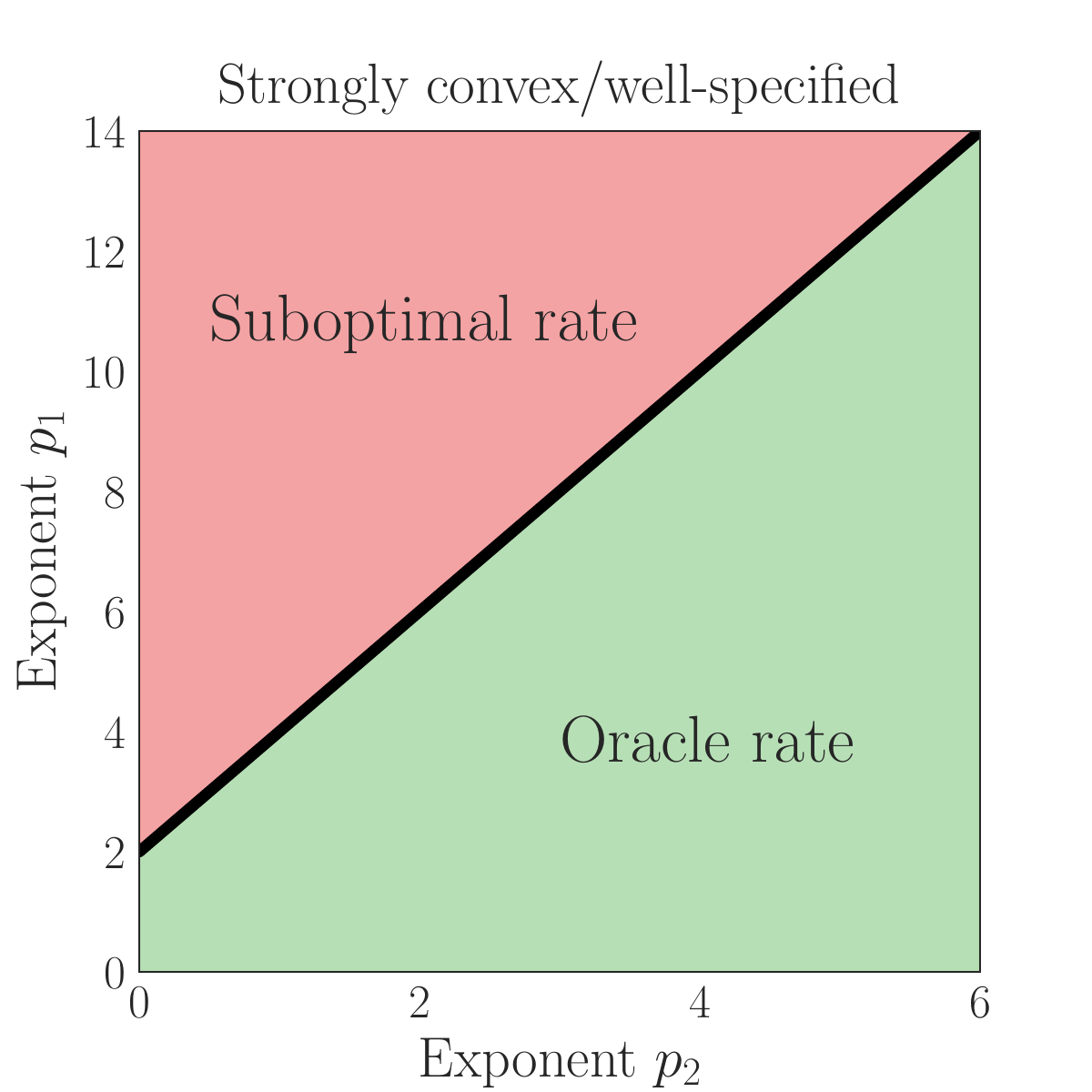}
\end{center}
\caption{Relationship between first and second stage for oracle rates; well-specified case.}
\label{fig:oracle_well_specified}
  \end{figure}

  \revoneedit{\pref{thm:oracle_well_specified} is proven by combining the main theorem (\pref{thm:generic_strongly_convex}) with algorithm-specific upper bounds on $\Rate(\target,\cdots)$ and $\Rate(\nuisance,\cdots)$. \pref{fig:oracle_well_specified} summarizes the sufficient conditions under which \pref{thm:oracle_well_specified} leads to the oracle rate $\Theta\prn{n^{-\frac{2}{2+p_2}}}$ \citep{yang1999information}}. In particular, whenever $\target$ is a parametric class (i.e. $\cH_{2}(\target,\veps,n)\propto\dtwo\log(1/\veps)$), it suffices to take $\pone<2$, which recovers the usual setup for semiparametric inference.

\section{Discussion}
\label{sec:discussion}
This paper initiates the systematic study of prediction error and
excess risk guarantees in the presence of nuisance parameters and
Neyman orthogonality. Our results highlight that orthogonality is
beneficial for learning with nuisance parameters even in the presence
of possible model misspecification, and even when the target
parameters belong to large nonparametric classes. We also show that
many of the typical assumptions used to analyze estimation in the
presence of nuisance parameters can be relaxed when excess risk
is the focus. There are many promising future directions,
including weakening assumptions, obtaining sharper guarantees for
specific settings and losses of interest (e.g., doubly-robust guarantees), and analyzing further
algorithms for general function classes (along the lines of
\pref{sec:erm,sec:oracle}). \revoneedit{We refer to the appendix for additional
results, as well as empirical results.}\loose

\preprint{\paragraph{Acknowledgements}}
\journal{\section*{Acknowledgements}}
We are grateful to \preprint{the anonymous COLT reviewers and to }Xiaohong Chen for
pointing out additional related work. Part of this work was completed
while DF was an intern at Microsoft Research\preprint{, New England}. DF
acknowledges support from the Facebook PhD fellowship and NSF Tripods
grant \#1740751.

\journal{
\begin{supplement}
\stitle{Supplementary for ``Orthogonal Statistical Learning''}
\sdescription{The supplementary material contains proofs, experiments, and additional results deferred due to space constraints.}
\end{supplement}
}

\iftoggle{supp}{}{
\journal{\bibliographystyle{imsart-nameyear}}
\bibliography{refs}
}

\appendix

\journal{
{
\hypersetup{linkcolor=black}
\newpage
\begin{center}
\begin{Large}
\textbf{Supplementary File for ``Orthogonal Statistical Learning''}
\end{Large}
\end{center}
\renewcommand{\contentsname}{Contents of Appendix}
\tableofcontents
\addtocontents{toc}{\protect\setcounter{tocdepth}{3}}
}
}

\clearpage

\preprint{\section*{Organization of Appendix}}
\journal{\section*{Organization}}
The appendix is organized as
follows.
\begin{itemize}
\item \pref{part:experiments} contains experimental results evaluating
  the performance of orthogonal risk minimization methods for examples
  considered in \pref{sec:orthogonal}.
  \item \pref{part:additional} contains supplemental theoretical results: Algorithms
  omitted from the main body for space
  (\pref{app:additional_algorithms}), user-friendly variants of the
  main theorems (\pref{sec:user_friendly}), construction of
  orthogonal losses (\pref{sec:orthogonal_loss}), sufficient
  conditions to apply the main theorems,
  (\pref{sec:sufficient}), sufficient conditions for oracle rates that
  extend \pref{sec:oracle} (\pref{sec:oracle_lipschitz}), applications
  of the main results (\pref{sec:applications}), further results
  regarding plug-in empirical risk minimization (generalization
  guarantees for specific function classes
  (\pref{sec:specific_classes}), and an application of the guarantees
  for plug-in variance-penalized ERM developed in \pref{sec:erm} to VC
  classes (\pref{app:vc}).
\item   \pref{part:proofs} contains proofs for the results presented in the
  main body of the paper: \pref{app:preliminaries} contains
  preliminaries, \pref{app:orthogonal} proves the main theorems for
  the sample-splitting meta-algorithm, \pref{app:constrained} proves
  the main results for plug-in empirical risk minimization (building
  on technical tools developed in \pref{app:general_m}), and
  \pref{app:rates} proves our results concerning oracle rates.
\end{itemize}

\part{Experiments}
\label{part:experiments}
\section{Experiments}
\newcommand{\E}{\mathbb{E}}

We conducted experiments to evaluate of the benefit of risk minimization with orthogonal loss functions as compared to non-orthogonal losses, focusing on the main applications considered in \pref{sec:orthogonal}: conditional average treatment effect (CATE) estimation under conditional exogeneity (\pref{sec:treatment_effect}) and policy learning (\pref{sec:policy_learning_body}).

For both applications, we use synthetic data. Each dataset contains $n$ i.i.d. examples, where each example is a tuple consisting of a $d$-dimensional random vector of confounders ($X\in \R^d$), a random binary treatment ($T\in \{0, 1\}$), and a random scalar outcome ($Y\in \R$). These random variables are drawn from a data generating process of the form:
\begin{align}
  \begin{aligned}
    T \sim~& \text{Bernoulli}(e_0(X)),& X \sim~& D_X\\
    Y =~& \tau_0(X)\, (T - 1/2) + b_0(X) + \epsilon,& \epsilon \sim~& \cN(0, \sigma^2)
  \end{aligned}\label{eq:exp_dgp}
\end{align}
We refer to $e_0(X)$ as the propensity function, $b_0(X)$ as the base response function, and $\tau_0(X)$ as the treatment effect function. The dimension $d$ and noise level $\sigma^2$ are parameters that vary across experiments.

Inspired by the set of experiments in the prior work of \cite{nie2017quasi,athey2017efficient} we consider six variants of the data-generating process in \pref{eq:exp_dgp} that capture different qualitative challenges for CATE estimation and policy learning.
\begin{itemize}
\item In {\bf Setup $\mb{A}$} we consider $X\sim U([0, 1])^d$, $e_0(X)=\text{clip}(\sin(\pi\, X[0]\, X[1]), .2, .8)$, $\tau_0(X)=.2 + \left(X[0] + X[1]\right)/2$, and $b_0(X)=\sin(\pi\, X[0]\, X[1]) + 2\,(X[2] - .5)^2 + X[3] + .5\, X[4]$, where $\text{clip}(x, a, b)\ldef \max\{a, \min\{b, x\}\}$ denotes the operator that clips $x$ to the interval $[a, b]$. This setup has a complicated propensity function and base response, but a relatively simple treatment effect function $\tau_0(X)$.
\item In {\bf Setup $\mb{B}$}, we consider $X\sim U([-.5, .5])^d$, $e_0(X)=.5$, $b_0(X)=\max\{0, X[0] + X[1], X[2]\} + \max\{0, X[3]+X[4]\}$, and $\tau_0(X)=X[0] + \log(1 + \exp\{X[1]\})$. This setup is a randomized trial with a complex base response and treatment effect function.
\item In {\bf Setup $\mb{C}$}, we consider $X\sim U([-.5, .5])^d$, $e_0(X)=\left(1 + \exp\{X[1] + X[2]\}\right)^{-1}$, $b_0(X)=2\, \log(1 + \exp\{X[0] + X[1] + X[2]\})$, and $\tau_0(X)=1$. This setup has a complex propensity and base response function, but a constant treatment effect.
\item In {\bf Setup $\mb{D}$}, we consider $X\sim U([-.5, .5])^d$, $e_0(X)=\left(1 + \exp\{-X[0]\} + \exp\{-X[1]\}\right)^{-1}$, $b_0(X)= .5\, \max\{0, X[0]+X[1]+X[2]\} + .5\, \max\{0, X[3] + X[4]\}$, and $\tau_0(X)=\max\{0, X[0]+X[1]+X[2]\} - \max\{0, X[3]+X[4]\}$. In this setup, the response under treatment is independent of the response under control, and hence there is no statistical benefit of jointly learning the two responses.
\item In {\bf Setup $\mb{E}$}, we consider $X\sim U([-.5, .5])^d$, $e_0(X)=\left(1 + \exp\{3\,X[1] + 3\, X[2]\}\right)^{-1}$, $b_0(X)=5\, \max\{0, X[0] + X[1]\}$, and $\tau_0(X)=2\, ((X[0]>.1)\;\vee\;(X[1] > .1)) - 1$. In this setup, we have a very wide range of responses that is largely explained by the confounder. Hence, failing to appropriately center the response around the predictable part can lead to large variance in the learning objective. Moreover, the treatment effect function is discontinuous.
  \item In {\bf Setup $\mb{F}$}, we consider $X\sim U([-.5, .5])^d$, $e_0(X)=\left(1 + \exp\{3\,X[1] + 3\, X[2]\}\right)^{-1}$, $b_0(X)=5\, \max\{0, X[0] + X[1]\}$, $\tau_0(X)=X[0] + \log(1 + \exp\{X[1]\})$. This setup is similar to setup $E$, but with a smooth treatment effect function.
\end{itemize}

For each setup, we consider three values for the sample size ($n\in \{500, 1000, 3000\}$), two values for the dimension of the confounders ($d\in \{6, 12\}$), and three values for the scale of the response noise ($\sigma\in \{.5, 1, 2\}$). For each such parameter setting, we run $100$ experiments and report average metrics and standard errors for the average metric (for the appropriate metric for each task).

\paragraph{Reproducibility}
  Code for reproducing the results is available at \url{https://github.com/vsyrgkanis/orthogonal_learning}. We note that because the FLAML package used for nuisance model estimation does not allow for control of random seeds for fully reproducible execution, results will have slight variation across multiple executions, albeit not to an extent that affects the qualitative take-aways.

\subsection{Conditional Average Treatment Effect Estimation}

We first consider the task of estimating the conditional average treatment effect $\tau_0(x)$ (i.e., we have $\theta_0=\tau_0$) with respect to the mean squared error metric $\E\brk[\big]{\prn[\big]{\hat{\theta}(X)-\theta_0(X)}^2}$. We use the two orthogonal losses considered in \pref{sec:treatment_effect}. The first is the residual-on-residual loss:
\begin{align}
    \poprisk(\theta,g) = \E\left[(Y - q(X) - \theta(X)\, (T - e(X))^2\right],
\end{align}
where $g=\crl{q,e}$, with the true values of the nuisance parameters given by $q_0(X)=\E[Y\mid X]$ and $e_0(X)=\E[T\mid X]$. Second, we consider the doubly-robust loss:
\begin{align}
    \poprisk(\theta,g) = \E\left[\left(f(1, X) - f(0, X) + \frac{e(X)-T}{e(X) (1 - e(X))}\, \left(Y - f(T, X)\right) - \theta(X)\right)^2\right],
\end{align}
where $g=\crl{f,e}$, with the true values of the nuisance parameters given by $f_0(T, X)=\E[Y\mid T, X]$ and $e_0(X)=\E[T\mid X]$. For each loss, we consider two risk minimization approaches: a semi-crossfitting approach that allows for some sample re-use, and a pure sample splitting approach.
\begin{itemize}
\item In the semi-crossfitting approach, we perform model selection using the FLAML \citep{wang2021flaml} automated machine learning package, which performs hyperparameter tuning and selection using a broad class of forest based models. Using the model class and hyperparameters chosen by FLAML, we perform cross-fitting: we train a set of the nuisance models (e.g. $\hat{e}, \hat{q}$) on half of the data and compute their values (e.g. $\hat{e}_i := e(X_i), \hat{q}_i := q(X_i)$) on the other half, and vice-versa. We then minimize the loss function over $\theta$ using all the samples with their computed nuisance values (e.g. $\frac{1}{n}\sum_{i=1}^n (Y_i - \hat{q}_i - \theta(X_i)\, (T_i - \hat{e}_i))^2$ for the residual-on-residual loss); this phase is also performed using FLAML.
\item In the pure sample-splitting approach, we split the dataset in half and perform hyperparameter tuning, model selection, and model fitting with FLAML on the first half. We then compute nuisance values on the second half. Finally, we estimate the target function $\hat{\theta}$, also using the second half of the dataset. This approach corresponds to the meta-algorithm presented in the main body (\pref{alg:sample_splitting}). 
\end{itemize}
We note that while the theoretical results in this paper only concern the pure sample-splitting approach, we expect that the semi-crossfitting approach will also benefit from the orthogonality of the loss, and it has been analyzed in prior literature for estimation of finite-dimensional target parameters \citep{chernozhukov2016double}.

We refer to the estimation approach that uses the residual-on-residual loss with semi-crossfitting as $\text{dml}$, and refer to the variant with pure sample splitting as $\text{dml\_split}$. Similarly, we refer to the two variants that use the doubly robust loss as $\text{dr}$ and $\text{dr\_split}$. Finally, as a benchmark we also consider the non-orthogonal loss:
\begin{align}
  \poprisk(\theta,\crl{f}) = \E\left[\left(f(1, X) - f(0, X) - \theta(X)\right)^2\right],
\end{align}
which we apply without any sample splitting or cross-fitting. We refer to this method as the $\text{slearner}$.

For each method, to perform risk minimization for the target parameter $\theta$, we applied the FLAML package, which performs automated machine learning and hyperparameter tuning across many candidate forest models. In order to make the target parameter inherently simpler than then nuisance parameters, we constrained the final stage model to depend only on the first $n_x < d$ coordinates of the variables $X$. For all setups except $D$, we used $n_x=4$, while for setup $D$ we used $n_x=5$ to maintain well-specification of the target model.

\paragraph{Results}
The average mean-squared error for the estimated target parameters $\hat{\theta}$ across the $100$ experiments and across the different setups are depicted in Figures~\ref{fig:setupA-cate}, \ref{fig:setupB-cate}, \ref{fig:setupC-cate}, \ref{fig:setupD-cate}, \ref{fig:setupE-cate}, \ref{fig:setupF-cate}. We find that the methods that use orthogonal losses either significantly out-perform or are comparable to the non-orthogonal loss in most domains. The superior performance vanishes only in the presence of small samples and high noise level $\sigma$, potentially due to large errors in the nuisance functions. Moreover, we find that for relatively small sample size, the orthogonal losses perform comparably to the oracle method, which uses the true models $p_0, q_0$ and minimizes the residual-on-residual loss (denoted by $\text{oracle}$). Moreover, we find that sample-splitting is sufficient to get a substantial boost in performance, and that cross-fitting tends to give small further improvement over sample-splitting.

\begin{figure}[H]
\scriptsize
\centering
\begin{tabular}{lllllllll}
\toprule
         &        &            &         oracle &            dml &     dml\_split &             dr &      dr\_split &       slearner \\
\midrule
\multirow{6}{*}{$n=500$} & \multirow{3}{*}{$d=6$} & $\sigma=0.5$ &  0.039 (0.003) &  0.057 (0.003) &  0.076 (0.004) &  0.077 (0.004) &  0.084 (0.005) &  0.113 (0.006) \\
         &        & $\sigma=1$ &  0.090 (0.011) &  0.083 (0.006) &  0.124 (0.011) &  0.092 (0.006) &  0.124 (0.010) &  0.137 (0.007) \\
         &        & $\sigma=2$ &  0.240 (0.047) &  0.181 (0.031) &  0.382 (0.074) &  0.171 (0.020) &  0.322 (0.093) &  0.181 (0.012) \\
\cline{2-9}
         & \multirow{3}{*}{$d=12$} & $\sigma=0.5$ &  0.041 (0.004) &  0.062 (0.003) &  0.076 (0.005) &  0.076 (0.003) &  0.085 (0.005) &  0.107 (0.006) \\
         &        & $\sigma=1$ &  0.073 (0.006) &  0.080 (0.005) &  0.161 (0.026) &  0.089 (0.005) &  0.134 (0.014) &  0.111 (0.007) \\
         &        & $\sigma=2$ &  0.187 (0.025) &  0.136 (0.009) &  0.219 (0.016) &  0.149 (0.010) &  0.290 (0.031) &  0.208 (0.013) \\
\cline{1-9}
\cline{2-9}
\multirow{6}{*}{$n=1000$} & \multirow{3}{*}{$d=6$} & $\sigma=0.5$ &  0.025 (0.002) &  0.055 (0.002) &  0.062 (0.003) &  0.071 (0.003) &  0.074 (0.003) &  0.118 (0.006) \\
         &        & $\sigma=1$ &  0.070 (0.017) &  0.061 (0.004) &  0.082 (0.007) &  0.078 (0.004) &  0.088 (0.004) &  0.116 (0.006) \\
         &        & $\sigma=2$ &  0.112 (0.015) &  0.119 (0.007) &  0.224 (0.036) &  0.102 (0.008) &  0.159 (0.015) &  0.155 (0.010) \\
\cline{2-9}
         & \multirow{3}{*}{$d=12$} & $\sigma=0.5$ &  0.022 (0.001) &  0.057 (0.003) &  0.067 (0.004) &  0.070 (0.003) &  0.074 (0.003) &  0.107 (0.006) \\
         &        & $\sigma=1$ &  0.058 (0.004) &  0.062 (0.003) &  0.088 (0.007) &  0.075 (0.004) &  0.086 (0.006) &  0.125 (0.006) \\
         &        & $\sigma=2$ &  0.120 (0.015) &  0.112 (0.007) &  0.180 (0.020) &  0.118 (0.012) &  0.146 (0.011) &  0.171 (0.009) \\
\cline{1-9}
\cline{2-9}
\multirow{6}{*}{$n=3000$} & \multirow{3}{*}{$d=6$} & $\sigma=0.5$ &  0.017 (0.001) &  0.038 (0.001) &  0.045 (0.002) &  0.073 (0.002) &  0.074 (0.002) &  0.115 (0.005) \\
         &        & $\sigma=1$ &  0.034 (0.002) &  0.048 (0.003) &  0.053 (0.002) &  0.071 (0.003) &  0.070 (0.003) &  0.112 (0.005) \\
         &        & $\sigma=2$ &  0.075 (0.009) &  0.064 (0.003) &  0.079 (0.004) &  0.078 (0.004) &  0.088 (0.007) &  0.142 (0.006) \\
\cline{2-9}
         & \multirow{3}{*}{$d=12$} & $\sigma=0.5$ &  0.017 (0.001) &  0.049 (0.002) &  0.052 (0.002) &  0.074 (0.002) &  0.076 (0.002) &  0.138 (0.004) \\
         &        & $\sigma=1$ &  0.034 (0.003) &  0.050 (0.002) &  0.058 (0.003) &  0.075 (0.003) &  0.079 (0.003) &  0.142 (0.004) \\
         &        & $\sigma=2$ &  0.094 (0.012) &  0.062 (0.003) &  0.079 (0.005) &  0.076 (0.004) &  0.098 (0.007) &  0.131 (0.007) \\
\bottomrule
\end{tabular}

\caption{Setup A. Conditional Average Treatment Effect (CATE) estimation. Mean-Squared-Error (MSE), averaged across $100$ experiments, with standard error.}
\label{fig:setupA-cate}
\end{figure}
 
\begin{figure}[H]
\scriptsize
\centering
\begin{tabular}{lllllllll}
\toprule
         &        &            &         oracle &            dml &     dml\_split &             dr &      dr\_split &       slearner \\
\midrule
\multirow{6}{*}{$n=500$} & \multirow{3}{*}{$d=6$} & $\sigma=0.5$ &  0.037 (0.002) &  0.063 (0.002) &  0.075 (0.002) &  0.063 (0.002) &  0.074 (0.003) &  0.154 (0.011) \\
         &        & $\sigma=1$ &  0.100 (0.007) &  0.084 (0.003) &  0.112 (0.006) &  0.088 (0.005) &  0.115 (0.005) &  0.186 (0.010) \\
         &        & $\sigma=2$ &  0.206 (0.018) &  0.142 (0.009) &  0.240 (0.023) &  0.171 (0.019) &  0.418 (0.143) &  0.263 (0.015) \\
\cline{2-9}
         & \multirow{3}{*}{$d=12$} & $\sigma=0.5$ &  0.042 (0.002) &  0.065 (0.001) &  0.075 (0.002) &  0.063 (0.001) &  0.079 (0.003) &  0.182 (0.011) \\
         &        & $\sigma=1$ &  0.090 (0.005) &  0.084 (0.003) &  0.104 (0.004) &  0.096 (0.007) &  0.137 (0.017) &  0.188 (0.010) \\
         &        & $\sigma=2$ &  0.224 (0.026) &  0.172 (0.013) &  0.283 (0.038) &  0.159 (0.007) &  0.540 (0.207) &  0.302 (0.018) \\
\cline{1-9}
\cline{2-9}
\multirow{6}{*}{$n=1000$} & \multirow{3}{*}{$d=6$} & $\sigma=0.5$ &  0.033 (0.001) &  0.056 (0.001) &  0.060 (0.001) &  0.055 (0.001) &  0.058 (0.001) &  0.139 (0.011) \\
         &        & $\sigma=1$ &  0.072 (0.008) &  0.071 (0.001) &  0.082 (0.003) &  0.070 (0.001) &  0.082 (0.002) &  0.174 (0.011) \\
         &        & $\sigma=2$ &  0.141 (0.010) &  0.109 (0.005) &  0.155 (0.012) &  0.137 (0.018) &  0.170 (0.012) &  0.228 (0.012) \\
\cline{2-9}
         & \multirow{3}{*}{$d=12$} & $\sigma=0.5$ &  0.028 (0.001) &  0.057 (0.001) &  0.064 (0.001) &  0.059 (0.001) &  0.064 (0.001) &  0.197 (0.011) \\
         &        & $\sigma=1$ &  0.063 (0.004) &  0.069 (0.001) &  0.088 (0.003) &  0.072 (0.002) &  0.086 (0.002) &  0.163 (0.010) \\
         &        & $\sigma=2$ &  0.155 (0.021) &  0.113 (0.005) &  0.233 (0.027) &  0.138 (0.013) &  0.181 (0.020) &  0.228 (0.012) \\
\cline{1-9}
\cline{2-9}
\multirow{6}{*}{$n=3000$} & \multirow{3}{*}{$d=6$} & $\sigma=0.5$ &  0.021 (0.001) &  0.057 (0.000) &  0.058 (0.001) &  0.057 (0.001) &  0.056 (0.001) &  0.109 (0.009) \\
         &        & $\sigma=1$ &  0.040 (0.003) &  0.064 (0.001) &  0.067 (0.001) &  0.063 (0.001) &  0.066 (0.001) &  0.146 (0.011) \\
         &        & $\sigma=2$ &  0.096 (0.009) &  0.074 (0.002) &  0.090 (0.003) &  0.076 (0.002) &  0.098 (0.006) &  0.205 (0.011) \\
\cline{2-9}
         & \multirow{3}{*}{$d=12$} & $\sigma=0.5$ &  0.022 (0.001) &  0.058 (0.001) &  0.059 (0.001) &  0.058 (0.001) &  0.057 (0.002) &  0.078 (0.006) \\
         &        & $\sigma=1$ &  0.040 (0.003) &  0.063 (0.001) &  0.066 (0.001) &  0.062 (0.001) &  0.065 (0.002) &  0.149 (0.011) \\
         &        & $\sigma=2$ &  0.087 (0.008) &  0.074 (0.002) &  0.090 (0.003) &  0.075 (0.002) &  0.090 (0.002) &  0.143 (0.010) \\
\bottomrule
\end{tabular}

\caption{Setup B. Conditional Average Treatment Effect (CATE) estimation. Mean-Squared-Error (MSE), averaged across $100$ experiments, with standard error.}
\label{fig:setupB-cate}
\end{figure}

\begin{figure}[H]
\scriptsize
\centering
\begin{tabular}{lllllllll}
\toprule
         &        &            &         oracle &            dml &     dml\_split &             dr &      dr\_split &       slearner \\
\midrule
\multirow{6}{*}{$n=500$} & \multirow{3}{*}{$d=6$} & $\sigma=0.5$ &  0.012 (0.004) &  0.033 (0.004) &  0.042 (0.004) &  0.035 (0.004) &  0.059 (0.006) &  0.238 (0.021) \\
         &        & $\sigma=1$ &  0.036 (0.006) &  0.053 (0.006) &  0.126 (0.014) &  0.066 (0.006) &  0.176 (0.075) &  0.246 (0.017) \\
         &        & $\sigma=2$ &  0.159 (0.029) &  0.171 (0.018) &  0.225 (0.011) &  0.167 (0.011) &  0.288 (0.019) &  0.398 (0.026) \\
\cline{2-9}
         & \multirow{3}{*}{$d=12$} & $\sigma=0.5$ &  0.009 (0.001) &  0.033 (0.004) &  0.044 (0.005) &  0.043 (0.005) &  0.060 (0.006) &  0.252 (0.021) \\
         &        & $\sigma=1$ &  0.030 (0.005) &  0.083 (0.010) &  0.108 (0.011) &  0.090 (0.011) &  0.128 (0.013) &  0.261 (0.018) \\
         &        & $\sigma=2$ &  0.127 (0.023) &  0.161 (0.009) &  0.255 (0.019) &  0.166 (0.010) &  0.250 (0.017) &  0.429 (0.029) \\
\cline{1-9}
\cline{2-9}
\multirow{6}{*}{$n=1000$} & \multirow{3}{*}{$d=6$} & $\sigma=0.5$ &  0.006 (0.001) &  0.013 (0.002) &  0.019 (0.003) &  0.021 (0.004) &  0.037 (0.005) &  0.185 (0.018) \\
         &        & $\sigma=1$ &  0.017 (0.003) &  0.042 (0.005) &  0.062 (0.006) &  0.050 (0.006) &  0.070 (0.006) &  0.233 (0.020) \\
         &        & $\sigma=2$ &  0.084 (0.014) &  0.104 (0.013) &  0.150 (0.015) &  0.099 (0.008) &  0.168 (0.016) &  0.271 (0.022) \\
\cline{2-9}
         & \multirow{3}{*}{$d=12$} & $\sigma=0.5$ &  0.006 (0.002) &  0.016 (0.002) &  0.039 (0.005) &  0.026 (0.003) &  0.032 (0.004) &  0.244 (0.019) \\
         &        & $\sigma=1$ &  0.022 (0.004) &  0.050 (0.006) &  0.073 (0.007) &  0.049 (0.005) &  0.063 (0.006) &  0.158 (0.014) \\
         &        & $\sigma=2$ &  0.073 (0.014) &  0.109 (0.012) &  0.177 (0.018) &  0.098 (0.007) &  0.159 (0.009) &  0.224 (0.020) \\
\cline{1-9}
\cline{2-9}
\multirow{6}{*}{$n=3000$} & \multirow{3}{*}{$d=6$} & $\sigma=0.5$ &  0.002 (0.000) &  0.012 (0.002) &  0.015 (0.002) &  0.015 (0.002) &  0.018 (0.003) &  0.182 (0.021) \\
         &        & $\sigma=1$ &  0.010 (0.002) &  0.021 (0.003) &  0.035 (0.005) &  0.027 (0.004) &  0.033 (0.004) &  0.203 (0.020) \\
         &        & $\sigma=2$ &  0.051 (0.010) &  0.063 (0.006) &  0.077 (0.006) &  0.068 (0.006) &  0.076 (0.008) &  0.185 (0.017) \\
\cline{2-9}
         & \multirow{3}{*}{$d=12$} & $\sigma=0.5$ &  0.003 (0.001) &  0.008 (0.001) &  0.013 (0.002) &  0.017 (0.003) &  0.018 (0.002) &  0.161 (0.019) \\
         &        & $\sigma=1$ &  0.008 (0.001) &  0.026 (0.004) &  0.044 (0.005) &  0.022 (0.003) &  0.040 (0.005) &  0.224 (0.022) \\
         &        & $\sigma=2$ &  0.039 (0.008) &  0.053 (0.006) &  0.091 (0.006) &  0.048 (0.005) &  0.076 (0.006) &  0.242 (0.022) \\
\bottomrule
\end{tabular}

\caption{Setup C. Conditional Average Treatment Effect (CATE) estimation. Mean-Squared-Error (MSE), averaged across $100$ experiments, with standard error.}
\label{fig:setupC-cate}
\end{figure}
 
\begin{figure}[H]
\scriptsize
\centering
\begin{tabular}{lllllllll}
\toprule
         &        &            &         oracle &            dml &     dml\_split &             dr &      dr\_split &       slearner \\
\midrule
\multirow{6}{*}{$n=500$} & \multirow{3}{*}{$d=6$} & $\sigma=0.5$ &  0.088 (0.003) &  0.117 (0.003) &  0.136 (0.009) &  0.120 (0.003) &  0.129 (0.004) &  0.135 (0.001) \\
         &        & $\sigma=1$ &  0.143 (0.005) &  0.156 (0.009) &  0.176 (0.013) &  0.145 (0.005) &  0.202 (0.017) &  0.140 (0.001) \\
         &        & $\sigma=2$ &  0.274 (0.022) &  0.240 (0.010) &  0.404 (0.038) &  0.295 (0.032) &  0.373 (0.027) &  0.152 (0.006) \\
\cline{2-9}
         & \multirow{3}{*}{$d=12$} & $\sigma=0.5$ &  0.088 (0.002) &  0.114 (0.002) &  0.126 (0.003) &  0.116 (0.003) &  0.129 (0.005) &  0.134 (0.001) \\
         &        & $\sigma=1$ &  0.161 (0.008) &  0.141 (0.004) &  0.178 (0.007) &  0.167 (0.010) &  0.212 (0.018) &  0.141 (0.001) \\
         &        & $\sigma=2$ &  0.246 (0.013) &  0.266 (0.029) &  0.366 (0.024) &  0.262 (0.016) &  0.375 (0.020) &  0.146 (0.004) \\
\cline{1-9}
\cline{2-9}
\multirow{6}{*}{$n=1000$} & \multirow{3}{*}{$d=6$} & $\sigma=0.5$ &  0.076 (0.002) &  0.111 (0.002) &  0.117 (0.003) &  0.115 (0.003) &  0.118 (0.003) &  0.132 (0.001) \\
         &        & $\sigma=1$ &  0.109 (0.004) &  0.132 (0.004) &  0.146 (0.005) &  0.138 (0.005) &  0.157 (0.009) &  0.137 (0.001) \\
         &        & $\sigma=2$ &  0.206 (0.018) &  0.182 (0.009) &  0.234 (0.011) &  0.183 (0.006) &  0.263 (0.026) &  0.142 (0.001) \\
\cline{2-9}
         & \multirow{3}{*}{$d=12$} & $\sigma=0.5$ &  0.079 (0.002) &  0.111 (0.002) &  0.115 (0.002) &  0.111 (0.002) &  0.119 (0.003) &  0.131 (0.001) \\
         &        & $\sigma=1$ &  0.113 (0.005) &  0.137 (0.004) &  0.149 (0.008) &  0.140 (0.005) &  0.152 (0.006) &  0.138 (0.001) \\
         &        & $\sigma=2$ &  0.211 (0.018) &  0.178 (0.006) &  0.245 (0.023) &  0.242 (0.032) &  0.244 (0.010) &  0.143 (0.001) \\
\cline{1-9}
\cline{2-9}
\multirow{6}{*}{$n=3000$} & \multirow{3}{*}{$d=6$} & $\sigma=0.5$ &  0.058 (0.002) &  0.107 (0.000) &  0.108 (0.000) &  0.107 (0.000) &  0.107 (0.000) &  0.127 (0.001) \\
         &        & $\sigma=1$ &  0.085 (0.002) &  0.114 (0.003) &  0.117 (0.003) &  0.112 (0.002) &  0.135 (0.005) &  0.132 (0.001) \\
         &        & $\sigma=2$ &  0.153 (0.010) &  0.148 (0.005) &  0.162 (0.005) &  0.162 (0.014) &  0.158 (0.005) &  0.137 (0.001) \\
\cline{2-9}
         & \multirow{3}{*}{$d=12$} & $\sigma=0.5$ &  0.059 (0.002) &  0.109 (0.001) &  0.108 (0.001) &  0.106 (0.000) &  0.108 (0.001) &  0.129 (0.001) \\
         &        & $\sigma=1$ &  0.091 (0.003) &  0.126 (0.004) &  0.126 (0.004) &  0.123 (0.004) &  0.127 (0.004) &  0.131 (0.001) \\
         &        & $\sigma=2$ &  0.143 (0.009) &  0.133 (0.004) &  0.148 (0.004) &  0.143 (0.004) &  0.157 (0.005) &  0.139 (0.000) \\
\bottomrule
\end{tabular}

\caption{Setup D. Conditional Average Treatment Effect (CATE) estimation. Mean-Squared-Error (MSE), averaged across $100$ experiments, with standard error.}
\label{fig:setupD-cate}
\end{figure}
 
\begin{figure}[H]
\scriptsize
\centering
\begin{tabular}{lllllllll}
\toprule
         &        &            &         oracle &            dml &     dml\_split &             dr &      dr\_split &       slearner \\
\midrule
\multirow{6}{*}{$n=500$} & \multirow{3}{*}{$d=6$} & $\sigma=0.5$ &  0.121 (0.004) &  0.548 (0.012) &  0.595 (0.013) &  0.568 (0.013) &  0.611 (0.014) &  0.978 (0.007) \\
         &        & $\sigma=1$ &  0.248 (0.014) &  0.567 (0.014) &  0.655 (0.017) &  0.583 (0.014) &  0.693 (0.019) &  0.954 (0.009) \\
         &        & $\sigma=2$ &  0.660 (0.033) &  0.730 (0.024) &  0.905 (0.038) &  0.757 (0.023) &  0.894 (0.023) &  0.984 (0.005) \\
\cline{2-9}
         & \multirow{3}{*}{$d=12$} & $\sigma=0.5$ &  0.138 (0.011) &  0.507 (0.012) &  0.564 (0.016) &  0.544 (0.012) &  0.599 (0.017) &  0.965 (0.011) \\
         &        & $\sigma=1$ &  0.271 (0.015) &  0.548 (0.012) &  0.672 (0.027) &  0.588 (0.014) &  0.644 (0.018) &  0.984 (0.004) \\
         &        & $\sigma=2$ &  0.785 (0.073) &  0.793 (0.042) &  0.897 (0.039) &  0.790 (0.032) &  0.926 (0.044) &  0.985 (0.010) \\
\cline{1-9}
\cline{2-9}
\multirow{6}{*}{$n=1000$} & \multirow{3}{*}{$d=6$} & $\sigma=0.5$ &  0.075 (0.005) &  0.450 (0.007) &  0.517 (0.012) &  0.460 (0.011) &  0.520 (0.011) &  0.951 (0.013) \\
         &        & $\sigma=1$ &  0.190 (0.013) &  0.522 (0.011) &  0.555 (0.012) &  0.517 (0.010) &  0.601 (0.018) &  0.969 (0.009) \\
         &        & $\sigma=2$ &  0.428 (0.030) &  0.558 (0.015) &  0.810 (0.046) &  0.588 (0.013) &  0.710 (0.019) &  0.985 (0.004) \\
\cline{2-9}
         & \multirow{3}{*}{$d=12$} & $\sigma=0.5$ &  0.076 (0.004) &  0.492 (0.007) &  0.522 (0.009) &  0.522 (0.005) &  0.548 (0.010) &  0.976 (0.008) \\
         &        & $\sigma=1$ &  0.167 (0.009) &  0.515 (0.009) &  0.554 (0.014) &  0.531 (0.007) &  0.580 (0.012) &  0.970 (0.008) \\
         &        & $\sigma=2$ &  0.451 (0.028) &  0.579 (0.018) &  0.740 (0.030) &  0.603 (0.014) &  0.784 (0.031) &  0.972 (0.008) \\
\cline{1-9}
\cline{2-9}
\multirow{6}{*}{$n=3000$} & \multirow{3}{*}{$d=6$} & $\sigma=0.5$ &  0.056 (0.008) &  0.475 (0.005) &  0.504 (0.007) &  0.510 (0.005) &  0.529 (0.008) &  0.905 (0.020) \\
         &        & $\sigma=1$ &  0.091 (0.004) &  0.475 (0.006) &  0.515 (0.009) &  0.522 (0.005) &  0.526 (0.009) &  0.915 (0.018) \\
         &        & $\sigma=2$ &  0.213 (0.014) &  0.510 (0.009) &  0.568 (0.013) &  0.561 (0.009) &  0.590 (0.013) &  0.960 (0.012) \\
\cline{2-9}
         & \multirow{3}{*}{$d=12$} & $\sigma=0.5$ &  0.036 (0.004) &  0.494 (0.006) &  0.502 (0.007) &  0.525 (0.005) &  0.532 (0.008) &  0.957 (0.012) \\
         &        & $\sigma=1$ &  0.119 (0.008) &  0.492 (0.007) &  0.514 (0.008) &  0.526 (0.006) &  0.532 (0.009) &  0.945 (0.015) \\
         &        & $\sigma=2$ &  0.215 (0.014) &  0.543 (0.010) &  0.576 (0.012) &  0.550 (0.008) &  0.575 (0.012) &  0.967 (0.009) \\
\bottomrule
\end{tabular}

\caption{Setup E. Conditional Average Treatment Effect (CATE) estimation. Mean-Squared-Error (MSE), averaged across $100$ experiments, with standard error.}
\label{fig:setupE-cate}
\end{figure}
 
\begin{figure}[H]
\scriptsize
\centering
\begin{tabular}{lllllllll}
\toprule
         &        &            &         oracle &            dml &     dml\_split &             dr &      dr\_split &       slearner \\
\midrule
\multirow{6}{*}{$n=500$} & \multirow{3}{*}{$d=6$} & $\sigma=0.5$ &  0.045 (0.002) &  0.164 (0.010) &  0.184 (0.012) &  0.169 (0.006) &  0.189 (0.008) &  0.534 (0.011) \\
         &        & $\sigma=1$ &  0.104 (0.005) &  0.180 (0.008) &  0.242 (0.018) &  0.196 (0.010) &  0.239 (0.014) &  0.547 (0.009) \\
         &        & $\sigma=2$ &  0.341 (0.050) &  0.251 (0.014) &  0.332 (0.019) &  0.249 (0.013) &  0.420 (0.048) &  0.522 (0.013) \\
\cline{2-9}
         & \multirow{3}{*}{$d=12$} & $\sigma=0.5$ &  0.050 (0.003) &  0.171 (0.008) &  0.189 (0.010) &  0.183 (0.008) &  0.216 (0.020) &  0.568 (0.007) \\
         &        & $\sigma=1$ &  0.148 (0.049) &  0.181 (0.009) &  0.231 (0.017) &  0.208 (0.013) &  0.228 (0.012) &  0.529 (0.012) \\
         &        & $\sigma=2$ &  0.322 (0.074) &  0.244 (0.013) &  0.378 (0.030) &  0.254 (0.013) &  0.463 (0.079) &  0.550 (0.010) \\
\cline{1-9}
\cline{2-9}
\multirow{6}{*}{$n=1000$} & \multirow{3}{*}{$d=6$} & $\sigma=0.5$ &  0.036 (0.002) &  0.107 (0.003) &  0.179 (0.017) &  0.130 (0.004) &  0.155 (0.006) &  0.519 (0.013) \\
         &        & $\sigma=1$ &  0.087 (0.007) &  0.135 (0.005) &  0.192 (0.011) &  0.157 (0.006) &  0.187 (0.012) &  0.536 (0.011) \\
         &        & $\sigma=2$ &  0.177 (0.022) &  0.178 (0.009) &  0.246 (0.012) &  0.178 (0.007) &  0.278 (0.017) &  0.526 (0.010) \\
\cline{2-9}
         & \multirow{3}{*}{$d=12$} & $\sigma=0.5$ &  0.035 (0.002) &  0.131 (0.004) &  0.161 (0.009) &  0.162 (0.005) &  0.173 (0.006) &  0.518 (0.012) \\
         &        & $\sigma=1$ &  0.091 (0.016) &  0.154 (0.011) &  0.174 (0.008) &  0.156 (0.008) &  0.181 (0.009) &  0.518 (0.013) \\
         &        & $\sigma=2$ &  0.196 (0.022) &  0.171 (0.007) &  0.230 (0.012) &  0.177 (0.007) &  0.310 (0.047) &  0.521 (0.012) \\
\cline{1-9}
\cline{2-9}
\multirow{6}{*}{$n=3000$} & \multirow{3}{*}{$d=6$} & $\sigma=0.5$ &  0.026 (0.001) &  0.164 (0.006) &  0.166 (0.006) &  0.185 (0.006) &  0.178 (0.007) &  0.546 (0.013) \\
         &        & $\sigma=1$ &  0.049 (0.005) &  0.145 (0.006) &  0.151 (0.006) &  0.150 (0.005) &  0.163 (0.006) &  0.544 (0.013) \\
         &        & $\sigma=2$ &  0.145 (0.018) &  0.134 (0.005) &  0.163 (0.008) &  0.165 (0.007) &  0.160 (0.006) &  0.509 (0.014) \\
\cline{2-9}
         & \multirow{3}{*}{$d=12$} & $\sigma=0.5$ &  0.025 (0.001) &  0.124 (0.005) &  0.143 (0.006) &  0.154 (0.005) &  0.157 (0.005) &  0.530 (0.013) \\
         &        & $\sigma=1$ &  0.043 (0.003) &  0.146 (0.006) &  0.163 (0.007) &  0.171 (0.006) &  0.185 (0.007) &  0.514 (0.016) \\
         &        & $\sigma=2$ &  0.118 (0.015) &  0.125 (0.004) &  0.156 (0.006) &  0.151 (0.004) &  0.184 (0.007) &  0.499 (0.015) \\
\bottomrule
\end{tabular}

\caption{Setup F. Conditional Average Treatment Effect (CATE) estimation. Mean-Squared-Error (MSE), averaged across $100$ experiments, with standard error.}
\label{fig:setupF-cate}
\end{figure}

\subsection{Policy Learning} Our second set of experiments concern the task of policy learning, where the goal is to minimize the loss $\poprisk(\theta) = -\E\left[\theta(X)\, (\tau_0(X) - c)\right]$ (or equivalently, maximize the reward $\E[\theta(X)\, (\tau_0(X) - c)]$), where $\theta: \cX\to \{0, 1\}$ is the treatment policy (the target parameter) and $c\in\bbR$ is a pre-defined treatment cost; for all experiments, we use the mean of the heterogeneous treatment effect as the cost (i.e. $c:=\E[\tau_0(X)]$). We implemented the orthogonal doubly-robust loss (denoted $\text{dr}$) described in \pref{sec:policy_learning_body} (Eq. \pref{eq:binary_treatment}):
\begin{align}
  \poprisk(\theta,g) = -\E\left[\theta(X)\, \left(f(1, X) - f(0, X) + \frac{e(X)-T}{e(X) (1 - e(X))}\, \left(Y - f(T, X)\right) - c\right)\right],
\end{align}
where $g=\crl{f,e}$ is the nuisance parameter, with $f_0(T,X)=\En\brk*{Y\mid{}T,X}$ and $e_0(X)\ldef{}\En\brk*{T\mid{}X}$. We consider two non-orthogonal losses as benchmarks. The first benchmark loss (denoted $\text{direct}$) is based solely on the regression model $f$:
\begin{align}
    \poprisk(\theta,\crl{f}) = -\E\left[\theta(X)\, \left(f(1, X) - f(0, X) - c\right)\right],
\end{align}
while the second benchmark loss (denoted $\text{ips}$), is based solely on the propensity model $e$:
\begin{align}
  \poprisk(\theta,\crl{e}) = -\E\left[\theta(X)\, \left(\frac{e(X) - T}{e(X) (1 - e(X))}\, Y - c\right)\right].
\end{align}
We consider the same data generating process setups (A-F) as in the CATE estimation setting (in particular, the true nuisance parameters $f_0$ and $e_0$ are chosen in the same fashion). As in the CATE estimation experiments, we used FLAML to fit the nuisance and target parameters. We performed hyperparameter tuning, model selection and model fitting on half of each dataset, and performed loss minimization over $\theta$ on the other half. For the target parameter (the treatment policy) we searched over binary decision trees of depth at most $2$, with a minimum leaf size of at least $20$ samples. These trees where constructed using a greedy method, where---starting at the root node---we recursively choose at each node a split that greedily leads to the largest improvement in the target criterion.

\paragraph{Results}
We present our results in Figures~\ref{fig:setupA-policy}, \ref{fig:setupB-policy}, \ref{fig:setupD-policy}, \ref{fig:setupE-policy}, \ref{fig:setupF-policy}, which display the mean reward (negative of loss) for the policies learned by the approaches above; we omit setup C because the treatment effect is constant, which renders the policy learning problem trivial.
Similar to our CATE results, we find that in most settings, the orthogonal loss approach gives comparable results to the non-orthogonal losses, but for setups $E$ and $F$, the orthogonal loss method significantly outperforms the non-orthogonal losses. As a secondary result, we compare to
\begin{itemize}
  \item The unrestricted optimal policy $\theta_{\mathrm{opt}}(X)=1\{\tau(X) - c>0\}$ (denoted $\text{opt}$).
  \item An oracle method (denoted as $\text{or}$) that uses the true nuisance functions and optimizes the oracle objective over the same space of target policies as in $\text{dr}$, $\text{direct}$, and $\text{ips}$.
\end{itemize}
We find that when the sample size is large and the variance of the outcome is relatively small, the learned policies perform comparably to the oracle policy.

\begin{figure}[H]
\scriptsize
\centering
\begin{tabular}{llllllll}
\toprule
         &        &            &            opt &             or &             dr &         direct &            ips \\
\midrule
\multirow{6}{*}{$n=500$} & \multirow{3}{*}{$d=6$} & $\sigma=0.5$ &  0.083 (0.000) &  0.077 (0.000) &  0.021 (0.003) &  0.028 (0.002) &  0.055 (0.002) \\
         &        & $\sigma=1$ &  0.083 (0.000) &  0.077 (0.000) &  0.005 (0.003) &  0.012 (0.002) &  0.043 (0.003) \\
         &        & $\sigma=2$ &  0.083 (0.000) &  0.077 (0.000) &  0.004 (0.002) &  0.001 (0.002) &  0.033 (0.003) \\
\cline{2-8}
         & \multirow{3}{*}{$d=12$} & $\sigma=0.5$ &  0.083 (0.000) &  0.077 (0.000) &  0.019 (0.002) &  0.034 (0.003) &  0.047 (0.003) \\
         &        & $\sigma=1$ &  0.083 (0.000) &  0.077 (0.000) &  0.005 (0.002) &  0.012 (0.002) &  0.041 (0.003) \\
         &        & $\sigma=2$ &  0.083 (0.000) &  0.077 (0.000) &  0.002 (0.002) &  0.004 (0.002) &  0.027 (0.002) \\
\cline{1-8}
\cline{2-8}
\multirow{6}{*}{$n=1000$} & \multirow{3}{*}{$d=6$} & $\sigma=0.5$ &  0.082 (0.000) &  0.077 (0.000) &  0.029 (0.002) &  0.038 (0.003) &  0.060 (0.002) \\
         &        & $\sigma=1$ &  0.082 (0.000) &  0.077 (0.000) &  0.013 (0.002) &  0.019 (0.002) &  0.053 (0.002) \\
         &        & $\sigma=2$ &  0.082 (0.000) &  0.077 (0.000) &  0.005 (0.002) &  0.007 (0.002) &  0.039 (0.003) \\
\cline{2-8}
         & \multirow{3}{*}{$d=12$} & $\sigma=0.5$ &  0.084 (0.000) &  0.078 (0.000) &  0.033 (0.003) &  0.032 (0.003) &  0.060 (0.002) \\
         &        & $\sigma=1$ &  0.084 (0.000) &  0.078 (0.000) &  0.009 (0.002) &  0.016 (0.002) &  0.053 (0.002) \\
         &        & $\sigma=2$ &  0.084 (0.000) &  0.078 (0.000) &  0.005 (0.002) &  0.004 (0.002) &  0.030 (0.003) \\
\cline{1-8}
\cline{2-8}
\multirow{6}{*}{$n=3000$} & \multirow{3}{*}{$d=6$} & $\sigma=0.5$ &  0.083 (0.000) &  0.078 (0.000) &  0.054 (0.002) &  0.040 (0.002) &  0.068 (0.002) \\
         &        & $\sigma=1$ &  0.083 (0.000) &  0.078 (0.000) &  0.032 (0.003) &  0.033 (0.003) &  0.069 (0.001) \\
         &        & $\sigma=2$ &  0.083 (0.000) &  0.078 (0.000) &  0.016 (0.002) &  0.015 (0.002) &  0.063 (0.001) \\
\cline{2-8}
         & \multirow{3}{*}{$d=12$} & $\sigma=0.5$ &  0.080 (0.000) &  0.075 (0.000) &  0.041 (0.003) &  0.028 (0.003) &  0.067 (0.001) \\
         &        & $\sigma=1$ &  0.080 (0.000) &  0.075 (0.000) &  0.017 (0.003) &  0.023 (0.003) &  0.066 (0.001) \\
         &        & $\sigma=2$ &  0.080 (0.000) &  0.075 (0.000) &  0.001 (0.002) &  0.013 (0.002) &  0.056 (0.002) \\
\bottomrule
\end{tabular}

\caption{Setup A. Policy Learning. Learned policy value, averaged across $100$ experiments, with standard error.}
\label{fig:setupA-policy}
\end{figure}
 
\begin{figure}[H]
\scriptsize
\centering
\begin{tabular}{llllllll}
\toprule
         &        &            &            opt &             or &             dr &         direct &            ips \\
\midrule
\multirow{6}{*}{$n=500$} & \multirow{3}{*}{$d=6$} & $\sigma=0.5$ &  0.135 (0.000) &  0.124 (0.000) &  0.078 (0.004) &  0.076 (0.004) &  0.042 (0.005) \\
         &        & $\sigma=1$ &  0.135 (0.000) &  0.124 (0.000) &  0.026 (0.005) &  0.037 (0.005) &  0.031 (0.005) \\
         &        & $\sigma=2$ &  0.135 (0.000) &  0.124 (0.000) &  0.011 (0.004) &  0.013 (0.003) &  0.005 (0.004) \\
\cline{2-8}
         & \multirow{3}{*}{$d=12$} & $\sigma=0.5$ &  0.135 (0.000) &  0.124 (0.000) &  0.058 (0.005) &  0.064 (0.005) &  0.043 (0.005) \\
         &        & $\sigma=1$ &  0.135 (0.000) &  0.124 (0.000) &  0.027 (0.004) &  0.032 (0.005) &  0.020 (0.004) \\
         &        & $\sigma=2$ &  0.135 (0.000) &  0.124 (0.000) &  0.008 (0.002) &  0.004 (0.002) &  0.008 (0.003) \\
\cline{1-8}
\cline{2-8}
\multirow{6}{*}{$n=1000$} & \multirow{3}{*}{$d=6$} & $\sigma=0.5$ &  0.135 (0.000) &  0.124 (0.000) &  0.104 (0.003) &  0.097 (0.002) &  0.081 (0.004) \\
         &        & $\sigma=1$ &  0.135 (0.000) &  0.124 (0.000) &  0.058 (0.005) &  0.052 (0.004) &  0.049 (0.005) \\
         &        & $\sigma=2$ &  0.135 (0.000) &  0.124 (0.000) &  0.015 (0.004) &  0.010 (0.002) &  0.014 (0.004) \\
\cline{2-8}
         & \multirow{3}{*}{$d=12$} & $\sigma=0.5$ &  0.136 (0.000) &  0.125 (0.000) &  0.094 (0.004) &  0.087 (0.004) &  0.068 (0.005) \\
         &        & $\sigma=1$ &  0.136 (0.000) &  0.125 (0.000) &  0.043 (0.005) &  0.051 (0.005) &  0.035 (0.004) \\
         &        & $\sigma=2$ &  0.136 (0.000) &  0.125 (0.000) &  0.007 (0.003) &  0.013 (0.003) &  0.014 (0.003) \\
\cline{1-8}
\cline{2-8}
\multirow{6}{*}{$n=3000$} & \multirow{3}{*}{$d=6$} & $\sigma=0.5$ &  0.136 (0.000) &  0.125 (0.000) &  0.116 (0.002) &  0.106 (0.002) &  0.108 (0.003) \\
         &        & $\sigma=1$ &  0.136 (0.000) &  0.126 (0.000) &  0.102 (0.003) &  0.096 (0.003) &  0.097 (0.004) \\
         &        & $\sigma=2$ &  0.136 (0.000) &  0.126 (0.000) &  0.036 (0.005) &  0.050 (0.004) &  0.045 (0.005) \\
\cline{2-8}
         & \multirow{3}{*}{$d=12$} & $\sigma=0.5$ &  0.131 (0.000) &  0.121 (0.000) &  0.113 (0.002) &  0.107 (0.001) &  0.106 (0.002) \\
         &        & $\sigma=1$ &  0.131 (0.000) &  0.121 (0.000) &  0.091 (0.004) &  0.085 (0.003) &  0.074 (0.005) \\
         &        & $\sigma=2$ &  0.131 (0.000) &  0.121 (0.000) &  0.028 (0.005) &  0.038 (0.004) &  0.032 (0.005) \\
\bottomrule
\end{tabular}

\caption{Setup B. Policy Learning. Learned policy value, averaged across $100$ experiments, with standard error.}
\label{fig:setupB-policy}
\end{figure}
 
\begin{figure}[H]
\scriptsize
\centering
\begin{tabular}{llllllll}
\toprule
         &        &            &            opt &             or &             dr &          direct &            ips \\
\midrule
\multirow{6}{*}{$n=500$} & \multirow{3}{*}{$d=6$} & $\sigma=0.5$ &  0.133 (0.000) &  0.079 (0.001) &  0.039 (0.003) &   0.033 (0.003) &  0.045 (0.003) \\
         &        & $\sigma=1$ &  0.133 (0.000) &  0.080 (0.001) &  0.015 (0.004) &   0.011 (0.002) &  0.023 (0.003) \\
         &        & $\sigma=2$ &  0.133 (0.000) &  0.079 (0.001) &  0.009 (0.003) &  -0.001 (0.002) &  0.011 (0.004) \\
\cline{2-8}
         & \multirow{3}{*}{$d=12$} & $\sigma=0.5$ &  0.133 (0.000) &  0.079 (0.001) &  0.022 (0.003) &   0.014 (0.002) &  0.026 (0.003) \\
         &        & $\sigma=1$ &  0.133 (0.000) &  0.080 (0.001) &  0.007 (0.003) &   0.001 (0.001) &  0.013 (0.003) \\
         &        & $\sigma=2$ &  0.133 (0.000) &  0.079 (0.001) &  0.005 (0.003) &   0.001 (0.001) &  0.005 (0.003) \\
\cline{1-8}
\cline{2-8}
\multirow{6}{*}{$n=1000$} & \multirow{3}{*}{$d=6$} & $\sigma=0.5$ &  0.133 (0.000) &  0.080 (0.001) &  0.056 (0.003) &   0.051 (0.003) &  0.050 (0.003) \\
         &        & $\sigma=1$ &  0.133 (0.000) &  0.081 (0.000) &  0.032 (0.003) &   0.017 (0.003) &  0.019 (0.003) \\
         &        & $\sigma=2$ &  0.133 (0.000) &  0.081 (0.000) &  0.012 (0.003) &   0.003 (0.002) &  0.011 (0.003) \\
\cline{2-8}
         & \multirow{3}{*}{$d=12$} & $\sigma=0.5$ &  0.136 (0.000) &  0.086 (0.000) &  0.051 (0.003) &   0.039 (0.003) &  0.050 (0.003) \\
         &        & $\sigma=1$ &  0.136 (0.000) &  0.085 (0.000) &  0.024 (0.003) &   0.011 (0.002) &  0.027 (0.003) \\
         &        & $\sigma=2$ &  0.136 (0.000) &  0.084 (0.001) &  0.008 (0.002) &   0.006 (0.002) &  0.009 (0.003) \\
\cline{1-8}
\cline{2-8}
\multirow{6}{*}{$n=3000$} & \multirow{3}{*}{$d=6$} & $\sigma=0.5$ &  0.135 (0.000) &  0.086 (0.000) &  0.072 (0.001) &   0.074 (0.001) &  0.072 (0.001) \\
         &        & $\sigma=1$ &  0.135 (0.000) &  0.086 (0.000) &  0.050 (0.003) &   0.051 (0.003) &  0.051 (0.003) \\
         &        & $\sigma=2$ &  0.135 (0.000) &  0.086 (0.000) &  0.025 (0.004) &   0.011 (0.002) &  0.023 (0.004) \\
\cline{2-8}
         & \multirow{3}{*}{$d=12$} & $\sigma=0.5$ &  0.133 (0.000) &  0.083 (0.000) &  0.067 (0.002) &   0.067 (0.002) &  0.062 (0.002) \\
         &        & $\sigma=1$ &  0.133 (0.000) &  0.082 (0.000) &  0.042 (0.003) &   0.038 (0.003) &  0.050 (0.003) \\
         &        & $\sigma=2$ &  0.133 (0.000) &  0.083 (0.000) &  0.008 (0.003) &   0.005 (0.002) &  0.013 (0.003) \\
\bottomrule
\end{tabular}

\caption{Setup D. Policy Learning. Learned policy value, averaged across $100$ experiments, with standard error.}
\label{fig:setupD-policy}
\end{figure}
 
\begin{figure}[H]
\scriptsize
\centering
\begin{tabular}{llllllll}
\toprule
         &        &            &            opt &             or &             dr &         direct &            ips \\
\midrule
\multirow{6}{*}{$n=500$} & \multirow{3}{*}{$d=6$} & $\sigma=0.5$ &  0.463 (0.000) &  0.454 (0.001) &  0.392 (0.009) &  0.185 (0.015) &  0.068 (0.011) \\
         &        & $\sigma=1$ &  0.463 (0.000) &  0.453 (0.001) &  0.243 (0.014) &  0.094 (0.012) &  0.041 (0.008) \\
         &        & $\sigma=2$ &  0.463 (0.000) &  0.454 (0.001) &  0.080 (0.013) &  0.054 (0.010) &  0.022 (0.008) \\
\cline{2-8}
         & \multirow{3}{*}{$d=12$} & $\sigma=0.5$ &  0.460 (0.000) &  0.451 (0.001) &  0.376 (0.010) &  0.137 (0.014) &  0.032 (0.009) \\
         &        & $\sigma=1$ &  0.460 (0.000) &  0.451 (0.001) &  0.227 (0.016) &  0.047 (0.009) &  0.030 (0.007) \\
         &        & $\sigma=2$ &  0.460 (0.000) &  0.453 (0.001) &  0.077 (0.012) &  0.031 (0.008) &  0.014 (0.007) \\
\cline{1-8}
\cline{2-8}
\multirow{6}{*}{$n=1000$} & \multirow{3}{*}{$d=6$} & $\sigma=0.5$ &  0.454 (0.000) &  0.449 (0.001) &  0.436 (0.002) &  0.284 (0.014) &  0.099 (0.012) \\
         &        & $\sigma=1$ &  0.454 (0.000) &  0.449 (0.001) &  0.360 (0.010) &  0.183 (0.015) &  0.065 (0.010) \\
         &        & $\sigma=2$ &  0.454 (0.000) &  0.449 (0.001) &  0.136 (0.015) &  0.048 (0.008) &  0.034 (0.007) \\
\cline{2-8}
         & \multirow{3}{*}{$d=12$} & $\sigma=0.5$ &  0.464 (0.000) &  0.459 (0.001) &  0.441 (0.003) &  0.217 (0.015) &  0.067 (0.010) \\
         &        & $\sigma=1$ &  0.464 (0.000) &  0.459 (0.001) &  0.361 (0.010) &  0.136 (0.013) &  0.038 (0.007) \\
         &        & $\sigma=2$ &  0.464 (0.000) &  0.459 (0.001) &  0.088 (0.011) &  0.053 (0.009) &  0.020 (0.007) \\
\cline{1-8}
\cline{2-8}
\multirow{6}{*}{$n=3000$} & \multirow{3}{*}{$d=6$} & $\sigma=0.5$ &  0.459 (0.000) &  0.457 (0.000) &  0.452 (0.001) &  0.356 (0.009) &  0.122 (0.013) \\
         &        & $\sigma=1$ &  0.459 (0.000) &  0.457 (0.000) &  0.439 (0.002) &  0.291 (0.013) &  0.120 (0.013) \\
         &        & $\sigma=2$ &  0.459 (0.000) &  0.458 (0.000) &  0.350 (0.010) &  0.193 (0.014) &  0.092 (0.012) \\
\cline{2-8}
         & \multirow{3}{*}{$d=12$} & $\sigma=0.5$ &  0.453 (0.000) &  0.451 (0.000) &  0.445 (0.001) &  0.347 (0.012) &  0.114 (0.013) \\
         &        & $\sigma=1$ &  0.453 (0.000) &  0.451 (0.000) &  0.429 (0.003) &  0.259 (0.014) &  0.112 (0.013) \\
         &        & $\sigma=2$ &  0.453 (0.000) &  0.452 (0.000) &  0.297 (0.014) &  0.102 (0.012) &  0.049 (0.009) \\
\bottomrule
\end{tabular}

\caption{Setup E. Policy Learning. Learned policy value, averaged across $100$ experiments, with standard error.}
\label{fig:setupE-policy}
\end{figure}
 
\begin{figure}[H]
\scriptsize
\centering
\begin{tabular}{llllllll}
\toprule
         &        &            &            opt &             or &             dr &         direct &             ips \\
\midrule
\multirow{6}{*}{$n=500$} & \multirow{3}{*}{$d=6$} & $\sigma=0.5$ &  0.135 (0.000) &  0.124 (0.000) &  0.040 (0.005) &  0.022 (0.004) &   0.007 (0.003) \\
         &        & $\sigma=1$ &  0.135 (0.000) &  0.124 (0.000) &  0.015 (0.004) &  0.013 (0.003) &  -0.000 (0.003) \\
         &        & $\sigma=2$ &  0.135 (0.000) &  0.123 (0.000) &  0.005 (0.003) &  0.007 (0.003) &   0.001 (0.004) \\
\cline{2-8}
         & \multirow{3}{*}{$d=12$} & $\sigma=0.5$ &  0.135 (0.000) &  0.124 (0.000) &  0.034 (0.005) &  0.010 (0.003) &   0.002 (0.002) \\
         &        & $\sigma=1$ &  0.135 (0.000) &  0.124 (0.000) &  0.009 (0.003) &  0.006 (0.002) &   0.003 (0.003) \\
         &        & $\sigma=2$ &  0.135 (0.000) &  0.124 (0.000) &  0.004 (0.003) &  0.003 (0.002) &   0.003 (0.003) \\
\cline{1-8}
\cline{2-8}
\multirow{6}{*}{$n=1000$} & \multirow{3}{*}{$d=6$} & $\sigma=0.5$ &  0.135 (0.000) &  0.124 (0.000) &  0.064 (0.004) &  0.032 (0.003) &   0.011 (0.003) \\
         &        & $\sigma=1$ &  0.135 (0.000) &  0.124 (0.000) &  0.030 (0.004) &  0.022 (0.003) &   0.007 (0.002) \\
         &        & $\sigma=2$ &  0.135 (0.000) &  0.124 (0.000) &  0.010 (0.004) &  0.008 (0.002) &   0.003 (0.003) \\
\cline{2-8}
         & \multirow{3}{*}{$d=12$} & $\sigma=0.5$ &  0.136 (0.000) &  0.125 (0.000) &  0.058 (0.005) &  0.026 (0.004) &   0.009 (0.002) \\
         &        & $\sigma=1$ &  0.136 (0.000) &  0.125 (0.000) &  0.017 (0.003) &  0.013 (0.003) &   0.002 (0.002) \\
         &        & $\sigma=2$ &  0.136 (0.000) &  0.125 (0.000) &  0.005 (0.002) &  0.008 (0.003) &  -0.003 (0.002) \\
\cline{1-8}
\cline{2-8}
\multirow{6}{*}{$n=3000$} & \multirow{3}{*}{$d=6$} & $\sigma=0.5$ &  0.136 (0.000) &  0.125 (0.000) &  0.108 (0.002) &  0.060 (0.004) &   0.016 (0.002) \\
         &        & $\sigma=1$ &  0.136 (0.000) &  0.126 (0.000) &  0.074 (0.005) &  0.051 (0.004) &   0.015 (0.002) \\
         &        & $\sigma=2$ &  0.136 (0.000) &  0.126 (0.000) &  0.023 (0.004) &  0.026 (0.003) &   0.004 (0.002) \\
\cline{2-8}
         & \multirow{3}{*}{$d=12$} & $\sigma=0.5$ &  0.131 (0.000) &  0.120 (0.000) &  0.102 (0.003) &  0.057 (0.004) &   0.012 (0.002) \\
         &        & $\sigma=1$ &  0.131 (0.000) &  0.121 (0.000) &  0.060 (0.005) &  0.036 (0.004) &   0.007 (0.002) \\
         &        & $\sigma=2$ &  0.131 (0.000) &  0.121 (0.000) &  0.009 (0.003) &  0.017 (0.003) &  -0.000 (0.001) \\
\bottomrule
\end{tabular}

\caption{Setup F. Policy Learning. Learned policy value, averaged across $100$ experiments, with standard error.}
\label{fig:setupF-policy}
\end{figure}

\part{Additional Results}
\label{part:additional}

\section{Additional Algorithms}
\label{app:additional_algorithms}

\pref{alg:variance_penalization} contains pseudocode for the variance-penalized plug-in empirical risk minimization method described in \pref{sec:variance_penalized}, which is omitted from the main body due to space constraints.

\pref{alg:cross_fitting} and \pref{alg:cross_fitting2} present variants of \pref{alg:sample_splitting} that employ cross-fitting rather than sample splitting, and serve as statistical learning counterparts to the DML1 and DML2 methods described in \cite{chernozhukov2016double}. \pref{alg:cross_fitting2} is specialized to M-estimation losses of the type considered in \pref{sec:erm}, in which $\poprisk(\theta,g)=\En\brk*{\ls(\theta(x),g(w)\midsem{}z)}$ for a point-wise loss function $\ls$. 

\begin{figure}[H]
\begin{framedalgorithm}[Plug-In ERM with Centered Second Moment Penalization]~\\
\textbf{Input}: Sample set $S=z_1,\ldots,z_n$.
\begin{itemize}
\item Split $S$ into subsets $\sampleone$, $\sampletwo$, and $\samplethree$ of equal size.
\item Let $\estone$ be the output of $\mathrm{Alg}(\nuisance,\sampleone)$.
  \item Let $\muhat = \inf_{\theta\in \target} L_{S_3}(\theta, \hat{g})$.
  \item Return $\esttwo = \arg\min_{\theta\in \target} L_{\sampletwo}(\theta, \hat{g}) + 36\delta_nR^{-1}\|\ell(\theta(\cdot),\estone(\cdot);\cdot)-\muhat\|_{L_2(\sampletwo)}$.
\end{itemize}
\label{alg:variance_penalization}
\end{framedalgorithm}
\end{figure}

\begin{figure}[H]
\begin{framedmetaalgorithm}[Two-Stage Estimation with Cross-Fitting (DML1)]~\\
\textbf{Input}: Sample set $S=z_1,\ldots,z_n$, Number of folds $K$.
\begin{itemize}
\item Let $S_1,\ldots,S_K$ be a random $K$-fold partition of $S$ such that each fold $S_i$ has size $n/K$.
  \item For $i=1,\ldots,K$:
    \begin{itemize}
    \item Let $\estone_i$ be the output of $\mathrm{Alg}(\nuisance,S_i^{c})$, where $S_i^{c}=S\setminus{}S_i$.
    \item Let $\esttwo_i$ be the output of $\mathrm{Alg}(\target,S_i\midsem\estone_i)$.
    \end{itemize}
  \item Return $\esttwo = \frac{1}{K}\sum_{i=1}^{K}\esttwo_i$.
\end{itemize}
\label{alg:cross_fitting}
\end{framedmetaalgorithm}
\end{figure}
\vspace{-10pt}
\begin{figure}[H]
\begin{framedmetaalgorithm}[Two-Stage Estimation with Cross-Fitting (DML2)]~\\
\textbf{Input}: Sample set $S=z_1,\ldots,z_n$, Number of folds $K$.
\begin{itemize}
\item Let $S_1,\ldots,S_K$ be a random $K$-fold partition of $S$ such that each fold $S_i$ has size $n/K$.
  \item For $i=1,\ldots,K$:
    \begin{itemize}
    \item Let $\estone_i$ be the output of $\mathrm{Alg}(\nuisance,S_i^{c})$, where $S_i^{c}=S\setminus{}S_i$.
    \end{itemize}
  \item Use any algorithm $\mathrm{Alg}(\target,\sampletwo\midsem \{\estone_i\}_{i=1}^K)$ that returns $\esttwo$ which achieves average plug-in excess risk:
  \begin{align}
    \frac{1}{K} \sum_{k=1}^K \left( L_D(\theta; \hat{g}^{(k)}) -L_D(\theta^*; \hat{g}^{(k)})\right) \leq \Rate(\target,  S, \delta\midsem{}g)
    \end{align}
\end{itemize}
\label{alg:cross_fitting2}
\end{framedmetaalgorithm}
\end{figure}

We note that most the natural instantiation of the DML2
  meta-algorithm is to use a second-stage algorithm that minimizes the
  average empirical loss across the folds. In particular, assuming the empirical loss takes the form
  $L_{S}(\theta,g)\ldef{}\sum_{i\in{}S}\ls(\theta(x_i),g(w_i); z_i)$,
  one can apply plug-in ERM to the average empirical risk across the folds
\begin{align}
  \esttwo = \argmin_{\theta\in\Theta}\frac{1}{K}\sum_{k=1}^{K}L_{S_k}(\theta, \estone\ind{k}).
\end{align}
The localized Rademacher complexity techniques we develop in this
paper can be adapted to this method to provide average excess risk
bounds of order $\delta_{n/K}\cdot\nrm{\esttwo -
  \besttwo}_{L_2(\ell_2,\cD)}+\delta_{n/K}^2$, where $\delta_n$ is as
described in \pref{thm:fast_erm}. This guarantee can then be combined
with our main theorems to achieve oracle excess risk bounds with
second order dependence. For example, the follow-up work of
\cite{dao2020knowledge} invokes such an analysis in the context of
knowledge distillation. \pref{alg:cross_fitting2} is stated in a more
general form to allow for second stage algorithms that go beyond
plug-in ERM (e.g., penalized ERM variants or aggregation methods).

%
%
%
%
%
%

%
%
%
%
 
\section{Orthogonal Statistical Learning: User-Friendly Tools}
\label{sec:user_friendly}
Our main results, \pref{thm:generic_strongly_convex} and \pref{thm:orthogonal_slow}, give excess risk bounds for \pref{alg:sample_splitting} (with generic nuisance and target estimators) under Neyman orthogonality. In this section we provide some additional consequences and variants of these results which will prove useful in deriving guarantees for specific estimators.

The first result gives a consequence of \pref{thm:generic_strongly_convex} for the case where the target estimator satisfies a certain self-bounding property. Such is the case for plug-in empirical risk minimization.
\begin{lemma}
  \label{lem:self_bounding}
  Suppose that the conditions of \pref{thm:generic_strongly_convex} hold, and that for all $g\in\nuisance$,
  \[
    \Rate(\target, n, \delta;\esttwo,g)
    \leq{}
    \veps_n(\delta)\cdot{}\nrm[\big]{\esttwo-\besttwo}_{\target}
    + \alpha_n(\delta),
  \]
  for functions $\veps_n(\delta)$ and $\alpha_n(\delta)$. Then the sample splitting meta-algorithm (\pref{alg:sample_splitting}) produces an estimate $\esttwo$ such that with probability at least $1-\delta$,
\begin{equation}
\label{eq:self_bounding1}
\nrm[\big]{\esttwo-\besttwo}_{\target}^{2} \leq{}
C_1^2\cdot{}\veps^2_{n/2}(\delta/2) +
2C_1\cdot\alpha_{n/2}(\delta/2) + 
2C_2\cdot{}\prn*{
\Rate(\nuisance,\sampleone,\delta/2)
}^{\frac{4}{1+r}},
\end{equation}
and
\begin{equation}
\label{eq:self_bounding2}
\begin{aligned}
L_{\cD}(\esttwo,\gtone) - L_{\cD}(\besttwo, \gtone)
\leq{}
2\beta_1C_1^2\cdot{}\veps^2_{n/2}(\delta/2) + 2\beta_1C_1\alpha_{n/2}(\delta/2)+
2\beta_1C_2\cdot{}\prn*{
\Rate(\nuisance,\sampleone,\delta/2)
}^{\frac{4}{1+r}},
\end{aligned}
\end{equation}
where $C_1$ and $C_2$ are defined as in \pref{thm:generic_strongly_convex}.
\end{lemma}

\begin{proof}[\pfref{lem:self_bounding}]
  \pref{thm:generic_strongly_convex} implies that with probability at least $1-\delta$,
  \begin{align*}
\nrm[\big]{\esttwo-\besttwo}_{\target}^{2} &\leq{}
                                             C_1\cdot\Rate(\target,\sampletwo,\delta/2\midsem{}\esttwo,\estone)
                                             +
C_2\cdot
\prn*{\Rate(\nuisance,\sampleone,\delta/2)
                                             }^{\frac{4}{1+r}},\\
                                           &\leq{} C_1\cdot\veps_{n/2}(\delta/2)\cdot{}\nrm[\big]{\esttwo-\besttwo}_{\target}+C_1\cdot\alpha_{n/2}(\delta/2)
+ C_2\cdot{}\prn*{
\Rate(\nuisance,\sampleone,\delta/2)
                                             }^{\frac{4}{1+r}},
                                             \end{align*}
where $C_1$ and $C_2$ are as in \pref{thm:generic_strongly_convex}. By the AM-GM inequality, we can upper bound this by
\begin{align*}
                                           &\leq{}
                                             \frac{C_1^2}{2}\veps^2_{n/2}(\delta/2)
                                             + C_1\alpha_{n/2}(\delta/2)
                                             + \frac{1}{2}\nrm[\big]{\esttwo-\besttwo}_{\target}^2+
C_2\cdot{}\prn*{
\Rate(\nuisance,\sampleone,\delta/2)
                                             }^{\frac{4}{1+r}}.
  \end{align*}
Rearranging yields \pref{eq:self_bounding1}. To prove
\pref{eq:self_bounding2}, we begin by applying,
\pref{ass:smooth_loss}, which gives that
\[
  L_{\cD}(\esttwo,\gtone) - L_{\cD}(\besttwo, \gtone)
  \leq D_{\theta}L_{\cD}(\besttwo,\gtone)[\esttwo-\besttwo]
  + \frac{\beta_1}{2}\nrm*{\theta-\besttwo}_{\target}^{2}
\]
Applying \pref{lem:first_order_bound}, we have
\begin{align*}
D_{\theta}L_{\cD}(\besttwo,\gtone)[\esttwo-\besttwo] \leq{}  
 \Rate(\target,\sampletwo,\delta/2\midsem{}\esttwo,\estone) +
 C \cdot{}\prn*{
\Rate(\nuisance,\sampleone,\delta/2)
}^{\frac{4}{1+r}}
 -\frac{\lambda}{4}\cdot{}\nrm[\big]{\esttwo-\besttwo}_{\target}^{2},
\end{align*}
 where $C\leq{}
 \frac{\lambda}{2}\prn[\bigg]{\prn[\bigg]{\frac{\beta_2}{\lambda}}^{\frac{2}{1+r}}
   + \frac{\kappa}{\lambda}}$. We further bound this by
   \begin{align*}
     D_{\theta}L_{\cD}(\besttwo,\gtone)[\esttwo-\besttwo]
     &\leq{}
           \veps_{n/2}(\delta/2)\cdot{}\nrm[\big]{\esttwo-\besttwo}_{\target}
    + \alpha_{n/2}(\delta/2)
+
 C \cdot{}\prn*{
\Rate(\nuisance,\sampleone,\delta/2)
}^{\frac{4}{1+r}}
       -\frac{\lambda}{4}\cdot{}\nrm[\big]{\esttwo-\besttwo}_{\target}^{2}, \\
            &\leq{}
              \frac{1}{\lambda}\veps^2_{n/2}(\delta/2)
    + \alpha_{n/2}(\delta/2)
+
 C \cdot{}\prn*{
\Rate(\nuisance,\sampleone,\delta/2)
}^{\frac{4}{1+r}},
   \end{align*}
              where we have applied the AM-GM inequality. It follows
              that
              \begin{align*}
                &L_{\cD}(\esttwo,\gtone) - L_{\cD}(\besttwo, \gtone) \\
                &\leq{} \prn*{\frac{\beta_1C_1^2}{2}+\frac{1}{\lambda}}\cdot{}\veps^2_{n/2}(\delta/2) + (1+\beta_1C_1)\alpha_{n/2}(\delta/2)+
(\beta_1C_2+C)\cdot{}\prn*{
\Rate(\nuisance,\sampleone,\delta/2)
                  }^{\frac{4}{1+r}}, \\
                                &\leq{} 2\beta_1C_1^2\cdot{}\veps^2_{n/2}(\delta/2) + 2\beta_1C_1\alpha_{n/2}(\delta/2)+
2\beta_1C_2\cdot{}\prn*{
\Rate(\nuisance,\sampleone,\delta/2)
}^{\frac{4}{1+r}},
              \end{align*}
              where we have applied \pref{eq:self_bounding1} and then used
              that $\lambda\leq{}\beta_1$ to simplify.
\end{proof}

The next lemma we provide is a variant of
\pref{thm:generic_strongly_convex} which gives a bound on the first
derivative of the oracle risk for \pref{alg:sample_splitting}.

\begin{lemma}[Variant of \pref{thm:generic_strongly_convex}]
\label{lem:first_order_bound}
Suppose there exists
$\besttwo\in\argmin_{\theta\in\target}\poprisk(\theta,\gtone)$ such
that
\pref{ass:orthogonal,ass:well_specified,ass:strong_convex_loss,ass:smooth_loss},
are satisfied. Then the sample splitting meta-algorithm
(\pref{alg:sample_splitting}) produces a \param $\esttwo$ such that
with probability at least $1-\delta$,
\begin{equation}
  \label{eq:first_order_bound}
D_{\theta}L_{\cD}(\besttwo,\gtone)[\esttwo-\besttwo] \leq{}  
 \Rate(\target,\sampletwo,\delta/2\midsem{}\esttwo,\estone) +
 C \cdot{}\prn*{
\Rate(\nuisance,\sampleone,\delta/2)
}^{\frac{4}{1+r}}
 -\frac{\lambda}{4}\cdot{}\nrm[\big]{\esttwo-\besttwo}_{\target}^{2},
\end{equation}
where $C\leq{} \frac{\lambda}{2}\prn[\bigg]{\prn[\bigg]{\frac{\beta_2}{\lambda}}^{\frac{2}{1+r}} + \frac{\kappa}{\lambda}}$.
\end{lemma}
\begin{proof}[\pfref{lem:first_order_bound}]
See \pref{eq:strongly_convex_quadratic} within the proof of \pref{thm:generic_strongly_convex}.
\end{proof}

The next lemma (\pref{lem:pseudo_risk}) uses orthogonality to
  provide a bound on the plug-in excess risk in terms of the target
  error and nuisance error. Recall that our main theorems give a bound
  on the \emph{oracle} excess risk as a function of the plug-in excess
  risk. \pref{lem:pseudo_risk} is---in some sense---a converse, upper
  bounding the plug-in excess risk with the estimation error
  $\nrm*{\theta-\gttwo}_{\target}$ under orthogonality. This type of
  guarantee can be useful as an intermediate result when analyzing the
  plug-in excess risk for second stage algorithms. In this paper, we
  use it in the context of Skeleton Aggregation (\pref{app:rates}), to control misspecification error for the pseudo-excess risk.

\begin{lemma}
  \label{lem:pseudo_risk}Suppose that \pref{ass:orthogonal} holds, that
  \[
    D_{\theta}\poprisk(\gttwo,\gtone)[\theta-\gttwo]= 0, \quad{}\forall{}\theta\in\starhull(\target,\gttwo),
  \]
  and that the following derivative bounds hold:
   \begin{itemize}
    \item $\forall{}\theta\in\target,\bar{\theta}\in\starhull(\target,\gttwo),g\in\nuisance$:
    \begin{align*}
      D^{2}_{\theta}\poprisk(\bar{\theta},g)[\theta-\gttwo,\theta-\gttwo] \leq{}\beta_1\cdot \nrm*{\theta-\gttwo}_{\target}^{2}
        + \kappa_1\cdot \nrm*{g-\gtone}_{\nuisance}^{4}
    \end{align*}
    \item $\forall{}\theta\in\target,g\in\nuisance,\bar{g}\in\starhull(\cG,\gtone)$:
    \begin{align*}
    \abs*{D^{2}_{g}D_{\theta}\poprisk(\gttwo,\bar{g})[\theta-\gttwo,g-\gtone,g-\gtone]} 
      \leq{}~ \beta_2\cdot \nrm*{\theta-\gttwo}_{\target}^2 +\kappa_2\cdot\nrm*{g-\gtone}_{\nuisance}^4
  \end{align*}
  \end{itemize}
  Then for all $\theta\in\target$,
  \begin{equation}
    \label{eq:pseudo_risk}
    \poprisk(\theta,\estone) - \poprisk(\gttwo,\estone)
    \leq{} \frac{1}{2}(\beta_1+\beta_2)\cdot \nrm*{\theta-\gttwo}_{\target}^{2}
      + \frac{1}{2}(\kappa_1+\kappa_2)\cdot \nrm*{g-\gtone}_{\nuisance}^{4}.
  \end{equation}
\end{lemma}

\begin{proof}[\pfref{lem:pseudo_risk}]
Let $\theta$ be fixed. Using a second-order Taylor expansion, there exists $\bar{\theta}\in\starhull(\target,\gttwo)$ such that
\begin{align*}
  \poprisk(\theta,\estone) - \poprisk(\gttwo,\estone)
  =
D_{\theta}\poprisk(\gttwo,\estone)[\theta-\gttwo] + \frac{1}{2}D^{2}_{\theta}\poprisk(\bar{\theta},\estone)[\theta-\gttwo,\theta-\gttwo].
\end{align*}
Using another second-order Taylor expansion, there exists $\bar{g}\in\starhull(\nuisance,\gtone)$ for which
\begin{align*}
D_{\theta}\poprisk(\gttwo,\estone)[\theta-\gttwo]
&= D_{\theta}\poprisk(\gttwo,\gtone)[\theta-\gttwo]
+ D_{g}D_{\theta}\poprisk(\gttwo,\gtone)[\theta-\gttwo,\estone-\gtone]\\
  &~~~~+ \frac{1}{2}D^{2}_{g}D_{\theta}\poprisk(\gttwo,\bar{g})[\theta-\gttwo,\estone-\gtone,\estone-\gtone]\\
  &=\frac{1}{2}D^{2}_{g}D_{\theta}\poprisk(\gttwo,\bar{g})[\theta-\gttwo,\estone-\gtone,\estone-\gtone],
\end{align*}
where the second equality uses the first-order condition and orthogonality. The result now follows from the assumed derivative bounds.
  
\end{proof}

\section{Construction of Orthogonal Losses}
\label{sec:orthogonal_loss}

While orthogonal losses are already known for many problem settings
and statistical models (treatment effect estimation, policy learning,
regression with missing/censored data, and so on), for new
problems we often begin with a loss which is \emph{not} necessarily
orthogonal. A natural question, which we address now, is whether one
can modify the loss to satisfy orthogonality so that our main theorems
can be applied.

Suppose we begin with a loss $\ls(\theta(x),g;z)$ such that the
nuisance and target parameter are specified by the moment equations
\begin{equation}
  \label{eq:non_orthogonal_moment}
  \begin{aligned}
    \En\brk*{\grad_{\zeta}\ls(\gttwo(x),\gtone;z)\mid{}x}=0,\\
    \En\brk*{u - \gtone(w)\mid{}w} = 0,
  \end{aligned}
\end{equation}
where $u\subseteq{}z$ is a random variable, $x\subseteq{} w$, and $\grad_{\zeta}$
denotes the derivative with respect to the first
argument. If $\poprisk(\theta,g)=\En_{z}\brk*{\ls(\theta(x),g(w);z)}$
is not orthogonal, we can construct an orthogonal loss using a generalization of a construction in
\cite{chernozhukov2018plugin}. For simplicity, we sketch the approach
for the special case where $\gttwo$ is scalar-valued.

To begin, assume that there exists a function $a_0$ such that
for all $x\in\cX$, we have
\begin{equation}\label{eqn:correction}
D_{g}\En\brk*{\grad_{\zeta}\ls(\gttwo(x),\gtone;z)\mid{} x}[g - \gtone] = \En\brk*{\tri*{a_0(w),g(w) - \gtone(w)} \mid{} x}.
\end{equation}
Under this assumption, we can expand our nuisance parameters to
include $a_0$---that is, define $\tilde{g}_0\ldef\crl*{g_0,a_0}$---and construct a new orthogonal loss:
\begin{equation}
  \label{eq:orthogonalized_loss}
  \tilde{\ls}(\theta(x),\tilde{g};z) \ldef{} \ls(\theta(x),g;z) + \tri*{a(w),u-g(w)}\cdot{}\theta(x).
\end{equation}
Letting $\tilde{L}_{\cD}(\theta,\tilde{g})=\En\brk[\big]{\tilde{\ls}(\theta(x),\tilde{g};z)}$
be the new population risk, we have the following claim.
\begin{lemma}
  \label{lem:orthogonal_loss}
The population risk $\tilde{L}_{\cD}(\theta,\tilde{g})$ satisfies \pref{ass:orthogonal}
and \pref{ass:well_specified}.
\end{lemma}

\begin{proof}[\pfref{lem:orthogonal_loss}]
The first-order condition (\pref{ass:well_specified}) follows immediately by using that
$\En\brk*{\grad_{\zeta}\ls(\gttwo(x),\gtone;z)\mid{}x}=0$, and by the assumption that $\En\brk*{u\mid{}w}=g_0(w)$. To establish orthogonality with respect to the nuisance parameter $g$, we compute
\begin{align*}
  &D_{g}D_{\theta}\tilde{L}_{\cD}(\gttwo,\crl*{\gtone,a_0})[\theta-\theta_0,g-g_0]\\
  &=
\En\brk*{D_g\En\brk*{\grad_{\zeta}\ls(\gttwo(x),\gtone(w);z)\mid{}x}\brk{g-g_0}(\theta(x)-\theta_0(x))}\\
  &~~~~-\En\brk*{\tri*{a_0(w),g(w)-g_0(w)}(\theta(x)-\theta_0(x))}\\
    &=
    \En\brk*{\En\brk*{\tri*{a_0(w),g(w)-g_0(w)}\mid{}x}(\theta(x)-\theta_0(x))} - \En\brk*{\tri*{a_0(w),g(w)-g_0(w)}(\theta(x)-\theta_0(x))}\\
  &=0,
\end{align*}
where the second equality follows from the definition of $a_0$ and the final inequality is the law of total expectation. For
orthogonality with respect to $a$, we have
\begin{align*}
  D_{\theta}D_{a}\tilde{L}_{\cD}(\gttwo,\crl*{\gtone,a_0})[\theta-\theta_0,a-a_0]
  = \En\brk*{\tri*{a(w)-a_0(w),u-g_0(w)}(\theta(x)-\theta_0(x))}
  = 0,
\end{align*}
where the final inequality uses that $x\subseteq{}w$ and $\En\brk*{u\mid{}w}=g_0(w)$.
\end{proof}

As a first example, in the special case where the loss depends on $g_0$ only
through its evaluation at $w$ (i.e., \pref{eq:non_orthogonal_moment} simplifies to
$\En\brk*{\grad_{\zeta}\ls(\gttwo(x),\gtone(w); z)\mid{} x}=0$), then
we can take
\begin{equation}
a_0(w) = \En\brk*{\grad_{\gamma}\grad_{\zeta}\ls(\gttwo(x), \gtone(w);z)\mid{}w}.
\end{equation}
Of course, to make use of the lemma, we must be able to estimate the
new nuisance parameter $a_0$. This can be accomplished through an
additional plug-in estimation step based on sample splitting: Split
$S$ into folds $S_1$, $S_2$, $S_3$, and $S_4$ of equal size. Estimate
$\wh{g}$ on $S_1$, then obtain an initial estimate
$\wh{\theta}_{\mathrm{init}}$ for $\theta_0$ by solving
$\argmin_{\theta\in\Theta}L_{S_2}(\theta,\wh{g})$, where $L_{S_2}$
denotes the empirical loss over $S_2$. Next, use the
initial estimator to compute a plug-in
estimator $\wh{a}$ for $a_0$ by regressing onto the ``targets''
$\grad_{\zeta}\grad_{\gamma}\ls(\wh{\theta}_{\mathrm{init}}(x),\wh{g}(w);z)$
on $S_3$; \revoneedit{this requires access to an additional function class
$\Aclass$ containing $\gta$, and the resulting guarantee will have
(second-order) dependence on the complexity of this class.} Finally, produce the main estimator for the
target \param by solving
$\wh{\theta}=\argmin_{\theta\in\Theta}\wt{L}_{S_4}(\theta,\crl*{\wh{g},\wh{a}})$. \revoneedit{A
  full description is given in \pref{alg:automatic_debiasing}.}
\begin{figure}[H]
\begin{framedalgorithm}[Plug-In ERM with Automatic Debiasing]~\\
\textbf{Input}: Sample set $S=z_1,\ldots,z_n$, Function class $\cA$ containing $a_0$ (cf. \pref{eqn:correction})
\begin{itemize}
\item Split $S$ into subsets $\sampleone$, $\sampletwo$, $\samplethree$, and $\samplefour$ of equal size.
  \item Let $\estone$ be the output of $\mathrm{Alg}(\nuisance,\sampleone)$.
\item Let $\wh{\theta}_{\mathrm{init}}=\argmin_{\theta\in\Theta}L_{S_2}(\theta,\wh{g})$, where $L_{S_2}$
  denotes the empirical loss over $S_2$.
\item Estimate $a_0$ via $\esta=\argmin_{a\in\Aclass}\frac{1}{\abs{\samplethree}}\sum_{z\in\samplethree}\nrm[\big]{a(w) - \grad_{\zeta}\grad_{\gamma}\ls(\wh{\theta}_{\mathrm{init}}(x),\wh{g}(w);z)}_{2}^{2}$.
\item Return $\esttwo=\argmin_{\theta\in\Theta}\wt{L}_{S_4}(\theta,\crl*{\wh{g},\wh{a}})$, where $\tilde{\ls}$ is defined as in \pref{eq:orthogonalized_loss}.
\end{itemize}
\label{alg:automatic_debiasing}
\end{framedalgorithm}
\end{figure}
\vspace{-10pt}

The key idea behind this scheme is that the initial estimator
$\wh{\theta}_{\mathrm{init}}$ will not be able to take advantage of
orthogonality, but its estimation error will only enter the final
bound through the error of $\wh{a}$, and thus will only have
higher-order impact on the rate. \revoneedit{We omit details, which
  are beyond the scope of the present paper.}
%
%
%
%
%
%
%

This approach is applicable for the problem of estimating
utility functions in models of strategic competition, as used in
\cite{chernozhukov2018plugin}; see \pref{app:orthogonal_loss_examples}
for a detailed example. For some models---including utility function
estimation---$a_0$ is a known function of $\theta_0$ and $g_0$, so
that the extra regression to estimate $\wh{a}$ given the initial
estimators is not required.

A more general setting where the loss has the form in
\pref{eqn:correction} is as follows. Suppose that all functions
$g\in\cG$ are conditionally square-integrable in the sense that for
all $x$, $\En\brk{g^2(w)\mid{} x}<\infty$, and suppose there exist
functions $\beta_0$, $T_x$ such that we can
write
\[
	\En\brk*{\grad_{\zeta}\ls(\gttwo(x), g;z)\mid{}x} = \beta_0(x) + T_x(g),
\]
where $T_x(g)$ is a linear operator on $g$ with uniformly bounded operator norm:
\begin{equation}
\sup_{x}\|T_x(g)\|_{\text{op}}:= \sup_{x, g \neq 0} \frac{\En\brk*{\grad_{\zeta}\ls(\gttwo(x), g;z)\mid{}x}}{\sqrt{\En\brk*{g^2(w)\mid{}x}}}  < \infty.
\end{equation}
By the Riesz-Frechet representation theorem, we can express the
operator $T_x$ as
\begin{align}
T_x(g) = \En\brk*{a_0(w)\, g(w)\mid{} x},
\end{align}
where we have used that $x\subseteq w$ to simplify. Hence, we have
\begin{align}
\En\brk*{\grad_{\zeta}\ls(\gttwo(x), g;z)\mid{}x} = \beta_0(x) + \En\brk*{a_0(w)\, g(w)\mid{} x},
\end{align}
so that
\pref{eqn:correction} is satisfied for $a_0$ induced by the family of Riesz representers for operators $T_x$. This is a variant of the Riesz representer approach presented
in \cite{chernozhukov2018riesz}. In
\pref{app:orthogonal_loss_examples}, we show that this construction
recovers the treatment effect estimation example presented in the introduction.

We mention in passing that that another, perhaps more standard,
approach to constructing an orthogonal loss is to derive the influence function for
the risk function $\poprisk$ using standard calculations from
semiparametric theory \citep{van2003unifiedb,tsiatis2007semiparametric,kosorok2008introduction,kennedy2016semiparametric}.

\subsection{Orthogonal Loss Construction: Examples}
\label{app:orthogonal_loss_examples}
We now walk through concrete examples of
the orthogonal loss construction approach outlined in
\pref{sec:orthogonal_loss}. Both examples use the loss structure in
\pref{eq:non_orthogonal_moment}.
\paragraph{Estimating utility functions in models of strategic competition}
In this setting, we have
\[
\theta(x) = \begin{pmatrix} \psi(x)\\ \Delta\end{pmatrix},\quad x=w,\quad \text{and}\quad \grad_{\zeta}\ls(\theta(x), g;z)=\left({\cal L}(\psi(x) + \Delta\, g(x)) - y\right)\cdot \begin{pmatrix} 1\\ g(x)\end{pmatrix},
\] where ${\cal L}$ is the logistic function. The motivation of this problem stems from estimating games of incomplete information, where $y$ is the entry decision of one player, $u$ is the entry decision of the opponent, $x$ is a featurized state of the world, $\psi(x)$ is the non-strategic part of the utility of the player and $\Delta\, g$ is the competitive part of the utility, i.e. the effect of the opponent's entry decision on the player's utility. For this setting, we can take the auxiliary nuisance variable to be
\begin{equation}
a_0(x) = \En\brk*{\grad_{\gamma}\grad_{\zeta}\ls(\gttwo(x), \gtone(x);z)\mid{}x} = \Delta_0 {\cal L}'(\psi_0(x) + \Delta_0\, g_0(x))
\begin{pmatrix} 1, & g_0(x)
\end{pmatrix}
\end{equation}
Thus, $a_0(w)$ is a known function of $\theta_0$ and $g_0$, and to estimate $a_0$ it suffices to construct preliminary estimates for $\theta_0$ and $g_0$ (e.g., using plug-in estimation with the original non-orthogonal loss).
We can simplify the final orthogonal loss based our generic construction to
\begin{equation}
  \tilde{\ls}(\theta(x),\tilde{g};z) \ldef{} \ls(\theta(x),g;z) + (u-g(w))\cdot{} \tilde{\Delta} {\cal L}'(\tilde{\psi}(x) + \tilde{\Delta}\, g(x)) \cdot{} (\theta(x) + \Delta g(x)),
\end{equation}
where $\tilde{g}=\crl{g, \tilde{\theta}}$ and $\tilde{\theta}=\begin{pmatrix} \tilde{\psi}(x)\\ \tilde{\Delta}\end{pmatrix}$ is a preliminary estimate for $\theta_0$.

\paragraph{Treatment effect estimation}
In this setting, we denote the treatment as $d\in \{0, 1\}$ and $w=(d,
x, v)$, for some vector of extra control variables $v$.
In this setting, we take $z = (x,w,y)$ and $w=(x,v,d)$, where $v$ is a
vector of extra control variables and $d\in\crl*{0,1}$ is a binary
treatment. We begin from the moment equation
\begin{align}
\grad_{\zeta}\ls(\theta(x),g;z) :=~&\theta(x) - g(x, v,1) + g(x, v,0),
\end{align}
with $g_0$ identified by the local moment equation $\En\brk*{y -
  g_0(x, v, d) \mid{} x, v, d}=0$. In this case, the Riesz representer
takes the form $a_0(w) = \frac{d - (1 - d)}{\Pr\brk*{D=d\mid{} x,
    v}}$, and the resulting orthogonal loss created using the generic
construction takes the form
\begin{align}
\tilde{\ell}(\theta(x), \tilde{g}; z) =~& (\theta(x) - g(1, x, v) + g(0, x, v))^2 + (y - g(x, v, d))\cdot a_0(w)\cdot \theta(x),
\end{align}
where we recall $\tilde{g}=\crl{g,a}$. Minimizing this over $\theta$ is equivalent to minimizing the loss
\begin{align}
\tilde{\ell}(\theta(x), \tilde{g}; z) =~& (\theta(x) - g(1, x, v) + g(0, x, v) + (y - g(x, v, d))\cdot a_0(w))^2 ,
\end{align}
which is precisely the loss presented in the introduction of the paper.

\section{Sufficient Conditions for Theorems
  \journal{\ref{thm:generic_strongly_convex}} and
  \journal{\ref{thm:orthogonal_slow}}: Single Index Losses}
\label{sec:sufficient}

Our setup and main results in \pref{sec:orthogonal}
(\pref{thm:generic_strongly_convex} and \pref{thm:orthogonal_slow}) are stated at a high level of generality, with abstract assumptions on the structure of the population risk. In this section of the appendix we provide conditions under which these assumptions follow from concrete structural assumptions on the risk. We give sufficient guarantees for general families of loss functions. In particular, the conditions we give here suffice to derive guarantees for the applications considered in \pref{sec:applications}.

\subsection{Fast Rates}
\label{ssec:single_index}

In this section we give a broad class of losses under which the conditions for fast rates in \pref{ssec:orthogonal_strongly_convex} are satisfied. The population loss $\poprisk$ is defined as the expectation of a point-wise loss $\ls(\zeta,\gamma;z)$ acting on the predictions of the nuisance and target parameters. We assume existence of functions $\Phi$ and $\Lambda$ such that the loss has the structure
\[
\ls(\zeta,\gamma;z) = \Phi(\tri{\Lambda(\gamma,v),\zeta}, \gamma, z),
\]
so that
\begin{equation}\label{eqn:single_index}
\ls(\theta(\vartwo),g(\varone);z) = \Phi(
\tri{\Lambda(g(\varone), v), \theta(\vartwo)}, g(\varone), \varall), \quad\poprisk(\theta,g)=\En\brk{\ls(\theta(\vartwo),g(\varone); z)}.
\end{equation}
Here we recall from \pref{sec:setup} that $x$, $w$ are subsets of the data $z$, and let $v\subseteq{}z$ be an auxiliary subset of the data. We also assume existence of functions $\phi(\zeta)$ and $\Gamma(\gamma,z)$ such that the partial derivative of $\Phi$ may be written as
\begin{equation}
\frac{\partial }{\partial t} \Phi(t, \gamma, \varall) = \phi(t) - \Gamma(\gamma, z),
\end{equation}
where $\phi$ is non-decreasing and Lipschitz. A simple example is square loss regression, where $\Phi(t, \gamma, z) = \frac{1}{2}(t - \Gamma(\gamma, z))^2$. 

\begin{revone}
To provide fast rates of the type in \pref{sec:orthogonal} for losses with this structure, we let $p\geq{}1$ be given, define $\frac{1}{q}=1-\frac{1}{p}$, and consider the norms
\begin{align}
  \label{eq:single_index_target_norm}
  &\nrm*{\theta}_{\target} \ldef \prn[\Big]{\En_{z}\brk[\big]{
    \tri*{\Lambda(g_{0}(\varone),v), \theta(\vartwo)}^{2}}}^{1/2},\intertext{and}
  &\nrm*{g}_{\nuisance} \ldef \nrm*{g}_{L_{2p}(\ell_2, \cD)}.  \label{eq:single_index_nuisance_norm}
\end{align}
Taking $p$ large leads to a more stringent notion of distance for the nuisance parameter, but weakens the regularity conditions we consider, which depend on the dual parameter $q$. In particular, we make the following assumptions.
\end{revone}

\begin{assumption}[Sufficient Conditions for Fast Rates: Single Index Losses]
\label{ass:single_index}
The loss $\ls$ satisfies the following conditions for parameters $(\musi,\Tsi,\tausi,\Lsi, \gammasi, \rsi, \Rsi)$:
\begin{align}
\En\brk*{\nabla_{\gamma}\nabla_{\zeta} \ell(\besttwo(\vartwo), \gtone(\varone); z) \mid \varone} = 0. & &\text{($\Rightarrow$ \pref{ass:orthogonal})}\\
\En\brk*{\nabla_{\zeta} \ell(\besttwo(\vartwo), \gtone(\varone); z) \cdot (\theta(\vartwo) - \besttwo(\vartwo))} \geq 0,  & \quad\forall\theta\in\starhull(\targetemp,\besttwo).& \text{($\Rightarrow$ \pref{ass:well_specified})}\\
\nrm*{\En\brk*{\nabla_{\gamma\gamma}^2 \nabla_{\zeta_i} \ell(\besttwo(\vartwo), g(\varone);z) \mid \varone}}_{\sigma}\leq \musi, & \quad \forall i, \forall{}g\in\starhull(\nuisance,\gtone). & \text{($\Rightarrow$ 
\preprint{\savehyperref{ass:smooth_loss}{Assumption \ref*{ass:smooth_loss}}}
\journal{Assumption \savehyperref{ass:smooth_loss}{\ref*{ass:smooth_loss}}})}\\
\Tsi\geq \phi'(t)\geq \tausi. & & \text{($\Rightarrow$ \pref{ass:strong_convex_loss,ass:smooth_loss})}\\
\nrm*{\theta-\besttwo}_{\target}^{1-\rsi}
\geq \gammasi^{1/2}\nrm{\theta-\besttwo}_{L_q(\cD,\ls_2)},&\quad\forall\theta\in\targetemp.
                                                                                                                                 & \text{($\Rightarrow$ \pref{ass:strong_convex_loss,ass:smooth_loss})}\label{eq:single_index_eigenvalue}\\
\Lambda(\gamma, v)~\textnormal{is $\Lsi$-Lipschitz in $\gamma$  w.r.t $\nrm*{\cdot}_2$ a.s.} & & \text{($\Rightarrow$ \pref{ass:strong_convex_loss})}\\
\sup_{x\in \cX, \theta \in \target} \nrm*{\theta(x)}_2 \leq \Rsi. & & \text{($\Rightarrow$ \pref{ass:strong_convex_loss})}
\end{align}
\end{assumption}
Let us highlight the condition \pref{eq:single_index_eigenvalue}, which involves the dual exponent $q$. Taking $p$ large strengthens the notion of distance for the nuisance parameter, but weakens the condition \pref{eq:single_index_eigenvalue}. For example, with $\rsi=0$, when we take $p=2$, \pref{eq:single_index_eigenvalue} becomes an average-case eigenvalue-type condition, but we require $L_4$ distance on the nuisance parameter. On the other hand, with $p=1$, \pref{eq:single_index_eigenvalue} essentially requires that parameter error ($L_{\infty}$) and $L_2$ error are equivalent for the target, but leads to $L_2$ distance for the nuisance parameter, which is more permissive than $L_4$. As a concrete example, consider the logistic loss, where $\Phi(t, \gamma, z) = y\cdot \log(\loglink(t)) + (1-y)\cdot \log(1-\loglink(t))$, where the target class $y\in \{0,1\}$ is a subset of the data $z$ and $\loglink(t)\ldef{}1/(1+e^{-t})$ is the logistic function, so that $\frac{\partial }{\partial t} \Phi(t, \gamma, z)=\loglink(t) - y$. Observe that the gradient of the loss with respect to the target index value can be written as
\begin{align*}
\nabla_{\zeta} \ell(\zeta, \gamma; z) =~&  \prn*{\phi(\tri{\Lambda(\gamma, v), \zeta}) - \Gamma(\gamma, z)}\Lambda(\gamma, v).
\end{align*}
Moreover, whenever the arguments to the loss are bounded, the Hessian can be bounded above and below via
\begin{equation}\label{eqn:hessian_bounds_single_index}
\Tsi \cdot \Lambda(\gamma, v) \Lambda(\gamma, v)^{\trn}
\succeq
\nabla_{\zeta\zeta}^2 \ell(\zeta, \gamma; z) = \phi'(\tri{\Lambda(\gamma, v), \zeta}) \cdot  \Lambda(\gamma, v) \Lambda(\gamma, v)^{\trn}
\succeq \tausi \cdot \Lambda(\gamma, v) \Lambda(\gamma, v)^{\trn},
\end{equation}
and thus, when $p=2$,  the ratio condition in \pref{eq:single_index_eigenvalue} is implied by a minimum eigenvalue assumption on the conditional covariance matrix $\En\brk*{\Lambda(\gtone(\varone), v) \Lambda(\gtone(\varone), v)^{\trn} \mid \vartwo}$.

We now show that these conditions are sufficient to satisfy the assumptions of \pref{thm:generic_strongly_convex}, and thus guarantee higher order impact from the nuisance parameters. Our main result is as follows.
\begin{lemma}\label{lem:suff_single_index}
  If \pref{ass:single_index} holds, then the \pref{ass:orthogonal,ass:well_specified,ass:strong_convex_loss,ass:smooth_loss} required by \pref{thm:generic_strongly_convex} are satisfied with constants $r=\rsi$, $\lambda = \frac{\tausi}{4}$,  $\kappa=8\tausi(\Lsi^4\Rsi^2\lambdasi^{-1})^{\frac{1}{1+\rsi}}$, $\beta_1=\Tsi$ and $\beta_2= \frac{\musi \sqrt{\dimtwo}}{\sqrt{\gammasi}}$, and with respect to the norms $\nrm{\cdot}_{\target}$ and $\nrm{\cdot}_{\nuisance}$ defined in \pref{eq:single_index_target_norm,eq:single_index_nuisance_norm}.
\end{lemma}

Combining this lemma with the guarantee from \pref{thm:generic_strongly_convex} directly yields an oracle excess risk bound.
\begin{corollary}
\label{cor:single_index}
Suppose that there is some $\besttwo\in\argmin_{\theta\in\nuisance}\poprisk(\theta,\gtone)$ such that \pref{ass:single_index} is satisfied. The sample splitting meta-algorithm \pref{alg:sample_splitting} produces a predictor $\esttwo$ that guarantees that with probability at least $1-\delta$,
\begin{align*}
\poprisk(\esttwo,\gtone) - \poprisk(\besttwo,\gtone)
\leq{}& C_1\cdot\Rate(\target, \sampletwo, \delta/2\midsem{}\esttwo,\estone) 
        + C_2\cdot
        \,\prn*{\Rate(\nuisance,\sampleone,\delta/2)}^{\frac{4}{1+\rsi}},
\end{align*}
where $C_1=\frac{16}{\tausi}$ and $C_2=2\prn[\Big]{        \prn*{\frac{16\musi^2\dimtwo}{\tausi^2\gammasi}}^{\frac{1}{1+\rsi}}        + 32\prn*{\frac{\Lsi^4\Rsi^2}{\gammasi}}^{\frac{1}{1+\rsi}}     }$.
Furthermore, with $\nrm*{\cdot}_{\target}$ as in \pref{lem:suff_single_index}, the following prediction error guarantee is satisfied with probability at least $1-\delta$,
\begin{align*}
\nrm[\big]{\esttwo - \besttwo}_{\target}^{2}
\leq{}~&  \frac{\Tsi}{2}\cdot{}C_1\cdot\Rate(\target, \sampletwo, \delta/2\midsem{}\esttwo,\estone) 
         +\frac{\Tsi}{2}\cdot{}C_2\cdot
         \prn*{\Rate(\nuisance,\sampleone,\delta/2)}^{\frac{4}{1+\rsi}}.
\end{align*}
\end{corollary}
Observe that for both of the bounds in \pref{cor:single_index}, the only problem-dependent parameters that multiply the target estimation rate $\Rate(\target,\cdots)$ are $\Tsi$ and $\tausi^{-1}$. Importantly, this implies that if the more restrictive parameters $\gammasi,\Lsi,\musi, \Rsi,\dimtwo$ and so forth are held constant as $n$ grows, they have negligible impact asymptotically, so long as the nuisance parameter can be estimated quickly enough. For the square loss this is particularly desirable, since $\Tsi=\tausi=1$.

En route to proving \pref{lem:sufficient_slow}, we also prove the
following, slightly stronger, smoothness bound, which is used for
several results.
\begin{lemma}
\label{lem:single_index_smooth}
Suppose \pref{ass:single_index} holds with $\rsi=0$. Then for all functions $\theta\in\targetemp$, $\bar{\theta}\in\starhull(\targetemp,\besttwo)$, and $g\in\nuisance$, 
\begin{equation}
D^{2}_{\theta}\poprisk(\bar{\theta}, g)[\theta{}-\besttwo, \theta{}-\besttwo] \leq~ 3\Tsi \, \nrm*{\theta{}-\besttwo}_{\target}^2 + \frac{4\Tsi \Lsi^4  \Rsi^2}{\gammasi}\nrm*{g - \gtone}_{\nuisance}^4.
\end{equation}
\end{lemma}

\paragraph{Discussion: Estimation for the first stage}
When we work with $L_2$ error for the target parameter (that is, $p=2$), \pref{cor:single_index} (with $\rsi=0$) provides guarantees in terms of the $L_{4}$ estimation rate for the nuisance parameters, i.e.
\[
\Rate(\nuisance,\sampleone,\delta) = \nrm*{\estone-\gtone}_{L_{4}(\ls_{2},\cD)} =  \prn*{\En_{w}\nrm*{\estone(w)-\gtone(w)}_{2}^{4}}^{\frac{1}{4}}.
\]
Since $L_{4}$ error rates are somewhat less common than $L_{2}$ (i.e., square loss) estimation rates, let us briefly discuss conditions under which out-of-the box algorithms can be used to give guarantees on the $L_{4}$ error. 

First, for many nonparametric classes of interest, minimax $L_{p}$ error rates have been characterized and can be applied directly. This includes smooth classes \citep{stone1980optimal,stone1982optimal}, \Holder{} classes \citep{lepskii1992asymptotically,kerkyacharian2001nonlinear,kerkyacharian2008nonlinear}, Besov classes \citep{delyon1996minimax}, Sobolev classes \citep{tsybakov1998pointwise}, and convex regression \citep{guntuboyina2015global}

Second, whenever the $\nuisance$ is a linear class or more broadly a parametric class, classical statistical theory \citep{lehmann2006theory} guarantees parameter recovery. Up to problem-dependent constants, this implies a bound on the $L_{4}$ error as soon as the fourth moment is bounded. This approach also extends to the high-dimesional setting \citep{hastie2015statistical}.

Last, if the class $\nuisance$ has well-behaved moments in the sense that $\nrm*{g-\gtone}_{L_{4}(\ls_2,\cD)}\leq{} C\nrm*{g-\gtone}_{L_{2}(\ls_2,\cD)}$ for all $g\in\nuisance$, we can directly appeal to square loss regression algorithms for the first stage; this is the approach taken in \pref{sec:algsandrates}. This condition is related to the so-called ``subgaussian class'' assumption, and both have been explored in recent works \citep{lecue2013learning,mendelson2014learning,liang2015learning}.

Of course, whenever $L_{\infty}$ guarantees are available for the target class itself (e.g., for parametric models), we can instead take $p=1$, which permits the use of the more standard $L_2$ distance for the nuisance parameter.

\subsection{Slow Rates} 
In the single-index setup, assumptions much weaker than \pref{ass:single_index} suffice to obtain slow rates via \pref{thm:orthogonal_slow}. In particular, the following conditions are sufficient.
\begin{assumption}[Sufficient Slow Rate Conditions for Single Index Losses]\label{ass:gradient_conditions_slow}
\begin{align}
\En\brk*{\grad{}_{\zeta}\grad_{\gamma}\ell(\theta(\vartwo),g_0(\varone);z)\mid{}\varone}= 0, & \quad\quad \forall \theta\in\starhull(\targetemp,\besttwo) + \starhull(\targetemp-\besttwo,0) \tag{$\Rightarrow$ \pref{ass:universal_orthogonality}}.\\
\En\brk*{\nabla_{\gamma\gamma}^2 \ell(\theta(\vartwo), g(\varone); z) \mid \varone}\preceq \betasi I, & \quad\quad\forall\theta\in\starhull(\targetemp,\besttwo),\; g\in\starhull(\nuisance,\gtone). \tag{$\Rightarrow$ \pref{ass:smooth_loss_slow}}
\end{align}
\end{assumption}Compared to \pref{ass:single_index}, the most important difference is that since we require \emph{universal} orthogonality, the first condition is required to hold for all $\theta$, not just at $\theta_0$. \pref{ass:gradient_conditions_slow} has the following immediate consequences.
\begin{lemma}
\label{lem:sufficient_slow}
If \pref{ass:gradient_conditions_slow} holds, then \pref{ass:universal_orthogonality} is satisfied and \pref{ass:smooth_loss_slow} is satisfied with constant $\beta=\betasi$ and with respect to $\nrm*{\cdot}_{\nuisance} = \nrm*{\cdot}_{L_2(\ell_2, \cD)}$.
\end{lemma}
\begin{corollary}
  \label{cor:sufficient_slow}
Suppose \pref{ass:gradient_conditions_slow} holds. Then with probability at least $1-\delta$, the target predictor $\esttwo$ produced by \pref{alg:sample_splitting} enjoys the excess risk bound
\[
L_{\cD}(\esttwo,\gtone) - L_{\cD}(\besttwo, \gtone) 
\leq{}  \Rate(\target, \sampletwo, \delta/2\midsem{}\esttwo,\estone) + \betasi \cdot\prn*{\Rate(\nuisance,\sampleone,\delta/2)}^{2}.
\]
\end{corollary}

\subsection{Proofs}
\label{app:sufficient}

\begin{proof}[\pfref{lem:suff_single_index}] We prove that the
  assumptions required by \pref{thm:generic_strongly_convex} are implied by our conditions one by one.
~
\paragraph{\pref{ass:orthogonal}}From the definition of the
directional derivative, the law of iterated expectations, and the fact that
$\Vartwo \subseteq \Varone$, we have
\begin{align*}
D_{g}D_{\theta}\poprisk(\besttwo,\gtone)[\theta-\besttwo,g-\gtone] =~& \En\brk*{(\theta(\vartwo) - \besttwo(\vartwo))^{\trn}\cdot \nabla_{\gamma}\nabla_{\zeta} \ell(\besttwo(\vartwo), \gtone(\varone), z) (g(\varone) - \gtone(\varone))} \\ 
=~& \En\brk*{(\theta(\vartwo) - \besttwo(\vartwo))^{\trn}\cdot \En\brk*{\nabla_{\gamma}\nabla_{\zeta} \ell(\besttwo(\vartwo), \gtone(\varone), z) \mid \varone}  (g(\varone) - \gtone(\varone))}\\
 =~& 0.
\end{align*}

\paragraph{\pref{ass:well_specified}}  This likewise follows immediately by expanding the directional derivative and applying the law of iterated expectation:
\begin{align*}
D_{\theta}\poprisk(\besttwo,\gtone)[\theta-\besttwo] = \En\brk*{\nabla_{\zeta} \ell(\besttwo(\vartwo), \gtone(\varone), z) \cdot (\theta(\vartwo) - \besttwo(\vartwo))} \geq 0.
\end{align*}

We now argue about the remaining assumptions. We will repeatedly invoke the following expression for the second derivative of the population risk.
\begin{align}
D^{2}_{\theta}\poprisk(\bar{\theta},g)[\theta-\theta', \theta-\theta']
= \En_{z}\brk*{
\phi'\prn*{\tri*{\Lambda(g(\varone),v), \bar{\theta}(\vartwo)}}\cdot \tri*{\Lambda(g(\varone),v), \theta(\vartwo)-\theta'(\vartwo)}^{2}
}.\label{eqn:single_index_hessian}
\end{align}

Let us introduce some additional notation.
Let $g\in\nuisance$ and $\theta\in\target$ be the free variables in
the statements of \pref{ass:smooth_loss} and \pref{ass:strong_convex_loss}. We define the following vector- and matrix-valued random variables:
\begin{align}
W_0 =~& \Lambda(\gtone(\varone), v), &
W_n =~& \Lambda(g(\varone), v),\\
X_0 =~& \besttwo(\vartwo), & X_n =~& \theta(\vartwo), &
\\
V_0=~& \gtone(\varone), & V_n=~&g(\varone),\\
A_0 =~& \En\brk*{W_0 W_0^{\trn}\mid{}\vartwo}.
\end{align}

To prove the lemma, it suffices to verify that
\pref{ass:strong_convex_loss} and \pref{ass:smooth_loss} hold for the
norms $\nrm*{\theta
  -\theta'}_{\target}^2 = \En\brk[\big]{\tri*{W_0, \theta(\vartwo) -
    \theta'(\vartwo)}^2}$ and $\nrm*{g - g'}_{\nuisance} = \nrm*{g-g'}_{L_{2p}(\ell_2,
  \cD)}$.

\preprint{\newcommand{\afoura}{\savehyperref{ass:smooth_loss}{Assumption \ref*{ass:smooth_loss}(a)}}}
\journal{\newcommand{\afoura}{Assumption \savehyperref{ass:smooth_loss}{\ref*{ass:smooth_loss}(a)}}}
\preprint{\newcommand{\afourb}{\savehyperref{ass:smooth_loss}{Assumption \ref*{ass:smooth_loss}(b)}}}
\journal{\newcommand{\afourb}{Assumption \savehyperref{ass:smooth_loss}{\ref*{ass:smooth_loss}(b)}}}

\paragraph{\afoura} Using \pref{eqn:hessian_bounds_single_index} and \pref{eqn:single_index_hessian} we have:
\begin{align*}
D^{2}_{\theta}\poprisk(\bar{\theta}, g_0)[\theta{}-\besttwo, \theta{}-\besttwo] \leq~&  \Tsi \cdot \En_{z}\brk*{
\tri*{W_0, X_n - X_0}^{2}} \leq{} \Tsi \cdot \nrm*{\theta{}-\besttwo}_{\target}^2,
\end{align*}
since $\nrm*{\cdot}_{L_2(\cD)}\leq{}\nrm*{\cdot}_{L_p(\cD)}$ for all
$p\geq{}2$. It follows that \afoura{} is satisfied with $\beta_1 = \Tsi$.

\paragraph{\afourb} 
Define random vectors $X_0 = \besttwo(\vartwo)$, $X_n=\theta{}(\vartwo)$, $V_0=\gtone(\varone)$, $V_n=g(\varone)$ and $\Sigma(\varone) = \En\brk*{\nabla_{\gamma\gamma}^2 \nabla_{\zeta_i} \ell(\besttwo(\vartwo), \bar{g}(\varone), z) \mid \varone}$. Then, invoking the assumed structure for the loss function, we have
\begin{align*}
&\abs*{D^{2}_{g}D_{\theta}\poprisk(\besttwo,\bar{g})[\theta{}-\besttwo,g-\gtone,g-\gtone]} \\&=  \abs*{\sum_{i=1}^{\dimtwo{}}\En\brk*{(X_{ni} - X_{0i}) (V_n - V_0)^{\trn} \nabla_{\gamma\gamma}^2 \nabla_{\zeta_i} \ell(\besttwo(\vartwo), \bar{g}(\varone), z) (V_n - V_0)}}.
\end{align*}
Using that $x\subseteq{}w$, we have
\begin{align*}
=~&  \abs*{\sum_{i=1}^{\dimtwo{}}\En\brk*{(X_{ni} - X_{0i}) (V_n - V_0)^{\trn} \Sigma(\varone) (V_n - V_0)}}\\
\leq~& \sum_{i=1}^{\dimtwo{}}\En\brk*{\abs*{X_{ni} - X_{0i}}\cdot \abs*{(V_n - V_0)^{\trn} \Sigma(\varone) (V_n - V_0)}}\\
\leq~& \musi\sum_{i=1}^{\dimtwo{}}\En\brk*{\abs*{X_{ni} - X_{0i}}\cdot \nrm*{V_n - V_0}_2^2}\\
=~& \musi \En\brk*{\nrm*{X_{n} - X_{0}}_1 \cdot \nrm*{V_n - V_0}_2^2}.
\end{align*}
All that remains is to relate these norms to the norms appearing in
the lemma statement.
\begin{align*}
\leq~& \musi\sqrt{\dimtwo} \En\brk*{\nrm*{X_{n} - X_{0}}_2 \cdot \nrm*{V_n - V_0}_2^2}\\
\leq~& \musi\sqrt{\dimtwo}\cdot
       \nrm{X_n-X_0}_{L_q(\ls_2,\cD)}\cdot{}\nrm*{ \nrm*{V_n - V_0}_2^2}_{L_p(\cD)}
  \\
=~& \musi\sqrt{\dimtwo}\cdot \nrm*{\theta{}-\besttwo}_{L_q(\ell_2, \cD)} \cdot \nrm*{g-\gtone}_{\cG}^2\\
\leq~& \frac{\musi \sqrt{\dimtwo{}} }{\sqrt{\gammasi}} \nrm*{\theta{}-\besttwo}_{\target}^{1-\rsi} \cdot \nrm*{g-\gtone}_{\nuisance}^2.
\end{align*}
Thus, we have
\begin{equation}
\label{eq:sufficient_higher_order}
D^{2}_{\theta}\poprisk(\bar{\theta}, g_0)[\theta{}-\besttwo, \theta{}-\besttwo]
\leq{} \frac{\musi \sqrt{\dimtwo{}}}{\sqrt{\gammasi}} \nrm*{\theta{}-\besttwo}_{\target} \cdot \nrm*{g-\gtone}_{\nuisance}^2,
\end{equation}
so that \afourb{} is satisfied with $\beta_2= \frac{\musi
  \sqrt{\dimtwo{}} }{\sqrt{\gammasi}}$.

\paragraph{\pref{ass:strong_convex_loss}} By \pref{eqn:hessian_bounds_single_index} and \pref{eqn:single_index_hessian}, we have
\begin{align*}
D^{2}_{\theta}\poprisk(\bar{\theta}, g)[\theta{}-\besttwo, \theta{}-\besttwo] \geq~& \tausi \cdot \En_{z}\brk*{
\tri*{W_n, X_n - X_0}^{2}}\\
\geq~& \frac{\tausi}{2} \cdot \En_{z}\brk*{
\tri*{W_0, X_n - X_0}^{2}} - \tausi \cdot \En_{z}\brk*{
\tri*{W_0 - W_n, X_n - X_0}^{2}}\\
\geq~& \frac{\tausi}{2} \cdot \En_{z}\brk*{
\tri*{W_0, X_n - X_0}^{2}} - \tausi \cdot \En_{z}\brk*{
\nrm*{W_0 - W_n}_{2}^{2}\nrm*{X_n - X_0}^{2}_{2}}\\
\geq~& \frac{\tausi}{2} \cdot \nrm*{\theta{}-\besttwo}_{\target}^2 -
       \tausi\,\Lsi^2\, \En_{z}\brk*{\|V_n-V_0\|^2_{2} \|X_n -
       X_0\|_2^{2}}.
       \end{align*}
Using \Holder{}'s inequality, we have that
\begin{align*}
  \En_{z}\brk*{\|V_n-V_0\|^2_{2} \|X_n -
  X_0\|_2^{2}}
  &\leq{} \nrm*{ \nrm{V_n-V_0}_2^2}_{L_p(\cD)}\cdot{}\nrm{
  \nrm{X_n-X_0}^2_2}_{L_q(\cD)}\\
  &= \nrm*{g-g_0}_{\cG}^2\cdot{}\nrm{ \nrm{X_n-X_0}^2_2}_{L_q(\cD)}
  \\
  &\leq 2\Rsi\cdot{}\nrm*{g-g_0}_{\cG}^2\cdot{}\nrm{
    X_n-X_0}_{L_q(\ls_2,\cD)}  \\
  &\leq 2\Rsi\gammasi^{-1/2}\cdot{}\nrm*{g-\gtone}_{\cG}^2\cdot{}\nrm{ \theta-\besttwo}^{1-\rsi}_{\target}.
\end{align*}
Hence, by Young's inequality, we have that for any $\eta>0$
\begin{align*}
  D^{2}_{\theta}\poprisk(\bar{\theta}, g)[\theta{}-\besttwo,
  \theta{}-\besttwo]
  &\geq
                       \frac{\tausi}{2} \cdot \nrm*{\theta{}-\besttwo}_{\target}^2
    -2\tausi\Lsi^2\Rsi\gammasi^{-1/2} \cdot\prn*{
    \eta^{-\frac{1-\rsi}{1+\rsi}}\nrm*{g-\gtone}_{\cG}^{\frac{4}{1+\rsi}}
    +\frac{\eta}{2}\nrm*{\theta{}-\besttwo}_{\target}^2}.
\end{align*}
Choosing $\eta=(4\Lsi^2\Rsi\gammasi^{-1/2})^{-1}$ yields
\begin{align*}
  D^{2}_{\theta}\poprisk(\bar{\theta}, g)[\theta{}-\besttwo,
  \theta{}-\besttwo]  &\geq
                       \frac{\tausi}{4} \cdot \nrm*{\theta{}-\besttwo}_{\target}^2
                        -8\tausi(\Lsi^2\Rsi\lambdasi^{-1/2})^{\frac{2}{1+\rsi}}
                        \nrm*{g -
                       \gtone}_{\nuisance}^{\frac{4}{1+\rsi}} .
\end{align*}
Thus, \pref{ass:strong_convex_loss} is satisfied with $\lambda =
\frac{\tausi}{4}$ and $\kappa=8\tausi(\Lsi^2\Rsi\lambdasi^{-1/2})^{\frac{2}{1+\rsi}}$.

\end{proof}

\begin{proof}[\pfref{lem:single_index_smooth}]
We adopt the same notation as in the proof of \pref{lem:suff_single_index}. This proof is almost the same as the proof that \pref{ass:strong_convex_loss} is satisfied in that lemma.
\begin{align*}
D^{2}_{\theta}\poprisk(\bar{\theta}, g)[\theta{}-\besttwo, \theta{}-\besttwo] \leq~& \Tsi \cdot \En_{z}\brk*{
\tri*{W_n, X_n - X_0}^{2}}\\
\leq~& 2\Tsi \cdot \En_{z}\brk*{
\tri*{W_0, X_n - X_0}^{2}} + 2\,\Tsi \cdot \En_{z}\brk*{
\tri*{W_0 - W_n, X_n - X_0}^{2}}\\
\leq~& 2\Tsi \cdot \En_{z}\brk*{
\tri*{W_0, X_n - X_0}^{2}} + 2\,\Tsi \cdot \En_{z}\brk*{
\nrm*{W_0 - W_n}_{2}^{2}\nrm*{X_n - X_0}^{2}_{2}}\\
\leq~& 2\Tsi \cdot \nrm*{\theta{}-\besttwo}_{\target}^2 +
       2\,\Tsi\,\Lsi^2\, \En_{z}\brk*{\|V_n-V_0\|^2_{2} \|X_n -
       X_0\|_2^{2}} \\
                     \leq~& 2\Tsi \cdot \nrm*{\theta{}-\besttwo}_{\target}^2 +
                            4\,\Tsi\Rsi\,\Lsi^2\, \En_{z}\brk*{\|V_n-V_0\|^2_{2} \|X_n -
                            X_0\|_2} \\
                            \leq~& 2\Tsi \cdot \nrm*{\theta{}-\besttwo}_{\target}^2 +
                                   4\,\Tsi\Rsi\,\Lsi^2\,
                                   \nrm{g-\gtone}_{\nuisance}^2\cdot{}\nrm{\theta-\besttwo}_{L_q(\ls_2,\cD)}\\
                                                               \leq~& 2\Tsi \cdot \nrm*{\theta{}-\besttwo}_{\target}^2 +
                                                                      4\,\Tsi\Rsi\,\Lsi^2\gammasi^{-1/2}\, \nrm{g-\gtone}_{\nuisance}^2\cdot{}\nrm{\theta-\besttwo}_{\target}.
\end{align*}
Using the AM-GM inequality, this implies the inequality
\begin{align*}
D^{2}_{\theta}\poprisk(\bar{\theta}, g)[\theta{}-\besttwo, \theta{}-\besttwo] \leq~& 3\Tsi \, \nrm*{\theta{}-\besttwo}_{\target}^2 + \frac{4\Tsi \Lsi^4  \Rsi^2}{\gammasi}\nrm*{g - \gtone}_{\nuisance}^4.
\end{align*}
\end{proof}

\begin{proof}[\pfref{lem:sufficient_slow}] Immediate.
\end{proof}

\section{Sufficient Conditions for Oracle Rates: Further Results}
\label{sec:oracle_lipschitz}

The results in \pref{sec:oracle} provide sufficient conditions for fast oracle rates under the assumption of strong convexity and a well-specified model, which is not satisfied for all losses used in practice. For example, linear losses used in policy learning (cf. \pref{sec:policy_learning_body}) do not satisfy strong convexity property. In this section we extend the \revoneedit{sufficient conditions for oracle rates} given in \pref{sec:oracle}. First, we give guarantees for strongly convex losses in the presence of misspecification, and then give guarantees for any (potentially non-strongly convex) loss that is Lipschitz in the target prediction $\theta(x)$. We follow the same notation as in \pref{sec:oracle}.

\subsection{Oracle Rates for Square Losses with Misspecification}

Here we consider strongly convex losses with the same structure as in \pref{sec:oracle}, but consider the case in which the target parameter is \emph{misspecified}. This setting has been relatively unexplored in recent results on double machine learning \citep{chernozhukov2016double,mackey2017orthogonal,chernozhukov2018plugin}; this is perhaps not surprising since for many settings, assuming a well-specified model is critical to establish orthogonality. However, for certain settings including treatment effect estimation (\pref{sec:treatment_effect}), orthogonality can indeed hold even without model correctness. The following theorem shows that we can obtain oracle rates in the presence of both misspecification and nuisance parameters as long as the nuisance class has moderate complexity. 
\begin{theorem}[Oracle Rates, Misspecified Case]
  \label{thm:oracle_misspecified}
Suppose that the target class $\target$ is convex. Suppose that \pref{ass:orthogonal,ass:well_specified} and \pref{ass:square_oracle,ass:moment} hold. If the relationship
\begin{equation}
\pone < \max\crl*{2\ptwo+2,4\ptwo-2}
\end{equation}
holds, then for appropriate choice of sub-algorithms, the sample splitting meta-algorithm \pref{alg:sample_splitting} produces a predictor $\esttwo$ such that with probability at least $1-\delta$,
\begin{equation}
  \label{eq:oracle_misspecified}
\poprisk(\esttwo,\gtone) - \poprisk(\gttwo,\gtone)
\leq{}  \Ot\prn[\big]{n^{-\frac{2}{2+\ptwo}\wedge{}\frac{1}{\ptwo}}},
\end{equation}
\revoneedit{thereby matching the minimax rate in the absence of nuisance parameters. In particular, it suffices to use Skeleton Aggregation for the first stage and plug-in ERM for the second stage.}
\end{theorem}
\revoneedit{Like the previous result, \pref{thm:oracle_misspecified} is proven by combining the main theorem (\pref{thm:generic_strongly_convex}) with algorithm-specific upper bounds on $\Rate(\target,\cdots)$ and $\Rate(\nuisance,\cdots)$. \pref{fig:oracle_misspecified} summarizes sufficient conditions under which \pref{thm:oracle_misspecified} matches the oracle rate 
$\wt{\Theta}\prn{n^{-\frac{2}{2+\ptwo}\wedge{}\frac{1}{\ptwo}}}$ \citep{rakhlin2017empirical}.} Comparing to the well-specified case (\pref{thm:oracle_well_specified}/\pref{fig:oracle_well_specified}), we see that in the misspecified case, the condition on the metric entropy for the nuisance parameter is significantly more permissive when the target parameter class $\target$ is large (i.e. $\ptwo>2$). For example, if $\ptwo=5$ then we require $\pone<12$ for oracle rates in the well-specified case, but only require $\pone<18$ in the misspecified case. On the other hand, when $\ptwo\leq{}2$ the conditions on the nuisance metric entropy match the well-specified case. In particular, whenever $\target$ is a parametric class it again suffices to take $\pone<2$ so that the first stage has an $n^{-1/4}$-rate.

\begin{figure}[t]
  \begin{center}
  \includegraphics[width=.35\textwidth]{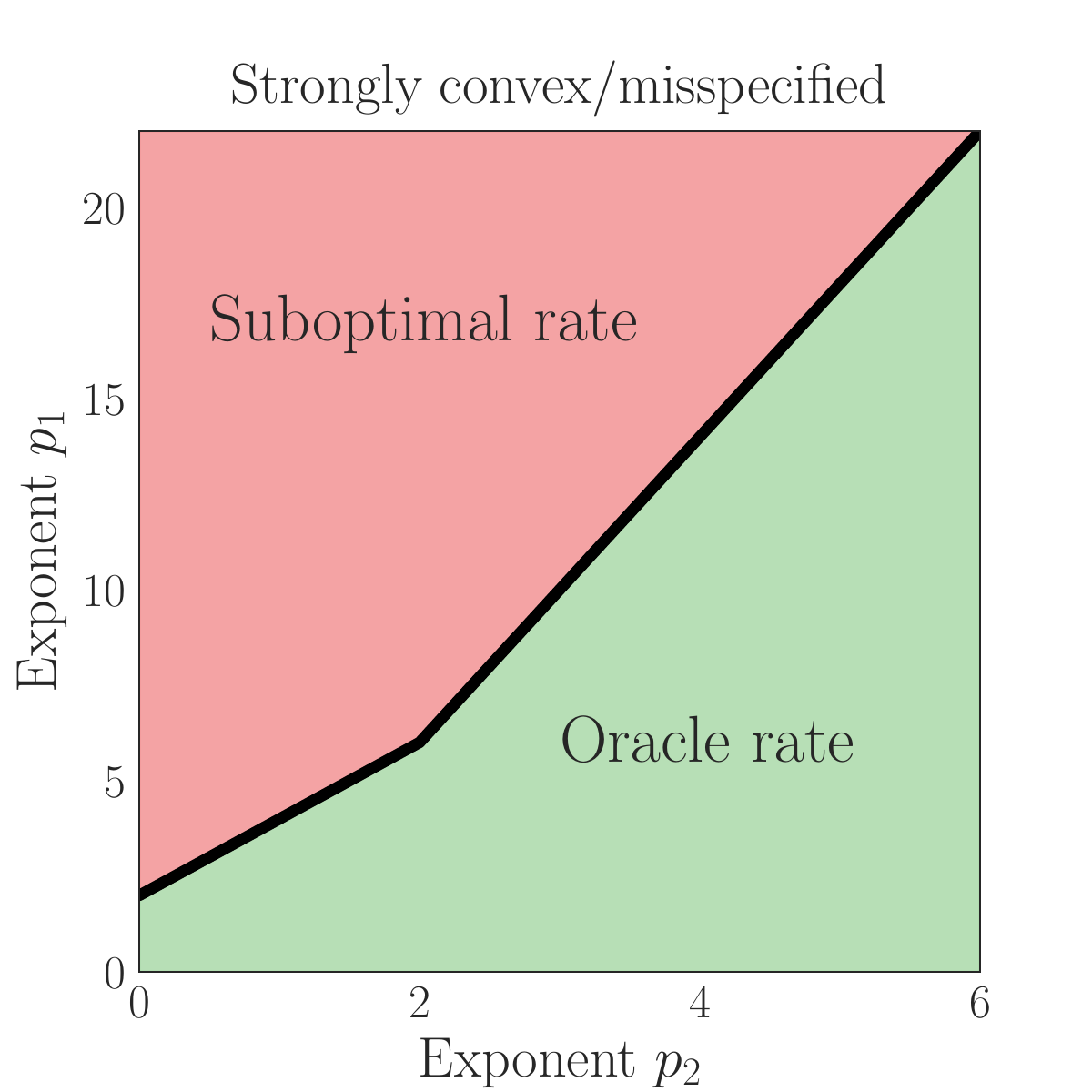}
\end{center}
\caption{Relationship between first and second stage for oracle rates; misspecified case.}
\label{fig:oracle_misspecified}
  \end{figure}

\subsection{Oracle Rates for Generic Lipschitz Losses}

For arbitrary Lipschitz losses (in particular, for the linear loss), the optimal rate in the absence of nuisance parameters is $\Ot\prn{n^{-\frac{1}{2}\wedge{}\frac{1}{\ptwo}}}$ under the metric entropy growth assumed in \pref{sec:algsandrates} (cf. Section 12.8/12.9 of \cite{StatNotes2012}). Our main theorem for this section shows that this rate is still obtained in the presence of nuisance parameters when the nuisance metric entropy parameter $p_1$ is not too much larger than $p_2$. We make the following regularity assumption on the loss.
\begin{assumption}
\label{ass:oracle_lipschitz}
The loss $\ls$ has absolute value bounded by $1$, the mapping $\zeta\mapsto\ls(\zeta,\gamma;z)$ is $O(1)$-Lipschitz with respect to $\ls_2$, and the mapping $\gamma\mapsto\ls(\zeta,\gamma;z)$ has $\bigoh(1)$-Lipschitz gradients with respect to $\ls_{2}$. \pref{ass:universal_orthogonality} holds for all $g\in\cG$ when $p_1\leq{}2$, and for all $g\in\nuisance+\starhull(\nuisance-\nuisance,0)$ if $p_1>2$.
\end{assumption}
Compared to the oracle rates for square losses, the assumptions required for our main Lipschitz loss theorem are relaxed significantly: We no longer require strong convexity, nor do we require any type of moment comparison for the nuisance class. On the other hand, we do require the additional universal orthogonality condition from \pref{ssec:orthogonal_slow_rate}.

\begin{theorem}
  \label{thm:oracle_slow}
Suppose that \pref{ass:oracle_lipschitz} is satisfied. If the relationship
\begin{equation}
\pone < \max\crl*{2,2\ptwo-2}
\end{equation}
holds, then for appropriate choice of sub-algorithms, the sample splitting meta-algorithm produces a predictor $\esttwo$ that guarantees that, with probability at least $1-\delta$,
\begin{equation}
  \label{eq:oracle_slow}
\poprisk(\esttwo,\gtone) - \poprisk(\gttwo,\gtone)
\leq{}  \Ot\prn[\big]{n^{-\frac{1}{2}\wedge{}\frac{1}{\ptwo}}},
\end{equation}
\revoneedit{thereby matching the minimax rate in the absence of nuisance parameters.}
\end{theorem}
\revoneedit{\pref{thm:oracle_slow} is proven as a corollary of \pref{thm:orthogonal_slow}, by proving upper bounds on $\Rate(\target,\dots)$ and $\Rate(\nuisance,\cdots)$ for plug-in empirical risk minimization and skeleton aggregation, respectively; see \pref{app:rates} for the proof.} The theorem is summarized in \pref{fig:oracle_slow}. In particular, \revoneedit{suppose that our aim is to match the minimax optimal rate in the absence of nuisance parameters, which is $\Ot\prn{n^{-\frac{1}{2}\wedge{}\frac{1}{\ptwo}}}$ \citep{StatNotes2012}.} Whenever $\target$ is a parametric class (i.e. $\cH_{2}(\target,\veps,n)\propto\dtwo\log(1/\veps)$), it suffices to take $\pone<2$, as in the well-specified and misspecified square loss setups in \pref{sec:square_oracle}, and as in standard semiparametric results \citep{newey1994asymptotic,van2000asymptotic,robins2008higher,zheng2010asymptotic,chernozhukov2016double}. That the condition matches is somewhat interesting given that the final rate in this case is slower than in the square loss setup. Note however that we cannot tolerate $\pone>2$ until $\ptwo>2$, as compared to our results strongly convex losses, where the admissible value of $p_{1}$ is growing for all values of $\ptwo$.
This result generalizes the condition given in \cite{athey2017efficient} for the specific case of policy learning, which applies only in the parametric setting, and only for a specific loss. 

\begin{figure}[t]
  \begin{center}
  \includegraphics[width=.35\textwidth]{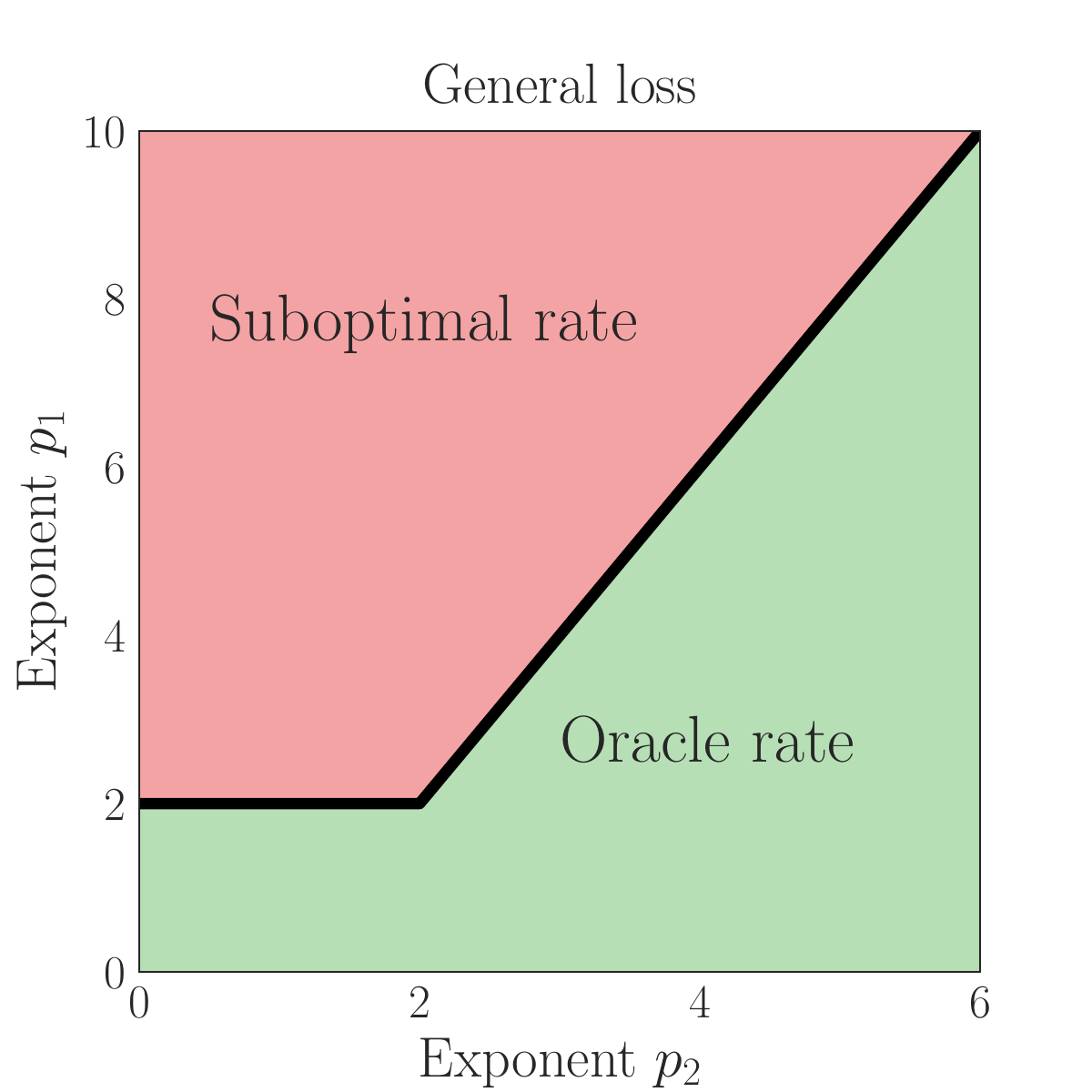}
\end{center}
\caption{Relationship between first and second stage for oracle rates; general loss case.}
\label{fig:oracle_slow}
  \end{figure}

\section{Additional Applications} 
\label{sec:applications}

In this section of the appendix we show how three additional families of applications---offline policy learning/optimization, domain adaptation/sample bias correction, and learning with missing data---fall into our general orthogonal statistical learning framework.  We also sketch some statistical consequences based on our general algorithmic tools, and show how these results generalize and extend previous work. 

\subsection{Policy Learning: Further Examples}\label{sec:policy}
In this section we show how to view some additional policy learning models in our framework. As in \pref{sec:policy_learning_body}, we consider the data-generating process in \pref{eq:treatment_control_dgp}, but here we go beyond binary treatments.

\paragraph{Multiple finite treatments} The binary setting above can easily be extended to the case of $N$ possible treatments, analyzed in \cite{zhou2018offline}. Formally, let $T\in \{\vec{e}_1, \ldots, \vec{e}_N\}$, where $\vec{e}_i\in \{0,1\}^N$ is the $i$-th standard basis vector. We still follow the data generating process \pref{eq:treatment_control_dgp}, but now $e_0:\cX\to\Delta_{N}$ and $f_0:\crl*{0,1}^{N}\times{}\cX\to\bbR$. To simplify notation, let $p_0(t, x) = \Pr\brk*{T=t \mid X=x}$ so that $p_0(t, x) = e_{0}(x)_t$. Then the following loss function is an unbiased estimate of the counterfactual loss:
\begin{equation}
\ell(t, Z; f_0, e_0) = f_{0}(t, X) + 1\{T=t\} \frac{(Y - f_0(t, X))}{p_0(t, X)}.
\end{equation}
 This formulation leads to the standard extension of the doubly-robust estimator to multiple outcomes \cite{dudik2011doubly,zhou2018offline}. Define an $N$-dimensional vector-valued function $\beta(f_0, e_0, Z)$ to have the $t$-th coordinate is equal to $\ell(t, Z; f_0, e_0)$. Then, as in the binary case, we can equivalently optimize a population risk that is linear in the target parameter:
\begin{equation}
L_{\cD}(\theta, \crl{f_0, e_0}) = \En\brk*{\tri*{\beta(f_0, e_0, Z), \theta(x)}}.
\end{equation}
This population risk is easily shown to satisfy universal orthogonality.

\paragraph{Counterfactual risk minimization and general continuous treatments} 
Counterfactual risk minimization (CRM) is a learning framework introduced by \cite{swaminathan2015counterfactual}. It is mathematically equivalent to the policy learning setup with arbitrary treatment and outcome spaces, but is motivated by a different set real-world learning scenarios and was developed in a parallel line of research in the machine learning literature. To highlight the relationship with policy learning and the applicability of our results to this setting we will present the CRM framework using the notation of policy learning. 

In counterfactual risk minimization we receive data $Z=(Y, T, X)$ from the policy learning data generating process \pref{eq:treatment_control_dgp}. The goal is to choose a hypothesis $\theta: \cX\rightarrow \Delta(\cT)$ (i.e., the policy takes as input covariates and returns a distribution over treatments) that minimizes the population risk:
\begin{equation}
\label{eq:crm_risk}
L_{\cD}^1(\theta, f_0) = \En_{Z} \En_{t\sim \theta(X)}\brk*{f_0(t, X)}.
\end{equation}
As in policy learning, we construct an unbiased estimate of this counterfactual loss via inverse propensity scoring. Let $p_0(t, X)$ denote the probability density of treatment $t$ conditional on covariates $X$ and (overloading notation) let $\theta(t,x)$ denote the density that hypothesis $\theta$ assigns to treatment $t$. Then we can formulate a new risk function that provides an unbiased estimate of the target risk \pref{eq:crm_risk}:
\begin{equation}
\label{eq:crm_risk2}
L_{\cD}^2(\theta, p_0) = \En\brk*{Y \frac{\theta(T, X)}{p_0(T,X)}}.
\end{equation}
In the CRM framework the propensity $p_0$ is assumed to be known \citep{swaminathan2015counterfactual,swaminathan2015self}. When the propensity is not known, we can treat it as a nuisance parameter to be estimated from data. However, the loss \pref{eq:crm_risk2} is not orthogonal to $p_0$. We can orthogonalize the population risk by also constructing an estimate of $f_0$ (see \pref{eq:treatment_control_dgp}) by regressing $Y$ on $(T,X)$. This leads to an analogue of the doubly robust formulation from the finite treatment setup:
\begin{equation}
\label{eq:crm_orthogonal}
L_{\cD}(\theta, \crl{f_0, p_0}) = \En\brk[\Bigg]{\underbrace{\prn*{f_0(T, X) + \frac{(Y - f_0(T,X))}{p_0(T,X)}}}_{\rdef\beta(f_0, p_0, Z)}\cdot \theta(T, X)}.
\end{equation}
For finite treatments this formulation is mathematically equivalent to the population risk for multiple finite treatments presented in the prequel.

For continuous treatments, the empirical version of the problem \pref{eq:crm_orthogonal} may be ill-posed, even if we assume that the propensity $p_0$ has density ower bounded by some constant (the analogue of the overlap condition). \cite{swaminathan2015counterfactual} proposed to regularize the empirical risk via variance penalization. A similar variance penalization approach is also proposed in recent work of \cite{bertsimas2018optimization}, who consider policy learning over arbitrary treatment spaces. The variance-penalized empirical risk minimization algorithm---in the context of \pref{alg:sample_splitting}---can be seen as a second stage algorithm that achieves a rate whose leading term scales with the variance of the optimal policy rather than some worst-case upper bound on the risk. Hence, it can be used in our framework to achieve variance-dependent excess risk bounds.

\cite{kallus2018policy} develop alternative algorithms for policy learning with continuous treatments via a kernel smoothing approach. This approach is equivalent to adding noise to a deterministic hypothesis space $\Pi$, e.g. $\theta(x) = \pi(x) + \zeta$ for each $\pi\in \Pi$, where $\zeta\sim \cN(0, \sigma^2)$. In our framework $\target$ is the space of density functions $\theta(t,x)$ induced by this construction. The value of a deterministic policy $\pi\in \Pi$ (or equivalently the value of its corresponding density $\theta\in \target$) is equal to
\begin{equation}
\En\brk*{\beta(f_0,p_0, Z) \cdot \cK_{\sigma}(T - \pi(X))},
\end{equation}
where $\cK_\sigma$ is the pdf of a normal distribution with standard deviation $\sigma$. This is equivalent to the formulation in \cite{kallus2018policy}, since the empirical version of this risk is the kernel-weighted loss: \begin{equation}
L_S(\pi, \crl{f_0, p_0}) = \frac{1}{n} \sum_{t=1}^n {\beta(f_0,p_0, Z_i) \cdot \cK_{\sigma}(T_i - \pi(X_i))},
\end{equation}
though we note that \cite{kallus2018policy} do not restrict themselves to only the gaussian kernel.

This idea falls into our framework by simply defining $\target$ to be this space of randomly perturbed policy functions. The resulting analysis in our framework is slightly different than that of \cite{kallus2018policy}, where kernel weighting is invoked to show consistency of the empirical risk, and subsequently optimization of the empirical risk over \emph{deterministic} policies is analyzed. With our framework, we directly calculate the regret with respect to \emph{randomized} policies by applying our general theorem. This implies that we enjoy robustness to errors in estimating $f_0$ and $p_0$. 

Observe that the rate for the second stage will depend on the amount of randomization $\sigma$, since the variance of the empirical risk is governed by $\sigma$. Consequently, if one is interested in regret against deterministic strategies, we can invoke Lipschitzness of the reward function $f_0$ to control the regret added by the extra randomness we are injecting, which would typically be of order $\sigma$. We can then choose an optimal $\sigma$ as a function of the number of samples to trade-off the \emph{bias} and \emph{variance} of the regret. If we wish to further optimize dependence on problem-dependent parameters in the resulting rates one can use variance penalization in the kernel-based framework to achieve a regret rate whose leading term scales with the variance of the optimal policy.

\subsection{Domain Adaptation and Sample Bias Correction}

Domain adaptation is a widely studied topic in machine learning \citep{daume2006domain,jiang2007instance,ben2007analysis,blitzer2008learning,mansour2009domain}. The goal is to choose a hypothesis that minimizes a given loss in expectation over a target data distribution, where the target distribution may be different from the distribution of data that is already collected. 

We consider a particular instance of domain adaptation called covariate shift, encountered in supervised learning \citep{shimodaira2000improving}. We assume that we have data $Z=(X, Y)$, where $X$ are co-variates drawn from some distribution $\cD_s$ with density $p_s$ and $Y$ are labels, drawn from some distribution $\cD_x$ conditional on $x$. Our goal is to choose a hypothesis $\theta$ from some hypothesis space $\target$, so as to minimize a loss function $\ell(\theta(x), y)$ in expectation over a different distribution of co-variates $\cD_t$ with density $p_t$. Both of the densities are unknown, and we solve this issue in the orthogonal statistical learning framework by treating their ratio as a nuisance parameter for an importance-weighted loss function. Let $f_0(x) = \frac{p_t(x)}{p_s(x)}$ and $g_0=\crl*{f_0}$, so that
\begin{equation}\label{eq:domain_adaptation}
\poprisk(\theta, g_0) \ldef \En_{\cD_s}\En_{y\mid{}x}\brk*{\ell(\theta(x), y)\cdot{}f_0(x)}=\En_{\cD_t}\En_{y\mid{}x}\brk*{\ell(\theta(x), y)}.
\end{equation}

We assume the hypothesis space satisfies a realizability condition, i.e. there exists $\theta_0\in \target$ such that:
\begin{equation}
\En\brk*{\nabla_{\zeta} \ell(\theta_0(x), y) \mid x}=0,
\end{equation}
where $\nabla_\zeta$ corresponds to the gradient with respect to the first input of $\ell$. For instance, for the case of the square loss $\ell(\zeta, y) = (\zeta - y)^2$, then this condition has the natural interpretation that there exists $\theta_0\in\target$ such that:
\begin{equation}
\En\brk*{y\mid x} = \theta_0(x).
\end{equation}
Observe that when we treat the density ratio $f_0(x)$ as a nuisance function, the loss function $\poprisk$ is orthogonal. Indeed,
\begin{align*}
D_f D_\theta \poprisk(\theta, \crl{f})[\theta - \theta_0, f - f_0] =~& \En\brk*{\nabla_{\zeta} \ell(\theta_0(x), y) \cdot (\theta(x) - \theta_0(x)) \cdot (f(x) - f_0(x))}\\
=~&\En\brk*{\En\brk{\nabla_{\zeta} \ell(\theta_0(x), y)\mid x} \cdot (\theta(x) - \theta_0(x)) \cdot (f(x) - f_0(x))} = 0.
\end{align*}
Also, note that this setup fits in to the single index structure from \pref{sec:sufficient} by writing $\poprisk$ as the expectation of a new loss $\tilde{\ls}(\theta(x),g(w),z)\ldef{}\ls(\theta(x),y)\cdot{}f(x)$. Focusing on the square loss $\ls(\zeta,y)=(\zeta-y)^{2}$ for concreteness, it is simple to show that all of the conditions of \pref{ass:single_index} are satisfied with $r=0$ as long as we have $g(x)\geq{}\eta>0$ for all $g\in\nuisance$. \pref{cor:single_index} then implies that with probability at least $1-\delta$, \pref{alg:sample_splitting} enjoys the bound
\[
\poprisk(\esttwo,\gtone) - \poprisk(\besttwo,\gtone)
\leq{} O\prn*{\Rate(\target, \sampletwo, \delta/2\midsem{}\esttwo,\estone) 
+ \mathrm{poly}(\eta^{-1})\cdot\prn*{\Rate(\nuisance,\sampleone,\delta/2)}^{4}
}.
\]
Note that whenever $\Rate(\nuisance,\sampleone,\delta/2)=o\prn*{\Rate(\target, \sampletwo, \delta/2\midsem{}\esttwo,\estone)^{1/4}}$ the dependence on $\eta^{-1}$ is negligible asymptotically. Of course, it is also important to develop algorithms for which the rate of the target class does not depend on $\eta^{-1}$. As one example, we can employ the variance-penalized ERM guarantee from \pref{thm:alexander}. When $\target$ has VC dimension $d$, and the variance of the loss at $(\theta_0,g_0)$ and the capacity function $\tau_0$ are bounded, this gives $\Rate(\nuisance,\sampleone,\delta/2)=O(\sqrt{d/n})$, with $\eta^{-1}$ entering only lower-order terms. The final result is that if $\Rate(\nuisance,\sampleone,\delta/2)=o(n^{-1/8})$, we get an excess risk bound for which the dominant term is $O(\sqrt{d/n})$, with no dependence on $\eta^{-1}$. 

\paragraph{Related work}
\cite{cortes2010learning} gave generalization error guarantees for the important weighted loss \pref{eq:domain_adaptation} in the case where the densities $p_{s}$ and $p_{t}$ are known. At the other extreme, \cite{ben2012hardness} showed strong impossibility results in the regime where the densities are unknown. Our results lie in the middle, and show that learning with unknown densities is possible in the regime where the weights belong to a nonparametric class that is not much more complex than the target predictor class $\target$. We remark in passing that algorithms based on discrepancy minimization \cite{ben2007analysis,mansour2009domain} offer another approach to domain adaptation that does not require importance weights, but these results are not directly comparable to our own.

\subsection{Missing Data}
\label{sec:missing_data}
As a final application, we apply our tools to the well-studied problem of regression with missing/censored data \citep{robins1995semiparametric,van2003unifiedb,rubin2005general,tsiatis2007semiparametric,wang2010nonparametric,graham2011efficiency}. In this setting we receive data is generated through standard regression model, but label/target variables are sometimes ``missing'' or unobserved. The learner observes whether or not the target is missing for each example, and the conditional probability that the target is missing is treated as an unknown nuisance parameter. As usual, the target is the unknown regression function.

To proceed, we formalize the setting through the following data-generating process for the observed variables $(X,W,T,\wt{Y})$:
\begin{equation}
\label{eq:missing_data_least_squares}
\begin{aligned}
&\wt{Y} = T\cdot{}Y,&&\\
&Y = \theta_0(X)+\veps_1,&&\En\brk*{\veps_1\mid{}X}=0,\\
&T=e_0(W) + \veps_2,&&\En\brk*{\veps_2\mid{}W}=0.
\end{aligned}
\end{equation}
Here $T\in\crl*{0,1}$ is an auxiliary variable (observed by the learner) that indicates whether the target variable is missing, and $e_0:\cX\to\brk*{0,1}$ is the unknown propensity for $T$. The parameter $\theta_0:\cX\to\bbR$ is the true regression function. We make the standard \emph{unobserved confounders} assumption that $X\subseteq{}W$ and $T\indep{}Y\mid{}W$.
We define $h_0(w)=-2\frac{(\En\brk*{Y\mid{}W=w}-\theta_0(x))}{e_0(w)}$, take $g_0=\crl*{h_0,e_0}$, and use the loss
\begin{equation}
\ls(\theta,\crl*{h,e},z) = \frac{T\prn[\big]{\wt{Y}-\theta(X)}^{2}}{e(W)} - \theta(X)h(W)\prn*{T-e(W)}.
\end{equation}
Observe that this loss has the property that
\[
\poprisk(\theta,g_0) = \En\prn*{Y-\theta(X)}^{2},
\]
so that the excess risk relative to $\theta_0$ precisely corresponds to prediction accuracy whenever the true nuisance parameter is plugged in.
\begin{proposition}
  \label{prop:missing_data}
This model satisfies \pref{ass:orthogonal} and \pref{ass:well_specified} whenever $\gttwo\in\target$, i.e. we have $D_{\theta}\poprisk(\gttwo,\crl*{h,e_0})[\theta-\theta_{0}]=0$, $D_{e}D_{\theta}\poprisk(\gttwo,\crl*{h_0,e_0})[\theta-\theta_{0}, e-e_0]=0$, and $D_{h}D_{\theta}\poprisk(\gttwo,\crl*{h_0,e_0})[\theta-\theta_{0}, h-h_0]=0$.
\end{proposition}
Note that the extra nuisance parameter $h$ is only required here because we consider the general setting in which $W\neq{}X$. Whenever $W=X$ this is unnecessary (and indeed $h_0=0$). This parameter can generally be estimated at a rate no worse than the rate for $e_0$ and $\theta_0$ (absent nuisance parameters); see \cite{chernozhukov2018plugin} for discussion.

As to rates and algorithms, the situation here is essentially the same as that of the domain adaptation example, so we discuss it only briefly. The setup has the single index structure from \pref{sec:sufficient}, and all of the sufficient conditions for fast rates from \pref{ass:single_index} are satisfied with $r=0$ as long as we have $e(W)\geq{}\eta>0$ for all $e$ in the nuisance class. Thus, with probability at least $1-\delta$, \pref{alg:sample_splitting} enjoys the bound
\[
\poprisk(\esttwo,\gtone) - \poprisk(\besttwo,\gtone)
\leq{} O\prn*{\Rate(\target, \sampletwo, \delta/2\midsem{}\esttwo,\estone) 
+ \mathrm{poly}(\eta^{-1})\cdot\prn*{\Rate(\nuisance,\sampleone,\delta/2)}^{4}
}.
\]
As with the previous example, the variance-penalized ERM guarantees from \pref{sec:erm} can be applied here to provide bounds on $\Rate(\target,\ldots)$ for which the dominant term in the excess risk does not scale with the inverse propensity range.

\begin{proof}[\pfref{prop:missing_data}]
We first show that the gradient vanishes in the sense of \pref{ass:well_specified} when evaluated at $\gttwo$. In particular, for any choice of $h$ we have
\begin{align*}
&D_{\theta}\poprisk(\gttwo,\crl*{h,e_0})[\theta-\theta_{0}] \\
&= \En\brk*{
\prn*{ -2T\frac{(\wt{Y}-\theta_0(X))}{e_{0}(W)}
- h(W)(T-e_{0}(W))
}(\theta(X)-\theta_0(X))
}. \\
\intertext{Using that $\En\brk*{T\mid{}W}=e_0(W)$ and $X\subseteq{}W$:}
&= \En\brk*{
\prn*{ -2T\frac{(\wt{Y}-\theta_0(X))}{e_{0}(W)}
}(\theta(X)-\theta_0(X))
}.
\intertext{Using that $T\in\crl*{0,1}$:}
&= \En\brk*{
\prn*{ -2T\frac{(Y-\theta_0(X))}{e_{0}(W)}
}(\theta(X)-\theta_0(X))
}.
\intertext{Using that $T\perp{}Y\mid{}W$ and that $X\subseteq{}W$:}
&= \En_{W}\brk*{
\prn*{ -2\En\brk*{T\mid{}W}\frac{\En\brk*{(Y-\theta_0(X))\mid{}W}}{e_{0}(W)}
}(\theta(X)-\theta_0(X))
} \\
&= -2\En\brk*{
(Y-\theta_0(X))
(\theta(X)-\theta_0(X))
}.\\
\intertext{Using that $\En\brk*{Y\mid{}X}=\theta_0(X)$:}
&=0.
\end{align*}
To establish orthogonality with respect to $e$, we have
\begin{align*}
&D_{e}D_{\theta}\poprisk(\gttwo,\crl*{h_0,e_0})[\theta-\theta_{0}, e-e_0] \\
&= \En\brk*{
\prn*{ 2T\frac{(Y-\theta_0(X))}{e_{0}^{2}(W)}
+ h_{0}(W)
}(\theta(X)-\theta_0(X))(e(W)-e_0(W))
} .
\intertext{Using that $X\subseteq{}W$:}
&= \En_{W}\brk*{
\prn*{ 2\frac{(\En\brk*{Y\mid{}W}-\theta_0(X))}{e_{0}(W)}
+ h_{0}(W)
}(\theta(X)-\theta_0(X))(e(W)-e_0(W))
}.
\intertext{The result follows immediately, using that $h_0(w) = -2\frac{(\En\brk*{Y\mid{}W=w}-\theta_0(x))}{e_{0}(w)}$.}
&=0.
\end{align*}
To establish orthogonality with respect to $h$, we have
\begin{align*}
&D_{h}D_{\theta}\poprisk(\gttwo,\crl*{h_0,e_0})[\theta-\theta_{0}, h-h_0] \\
&= \En\brk*{
(T-e_0(W))(\theta(X)-\theta_0(X))(h(W)-h_0(W))
}  \\
&= \En\brk*{
\veps_{2}\cdot{}(\theta(X)-\theta_0(X))(h(W)-h_0(W))
}.
\end{align*}
Using that $\En\brk*{\veps_2\mid{}W}=0$ and $X\subseteq{}W$, the expression above is seen to be equal to zero.
\end{proof}

\section{Plug-in Empirical Risk Minimization: Further Results}
\label{app:plugin}

\subsection{Plug-in Empirical Risk Minimization: Examples}
\label{sec:specific_classes}
In this section of the appendix, we instantiate the general plug-in
ERM framework from \pref{sec:erm} to give concrete guarantees for some
concrete classes of interest. In all examples we use $\wt{O}$ to hide dependence on problem-dependent constants, $\log{}n$ factors, and $\log(\delta^{-1})$ factors.

\paragraph{High-dimensional linear classes}
For our first set of examples, we focus on high-dimensional linear predictors. \cite{chernozhukov2018plugin} gave orthogonal/debiased estimation guarantees for high-dimensional predictors using Lasso-type algorithms. Our first example shows how to recover the type of guarantee they gave, and our second example shows that we can give similar guarantees under weaker assumptions by exploiting that we work in the excess risk / statistical learning (rather than parameter estimation) framework.
\begin{example}[High-Dimensional Linear Predictors with $\ls_{1}$ Constraint]
\label{ex:lasso}
Suppose that $\besttwo\in\bbR^{d}$ is an $s$-sparse linear function with support set $T\subset\brk*{d}$ and that $\nrm*{\besttwo}_{1}\leq{}1$ and $\nrm*{x}_{\infty}\leq{}1$ almost surely under $\cD$. Define the target class via
\[
\target=\crl*{x\mapsto\tri*{\theta,x}\mid\theta\in\bbR^{d},\,\,\nrm*{\theta}_{1}\leq{}\nrm*{\theta^{\star}}_{1}}.
\]
Given $\sampletwo=\xr[n]$, define the restricted eigenvalue for the target class as
\[
\gamma_{\mathrm{re}} = \inf_{\Delta: \nrm*{\Delta_{T^{c}}}_{1}\leq{}\nrm*{\Delta_{T}}_{1}}\frac{\frac{1}{n}\nrm*{X\Delta}_{2}^{2}}{\nrm*{\Delta}_{2}^{2}},
\]
where $X\in\bbR^{n\times{}d}$ has $\xr[n]$ as rows. Then under the assumptions of \pref{thm:fast_erm}, the empirical risk minimizer guarantees that with probability at least $1-\delta$,
\[
\poprisk(\esttwo,\gtone) - \poprisk(\besttwo,\gtone)
\leq{} \Ot\prn*{\frac{s\log{}d}{n\cdot\gamma_{\mathrm{re}}} + \prn*{\Rate(\nuisance,\sampleone,\delta/2)}^{\frac{4}{1+r}}}.
\]

\end{example}
For parameter estimation it is well known that restricted eigenvalue or related conditions are required to ensure parameter consistency. For prediction however, such as assumptions are not needed if we are willing to consider inefficient algorithms. The next example shows that ERM over predictors with a hard sparsity constraint obtains the optimal high-dimensional rate for prediction in the presence of nuisance parameters with \emph{no restricted eigenvalue assumption}.
\begin{example}[High-Dimensional Linear Predictors with Hard Sparsity]
\label{ex:hard_sparse}
Suppose that $\target$ is a class of high-dimensional linear predictors obeying exact or ``hard'' sparsity:
\[
\target = \crl*{x\mapsto\tri*{\theta,x}\mid\theta\in\bbR^{d},\,\,\nrm*{\theta}_{0}\leq{}s,\,\,\nrm*{\theta}_{1}\leq{}1},
\]
and suppose $\nrm*{x}_{\infty}\leq{}1$ almost surely under $\cD$. Then under the assumptions of \pref{thm:fast_erm}, the empirical risk minimizer guarantees that with probability at least $1-\delta$,
\[
\poprisk(\esttwo,\gtone) - \poprisk(\besttwo,\gtone)
\leq{} \Ot\prn*{\frac{s\log(d/s)}{n} + \prn*{\Rate(\nuisance,\sampleone,\delta/2)}^{\frac{4}{1+r}}}.
\]
\end{example}

\paragraph{Neural networks}
We now move beyond the classical linear setting to the case to the case where the target parameters belong to a class of neural networks, a considerably more expressive class of models. Let $\loglink(t)=(1+e^{-t})^{-1}$ be the logistic link function and let $\sigma_{\mathrm{relu}}(t)=\max\crl*{t,0}$ be the ReLU function.\footnoteorinline{For vector-valued inputs $x$ we overload $\loglink(t)$ and $\sigma_{\mathrm{relu}}(t)$ to denote element-wise application.}

Our first neural network example is inspired by \cite{chen1999improved,farrell2018deep}, who analyzed neural networks for nuisance parameter estimation with parametric target \params. We depart from their approach by using neural networks to estimate \emph{target parameters}.
\begin{example}
\label{ex:nn_fast}
Suppose that the target parameters are a class of neural networks $\target=\sigma_{\mathrm{log}}\circ{}\cF$, where
\begin{equation}
\label{eq:nn_slow}
\cF = \crl*{
f(x)\ldef{}A_{L}\cdot\sigma_{\mathrm{relu}}(A_{L-1}\cdot\sigma_{\mathrm{relu}}(A_{L-2}\cdot\sigma_{\mathrm{relu}}(A_{1}x)\ldots))
\mid{} A_{i}\in\bbR^{d_{i}\times{}d_{i-1}}, \nrm*{f}_{L_{\infty}(\cD)}\leq{}M
}, 
\end{equation}
and $d_{0}=d$ and $d_{L}=1$. Let $W=\sum_{i=1}^{L}d_{i}d_{i-1}$ denote the total number of weights in the network. Under the assumptions of \pref{thm:fast_erm}, the empirical risk minimizer guarantees that with probability at least $1-\delta$,
\[
\poprisk(\esttwo,\gtone) - \poprisk(\besttwo,\gtone)
\leq{} \Ot\prn*{\frac{WL\log{}W\log{}M}{n} + \prn*{\Rate(\nuisance,\sampleone,\delta/2)}^{\frac{4}{1+r}}}.
\]
\end{example}
The target class in this example is well-suited to estimation of binary treatment effects. Note that in this example our only quantitative assumption on the network weights is that they guarantee boundedness of the output. However, the bound scales linearly with the number of parameters $W$, and thus may be vacuous for modern \emph{overparameterized} neural networks. Our next example, which is based on neural networks covering bounds from \cite{bartlett2017spectrally}, shows that by making stronger assumptions on the weight matrices we can obtain weaker dependence on the number of parameters. This comes at the price of a slower rate---$n^{-\frac{1}{2}}$ vs. $n^{-1}$.
\begin{example}
\label{ex:nn_slow}
Suppose that the target parameters are a class of neural networks $\target=\sigma_{\mathrm{log}}\circ{}\cF$, where
\[
\cF = \crl*{
f(x)\ldef{}A_{L}\cdot\sigma_{\mathrm{relu}}(A_{L-1}\cdot\sigma_{\mathrm{relu}}(A_{L-2}\cdot\sigma_{\mathrm{relu}}(A_{1}x)\ldots))
\mid{} A_{i}\in\bbR^{d_{i}\times{}d_{i-1}}, \nrm*{A_i}_{\sigma}\leq{}s_i,\nrm*{A_i}_{2,1}\leq{}b_i
},
\]
and $\nrm*{A}_{2,1}$ denotes the sum of row-wise $\ls_{2}$-norms. Suppose $\nrm*{x}_{2}\leq{}1$ almost surely under $\cD$. Under the assumptions of \pref{thm:fast_erm}, the empirical risk minimizer guarantees that with probability at least $1-\delta$,
\[
\poprisk(\esttwo,\gtone) - \poprisk(\besttwo,\gtone)
\leq{} \Ot\prn*{\frac{\prod_{i=1}^{L}s_i\cdot\prn*{\sum_{i=1}^{L}\prn*{\nicefrac{b_i}{s_i}}^{2/3}}^{3/2}}{\sqrt{n}} + \prn*{\Rate(\nuisance,\sampleone,\delta/2)}^{\frac{4}{1+r}}}.
\]
\end{example}

Let us give a concrete example where the neural network guarantees above enable oracle rates for the target class while using a more flexible \param class for the nuisance. Suppose the target parameters belong to the class in \pref{ex:nn_fast} with $L_{2}$ layers and $W_{2}$ weights, and suppose the nuisance parameters also belong to a neural network class, but with $L_{1}$ layers and $W_{1}$ weights. In the next section we will show that under certain assumptions one can guarantee $\prn*{\Rate(\nuisance,\sampleone,\delta/2)}^{4}=\Ot((W_{1}L_{1}/n)^{2})$ for such a class. In this case, \pref{ex:nn_slow} shows that we obtain oracle rates whenever $W_{1}L_{1}=\littleo(\sqrt{W_{2}L_{2}n})$, meaning the number of parameters in the nuisance network can be significantly larger than for the target network. Similar guarantees can be derived for \pref{ex:nn_fast}.

Deriving tight generalization bounds for neural networks is an active area of research and there are many more results that can be used as-is to give guarantees for the second stage in our general framework \citep{golowich2017size,arora2019fine}.

\paragraph{Kernels} For our final example, we give rates for some basic kernel classes. These examples were chosen only for concreteness, and the machinery in this section and the subsequent sections can be invoked to give guarantees for more rich and general nonparametric classes.
\begin{example}[Gaussian Kernels]
\label{ex:gaussian}
Suppose that $\target\subset(\brk*{0,1}\to\bbR)$ is unit ball in the reproducing kernel Hilbert space with the Gaussian kernel $\cK(x,x')=e^{-\frac{1}{2}(x-x')^{2}}$. Suppose $x$ is drawn from the uniform distribution over $\brk*{0,1}$. Under the assumptions of \pref{thm:fast_erm}, the empirical risk minimizer guarantees that with probability at least $1-\delta$,
\[
\poprisk(\esttwo,\gtone) - \poprisk(\besttwo,\gtone)
\leq{} \Ot\prn*{\frac{1}{n} + \prn*{\Rate(\nuisance,\sampleone,\delta/2)}^{\frac{4}{1+r}}}.
\]
\end{example}

\begin{example}[Sobolev Spaces]
\label{ex:sobolev}
Suppose the target class is the Sobolev space
\[
\target = \crl*{\theta:\brk*{0,1}\to\bbR\mid{}\theta(0)=0,\;\textnormal{$f$ is absolutely continuous with $\theta'\in{}L_{2}[0,1]$}},
\]
and suppose that $x$ is drawn from the uniform distribution on $\brk*{0,1}$.
Under the assumptions of \pref{thm:fast_erm}, the empirical risk minimizer guarantees that with probability at least $1-\delta$,
\[
\poprisk(\esttwo,\gtone) - \poprisk(\besttwo,\gtone)
\leq{} \Ot\prn*{\frac{1}{n^{2/3}} + \prn*{\Rate(\nuisance,\sampleone,\delta/2)}^{\frac{4}{1+r}}}.
\]
\end{example}

\subsubsection{Proofs}\label{app:specific}
Throughout this section we adopt the shorthand
$\nrm*{\cdot}_{n,2}=\nrm*{\cdot}_{L_2(\zr[n])}$. We first recall some
basic technical lemmas which will be used in the proofs for the examples.
\begin{lemma}[\cite{mendelson2002improving}, Lemma 4.5]
\label{lem:star_entropy}
For any real-valued function class $\cF$ with $\nrm*{f}_{n,2}\leq{}1$ for all $f\in\cF$ and any $f^{\star}$ with $\nrm*{f^{\star}}_{n,2}\leq{}1$,
\begin{equation}
\cH_2(\starhull(\cF,f^{\star}),\veps,\zr[n]) \leq{} \cH_2(\cF,\veps/2,\zr[n]) + \log(4/\veps).
\end{equation}
\end{lemma}

\begin{lemma}[\cite{wainwright2019}, Proposition 14.1]
\label{lem:fixed_point_empirical}
Let $\delta_{n}$ be the minimal solution to
\[
\Rad(\cF,\delta)\leq{} \delta^{2},
\]
where $\cF\subseteq{}(\cZ\to\bbR)$ is a star-shaped set with $\sup_{f\in\cF}\sup_{z\in\cZ}\abs*{f(z)}\leq{}1$. Then with probability at least $1-\exp\prn{-cn\delta^{2}_n}$ over the draw of data $\zr[n]$, %
\[
\delta_{n}\leq{}34\wh{\delta}_{n},
\]
where $\wh{\delta}_{n}$ is the minimal solution to
\begin{equation}
\label{eq:empirical_fixed_point}
\Rad(\cF,\delta,\zr[n])\ldef{}\En_{\eps}\brk*{\sup_{f\in\cF:\nrm*{f}_{n,2}\leq{}\delta}\abs*{\frac{1}{n}\sum_{i=1}^{n}\eps_if(z_i)}} \leq{} \delta^{2}.
\end{equation}

\end{lemma}
The following result is an immediate consequence of \pref{lem:dudley}.
\begin{lemma}
\label{prop:dudley_fixed_point}
Define
\begin{equation}
\cF(\delta,\zr[n]) = \crl[\big]{f\in\cF: \nrm*{f}_{n,2}\leq{}\delta}.
\end{equation}
Then any minimal solution to
\begin{equation}
\label{eq:dudley_fixed_point}
\int_{\frac{\delta^{2}}{8}}^{\delta} \sqrt{\frac{\cH_2(\cF(\delta,\zr[n]),\veps,\zr[n])}{n}} d\veps \leq{} \frac{\delta^{2}}{20}.
\end{equation}
is an upper bound for the fixed point $\wh{\delta}_{n}$ for \pref{eq:empirical_fixed_point}.
\end{lemma}

\begin{proof}[\text{Proof for \pref{ex:lasso}}]
Let data $\xr[n]$ be fixed. Let $\cG = \crl*{x\mapsto{}\tri*{\theta,x}\mid{}\nrm*{\theta}_{1}\leq{}b}$. Under our assumption that $\nrm*{x_t}_{\infty}\leq{}1$, \cite{zhang2002covering}, Theorem 3, implies that
\[
\cH_2(\cG,\veps,\xr[n]) \leq{} O\prn*{\frac{b^{2}\log(d)}{\veps^{2}}}.
\]
We will now establish that \[\target(\delta,\xr[n])\ldef\crl*{x\mapsto\tri*{\theta-\besttwo,x}\mid{}\theta\in\target,\nrm*{\theta-\besttwo}_{n,2}\leq{}\delta}\subseteq{}\cG\] for an appropriate choice of $b$ using the restricted eigenvalue bound. 
Let \[\cC=\crl*{\Delta\in\bbR^{d}\mid{}\nrm*{\Delta_{T^{c}}}_{1}\leq{}\nrm*{\Delta_{T}}_{1}}.\] We first claim $\target-\besttwo\subset{}\cC$. Indeed, fix $\theta\in\target$ and let $\Delta=\theta-\besttwo$. Then we have
\[
\nrm*{\besttwo}_{1}\geq{}\nrm*{\theta}_{1} = \nrm*{\besttwo + \Delta}_1 = \nrm*{\besttwo + \Delta_{T}}_1 + \nrm*{\Delta_{T^{c}}}_1
\geq \nrm*{\besttwo}_{1}- \nrm*{\Delta_{T}}_1 + \nrm*{\Delta_{T^{c}}}_1.
\]
Rearranging, we get $\nrm*{\Delta_{T^{c}}}_{1}\leq{}\nrm*{\Delta_{T}}_{1}$ as desired. Now observe that for any $\Delta\in\cC$, we have
\[
\nrm*{\Delta}_{1}\leq{}\nrm*{\Delta_{T}}_{1} + \nrm*{\Delta_{T^{c}}}_{1}
\leq{}2\nrm*{\Delta_{T}}_{1} \leq{}2\sqrt{s}\nrm*{\Delta}_{2}\leq{}\frac{2\sqrt{s}}{\sqrt{\gamma_{\mathrm{re}}}}\cdot\frac{1}{\sqrt{n}}\nrm*{X\Delta}_{2}.
\]
This implies that $\target(\delta,\xr[n])\subseteq\cG_{b}$ for $b=\frac{2\sqrt{s}}{\sqrt{\gamma_{\mathrm{re}}}}\delta$, and as a consequence
\[
\cH_2(\target(\delta,\xr[n]),\veps,\xr[n]) \leq{} O\prn*{\frac{s\log(d)}{\gamma_{\mathrm{re}}\cdot\veps^{2}}\delta^{2}}.
\]
We plug this bound into \pref{eq:dudley_fixed_point} and derive an upper bound of
\begin{align*}
\int_{\frac{\delta^{2}}{8}}^{\delta} \sqrt{\frac{\cH_2(\starhull(\target-\besttwo,0),\veps,\xr[n])}{n}} d\veps &\leq{} 
O\prn*{\delta\cdot{}\sqrt{\frac{s\log(d)}{\gamma_{\mathrm{re}}n}}
\cdot\int_{\frac{\delta^{2}}{8}}^{\delta}\frac{1}{\veps}d\veps}\\
&\leq{} 
O\prn*{\delta\log(1/\delta)\cdot{}\sqrt{\frac{s\log(d)}{\gamma_{\mathrm{re}}n}}}.
\end{align*}
where we have used that $\starhull(\target-\besttwo,0)=\target-\besttwo$.
Using \pref{prop:dudley_fixed_point}, we may now take $\delta_{n}\leq{}O\prn*{\sqrt{\frac{s\log(d/s)}{n}}\log{}n}$ in \pref{eqn:critical_radius}, then combine with \pref{thm:fast_erm} and \pref{thm:generic_strongly_convex} to get the result.

\end{proof}

\begin{proof}[\text{Proof for \pref{ex:hard_sparse}}]
Since $\nrm*{\theta}_{1}\leq{}1$ and $\nrm*{x}_{\infty}\leq{}1$, the standard covering number bound for linear classes states that the covering number at scale $\veps$ for any fixed sparsity pattern is at most $C\cdot{}s\log(1/\veps)$. We take the union over all such covers for all ${ d\choose s}\leq{}\prn*{\frac{ed}{s}}^{s}$ sparsity patterns, which implies $\cH_2(\target-\besttwo,\veps,\xr[n])\propto{}s\prn*{\log(d/s) + \log(1/\veps)}$. \pref{lem:star_entropy} further implies that 
\[
\cH_2(\starhull(\target-\besttwo,0),\veps,\xr[n])\propto{}s\prn*{\log(d/s) + \log(1/\veps)}.
\]
It is now a standard calculation to show that 
\[
\int_{\frac{\delta^{2}}{8}}^{\delta} \sqrt{\frac{\cH_2(\starhull(\target-\besttwo,0),\veps,\xr[n])}{n}} d\veps \leq{} 
O\prn*{\delta\sqrt{\log(1/\delta)}\cdot{}\sqrt{\frac{s\log(d/s)}{n}}
}.
\]
Thus, via \pref{lem:fixed_point_empirical} and \pref{prop:dudley_fixed_point}, we may take $\delta_{n}\leq{}O\prn*{\sqrt{\frac{s\log(d/s)\log{}n}{n}}}$ in \pref{eqn:critical_radius}. The final result follows by combining \pref{thm:fast_erm} and \pref{thm:generic_strongly_convex}.
  
\end{proof}

\begin{proof}[\text{Proof for \pref{ex:nn_fast}}]
Since our the target class is bounded, \pref{lem:star_entropy} implies
\[
\cH_2(\starhull(\target-\besttwo),2\veps,\xr[n])
\leq{} \cH_2(\target-\besttwo,\veps,\xr[n]) + \log(2/\veps)
= \cH_2(\target,\veps,\xr[n]) + \log(2/\veps).
\]
Recall
\[
\cF = \crl*{
f(x)\ldef{}A_{L}\cdot\sigma_{\mathrm{relu}}(A_{L-1}\cdot\sigma_{\mathrm{relu}}(A_{L-2}\cdot\sigma_{\mathrm{relu}}(A_{1}x)\ldots))
\mid{} A_{i}\in\bbR^{d_{i}\times{}d_{i-1}},\nrm*{f}_{L_{\infty}(\cD)}\leq{}M
}.
\]
Then, since $\sigma_{\mathrm{relu}}$ is $1$-Lipschitz and positive-homogeneous, we have $\cH_2(\target,\veps,\xr[n])\leq{}\cH_2(\cF,\veps,\xr[n])$. Recall that for $\brk*{0,M}$-valued classes of regressors we can relate the empirical $L_{2}$ metric to an empirical $L_{1}$ metric for a closely related VC class as follows. Let $Y\sim{}\mathrm{unif}([0,M])$, let $f,f'\in\cG$, and write
\begin{align*}
\bbP_{n}(f(X)-f'(X))^{2} &= M^{2}\bbP_{n}(\bbP_{Y}(Y\leq{}f(X))-\bbP_{Y}(Y\leq{}f'(X)))^{2}\\
&\leq{} M^{2}(\bbP_{n}\times{}\bbP_{Y})(\ind\crl*{Y\leq{}f(X)}-\ind\crl*{Y\leq{}f'(X)}).
\end{align*}
Consequently, we see that the $L_{2}$ covering number for $\cF$ on the distribution $\bbP_{n}$ at scale $\veps$, is at most the size of the $L_{1}$ cover of the class $\cF'= \crl*{(x,y)\mapsto{}\crl*{y\leq{}f(x)}\mid{}f\in\cF}$ on distribution $\bbP_n\times{}\bbP_{Y}$ at scale $\veps^{2}/M$. Thus, invoking Haussler's $L_{1}$ covering number bound for VC classes \citep{haussler1995sphere}, we have
\begin{align*}
\cH_2(\target,\veps,\xr[n]) \leq{} 2\cdot\mathrm{vc}(\cF')\log\prn*{\frac{CM}{\veps}} =  2\cdot{}\mathrm{pdim}(\cF)\log\prn*{\frac{CM}{\veps}},
\end{align*}
where $\mathrm{vc}(\cdot)$ denotes the VC dimension and $\mathrm{pdim}(\cdot)$ denotes the pseudodimension. Using Theorem 14.1 from \cite{anthony1999neural} and Theorem 6 from \cite{bartlett2017nearly}, we have
\[
\mathrm{pdim}(\cF)\leq{}O(LW\log(W)).
\]
With this bound on the metric entropy we have
\begin{align*}
\int_{\frac{\delta^{2}}{8}}^{\delta} \sqrt{\frac{\cH_2(\starhull(\target-\besttwo,0),\veps,\xr[n])}{n}} d\veps &\leq{} 
O\prn*{
\sqrt{\frac{LW\log{}W\log{}M}{n}}
}\cdot{}
\int_{\frac{\delta^{2}}{8}}^{\delta}\sqrt{\log(1/\veps)}d\veps \\
&\leq{}
O\prn*{\sqrt{\frac{LW\log{}W\log{}M}{n}}\cdot\delta\sqrt{\log(1/\delta)}}.
\end{align*}
Thus, it suffices to take $\delta_{n}\leq{}O\prn*{
\sqrt{\frac{LW\log{}W\log{}M\log{}n}{n}}
}$ in \pref{eqn:critical_radius} and appeal to \pref{thm:fast_erm} and \pref{thm:generic_strongly_convex}.
\end{proof}

\begin{proof}[\text{Proof for \pref{ex:nn_slow}}]
As in \pref{ex:nn_fast}, we have $\cH_2(\starhull(\target-\besttwo),\veps,\xr[n])\leq{}\cH_2(\cF,\veps,\xr[n])$. Theorem 3.3 of \cite{bartlett2017spectrally} implies that under our assumptions,
\[
\cH_2(\cF,\veps,\xr[n]) \leq{} O\prn*{
\frac{n\log{}W}{\veps^{2}}\prod_{i=1}^{K}s_i^{2}\cdot\prn*{\sum_{i=1}^{L}\prn*{\nicefrac{b_i}{s_i}}^{2/3}}^{3}
}.
\]
The result follows by plugging this bound into \pref{prop:dudley_fixed_point} and proceeding exactly as in the previous examples.
\end{proof}
\begin{proof}[\text{Proof for \pref{ex:gaussian} and \pref{ex:sobolev}}]
Note that each target class $\target$ has range bounded by $1$. By examples 14.4 and 14.3 in \cite{wainwright2019}, we may take $\delta_{n}=c{\sqrt{\log{}n/n}}$ and $\delta_{n}=cn^{-1/3}$ in \pref{eqn:critical_radius} for the gaussian and Sobolev classes respectively.
We combine with this with \pref{thm:fast_erm} and \pref{thm:generic_strongly_convex}.
\end{proof}

\subsection{Plug-in Empirical Risk Minimization: Refined Guarantees for VC Classes}
\label{app:vc}
In this section of the appendix we use the general tools developed in \pref{sec:erm} to provide efficient/variance-dependent oracle rates for VC classes with general Lipschitz losses. Our main result shows that for VC classes with dimension $d$, the excess risk enjoyed by variance penalization grows exactly as $O(\sqrt{V^{\star}d/n})$ (where $V^{\star}$, as before, is the variance of the loss at the pair $(\besttwo,\gtone)$) so long as the nuisance estimator converges at a rate of $o(n^{-1/4})$. The key to our approach is to assume boundedness of the so-called \emph{Alexander capacity function}, a classical quantity that arises in the study of ratio type empirical processes \citep{gine2006concentration}.

To be more precise, for this example we assume that $\target$ is a class of binary predictors with VC dimension $d$, and let $\ls$ have the following policy learning structure:
\[
\ls(\theta,g;z) = \Gamma(g,z)\cdot\theta(x),\quad\quad\poprisk(\theta,g)=\En\brk{\ls(\theta,g;z)},
\]
where $\Gamma$ is a known function. Our goal is to derive a bound for which the leading term only scales with $V^{\star}$ rather than the loss range. Our results depend on a variant of the Alexander capacity function \citep{gine2006concentration,hanneke2014theory}. Letting \[\targeteps=\crl*{\theta\in\target:\En\brk*{\Gamma^2(\gtone,z)(\theta(x)-\besttwo(x))^{2}}\leq{}\veps^{2}},\] the capacity function is defined as
\begin{equation}
\tau^{2}(\veps) = \frac{\En\brk{\sup_{\theta\in\targeteps}\Gamma^{2}(\gtone,z)(\theta(x)-\besttwo(x))^{2}}}{\veps^{2}}.
\end{equation}
When $\Gamma$ is the unweighted classification loss, this definition recovers the classical definition of the capacity function \citep{gine2006concentration,hanneke2014theory}. Beyond boundedness of the capacity function, we make the following assumption.
\begin{assumption}
\label{ass:alexander}
\pref{ass:universal_orthogonality} holds along, with the following bounds:
\begin{itemize}
\item $\abs*{\Gamma(g,z)}\leq{}R$ almost surely for all $g\in\cG$, for some $R\geq{}1$.
\item $\En\brk*{\Gamma^{2}(g_0,z)\mid{}x}\geq{}\gamma$ almost surely.
\item $\prn*{\En(\Gamma(g,z)-\Gamma(g_0,z))^{4}}^{1/4}\leq{}L\nrm*{g-g_0}_{\nuisance}$ for all $g\in\nuisance$, for some seminorm $\nrm*{\cdot}_{\nuisance}$.
\item The first stage algorithm provides an estimation error bound with respect to $\nrm*{\cdot}_{\nuisance}$, i.e.
  $\nrm*{\estone-\gtone}_{\nuisance}\leq{}\Rate(\nuisance,S,\delta)$ with probability at least $1-\delta$ over the draw of $S$.
\item \pref{ass:smooth_loss_slow} holds with respect to $\nrm*{\cdot}_{\nuisance}$ with constant $\beta$.
\end{itemize}
\end{assumption}

\begin{theorem}
\label{thm:alexander}
Suppose that \pref{ass:alexander} holds, and define $\tau_0\ldef{}\sup_{\veps\geq{}\sqrt{\nicefrac{d}{n}}}\crl{\tau(\veps)}$.
Then variance-penalized empirical risk minimizer guarantees that with probability at least $1-\delta$,
\begin{align*}
\label{eq:policy_alexander}
&\poprisk(\esttwo,\gtone) - \poprisk(\besttwo,\gtone)\\
&\leq{} \Ot\prn*{\sqrt{\frac{V^{\star}d\log\tau_0}{n}} + \frac{(R+L)d\log\tau_0}{n} + (\beta + (R+L)^{2}\gamma^{-1/2})\cdot\prn*{\Rate(\nuisance,\sampleone,\delta/2)}^{2}}.
\end{align*}
\end{theorem}
Note that whenever $\Rate(\nuisance,\sampleone,\delta/2)=\littleo(n^{-1/4})$ the asymptotic rate depends only on the variance at $\besttwo$ and $\gtone$, not on the problem-dependent parameters $L/R/\beta\gamma$. Furthermore, whenever the capacity function is constant the asymptotic rate is exactly $O(\sqrt{V^{\star}d/n})$.

Variance-dependent bounds that obtain the efficient $O(\sqrt{V^{\star}d/n})$ rate have been the subject of much recent investigation, and there is much interest in understanding when the $O(\sqrt{V^{\star}d\log{}n/n})$ rate obtained by naive approaches can be improved. To give a brief survey, the seminal empirical variance bound due to \cite{maurer2009empirical} when applied directly to this setting gives a suboptimal $O(\sqrt{V^{\star}d\log{}n/n})$ rate. Recent work of \cite{athey2017efficient} shows that for a specific loss and nuisance parameter setup arising in policy learning, the $\log{}n$ can be replaced with a certain worst-case variance parameter. Our result, \pref{thm:alexander} is complementary, and shows that the $\log{}n$ can be replaced by the capacity function for \emph{general losses}. It appears unlikely that the $\log{}n$ factor can be removed without at least some type of assumption. Indeed the results in \cite{rakhlin2017empirical} imply that there are indeed VC classes for which the critical radius grows as $\sqrt{d\log{}n/n}$ in the worst case.

The proof of \pref{thm:alexander} can be broken into three parts: First, we apply the previous results of this section to show that the excess risk obtained by variance penalization depends on the critical radius of the class $\ls\circ\target$. Second, we show that in the absence of first-stage estimation error, the capacity function controls the critical radius. Finally, we show that the impact of nuisance estimation error on the capacity function is of second order.

\subsubsection{Proof of Theorem \ref*{thm:alexander}}
\newcommand{\plugincap}{\wh{\tau}}

Define the function class $\cF = \{z\mapsto{}\ell(\theta(x), \estone(w);z): \theta \in \target\}$. We assume for now that $\nrm*{f}_{\infty}\leq{}1$ for all $f\in\cF$, that $\Gamma$ is $1$-Lipschitz (i.e. $L=R=1$ in the theorem statement) and that $\En\brk*{\Gamma^{2}(g_0,z)\mid{}x}\geq{}\gamma$; the general case will be handled by rescaling at the end of the proof.. 

Our starting point is to appeal to \pref{thm:moment_penalized} . In particular, let $\delta_n \geq 0$ be any solution to the inequality:
\begin{equation}
\Radexp(\sh(\cF - f^{\star}),r) \leq \delta^2,
\end{equation}
where $f^{\star}(z) \ldef \ell(\besttwo(x), \estone(w);z)$. Then if $\esttwo$ is the outcome of variance-penalized ERM, we have that by \pref{thm:variance_penalized}, with probability at least $1-\delta$,
\begin{align*}
\journal{&}\poprisk(\esttwo, \gtone) - \poprisk(\besttwo, \gtone)  \journal{\\&}= O\left(\sqrt{V^{\star}} \left(\delta_n+\sqrt{\frac{\log(1/\delta)}{n}}\right) + \delta_n^2+\frac{\log(1/\delta)}{n}
+ \beta\cdot{}\prn*{\Rate(\nuisance,\sampleone,\delta/2)}^{2} + \Radexp^{2}(\cF)\right).
\end{align*}

\paragraph{Moving to capacity function at $g_0$}
Per the discussion in the prequel, we focus on bounding the critical radius in the case where $\cF$ is bounded by $1$ and $\Gamma$ is $1$-Lipschitz. We wish to make use of the capacity function, which is defined at $g_0$, but the local Rademacher complexity we need to bound is that of $\cF$, which evaluates the weight function $\Gamma$ at $\estone$. To make progress, we show how to use the capacity function defined in the theorem statement to bound the following ``plug-in'' variant:
\[
\plugincap^{2}(\veps) = \frac{\En\sup_{\theta:\En\brk*{\Gamma(\estone,z)^{2}(\theta(x)-\besttwo(x))^{2}}\leq{}\veps^{2}}\Gamma^{2}(\estone,z)(\theta(x)-\besttwo(x))^{2}}{\veps^{2}}.
\]

We first show how to relate the $L_2$ norm at $\estone$ to the $L_2$ norm at $\gtone$. 
Define $\nrm*{\theta}_{\target}=\sqrt{\En\Gamma(\estone,z)^{2}(\theta(x)-\besttwo(x))^{2}}$.  Then for any $\theta\in\target$ we have
\begin{align*}
\En\brk*{\Gamma(\estone,z)^{2}(\theta(x)-\besttwo(x))^{2}}
&\geq{} \frac{1}{2}\En\Gamma(\gtone,z)^{2}(\theta(x)-\besttwo(x))^{2} - \En(\Gamma(\estone,z)-\Gamma(\gtone,z))^{2}(\theta(x)-\besttwo(x))^{2}\\
  &= \frac{1}{2}\nrm*{\theta-\besttwo}_{\target}^{2} - \En(\Gamma(\estone,z)-\Gamma(\gtone,z))^{2}(\theta(x)-\besttwo(x))^{2}.
    \end{align*}
    Using AM-GM and boundedness of $\theta$, for any $\eta>0$ this is lower bounded by
    \begin{align*}
       \frac{1}{2}\nrm*{\theta-\besttwo}_{\target}^{2} - \frac{1}{2\eta}\En(\Gamma(\estone,z)-\Gamma(\gtone,z))^{4} - \frac{\eta}{2}\En(\theta(x)-\besttwo(x))^{2}.
        \end{align*}
        Using the Lipschitz assumption and conditional lower bound on $\Gamma$, we further lower bound by
        \begin{align*}
\geq \frac{1}{2}\nrm*{\theta-\besttwo}_{\target}^{2} - \frac{1}{2\eta}\nrm*{\estone-\gtone}_{\nuisance}^{4} - \frac{\eta}{2\gamma}\nrm*{\theta-\besttwo}_{\target}^{2}.
\end{align*}
Hence, by choosing $\eta=\gamma/2$ and rearranging, we get
\begin{equation}
\label{eq:l2_alexander}
\nrm*{\theta-\besttwo}_{\target}^{2} \leq{} 4\En\brk*{\Gamma(\estone,z)^{2}(\theta(x)-\besttwo(x))^{2}} + \frac{4}{\gamma}\nrm*{\estone-\gtone}_{\nuisance}^{4}.
\end{equation}
We proceed to bound the capacity function $\plugincap$. Let $\veps_0=\frac{2}{\gamma^{1/2}}\nrm*{\estone-\gtone}_{\nuisance}^{2}$. Let $\veps\geq{}\veps_0$ be fixed and let $\theta\in\target$ be any policy with $\En\brk*{\Gamma(\estone,z)^{2}(\theta(x)-\besttwo(x))^{2}}\leq{}\veps^{2}$. Then equation \pref{eq:l2_alexander} implies that that $\nrm*{\theta-\besttwo}_{\target}^{2}\leq{}5\veps^{2}$, and so
\[
\plugincap^{2}(\veps) \leq{} \frac{\En\sup_{\theta:\nrm*{\theta-\besttwo}_{\target}^{2}\leq{}5\veps^{2}}\Gamma^{2}(\estone,z)(\theta(x)-\besttwo(x))^{2}}{\veps^{2}}.
\]
To handle the term in the numerator we proceed similar to the proof of \pref{eq:l2_alexander}. We have
\begin{align*}
&\En\sup_{\theta:\nrm*{\theta-\besttwo}_{\target}^{2}\leq{}5\veps^{2}}\Gamma^{2}(\estone,z)(\theta(x)-\besttwo(x))^{2}\\
&\leq{} 2\En\sup_{\theta:\nrm*{\theta-\besttwo}_{\target}^{2}\leq{}5\veps^{2}}\Gamma^{2}(\gtone,z)(\theta(x)-\besttwo(x))^{2}
+ 2\En\sup_{\theta:\nrm*{\theta-\besttwo}_{\target}^{2}\leq{}5\veps^{2}}(\Gamma(\estone,z)-\Gamma(\gtone,z))^{2}(\theta(x)-\besttwo(x))^{2}.
\end{align*}
Fix any $\eta>0$. We use AM-GM and boundedness of policies to upper bound the second term as
\begin{align*}
&\En\sup_{\theta:\nrm*{\theta-\besttwo}_{\target}^{2}\leq{}5\veps^{2}}(\Gamma(\estone,z)-\Gamma(\gtone,z))^{2}(\theta(x)-\besttwo(x))^{2} \\
&\leq{}\frac{1}{\eta}\En(\Gamma(\estone,z)-\Gamma(\gtone,z))^{4}+ \eta\En_{x}\sup_{\theta:\nrm*{\theta-\besttwo}_{\target}^{2}\leq{}5\veps^{2}}(\theta(x)-\besttwo(x))^{2} \\
&\leq{}\frac{1}{\eta}\nrm*{\estone-\gtone}_{\nuisance}^{4}+ \frac{\eta}{\gamma}\En_{x}\sup_{\theta:\nrm*{\theta-\besttwo}_{\target}^{2}\leq{}5\veps^{2}}\En\brk*{\Gamma^{2}(\gtone,z)\mid{}x}(\theta(x)-\besttwo(x))^{2} \\
&\leq{}\frac{1}{\eta}\nrm*{\estone-\gtone}_{\nuisance}^{4}+ \frac{\eta}{\gamma}\En\sup_{\theta:\nrm*{\theta-\besttwo}_{\target}^{2}\leq{}5\veps^{2}}\Gamma^{2}(\gtone,z)(\theta(x)-\besttwo(x))^{2}
\intertext{We choose $\eta=\gamma$ and recall the definition of $\veps_0$, which gives}
&\leq{}\veps_{0}^{2}/4+ \En\sup_{\theta:\nrm*{\theta-\besttwo}_{\target}^{2}\leq{}5\veps^{2}}\Gamma^{2}(\gtone,z)(\theta(x)-\besttwo(x))^{2}.
\end{align*}
Putting everything together, we get
\[
\plugincap^{2}(\veps) \leq{} \frac{4\cdot{}\En\sup_{\theta:\nrm*{\theta-\besttwo}_{\target}^{2}\leq{}5\veps^{2}}\Gamma^{2}(\gtone,z)(\theta(x)-\besttwo(x))^{2} + \veps_0^{2}/2}{\veps^{2}}.
\]
Thus, for all $\veps\geq{}\veps_0$ we have 
\begin{equation}
\label{eq:alexander_nuisance}
\plugincap^{2}(\veps) \leq{} 20\cdot\tau^{2}(5\veps) + 3.
\end{equation}

\paragraph{Bounding the critical radius}
For any class $\cF$ we define $\cF_{\delta}=\crl*{f\in\cF:\nrm*{f}_{2}\leq\delta}$.
We work with the following empirical version of the local Rademacher complexity
\begin{equation}
\Rad(\cF,\delta,\zr[n]) = \En_{\epsilon}\brk*{\sup_{f\in \cF_\delta} \left| \frac{1}{n}\sum_{i=1}^n \epsilon_i f(z_i) \right| },
\end{equation}
which has $\Radexp(\cF,\delta) = \En_{\zr[n]}\brk*{\Rademp(\cF,\delta,\zr[n])}$. Let the draw of $\zr[n]$ be fixed. Invoking \pref{lem:dudley}, we have
\begin{align*}
\Rad(\starhull(\cF-f^{\star}),\delta,\zr[n]) \leq~& \inf_{\alpha\geq 0}\brk*{ 4\alpha + 10\int_{\alpha}^{\sup_{h\in \sh(\cF-\fstar)_\delta} \|h\|_n} \sqrt{\frac{\cH_2(\sh(\cF-\fstar)_\delta,\veps, z_{1:n})}{n}} d\veps}.
\end{align*}
Using that any $h\in\sh(\cF-\fstar)_{\delta}$ can be written as $r\cdot(f-f^{\star})$, where $\nrm*{f-f^{\star}}_{2}\leq{}\delta$ and $r\in\brk*{0,1}$, a simple discretization argument (cf. proof of \pref{lem:star_entropy}) shows that
\[
\cH_2(\sh(\cF-\fstar)_\delta, \veps,z_{1:n})\leq{}\cH_2((\cF-f^{\star})_\delta, \veps/2, z_{1:n}) + \log\prn[\bigg]{2\sup_{h\in(\cF-\fstar)_{\delta}}\nrm*{h}_{n}/\veps}.
\]
Let us adopt the shorthand $v_n=\sup_{h\in(\cF-\fstar)_{\delta}}\nrm*{h}_{n}$. It follows from the usual symmetrization argument that $\En{}v_n^{2}\leq{}\delta^{2} + 2\Radexp(\cF-f^{\star},\delta)$. Letting $\alpha=0$ be fixed, we can summarize our argument so far as
\[
\Rad(\starhull(\cF-f^{\star}),\delta,\zr[n]) \leq   10\int_{0}^{v_n} \sqrt{\frac{\cH_2((\cF-f^{\star})_{\delta}, \veps/2 ,z_{1:n})}{n}} d\veps
+ 10\int_{0}^{v_n}\sqrt{\log(2v_n/\veps)/n}d\veps.
\]
Furthermore, using a change of variables we have
\[
\int_{0}^{v_n}\sqrt{\log(2v_n/\veps)/n}d\veps
\leq{}v_n\int_{0}^{1}\sqrt{\log(2/\veps)/n}d\veps \leq{} C\cdot{}\frac{v_n}{\sqrt{n}}.
\]
We now handle the covering integral for $(\cF-f^{\star})_{\delta}$. 
Let \[
\target(\delta)=\crl*{\theta\in\target\mid{}\En\brk*{\Gamma^2(\estone,z)(\theta(x)-\besttwo(x))^{2}}\leq\delta^{2}}.
\]
Our approach is to upper bound the empirical $L_2$ covering number for $(\cF-f^{\star})_{\delta}$ by the covering number of the class $\target(\delta)$ with respect to Hamming error. Let the Hamming error on a set $S'=\crl*{x_1,\ldots,x_M}$ be defined via
\[
d_{H,S'}(\theta,\theta')=\frac{1}{M}\sum_{i=1}^{M}\indic\crl*{\theta(x'_i)\neq{}\theta'(x'_i)}.
\]
We claim that there is a choice for the dataset $S'=x'_1,\ldots,x'_M$ such that for all $g,g'\in(\cF-f^{\star})_{\delta}$, the empirical $L_2$ error on $S_2$ is upper bounded by the empirical Hamming error of associated policies $\theta,\theta'$ on $S'$.

Let $h=(f-f^{\star})$ and $h'=(f'-f^{\star})$ be fixed elements of $(\cF-f^{\star})_{\delta}$, and let $\theta,\theta'\in\target(\delta)$ be such that $f(z)=\Gamma(\estone,z)(\theta(x)-\besttwo(x))$ and likewise for $f'$ and $\theta'$. Define $\Gamma_i = \Gamma(\estone,z_i)$, and take $S'$ to contain of $m_{i}\ldef\ceil*{\frac{\sup_{\theta\in\target(\delta)}\Gamma_i^{2}(\theta(x_i)-\theta^{\star}(x_i))^{2}}{\delta^{2}}}$ copies of example $x_i$ for each $i$. With this choice, we have
\begin{align*}
d_{H,S'}(\theta,\theta')&=\frac{1}{M}\sum_{i=1}^{M}\indic\crl*{\theta(x'_i)\neq{}\theta'(x'_i)} \\
&=\frac{1}{M}\sum_{i=1}^{n}\ceil*{\frac{\sup_{\wt{\theta}\in\target(\delta)}\Gamma_i^{2}(\wt{\theta}(x_i)-\theta^{\star}(x_i))^{2}}{\delta^{2}}}\indic\crl*{\theta(x_i)\neq{}\theta'(x_i)} \\
&\geq\frac{1}{2M}\sum_{i=1}^{n}\frac{\Gamma_i^{2}(\theta(x_i)-\theta^{\star}(x_i))^{2}+ \Gamma_i^{2}(\theta'(x_i)-\theta^{\star}(x_i))^{2}}{\delta^{2}}\indic\crl*{\theta(x_i)\neq{}\theta'(x_i)} \\
&\geq\frac{1}{4M}\sum_{i=1}^{n}\frac{\Gamma_i^{2}(\theta(x_i)-\theta'(x_i))^{2}}{\delta^{2}}\indic\crl*{\theta(x_i)\neq{}\theta'(x_i)} \\
&=\frac{1}{4M}\sum_{i=1}^{n}\frac{\Gamma_i^{2}(\theta(x_i)-\theta'(x_i))^{2}}{\delta^{2}}\\
                        &=\frac{n}{4M\delta^{2}}\nrm*{h-h'}^{2}_{n,2}
\end{align*}
Thus, if we let $\veps'=\frac{n}{4M\delta^{2}}\veps^{2}$, then any $\veps'$-cover in Hamming error is an $\veps$-cover in $L_2$. Now define \[
  u_{n}^{2}=\frac{1}{n}\sum_{i=1}^{n}\sup_{\theta\in\target(\delta)}\Gamma_{i}^{2}(\theta(x_i)-\besttwo(x_i))^{2},\]
and note that $u_n^{2}\geq{}v_{n}^{2}$ by definition. We invoke the following facts.
\begin{itemize}
\item $M\leq{}n+\sum_{i=1}^{n}\frac{\sup_{\theta\in\target(\delta)}\Gamma_i^{2}(\theta(x_i)-\theta^{\star}(x_i))^{2}}{\delta^{2}}=n(1+ u_n^{2}/\delta^{2})$.
\item Haussler's bound \citep{haussler1995sphere} implies that any class with VC dimension $d$ admits a $\veps$--Hamming error cover of size $e(d+1)\prn*{\frac{2e}{\veps}}^{d}$.
\end{itemize}
Putting everything together, we have
\begin{align*}
\int_{0}^{v_n} \sqrt{\frac{\cH_2((\cF-f^{\star})_{\delta}, \veps/2 , z_{1:n})}{n}} d\veps
\leq{}
\int_{0}^{v_n} \sqrt{\frac{d\log(2e(\delta^{2}+u_{n}^{2})/\veps^{2})}{n}} d\veps
+ C\cdot{}v_n\sqrt{\log{}d/n}.
\end{align*}
It follows from the usual symmetrization argument that $\En{}\brk{v_n^{2}}\leq{}\delta^{2} + 2\Radexp(\cF-f^{\star},\delta)$. Furthermore, using the concentration bound in \pref{eq:talagrand1} (we use the assumed boundedness of elements of $\cF$ to simplify \pref{eq:talagrand1} to the form that appears on this page), there exists a constant $C\geq{}1$ such that for any $s>0$, with probability at least $1-e^{-s}$ over the draw of $\zr[n]$,
\[
v_{n}^{2} \leq{} C\prn*{\delta^{2} + \Radexp(\cF-f^{\star},\delta) + \frac{s}{n}} \rdef{} \tilde{\delta}^{2}.
\]
Thus, conditioning on this event, we have
\begin{align*}
\int_{0}^{v_n} \sqrt{\frac{d\log(2e(\delta^{2}+u_{n}^{2})/\veps^{2})}{n}} d\veps 
&\leq{}\int_{0}^{\tilde{\delta}} \sqrt{\frac{d\log(2e(\delta^{2}+u_{n}^{2})/\veps^{2})}{n}} d\veps\\
&\leq{}\tilde{\delta}\int_{0}^{1} \sqrt{\frac{d\log(2e(1+u_{n}^{2}/\delta^{2})/\veps^{2})}{n}} d\veps\\
&\leq{}C\cdot{}\tilde{\delta}\sqrt{\frac{d\log(2e(1+u_{n}^{2}/\delta^{2}))}{n}},
\end{align*}
where the second inequality uses a change of variables and that $\delta\leq{}\tilde{\delta}$. 

Now, to summarize our developments so far, we have shown that with probability at least $1-e^{-s}$,
\begin{align*}
&\Rad(\starhull(\cF-f^{\star}),\delta,\zr[n])\\
&\leq{}C\prn*{
\tilde{\delta}\sqrt{\frac{d\log(2e(1+u_{n}^{2}/\delta^{2}))}{n}}
+
\tilde{\delta}\sqrt{\log{}d/n}
} \\
&\leq{}C\prn*{
\tilde{\delta}\sqrt{\frac{d\log(2e(1+u_{n}^{2}/\delta^{2}))}{n}}
} \\
&=C\prn*{
(\delta+\sqrt{\Radexp(\cF-f^{\star},\delta)}+\sqrt{s/n})\sqrt{\frac{d\log(2e(1+u_{n}^{2}/\delta^{2}))}{n}}
}.
\end{align*}
Using Markov's inequality, we have that with probability at least $1-e^{-s}$, $u_{n}^{2}\leq{}\En{}\brk{u_n^{2}}\cdot{}e^{s}$. Thus, by union bound, with probability at least $1-2e^{-s}$,
\begin{align*}
&\Rad(\starhull(\cF-f^{\star}),\delta,\zr[n])\\
&\leq{}C\prn*{
\sqrt{s}(\delta+\sqrt{\Radexp(\cF-f^{\star},\delta)}+\sqrt{s/n})\sqrt{\frac{d\log(2e^{2}(1+\En_{\zr[n]}\brk{u_{n}^{2}}/\delta^{2}))}{n}}
}.
\end{align*}
Integrating out this tail bound, we get that
\begin{align*}
&\Radexp(\starhull(\cF-f^{\star}),\delta)\\
&\leq{}C\prn*{
(\delta+\sqrt{\Radexp(\cF-f^{\star},\delta)}+\sqrt{1/n})\sqrt{\frac{d\log(2e^{2}(1+\En_{\zr[n]}\brk{u_{n}^{2}}/\delta^{2}))}{n}}
}.
\end{align*}
Using AM-GM and that $\Radexp(\cF-f^{\star},\delta)\leq\Radexp(\starhull(\cF-f^{\star}),\delta)$, then rearranging, this implies
\begin{align*}
\Radexp(\starhull(\cF-f^{\star}),\delta) \leq{}
C\prn*{
\delta\sqrt{\frac{d\log(2e^{2}(1+\En_{\zr[n]}\brk{u_{n}^{2}}/\delta^{2}))}{n}}
+
\frac{d\log(2e(1+\En_{\zr[n]}\brk{u_{n}^{2}}/\delta^{2}))}{n}}.
\end{align*}
We now bound the ratio $\En_{\zr[n]}\brk*{u_{n}^{2}}/\delta^{2}$. We have
\begin{align*}
\En_{\zr[n]}\brk{u_{n}^{2}}=\frac{1}{n}\sum_{i=1}^{n}\En_{z_i}\sup_{\theta\in\target(\delta)}\Gamma^{2}(\estone,z_i)(\theta(x_i)-\theta^{\star}(x_i))^{2}
&=\En_{z}\sup_{\theta\in\target(\delta)}\Gamma^{2}(\estone,z)(\theta(x)-\theta^{\star}(x))^{2}
\\&\leq{}\plugincap^{2}(\delta)\delta^{2}.
\end{align*}
Thus, using the relationship between $\tau$ and $\plugincap$ established in the previous section of the proof, if $\delta>\veps_0$ we have
\[
\En_{\zr[n]}\brk*{u_{n}^{2}}/\delta^{2}\leq{}20\cdot\tau^{2}(5\delta) + 3,
\]
and so for all $\delta>\veps_0$,
\[
\Radexp(\starhull(\cF-f^{\star}),\delta) \leq{}
C\prn*{
\delta\sqrt{\frac{d\log\tau(5\delta)}{n}}
+
\frac{d\log\tau(5\delta)}{n}}.
\]
In particular, we can see from this expression that taking
\[
\delta_n\propto\sqrt{\frac{d\log\tau_0}{n}}+\veps_0
\]
yields a valid upper bound on the critical radius.
\paragraph{Final bound}
Putting together the excess risk bound and the critical radius bound, we have
\begin{align*}
&\poprisk(\esttwo, \gtone) - \poprisk(\besttwo, \gtone) \\
& = O\left(\sqrt{V^{\star}} \left(\delta_n+\sqrt{\frac{\log(1/\delta)}{n}}\right) + \delta_n^2+\frac{\log(1/\delta)}{n}
+ (1+\beta)\prn*{\Rate(\nuisance,\sampleone,\delta/2)}^{2} +\Radexp^{2}(\cF)\right) \\
&\leq{} O\prn*{
\sqrt{\frac{V^{\star}d\log(\tau_0/\delta)}{n}}
+ \frac{d\log(\tau_0/\delta)}{n} 
+ (\gamma^{-1/2}+\beta)\prn*{\Rate(\nuisance,\sampleone,\delta/2)}^{2}+\Radexp^{2}(\cF)
}.
\end{align*}
To handle the square Rademacher complexity term, we recall that since $\abs*{\Gamma}\leq{}1$, the main result of \cite{haussler1995sphere} implies that $\Radexp(\cF)\leq{}\sqrt{\frac{d}{n}}$; since this term is squared in the final bound, its contribution is of lower order.

To deduce the final bound in the general $R$-bounded $L$-Lipschitz case we divide the class by $(L+R)$, then rescale the final bound (observing that $\beta$, $\gamma$, and $V^{\star}$ all vary appropriately under the rescaling).

\qed

\part{Proofs for Main Results}
\label{part:proofs}

\section{Preliminaries}
\label{app:preliminaries}

We invoke the following version of Taylor's theorem and its
directional derivative generalization repeatedly.
\begin{proposition}[Taylor expansion]
\label{prop:taylor}
Let $a\leq{}b$ be fixed and let $f:I\to\bbR$, where $I\subseteq{}\bbR$
is an open interval containing $a,b$. If $f$ is $(k+1)$-times differentiable, then there exists $c\in\brk*{a,b}$ such that
\[
f(a) = f(b) + \sum_{i=1}^{k}\frac{1}{i!}f^{(i)}(b)(a-b)^{i} + \frac{1}{(k+1)!}f^{(k+1)}(c)(a-b)^{k+1}.
\]
Let $F:\cF\to\bbR$, where $\cF$ is a vector space of functions. For
any $g,g'\in\cF$, if $t\mapsto{}F(t\cdot{}g+(1-t)\cdot{}g')$ is
$(k+1)$-times differentiable over an open interval containing $\brk*{0,1}$, then there exists $\bar{g}\in\conv(\crl*{g,g'})$ such that
\[
F(g') = F(g) + \sum_{i=1}^{k}\frac{1}{i!}D_{g}^{i}F(g)[\underbrace{g'-g,\ldots,g'-g}_{\text{$i$ times}}] + \frac{1}{(k+1)!}D_{g}^{k+1}F(\bar{g})[\underbrace{g'-g,\ldots,g'-g}_{\text{$k+1$ times}}].
\]
\end{proposition}

\section{Proofs from Section \journal{\ref{sec:orthogonal}}\preprint{\ref*{sec:orthogonal}}}
\label{app:orthogonal}

\subsection{Omitted Proofs for Main Results}

\begin{proof}[\pfref{thm:orthogonal_slow}]
To begin, we use the guarantee for the second stage from \pref{def:algorithms} and perform straightforward manipulation to show
\[
L_{\cD}(\esttwo,\gtone) - L_{\cD}(\besttwo, \gtone) 
\leq{} (L_{\cD}(\esttwo,\gtone) - L_{\cD}(\esttwo, \estone)) + (L_{\cD}(\besttwo, \estone) - L_{\cD}(\besttwo, \gtone)) 
+ \Rate(\target, \sampletwo, \delta/2\midsem\esttwo,\estone).
\]
Using continuity guaranteed by \pref{ass:smooth_loss_slow}, we perform a second-order Taylor expansion with respect to $g$ for each pair of loss terms in the preceding expression to conclude that there exist $g,g'\in\starhull(\nuisance,\gtone)$ such that
\begin{align*}
&(L_{\cD}(\esttwo,\gtone) - L_{\cD}(\esttwo, \estone)) + (L_{\cD}(\besttwo, \estone) - L_{\cD}(\besttwo, \gtone)) \\
&=
-D_{g}L_{\cD}(\esttwo,\gtone)[\estone-\gtone] - \frac{1}{2}\cdot{}D^{2}_{g}L_{\cD}(\esttwo,g)[\estone-\gtone,\estone-\gtone]\\
&~~~~~+ D_{g}L_{\cD}(\besttwo,\gtone)[\estone-\gtone] + \frac{1}{2}\cdot{}D^{2}_{g}L_{\cD}(\besttwo,g')[\estone-\gtone,\estone-\gtone].
\intertext{Using the smoothness promised by \pref{ass:smooth_loss_slow}:}
&\leq{}
-D_{g}L_{\cD}(\esttwo,\gtone)[\estone-\gtone]
+ D_{g}L_{\cD}(\besttwo,\gtone)[\estone-\gtone] + \beta{}\nrm*{\estone-\gtone}_{\nuisance}^{2}.
\end{align*}
To relate the two derivative terms, we apply another second-order Taylor expansion (which is possible due to \pref{ass:smooth_loss_slow}), this time with respect to the target predictor.
\begin{align*}
&D_{g}L_{\cD}(\esttwo,\gtone)[\estone-\gtone] \\
&= D_{g}L_{\cD}(\besttwo,\gtone)[\estone-\gtone]
+ D_{\theta}D_{g}L_{\cD}(\besttwo,\gtone)[\estone-\gtone,\esttwo-\besttwo]
+ \frac{1}{2}\cdot{}D_{\theta}^{2}D_{g}L_{\cD}(\bar{\theta},\gtone)[\estone-\gtone,\esttwo-\besttwo, \esttwo-\besttwo],
\end{align*}
where $\bar{\theta}\in\conv(\crl{\esttwo,\besttwo})$. Universal orthogonality immediately implies that
\[
 D_{\theta}D_{g}L_{\cD}(\besttwo,\gtone)[\estone-\gtone,\esttwo-\besttwo] = 0.
\]
Furthermore, observe that
\begin{align*}
&D_{\theta}^{2}D_{g}L_{\cD}(\bar{\theta},\gtone)[\estone-\gtone,\esttwo-\besttwo, \esttwo-\besttwo] \\
&= \lim_{t\to{}0}\frac{D_{\theta}D_{g}L_{\cD}(\bar{\theta} + t(\esttwo-\besttwo),\gtone)[\estone-\gtone,\esttwo-\besttwo] - D_{\theta}D_{g}L_{\cD}(\bar{\theta},\gtone)[\estone-\gtone,\esttwo-\besttwo]}{t}.
\end{align*}
Since $\bar{\theta} + t(\esttwo-\besttwo)\in\starhull(\targetemp,\besttwo) + \starhull(\targetemp-\besttwo,0)$ for all $t\in\brk*{0,1}$, including $t=0$, universal orthogonality (\pref{ass:universal_orthogonality}) implies that both terms in the numerator are zero, and hence $D_{\theta}^{2}D_{g}L_{\cD}(\bar{\theta},\gtone)[\estone-\gtone,\esttwo-\besttwo, \esttwo-\besttwo]=0$. We conclude that $D_{g}L_{\cD}(\esttwo,\gtone)[\estone-\gtone]
= D_{g}L_{\cD}(\besttwo,\gtone)[\estone-\gtone]$. Using this identity in the excess risk upper bound, we arrive at
\begin{align*}
&L_{\cD}(\esttwo,\gtone) - L_{\cD}(\besttwo, \gtone) \\
&\leq{}  \Rate(\target, \sampletwo, \delta/2\midsem\esttwo,\estone) + \beta{}\cdot{}\nrm*{\estone-\gtone}_{\nuisance}^{2}\\
&\leq{}  \Rate(\target, \sampletwo, \delta/2\midsem\esttwo,\estone) + \beta{}\cdot{}\prn*{\Rate(\nuisance,\sampleone,\delta/2)}^{2}.
\end{align*}
\end{proof}

\subsection{Proofs for Examples}
\label{app:example_proofs}

\begin{proof}[\pfref{prop:treatment_example}]
  We first verify that the conditions of
  \pref{thm:generic_strongly_convex} are
  satisfied.
  To establish orthogonality (\pref{ass:orthogonal}) for the propensities $e$, let $\theta,\theta'$ be fixed. Then we have
\begin{align*}
&D_{e}D_{\theta}\poprisk(\theta,\crl*{m_0,e_0})[\theta'-\theta,e-e_{0}] \\ 
&= 2\cdot{}\En\brk*{
\prn*{(Y-m_{0}(X,W)) - (T-e_{0}(W))\theta(X)}(\theta'(X)-\theta(X))(e(W)-e_{0}(W))
}\\
&~~~~-2\cdot\En\brk*{
\prn*{\theta(X)(T-e_{0}(W))(\theta'(X)-\theta(X))(e(W)-e_{0}(W))
}
}.
\end{align*}
To handle the first term, we use that for any $x,w$,
\begin{align*}
&\En\brk*{\prn*{(Y-m_{0}(X,W)) - (T-e_{0}(W))\theta(X)}\mid{}X=x,W=w}\\
&= \En\brk*{\veps_{1} + \veps_{2}\cdot(\theta_{0}(X)-\theta(X))\mid{}X=x,W=w}=0.
\end{align*}
Similarly, the second term is handled by using that
\[
\En\brk*{\theta(X)(T-e_0(W))\mid{}X=x,W=w} = \En\brk*{\theta(X)\cdot\veps_2\mid{}X=x,W=w}=0.
\]
To establish orthogonality for the expected value parameter $m$, for any $\theta,\theta'$ we have
\begin{align*}
&D_{m}D_{\theta}\poprisk(\theta,\crl*{m_0,e_0})[\theta'-\theta,m-m_{0}] \\ 
&= 2\cdot{}\En\brk*{
(T-e_{0}(W))(\theta'(X)-\theta(X))(m(X,W)-m_{0}(X,W))
} \\
&= 2\cdot{}\En\brk*{
\veps_2\cdot{}(\theta'(X)-\theta(X))(m(X,W)-m_{0}(X,W))
}\\
&=0,
\end{align*}
which follows from the assumption $\En\brk*{\veps_2\mid{}X,W}=0$. Note
that both of these orthogonality proofs held for any choice of
$\theta$, not just $\theta_0$, and hence \pref{ass:orthogonal} is
satisfied for all $\besttwo$.

Next, we show that \pref{ass:well_specified} holds whenever the second
stage is well-specified (i.e. $\gttwo\in\target$); the fact that
\pref{ass:well_specified} whenever $\target$ is convex is immediate. We have
\begin{align*}
D_{\theta}\poprisk(\gttwo,\gtone)[\theta-\theta_{0}] = -2\cdot\En\brk*{
\prn*{(Y-m_{0}(X,W)) - (T-e_{0}(W))\theta_{0}(X)}(T-e_{0}(W))(\theta(X)-\theta_{0}(X))
}.
\end{align*}
In particular, for any $x$ we have
\begin{align*}
&\En\brk*{(Y-m_{0}(X,W)) - (T-e_{0}(W))\theta_{0}(X))(T-e_{0}(W))\mid{}X=x} \\
&= \En\brk*{(T\theta_{0}(W)+f_{0}(W)+\veps_1-e_{0}(X)\theta_0(W)-f_{0}(W)) - (T-e_{0}(W))\theta_{0}(X))(T-e_{0}(W))\mid{}X=x} \\
&= \En\brk*{\veps_1\cdot{}\veps_2\mid{}X=x} = \En\brk*{\En\brk*{\veps_1\mid{}X=x,\veps_2}\cdot{}\veps_2\mid{}X=x}=0,
\end{align*}
and so $D_{\theta}\poprisk(\gttwo,\gtone)[\theta-\theta_{0}]=0$.

It remains to verify \pref{ass:smooth_loss,ass:strong_convex_loss}. We do so appealing to \pref{lem:suff_single_index}. Following the notation of \pref{sec:sufficient}, we set $\Lambda(g(w),w)=(T-e(W))$, $\Gamma(g(w),z) = (Y-m(X,W))$, and $\phi(\zeta)=\zeta$. With this choice, we can take $\Tsi=\tausi=\Lsi=\Rsi=1$, and $\gammasi=\lambdare$. To bound the parameter $\musi$, for $\zeta\in\bbR$, $\gamma\in\bbR^2$, and $z=(X,W,T)$, we write
  \[
    \ls(\zeta,\gamma;z) = \prn*{(Y-\gamma_1) - (T-\gamma_2)\zeta}^2,
  \]
  which has
  \[
    \grad^2_{\gamma}\grad_{\zeta}\ls(\zeta,\gamma;z) = \left(
      \begin{array}{lr}
        0 & -2\\
        -2 & 2\zeta
      \end{array}
      \right).
    \]
    It follows that $\nrm{\nabla_{\gamma\gamma}^2 \nabla_{\zeta_i} \ell(\besttwo(\vartwo), g(\varone);z)}_{\sigma}\leq{}4\rdef\musi$ whenever $\abs{\besttwo(x)}\leq{}1$. As a result \pref{lem:suff_single_index} implies that \pref{ass:strong_convex_loss,ass:smooth_loss} are satisfied with constants $\lambda = \frac{1}{4}$,  $\kappa=4\lambdare^{-1}$, $\beta_1=1$ and $\beta_2= 4\lambdare^{-1/2}$. The result now follows from \pref{thm:generic_strongly_convex}.
    
\end{proof}

\begin{proof}[Proof and Details for \pref{ex:r_learner}]%
%
Formally, we consider the following setup \citep{nie2017quasi}. Let $\cH$ is an RKHS with norm $\nrm{\cdot}_{\cH}$ and kernel $\cK$. We assume that $\cX\subseteq\bbR^{d}$ is a compact metric space, and that $\cK$ is a kernel with respect to $\cD$. We define $T_{\cK}:L_2(\cD)\to{}L_2(\cD)$ via
\[
  \brk{\Tk{}f}(y) = \En\brk{\cK(X,y)f(X)}.
\]
Using Mercer's theorem \citep{cucker2002mathematical}, there exist eigenfunctions $(\psi_j)_{j=1}^{\infty}$ and eigenvalues $(\sigma_j)_{j=1}^{\infty}$ such that
\[
\cK(x,y) = \sum_{j=1}^{\infty}\sigma_j\psi_j(x)\psi_j(y).
\]
The main assumptions are as follows.
\begin{enumerate}
\item \emph{Eigenvalue decay.} There exists $p\in{}(0,1)$ such that
  \[
    G\ldef{}\sup_{j\geq{}1}j^{1/p}\sigma_j<\infty.
  \]
\item \emph{Approximation.} There exists $\alpha\in{}(0,1/2)$ such
  that the function $\theta_0$ has
  \[
    R\ldef{}\nrm*{\brk{\Tk^{\alpha}\theta_0}}_{\cH}<\infty.
  \]
\end{enumerate}
In addition, we assume that $\cK(x,y)\leq{}1$, that $\nrm{\psi_j}_{L_{\infty}(\cD)}\leq{}A$ (note that $\nrm{\psi_j}_{L_{2}(\cD)}=1$), and that $\abs{Y}\leq{}1$ almost surely. We also assume overlap, i.e. $\eta\leq{}e_0(X)\leq{}1-\eta$.
Throughout the proof, we use $\bigoht(\cdot)$ to suppress dependence on $G$, $R$, $A$, $\eta^{-1}(1-\eta)^{-1}$, and $\log(n)$. Recall that the estimator $\esttwo$ is obtained via plug-in empirical risk minimization with respect to the constrained function class
\[
\target = \crl*{\theta\in\cH\mid{}\nrm{\theta}_{\cH}\leq{}c, \nrm{\theta}_{L_{\infty}(\cD)}\leq{}1}
\]
for a parameter $c\geq{}1$. As noted in Eq. (19) of \cite{nie2017quasi}, if we define $\besttwo\ldef\argmin_{\theta\in\Theta}\poprisk(\theta,\gtone)$ (note that $\besttwo\neq\gttwo$, since $\Theta$ is constrained), we have
\begin{align}
  \poprisk(\besttwo,\gtone) - \poprisk(\gttwo,\gtone)
  \leq{} c^{2-\frac{1}{\alpha}}\nrm*{\brk{\Tk^{\alpha}\theta_0}}_{\cH}^{1/\alpha}.
  \label{eq:rlearner_approx}
\end{align}
In particular, by choosing $c\propto{}n^{\alpha/(p+(1-2\alpha))}$, the approximation error scales as $\bigoh(n^{-\frac{1-2\alpha}{p+(1-2\alpha)}})$ to $\gttwo$. In light of this approximation result, is remains to derive an oracle excess risk bound against $\besttwo$ for fixed $c$.

To proceed, we verify that the assumptions required by \pref{thm:generic_strongly_convex} (in particular, smoothness and strong convexity) are satisfied. We do so appealing to \pref{lem:suff_single_index}. Following the notation of \pref{sec:sufficient}, we set $\Lambda(g(w),w)=(T-e(W))$, $\Gamma(g(w),z) = (Y-m(X,W))$, and $\phi(\zeta)=\zeta$. With this choice, we can take $\Tsi=\tausi=\Lsi=\Rsi=1$. It remains to bound $\musi$ and $\gammasi$.
\begin{itemize}
  \item To bound the parameter $\musi$, for $\zeta\in\bbR$, $\gamma\in\bbR^2$, and $z=(X,W,T)$, we write
  \[
    \ls(\zeta,\gamma;z) = \prn*{(Y-\gamma_1) - (T-\gamma_2)\zeta}^2,
  \]
  which has
  \[
    \grad^2_{\gamma}\grad_{\zeta}\ls(\zeta,\gamma;z) = \left(
      \begin{array}{lr}
        0 & -2\\
        -2 & 2\zeta
      \end{array}
      \right).
    \]
    It follows that $\nrm{\nabla_{\gamma\gamma}^2 \nabla_{\zeta_i} \ell(\besttwo(\vartwo), g(\varone);z)}_{\sigma}\leq{}4\rdef\musi$ whenever $\abs{\besttwo(x)}\leq{}1$.
  \item To bound $\lambdasi$, we recall Lemma 5.1 of \cite{mendelson2010regularization}, which states that for all $\theta\in\cH$,
\[
  \nrm{\theta}_{L_{\infty}(\cD)}\leq{} C_{p}\nrm{\theta}_{\cH}^{p}  \nrm{\theta}_{L_{2}(\cD)}^{1-p},
\]
where $C_p$ is a constant depending on $p$, $G$, and $A$. In addition, we have $\nrm{\theta}_{L_{2}(\cD)}^2\leq{}\eta^{-1}(1-\eta)^{-1}\nrm*{\theta}_{\Theta}^2$ and $\nrm{\theta}_{\cH}\leq{}c$. Hence, we may take $\lambdasi^{-1}=\bigoht(c^{2p})$ and $\rsi=p$.
  \end{itemize}
  As a result, \pref{lem:suff_single_index} implies that \pref{ass:strong_convex_loss,ass:smooth_loss} are satisfied with constants $r=p$, $\lambda^{-1},\beta_1=\bigoht(1)$, $\kappa=\bigoht(c^{2p})$, and $\beta_2=\bigoht(c^p)$.
%
%
%
%
\pref{thm:fast_erm} now implies that with probability at least $1-\delta$, the plug-in empirical risk minimizer satisfies
\begin{align*}
  \poprisk(\esttwo,\gtone) - \poprisk(\besttwo,\gtone)
  &\leq \bigoht\prn*{\delta_n^{2} + \frac{\log(\delta^{-1})}{n} + c^{\frac{2p}{1+p}}\nrm*{\estone-\gtone}_{L_{2}(\ls_2,\cD)}^{\frac{4}{1+p}}},
\end{align*}
where $\delta_n$ is the solution to the fixed point equation in \pref{eqn:critical_radius}. From Lemma 5 of \cite{nie2017quasi}, we have
\[
\cR_n(\target,\delta)\leq{} \bigoh(\log(n))\cdot{} \frac{c^{p}\delta^{1-p}}{\sqrt{n}}.
\]
This implies that $\delta_n^2 = \bigoht\prn*{c^{\frac{2p}{1+p}}n^{-\frac{1}{(1+p)}}}$ is a valid fixed point, so we have
\begin{align*}
  \poprisk(\esttwo,\gtone) - \poprisk(\besttwo,\gtone)
  &\leq \bigoht\prn*{c^{\frac{2p}{1+p}}n^{-\frac{1}{(1+p)}} + \frac{\log(\delta^{-1})}{n} + c^{\frac{2p}{1+p}}\nrm*{\estone-\gtone}_{L_{2}(\ls_2,\cD)}^{\frac{4}{1+p}}}.
\end{align*}
Whenever $\nrm*{\estone-\gtone}_{L_{2}(\ls_2,\cD)}\leq{}\wt{o}(n^{-1/4})$, this simplifies to
\begin{align*}
  \poprisk(\esttwo,\gtone) - \poprisk(\besttwo,\gtone)
  &\leq \bigoht\prn*{c^{\frac{2p}{1+p}}n^{-\frac{1}{(1+p)}} + \frac{\log(\delta^{-1})}{n}},
\end{align*}
and combining with the approximation result in \pref{eq:rlearner_approx} yields
\begin{align*}
  \poprisk(\esttwo,\gtone) - \poprisk(\besttwo,\gtone)
  &\leq \bigoht\prn*{c^{\frac{2p}{1+p}}n^{-\frac{1}{(1+p)}} + c^{2-\frac{1}{\alpha}} + \frac{\log(\delta^{-1})}{n}}.
\end{align*}
Choosing $c=n^{\alpha/(p+(1-2\alpha))}$ gives excess risk $\bigoh(n^{-\frac{1-2\alpha}{p+(1-2\alpha)}})$.

%

%

%

%

\end{proof}

\begin{proof}[\pfref{prop:doubly_robust}]
    We verify that the conditions of
  \pref{thm:generic_strongly_convex} are
  satisfied. We do so by showing that \pref{ass:single_index} is satisfied and appealing to \pref{lem:suff_single_index}. In particular, following the notation of \pref{sec:sufficient}, we observe that the doubly-robust loss has the structure required by \pref{ass:single_index}, since we may take $\Lambda(g(w),w)=1$, $\Gamma(g,z) = \vphi(f,e;z)$, and $\phi(\zeta)=\zeta$. Furthermore,
  it can be written as a square loss of the form
  \begin{align*}
    \ls(\zeta,\gamma;z) = \prn*{\zeta - \vphi(\gamma;z)}^2,
  \end{align*}
  where, for $\gamma = (a_0,a_1,b)\in\bbR^{3}$, we overload notation and write
  \[
    \vphi(\gamma;z) = \frac{T-b}{b(1-b)}(Y-a_T) +a_1-a_0.
  \]

  To establish orthogonality (\pref{ass:orthogonal}), it suffices to show that
  \begin{align*}
    \En\brk*{\grad_{\gamma}\grad_{\zeta}\ls(\theta_0(X),g_0(X);z)\mid{}X}=0.
  \end{align*}
  We have
  \begin{align*}
    \En\brk*{\grad_{\gamma}\grad_{\zeta}\ls(\theta_0(X),g_0(X);z)\mid{}X}
    = -2 \En\brk*{\theta_0(X)\grad_{\gamma}\vphi(f_0,e_0;z)\mid{}X}, 
  \end{align*}
  and it is straightforward to verify that $\En\brk*{\grad_{\gamma}\vphi(f_0,e_0;z)\mid{}X}=0$ as a consequence of the double robustness property

  Next, since $\poprisk(\theta,g_0)=\En\brk*{(\theta(X)-\theta_0(X))^2} + \mathrm{Var}(\vphi(f_0,e_0;z))$, it is immediate that \pref{ass:well_specified} is satisfied.
  
  Finally, we verify the regularity conditions required to establish \pref{ass:smooth_loss,ass:strong_convex_loss}. With $p=2$, it is immediate that we may take $\Tsi=\tausi=\Rsi=\gammasi=1$, $\Lsi=0$, and $\rsi=0$ for \pref{ass:single_index} whenever $\abs{\theta(X)}\leq{}1$. To bound the parameter $\musi$, for $\zeta\in\bbR$, $\gamma\in\bbR^3$ and $z=(X,W,T)$ we observe that
  \[
    \grad^{2}_{\gamma}\grad_{\zeta}\ls(\zeta,\gamma;z) = -2\zeta\grad^{2}_{\gamma}\vphi(\gamma;z).
  \]
  Since $\abs{\zeta}\leq{}1$ by assumption, it suffices to bound $\nrm*{\grad^{2}_{\gamma}\vphi(\gamma;z)}_{\sigma}\leq{}3 \nrm*{\grad^{2}_{\gamma}\vphi(\gamma;z)}_{\infty}$. It is straightforward to verify that whenever $\abs{f(0,X)},\abs{f(1,X)}\leq{}1$, $\abs{Y}\leq{}1$, and $\eta\leq{}e(X)\leq{}1-\eta$, we have $\nrm*{\grad^{2}_{\gamma}\vphi(\gamma;z)}_{\infty}\leq{}4\eta^{-3}$, so we may take $\musi=24\eta^{-3}$. As a result, \pref{lem:suff_single_index} implies that \pref{ass:strong_convex_loss,ass:smooth_loss} are satisfied with constants $r=0$, $\lambda = \frac{1}{4}$,  $\kappa=0$, $\beta_1=1$ and $\beta_2= 24\eta^{-3}$. The result now follows from \pref{thm:generic_strongly_convex}.
\end{proof}

  \begin{proof}[\pfref{prop:policy_learning}]
    We first verify that \pref{ass:universal_orthogonality} is satisfied. Let $g$, $\theta$, and $\bar{\theta}$ be arbitrary. Abbreviate $f\ind{t}(x)=f(t,x)$. We have
    \begin{align*}
      &D_{f\ind{0}}D_{\theta}\poprisk(\thetabar,\crl{f_0\ind{0},f_0\ind{1},e_0})[\theta-\besttwo,f\ind{0}-f\ind{0}_0]\\
      &= \En\brk*{
      \prn*{\frac{1-T}{1-e_0(X)} -1}
      (\theta(X)-\besttwo(X))(f(0,X)-f_0(0,X))
        }= 0,
        \end{align*}
        and
        \begin{align*}
    &    D_{f\ind{1}}D_{\theta}\poprisk(\thetabar,\crl{f_0\ind{0},f_0\ind{1},e_0})[\theta-\besttwo,f\ind{1}-f\ind{1}_0] \\
            &=\En\brk*{
      \prn*{1 - \frac{T}{e_0(X)}}
      (\theta(X)-\besttwo(X))(f(1,X)-f_0(1,X))
              } =0.
    \end{align*}
    Similarly, we have
        \begin{align*}
      &D_eD_{\theta}\poprisk(\thetabar,\crl{f_0\ind{0},f_1\ind{0},e_0})[\theta-\besttwo,e-e_0]\\
      &= -\En\brk*{
      \prn*{T\frac{Y-f_0(1,X)}{e_0^2(X)} + (1-T)\frac{Y-f_0(0,X)}{(1-e_0(X))^2}}
      (\theta(X)-\besttwo(X))(e(X)-e_0(X))
        }\\
      &=0.
        \end{align*}
        This establishes the universal orthogonality property.

        We now verify that \pref{ass:smooth_loss_slow} is satisfied with $\nrm{\cdot}_{\cG}=\nrm*{\cdot}_{L_2(\ls_2,\cD)}$. Consider the function $B_Z:\bbR^{3}\to\bbR$ given by
        \[
          B_Z(v) = \prn*{v_2 -v_1 + T\frac{Y-v_2}{v_3} - (1-T)\frac{Y-v_1}{1-v_3}}.
        \]
        \pref{ass:smooth_loss_slow} holds whenever $\nrm*{\grad^{2}_v{}B_Z(v)}_{\sigma}$ is bounded for all $Z$ and all $v$ in the range of $\cG$. In particular, one can verify by inspection that $\nrm{\grad^{2}_v{}B_Z(v)}_{\sigma}\leq{}\bigoh(\eta^{-3})$ whenever $\abs{Y}\leq{}1$, all $g=\crl{f(0,\cdot),f(1,\cdot),e}\in\cG$ have $\abs{f(t,X)}\leq{}1$ and $e(X)\in\brk{\eta,1-\eta}$.

With both assumptions satisfied, the result now follows from \pref{thm:orthogonal_slow}.
  \end{proof}

\section{Technical Lemmas for Constrained $M$-Estimators} 
\label{app:general_m}

In this section of the appendix we give self-contained technical
results for $M$-estimation over general function classes, in the
absence of nuisance parameters. The results here serve as a building
block for the results of \pref{sec:erm}.

Let $\cF:\cX \rightarrow \R^d$ be the function class, and let
$\ell: \R^d\times \cZ\rightarrow \R$ be the loss. We receive a sample
set $S=z_1,\ldots,z_n$ drawn from distribution $\cD$ independently.
Let $\cL_{f}$ denote the random variable $\ell(f(\vartwo), \varall)$ and let
\begin{align}
\Pn \cL_{f} = \En\brk*{\ell(f(\vartwo), \varall)},\quad\text{and}\quad
\Pn_n \cL_{f} = \frac{1}{n} \sum_{i=1}^n \ell(f(\vartwo_i), \varall_i)
\end{align}
denote the population risk and empirical risk over $n$ samples. The
constrained empirical risk minimizer is given by
\begin{equation}
\hat{f} = \arg\min_{f \in \cF} \Pn_n\, \cL_{f}.
\end{equation}
\paragraph{Additional Notation}
To keep notation compact, we adopt the abbreviations
$\nrm*{f}_{p,q}\ldef{}\nrm*{f}_{L_p(\ls_q,\cD)}$ and
$\nrm*{f}_{n,p,q}\ldef{}\nrm*{f}_{L_p(\ls_q,\zr[n])}$. We drop the
subscript $q$ for real-valued function classes. We recall that
$\cF\rbar_t$ denotes the $t$th coordinate projection of $\cF$ and,
likewise, for any $f\in\cF$, the $t$th coordinate projection is $f_t\in\cF\rbar_t$.

The following lemmas provide a vector-valued extension of the
analysis of constrained ERM based on local Rademacher complexities
given in \cite{wainwright2019}. The key idea behind the extension is to invoke a vector-valued contraction theorem for
Rademacher complexity due to \cite{maurer2016vector}. For completeness
we give a include proofs for these lemmas, even though some parts are straightforward adaptations of lemmas in \cite{wainwright2019}.
\begin{lemma}\label{lem:prob_localized} Consider a function class $\cF: \cX\rightarrow \R^d$, with $\sup_{f\in \cF} \|f\|_{\infty,2}\leq 1$ and pick any $f^{\star}\in \cF$.
Assume that the loss $\ell$ is $L$-Lipschitz in its first argument with respect to the $\ell_2$-norm and let
\begin{equation}
Z_n(r) = \sup_{f\, \in\, \cF: \|f-f^{\star}\|_{2,2}\leq r} \left| \Pn_n (\cL_{f} - \cL_{f^{\star}}) - \Pn (\cL_{f} - \cL_{f^{\star}}) \right|.
\end{equation}
Then there are universal constants $c_1, c_2>0$ such that
\begin{equation}
\label{eq:talagrand1}
\Pr\prn*{Z_n(r)\geq 16\, L\, \sum_{t=1}^d \cR(\cF\rbar_t-f_t^{\star},r) + u} \leq c_1 \exp\prn*{-c_2\frac{n u^2}{L^2 r^2 + L u}}.
\end{equation}
Moreover, if $\delta_{n}$ is any solution to the inequalities
\begin{equation}
\forall t\in \{1,\ldots, d\}: \cR(\sh(\cF\rbar_t - f_t^{\star}),\delta) \leq \delta^2,
\end{equation}
then for each $r\geq \delta_n$,
\begin{equation}
\Pr[Z_n(r)\geq 16\, L\, d\, r\, \delta_n + u] \leq c_1 \exp\prn*{-\frac{c_2 n u^2}{L^2 r^2 + 2 L u}}.
\end{equation}
\end{lemma}

\begin{lemma}\label{lem:peeling} Consider a vector valued function
  class $\cF: \cX\rightarrow \R^d$ with $\sup_{f\in \cF}
  \|f\|_{\infty,2}\leq 1$, and pick any $f^{\star}\in \cF$. Let
  $\delta_n^2\geq \frac{4\,d\,\log(41\log(2c_2 n))}{c_2 n}$  be any
  solution to the system of inequalities
\begin{equation}
\forall t\in \{1,\ldots, d\}: \cR(\sh(\cF\rbar_t - f_t^{\star}),\delta) \leq \delta^2.
\end{equation}
Moreover, assume that the loss $\ell$ is $L$-Lipschitz in its first argument with respect to the $\ell_2$ norm. Consider the following event:
\begin{equation}
{\cal E}_1 = \left\{\exists f \in \cF: \|f - f^{\star} \|_{2,2} \geq
  \delta_n \text{ and } \left| \Pn_n (\cL_{f} - \cL_{f^{\star}}) - \Pn
    (\cL_{f} - \cL_{f^{\star}}) \right| \geq 18 L\,d\, \delta_n\, \|f -
  f^{\star}\|_{2,2}\right\}.
\end{equation}
Then for some universal constants $c_3, c_4$, $\Pr[{\cal E}_1] \leq c_3 \exp\prn{-c_4 n \delta_n^2}$.
\end{lemma}

\begin{lemma}\label{lem:tighter_peeling} Consider a vector valued function class $\cF: \cX\rightarrow \R^d$ and pick any $f^{\star}\in \cF$. Let $\delta_n\geq 0$  be any solution to the inequalities
\begin{equation}
\forall t\in \{1,\ldots, d\}: \cR(\sh(\cF\rbar_t - f_t^{\star}),\delta) \leq \delta^2.
\end{equation}
Suppose $\sup_{f\in \cF} \|f\|_{\infty,2}\leq 1$. Moreover, assume that the loss $\ell$ is $L$-Lipschitz in its first argument with respect to the $\ell_2$ norm and also linear, i.e. $\cL_{f + f'}=\cL_{f} + \cL_{f'}$ and $\cL_{\alpha f}=\alpha\cL_{f}$. Consider the following event:
\begin{equation}
{\cal E}_1 = \left\{\exists f \in \cF: \|f - f^{\star} \|_{2,2} \geq
  \delta_n \text{ and } \left| \Pn_n (\cL_{f} - \cL_{f^{\star}}) - \Pn
    (\cL_{f} - \cL_{f^{\star}}) \right| \geq 17 L\,d\, \delta_n\, \|f -
  f^{\star}\|_{2,2}\right\}.
\end{equation}
Then for some universal constants $c_3, c_4$, $\Pr[{\cal E}_1] \leq c_3 \exp\prn{-c_4 n \delta_n^2}$.
\end{lemma}

\begin{lemma}\label{lem:lipschitz_wainwright}
Consider a function class $\cF$, with $\sup_{f\in \cF} \|f\|_{\infty}\leq 1$, and pick any $f^{\star}\in \cF$. Let $\delta_n^2\geq \frac{4\,d\,\log(41\log(2c_2 n))}{c_2 n}$ be any solution to the inequalities:
\begin{equation}
\forall t\in \{1,\ldots, d\}: \cR(\sh(\cF\rbar_t - f_t^{\star}),\delta) \leq \delta^2.
\end{equation}
Moreover, assume that the loss $\ell$ is $L$-Lipschitz in its first argument with respect to the $\ell_2$ norm. Then for some universal constants $c_5, c_6$, with probability $1-c_5 \exp\prn{c_6 n \delta_n^2}$,
\begin{align}\label{eqn:lipschitz_loss_gen}
\left| \Pn_n (\cL_{f} - \cL_{f^{\star}}) - \Pn (\cL_{f} - \cL_{f^{\star}}) \right| \leq~& 18 L\, d\, \delta_n \{ \|f - f^{\star}\|_{2,2} + \delta_n\},\quad \forall f \in \cF.
\end{align}
Hence, the outcome $\hat{f}$ of constrained ERM satisfies that with the same probability,
\begin{align}
\Pn (\cL_{\hat{f}} - \cL_{f^{\star}}) \leq~& 18 L\,d\, \delta_n \{ \|\hat{f} - f^{\star}\|_{2,2} + \delta_n\}.
\end{align}
If the loss $\cL_{f}$ is also linear in $f$, i.e. $\cL_{f + f'}=\cL_{f} + \cL_{f'}$ and $\cL_{\alpha f}=\alpha\cL_{f}$, then the lower bound on $\delta_n^2$ is not required.
\end{lemma}

\preprint{\subsection{Proofs of Lemmas for Constrained $M$-Estimators}}
\journal{\subsection{Proofs}}

\begin{proof}[Proof of \pref{lem:prob_localized}]
By the Lipschitz condition on the loss and the boundedness of the functions, we have $\|\cL_f - \cL_{f^{\star}}\|_{\infty} \leq L\|f-f^{\star}\|_{\infty, 2}\leq 2L$. Moreover,
\begin{equation*}
\var(\cL_f - \cL_{f^{\star}}) \leq \cP (\cL_f - \cL_{f^{\star}})^2 \leq L^2 \|f-f^{\star}\|_{2,2}^2 \leq L^2 r^2 .
\end{equation*}
\dfedit{Thus, by Bousquet's concentration inequality (see Theorem 7.3 of
  \citet{bousquet2003concentration} or Theorem 12.5 of
  \citet{boucheron2013concentration}) we have that for all $u>0$,
  }
\begin{equation*}
\Pr\brk*{Z_n(r) \geq 2\En\brk*{Z_n(r)} + u} \leq c_1 \exp\prn*{-c_2 \frac{n u^2}{L^2 r^2 + L u}}
\end{equation*}
for absolute constants $c_1,c_2>0$. Moreover, by a standard symmetrization argument,
\begin{align*}
\En\brk*{Z_n(r)} \leq~& 2\, \En\brk*{\sup_{\|f-f^{\star}\|_{2,2}\leq r} \left| \frac{1}{n} \sum_{i=1}^n \epsilon_i \left\{\ell(f(\vartwo_i),\varall_i) - \ell(f^{\star}(\vartwo_i), \varall_i)\right\}\right|}\\
=~& 2\, \En\brk*{\sup_{\|f-f^{\star}\|_{2,2}\leq r,\, \delta\in \{-1, 1\}}  \delta \frac{1}{n} \sum_{i=1}^n \epsilon_i \left\{\ell(f(\vartwo_i),\varall_i) - \ell(f^{\star}(\vartwo_i),\varall_i)\right\}}\\
\leq~& 2\, \sum_{\delta\in \{-1, 1\}} \En\brk*{\sup_{\|f-f^{\star}\|_{2,2}\leq r}  \delta \frac{1}{n} \sum_{i=1}^n \epsilon_i \left\{\ell(f(\vartwo_i),\varall_i) - \ell(f^{\star}(\vartwo_i),\varall_i)\right\}}\\
\leq~& 4\, \En\brk*{\sup_{\|f-f^{\star}\|_{2,2}\leq r}  \frac{1}{n} \sum_{i=1}^n \epsilon_i \left\{\ell(f(\vartwo_i),\varall_i) - \ell(f^{\star}(\vartwo_i),\varall_i)\right\}},
\end{align*}
where the second inequality follows from the fact that each summand is
non-negative, since we can always choose $f=f^{\star}$. By invoking
the multivariate contraction inequality of \cite{maurer2016vector},
letting $\crl*{\eps_{i,t}}_{1\leq{}i\leq{}n, 1\leq{}t\leq{}d}$ be
independent Rademacher random variables, we have
\begin{align*}
\leq~& 8\,L\, \En\brk*{\sup_{\|f-f^{\star}\|_{2,2}\leq r} \frac{1}{n} \sum_{i=1}^n \sum_{t=1}^d \epsilon_{i,t} \left(f_t(\vartwo_i) - f_t^{\star}(\vartwo_i)\right)}\\
\leq~& 8\,L\, \sum_{t=1}^d \En\brk*{\sup_{\|f_t-f_t^{\star}\|_{2}\leq r} \left| \frac{1}{n} \sum_{i=1}^n \epsilon_{i,t} \left(f_t(\vartwo_i) - f_t^{\star}(\vartwo_i)\right)\right|}\\
=~& 8\, L\, \sum_{t=1}^d \cR(\cF\rbar_t-f_t^{\star},r),
\end{align*}
where we also used the fact that for any fixed $f^{\star}$, $\En\brk*{\epsilon_i \ell(f^{\star}(x_i),z_i)} = \En\brk*{\epsilon_{i, t} f_t^{\star}(x_i)}=0$. This completes the proof of the first part of the lemma. For the second part, observe that: $\cR(\cF\rbar_t-f_t^{\star},r)\leq \cR(\sh(\cF\rbar_t - f_t^{\star}),r)$. Moreover, for any star shaped function class $\cG$, the function $r \rightarrow \frac{\cR(\cG,r)}{r}$ is monotone non-increasing. Thus for any $r\geq \delta_n$,
\begin{align*}
\frac{\cR(\sh(\cF\rbar_t - f_t^{\star}),r)}{r} \leq \frac{\cR(\sh(\cF\rbar_t - f_t^{\star}),\delta_n)}{\delta_n} \leq \delta_n.
\end{align*}
Rearranging yields that $\cR(\sh(\cF\rbar_t - f_t^{\star}),r)\leq r \delta_n$. Hence, $\En\brk*{Z_n(r)}\leq 8\, L \, d \, r\, \delta_n$. This completes the proof of the second part of the lemma.
\end{proof}

\begin{proof}[Proof of \pref{lem:peeling}]
We invoke a peeling argument. Consider the events
\begin{equation*}
  {\cal S}_m = \{f \in \cF: \alpha^{m-1}\delta_n \leq \|f - f^{\star}\|_{2,2} \leq \alpha^m \delta_n\},
\end{equation*}
for $\alpha=18/17$. Since $\sup_{f \in \cF} \|f-f^{\star}\|_{2,2} \leq 2\sup_{f \in \cF} \|f\|_{\infty,2} \leq 2$, it must be that any $f \in \cF$ with $\|f-f^{\star}\|_{2,2}\geq \delta_n$ belongs to some ${\cal S}_m$ for $m\in \{1, 2, \ldots, M\}$, where $M\leq \frac{\log(2/\delta_n)}{\log_{\alpha}(e)}\leq 41 \log(2/\delta_n)$. Thus by a union bound we have
\begin{equation*}
\Pr[{\cal E}_1] \leq \sum_{m=1}^M \Pr[{\cal E}_1 \cap {\cal S}_m].
\end{equation*}
Moreover, observe that if the event ${\cal E}_1\cap {\cal S}_m$ occurs then there exists a $f\in \cF$ with $\|f - f^{\star}\|_{2,2} \leq \alpha^m \delta_n = r_m$, such that
\begin{align*}
\left| \Pn_n (\cL_{f} - \cL_{f^{\star}}) - \Pn (\cL_{f} - \cL_{f^{\star}}) \right|  \geq 18 L\,d\, \delta_n\, \|f - f^{\star}\|_2 &\geq 18 L\,d\, \delta_n\, \alpha^{m-1}\, \delta_n \\&= 17 L\, d\, \delta_n\, \alpha^{m}\, \delta_n = 17 L\, d\, \delta_n\, r_m.
\end{align*}
Thus, by the definition of $Z_n(r)$, we have
\begin{align*}
\Pr[{\cal E}_1 \cap {\cal S}_m] \leq \Pr[Z_n(r_m) \geq 17 L d \delta_n r_m].
\end{align*}
Applying \pref{lem:prob_localized} with $r=r_m$ and $u=L\, d\, r_m\,\delta_n$, yields that the latter probability is at most $c_1 \exp\prn{ - c_2 n \frac{L^2 r_m^2 \delta_n^2}{L^2 r_m^2 + L^2 d r_m \delta_n}}\leq c_1 \exp\prn{ - c_2 n \frac{\delta_n^2}{2d}}$, where we used the fact that $\delta_n \leq r_m$ in the last inequality.  Subsequently, taking a union bound over the $M$ events, we have
\begin{equation}
\Pr[{\cal E}_1] \leq c_1 M \exp\prn*{-c_2 n \frac{\delta_n^2}{2}} = c_1 \exp\prn*{-c_2 n \frac{\delta_n^2}{2d} + \log(M)}.
\end{equation}
Since, by assumption on the lower bound on $\delta_n$ we have $\log(M)\leq \log(41 \log(2/\delta_n)) \leq \log(41 \log(2 c_2 n)) \leq \frac{c_2 n \delta_n^2}{4d}$, we get
\begin{equation}
\Pr[{\cal E}_1] \leq  c_1 \exp\prn*{-c_2 n \frac{\delta_n^2}{4d}}.
\end{equation}
\end{proof}

\begin{proof}[Proof of \pref{lem:tighter_peeling}]
For simplicity, let $\|\cdot\|=\|\cdot\|_{2,2}$. Suppose that there
exists a function $f\in \cF$, with $\|f - f^{\star}\| \geq \delta_n$, such that
$$\left| \Pn_n \left(\cL_{f} - \cL_{f^{\star}}\right) - \Pn \left(\cL_{f} - \cL_{f^{\star}}\right)\right| \geq 17 L\,d\, \delta_n \|f - f^{\star}\|.$$
Then we will show that there exists a function $f'\in \sh(\cF - f^{\star})$, with $\|f' - f^{\star}\|=\delta_n$, such that
$$\left| \Pn_n \left(\cL_{f} - \cL_{f^{\star}}\right) - \Pn \left(\cL_{f} - \cL_{f^{\star}}\right)\right| \geq 17 L\,d\, \delta_n^2.$$
To do so, we simply choose $f'$ to satisfy
\begin{equation*}
f' - f^{\star} = \frac{\delta_n}{\|f - f^{\star}\|}\left(f-f^{\star}\right).
\end{equation*}
Since $\frac{\delta_n}{\|f-f^{\star}\|}\leq 1$ and by the definition
of the star hull, we know that $f'\in \sh(\cF - f^{\star})$. Moreover,
by the definition of $\theta'$, we also have that $\|f'-f^{\star}\|_n = \delta_n$. Moreover, by the linearity of the loss $\cL^{\star}_{f}$ with respect to $f$, we have:
\begin{align*}
\left| \Pn_n \left(\cL_{f'} - \cL_{f^{\star}}\right) - \Pn \left(\cL_{f'} - \cL_{f^{\star}}\right)\right| =~& \left| \Pn_n \cL_{f'-f^{\star}} - \Pn \cL_{f'-f^{\star}}\right|\\
=~& \frac{\delta_n}{\|f-f^{\star}\|}\left| \Pn_n \cL_{f-f^{\star}} - \Pn \cL_{f-f^{\star}}\right|\\
=~& \frac{\delta_n}{\|f-f^{\star}\|}\left| \Pn_n \left(\cL_{f} - \cL_{f^{\star}}\right) - \Pn \left(\cL_{f} - \cL_{f^{\star}}\right)\right|\\
\geq~& \frac{\delta_n}{\|f-f^{\star}\|} 17 L\,d\, \delta_n \|f-f^{\star}\| = 17 L\,d\, \delta_n^2.
\end{align*}
Thus we have that the probability of event ${\cal E}_1$ is upper bounded by the probability of the event
\begin{align*}
{\cal E}_1' &= \left\{\sup_{f\,\in\, \sh(\cF - f^{\star}): \|f - f^{\star}\|\leq \delta_n} \left| \Pn_n \left(\cL_{f} - \cL_{f^{\star}}\right) - \Pn \left(\cL_{f} - \cL_{f^{\star}}\right)\right| \geq 17 L\,d\, \delta_n^2\right\} \\&= \left\{Z_n(\delta_n) \geq 17 L\,d\, \delta_n^2\right\}.
\end{align*}
Invoking \pref{lem:prob_localized} with $r=\delta_n$ and $u=L\,d\,\delta_n^2$, we conclude that the probability of the second event is also at most $\mu_1' \exp\prn{-\mu_2' n \delta_n^2}$, for some universal constants $\mu_1',\mu_2'$. 
\end{proof}

\begin{proof}[Proof of \pref{lem:lipschitz_wainwright}]
Consider the events:
\begin{align*}
{\cal E}_0 =~& \{Z_n(\delta_n) \geq 17 L\,d\, \delta_n^2\},\\
{\cal E}_1 =~& \left\{\exists f \in \cF: \|f - f^{\star} \|_{2,2} \geq \delta_n \text{ and } \left| \Pn_n (\cL_{f} - \cL_{f^{\star}}) - \Pn (\cL_{f} - \cL_{f^{\star}}) \right| \geq 18 L\, d\, \delta_n \|f - f^{\star}\|_{2,2}\right\},
\end{align*}
with $Z_n(r)$ as defined in \pref{lem:prob_localized}. Observe that if \pref{eqn:lipschitz_loss_gen} is violated, then one of these events must occur. 
Applying \pref{lem:prob_localized} with $r = \delta_n$ and $u=L\,d\,\delta_n^2$ yields, that event ${\cal E}_0$ happens with probability at most $c_1\exp\prn{c_2' n \delta_n^2}$, where $c_2'=c_2/(L^2 + L\,d)$. Moreover, applying \pref{lem:peeling} we get that $\Pr[{\cal E}_1]\leq c_3 \exp\prn{-c_4 n \delta_n^2}$. Thus by a union bound with probability $1-c_5 \exp\prn{c_6 n \delta_n^2}$, neither events occur. If the loss $\cL_{f}$ is linear then we apply \pref{lem:tighter_peeling} instead of \pref{lem:peeling}, which does not require a lower bound on $\delta_n^2$.
\end{proof}

\section{Proofs from Section \journal{\ref{sec:erm}}\preprint{\ref*{sec:erm}}}
\label{app:constrained}
Let $\cL_{\theta,g}$ be shorthand for the random variable $\ell(\theta(\vartwo), g(\varone); \varall)$, and recall that
\begin{align}
\Pn \cL_{\theta, g} =\En\brk*{\ell(\theta(\vartwo), g(\varone); \varall)},\quad\text{and}\quad
\Pn_n \cL_{\theta, g} = \frac{1}{n} \sum_{i=1}^n \ell(\theta(\vartwo_i), g(\varone_i); \varall_i),
\end{align}
denote, respectively, the population risk and empirical risk over the $n$ samples in $\sampletwo$. We consider the two-stage plugin ERM algorithm
\begin{equation}
\esttwo = \arg\min_{\theta \in \target} \Pn_n \cL_{\theta, \estone}.
\end{equation}
As in \pref{app:general_m}, we adopt the abbreviations
$\nrm*{f}_{p,q}\ldef{}\nrm*{f}_{L_p(\ls_q,\cD)}$ and
$\nrm*{f}_{n,p,q}\ldef{}\nrm*{f}_{L_p(\ls_q,\sampletwo)}$; we drop the
subscript $q$ for real-valued function classes.

Throughout this section, we repeatedly make use of the following fact:
If $\delta_n$ solves a fixed point equation such as
\pref{eqn:critical_radius}, then $\delta'_n\ldef{}\delta_n +
C\sqrt{\frac{\log(1/\delta)}{n}}$ does as well. By expanding the
radius in this fashion, we can replace success probabilities of the
form $1-e^{-cn\delta_n^{2}}$ (e.g., in
\pref{thm:fast_erm} %
below) by  $1-\delta$, as long as $\delta_n$ is replaced by $\delta_n'$ in the final excess risk bound.

\subsection{Proof of Theorem~\ref*{thm:fast_erm}}
\label{sec:fast_erm_proof}

We prove \pref{thm:fast_erm} for the case where $R=1$. The general case follows by observing that ERM over $\Theta$ is equivalent to ERM over the class $\Theta/R$ with the loss $\tilde{\ls}(\zeta,g(w);z) \ldef{} \ls(R\cdot{}\zeta,g(w);z)$, with the problem parameters remapped as $L\mapsto{}LR$, $\beta_1\mapsto{}\beta_1R^2$, $\beta_2\mapsto\beta_2R$, $\kappa\mapsto\kappa$, $\lambda\mapsto\lambda{}R^{2}$, and $\delta_n\mapsto\delta_n/R$.

\begin{proof}[Proof of \pref{thm:fast_erm}]
Since $\esttwo$ is the outcome of the Plug-In ERM and since $\besttwo\in \target$, we have
\begin{equation*}
\Pn_n \left(\cL_{\esttwo, \estone} - \cL_{\besttwo, \estone}\right) \leq 0.
\end{equation*}
Applying \pref{lem:peeling}, with $\cF=\target$, $f^{\star}=\besttwo$
and $\cL_{\cdot} = \cL_{\cdot, \estone}$, we know that with
probability at least $1-c_1 \exp\prn{-c_2 n \delta_n^2}$, where
$c_1,c_2>0$ are numerical constants,
\begin{equation*}
\Pn (\cL_{\esttwo,\estone} - \cL_{\besttwo, \estone}) \leq 18
L\,\dimtwo\, \delta_n \|\esttwo - \besttwo\|_{2,2}
+18 L\,\dimtwo\, \delta_n^2.
\end{equation*}
Since
\pref{ass:orthogonal,ass:well_specified,ass:strong_convex_loss,ass:smooth_loss}
are satisfied, \pref{lem:self_bounding} (a corollary of
\pref{thm:generic_strongly_convex}) with $\veps_n(\delta)\ldef{}18
L\,\dimtwo\, \delta_n$ and $\alpha_n(\delta) \ldef 18 L\,\dimtwo\,
\delta_n^2$ implies that
\begin{align*}
  \nrm[\big]{\esttwo-\besttwo}_{\target}^{2} &\leq{}
C_1^2\cdot{}\veps^2_{n/2}(\delta/2) +
2C_1\cdot\alpha_{n/2}(\delta/2) + 
2C_2\cdot{}\prn*{
\Rate(\nuisance,\sampleone,\delta/2)
                                               }^{\frac{4}{1+r}}\\
                                               &\leq{}
3C_1^2\cdot{}\veps^2_{n/2}(\delta/2) +
2C_2\cdot{}\prn*{
\Rate(\nuisance,\sampleone,\delta/2)
}^{\frac{4}{1+r}},
\end{align*}
and 
\begin{align*}
\Pn (\cL_{\esttwo,\gtone} - \cL_{\besttwo, \gtone}) &\leq 
                                                      2\beta_1C_1^2\cdot{}\veps^2_{n/2}(\delta/2) +
                                                      2\beta_1C_1\cdot\alpha_{n/2}(\delta/2)
  + 
2\beta_1C_2\cdot{}\prn*{
\Rate(\nuisance,\sampleone,\delta/2)
  }^{\frac{1}{1+r}},\\
  &\leq 
    4\beta_1C_1^2\cdot{}\veps^2_{n/2}(\delta/2)
  + 
2\beta_1C_2\cdot{}\prn*{
\Rate(\nuisance,\sampleone,\delta/2)
}^{\frac{1}{1+r}},
\end{align*}
where $C_1\leq\frac{4}{\lambda}$ and
$C_2\leq2\prn[\Big]{\prn[\big]{\frac{\beta_2}{\lambda}}^{\frac{2}{1+r}}
  + \frac{\kappa}{\lambda}}$ are as in
\pref{thm:generic_strongly_convex}; note that we have used that
$\alpha_n(\delta)\leq\veps^2_n(\delta)$ and $C_1\geq{}1$ to simplify. This establishes the result.

\end{proof}

\subsection{Proof of Theorem~\ref*{thm:variance_penalized}}

We first prove a theorem about non-centered $\ell_2$-moment penalization, and then show that this implies \pref{thm:variance_penalized}.

\begin{theorem}[Moment-Penalized Plug-In ERM]\label{thm:moment_penalized}
Consider the function class $\cF = \{\ell(\theta(\cdot), \estone(\cdot); \cdot): \theta \in \target\}$, with $R\ldef\sup_{f\in \cF}\|f\|_{L_{\infty}(\cD)}$ and $\fstar\ldef\ell(\besttwo(\cdot), \estone(\cdot); \cdot)$. Let $\delta_n^2 \geq 0$ be any solution to the inequality
\begin{equation}
  \label{eq:moment_critical}
\Radexp(\sh(\cF - f^\star),\delta) \leq \frac{\delta^2}{R},
\end{equation}
and $\esttwo$ be the second moment-penalized empirical risk minimizer,
\begin{equation*}
\esttwo = \arg\min_{\theta\in \target} L_{\sampletwo}(\theta, \estone) + 36 \delta_n \|\ell(\theta(\cdot),\estone(\cdot);\cdot)\|_{L_2(\sampletwo)}.
\end{equation*}
Then with probability at least $1-\delta$, 
\begin{align*}
\poprisk(\esttwo, \estone) - \poprisk(\besttwo, \estone) =~& O\left(\sqrt{\En\brk*{\ell(\besttwo(\cdot), \estone(\cdot); \cdot)^2}}\cdot\left(\frac{\delta_n}{R}+\sqrt{\frac{\log(1/\delta)}{n}}\right) + \left(\frac{\delta_n^2}{R}+R\frac{\log(1/\delta)}{n}\right)\right).
\end{align*}
\end{theorem}

\begin{proof}[Proof of \pref{thm:moment_penalized}]
  We first consider the case where $R=1$.  We apply
  \pref{lem:lipschitz_wainwright} with $\cL_f\ldef{}f(x)$ (i.e., the
  identity loss). Observe that $\|f-f^{\star}\|_{2} \leq 2
  \|f-f^{\star}\|_{n,2}+\delta_n$ with probability $1-c_7\exp\prn{c_8
    n \delta_n^2}$ (via Theorem 4.1 of \cite{wainwright2019}), and
  that $\cL_{f}$ is linear. As a result, we have that for any $\delta_n\geq 0$ that satisfies the conditions of \pref{thm:moment_penalized}, with probability $1-c_9 \exp\prn{c_{10} n \delta_n^2}$,
\begin{align*}
\left| \Pn_n (\cL_{f} - \cL_{f^{\star}}) - \Pn (\cL_{f}- \cL_{f^{\star}}) \right| \leq~& 18 \delta_n \{\|f-f^{\star}\|_{2} + \delta_n\}\leq 36 \delta_n \{\|f-f^{\star}\|_{n, 2} + \delta_n\} & \forall f \in \cF.
\end{align*}
This implies that with probability at least $1-c_9 \exp\prn{c_{10} n \delta_n^2}$, 
\begin{align*}
\Pn (\cL_{\hat{f}}- \cL_{f^{\star}}) \leq~& \Pn_n (\cL_{\hat{f}} - \cL_{f^{\star}}) + 36 \delta_n \|\hat{f}-f^{\star}\|_{n, 2} + 36\delta_n^2\\
\leq~& \Pn_n (\cL_{\hat{f}} - \cL_{f^{\star}}) + 36 \delta_n \|\hat{f}\| + 36\delta_n\|f^{\star}\|_{n, 2} + 36\delta_n^2\\
\leq~& 72 \delta_n \|f^{\star}\|_{n, 2} + 36 \delta_n^2,
\end{align*}
where the second inequality follows by the definition of the moment-penalized algorithm, since
\begin{equation}
  \label{eq:moment_penalized_rescaled}
\Pn_n \cL_{f} + 36 \delta_n \|\hat{f}\|_{n, 2} \leq  \Pn_n \cL_{f^{\star}} + 36 \delta_n \|f^{\star}\|_{n, 2}.
\end{equation}
Now, for general values of $R$, we may apply the reasoning above to the normalized class $\cF/R$. In particular, the condition \pref{eq:moment_critical} implies that
\[
  \Radexp(\sh(\cF-\fstar)/R,\delta_n/R)
  = \frac{1}{R}\Radexp(\sh(\cF-\fstar),\delta_n)\leq{} \frac{\delta_n^{2}}{R^{2}},
\]
so we may take $\delta_n'\ldef{}\delta_n/R$ as the critical radius for
the normalized class. %
This, combined with \pref{eq:moment_penalized_rescaled}, implies that probability at least $1-\delta$,
\begin{align*}
&\frac{1}{R}\prn*{\poprisk(\esttwo, \estone) - \poprisk(\besttwo, \estone)} \\&\leq O\left(\frac{1}{R}\sqrt{\En\brk*{\ell(\besttwo(\cdot), \estone(\cdot); \cdot)^2}}\cdot\left(\frac{\delta_n}{R}+\sqrt{\frac{\log(1/\delta)}{n}}\right) + \left(\frac{\delta_n^2}{R^{2}}+\frac{\log(1/\delta)}{n}\right)\right),
\end{align*}
which establishes the result.
\end{proof}

\begin{proof}[\pfref{thm:variance_penalized}]
  We first sketch the intuition behind how to deduce the
  variance-based bound from \pref{thm:variance_penalized} from the
  second moment-based bound from \pref{thm:moment_penalized}, then
  give a formal proof. Observe that if the optimal loss $\mu^\star\ldef{}\poprisk(\besttwo,\gtone)$ is zero then the two are equivalent. This motivates the following approach: if one has access to a good preliminary estimate $\muhat$ of the value $\mu^\star$, then using moment penalization one can always attain a bound that depends on $\|f^\star - \muhat\|_{L_2(\cD)} = \sqrt{\var(f^\star)} + O(\abs{\muhat - \mu^\star})$. The latter is achieved by simply redefining the function class $\cF$ in \pref{thm:moment_penalized} to be the centered class of losses $\{\ell(\theta(\cdot), \estone(\cdot), \cdot)-\muhat: \theta \in \target\}$. This leads to the algorithm in the theorem statement, which penalizes the centered second moment:
\begin{equation}
\hat{\theta} = \arg\min_{\theta\in \target} L_{\sampletwo}(\theta, \hat{g}) + 36 \delta_n \|\ell(\theta(\cdot),\estone(\cdot);\cdot)-\muhat\|_{L_2(\sampletwo)}.
\end{equation}
As long as the error in the preliminary estimate is vanishing,
i.e. $\abs{\muhat - \mu^\star}\rdef\veps_n\rightarrow 0$, then the
impact of this error on the regret is only of second order, since the
final regret bound takes the form.
\begin{equation}
\label{eq:variance_penalized}
\poprisk(\esttwo, \estone) - \poprisk(\besttwo, \estone) = O\left( \delta_n \sqrt{\var(f^\star)} + \delta_n\veps_n + \delta_n^2\right) = O\left( \delta_n \sqrt{\var(f^\star)}+\veps_n^{2}+\delta_n^{2}\right).
\end{equation}

In more detail, using the bound from \pref{thm:moment_penalized}, we have
\begin{align*}
\journal{&}\poprisk(\esttwo, \estone) - \poprisk(\besttwo, \estone) \journal{\\&}= O\left(\prn*{\sqrt{\var\brk*{\ell(\besttwo(\cdot), \estone(\cdot); \cdot)}}+\veps_n}\cdot\left(\frac{\delta_n}{R}+\sqrt{\frac{\log(1/\delta)}{n}}\right) + \left(\frac{\delta_n^2}{R}+R\frac{\log(1/\delta)}{n}\right)\right),
\end{align*}
which, using the AM-GM inequality, we can bound by
\begin{align*}
\journal{&}\poprisk(\esttwo, \estone) - \poprisk(\besttwo, \estone) \journal{\\&}= O\left(\prn*{\sqrt{\var\brk*{\ell(\besttwo(\cdot), \estone(\cdot); \cdot)}}}\cdot\left(\frac{\delta_n}{R}+\sqrt{\frac{\log(1/\delta)}{n}}\right) + \left(\frac{\delta_n^2}{R}+R\frac{\log(1/\delta)}{n}\right) + \frac{\veps_n^{2}}{R}\right).
\end{align*}
We are nearly ready to apply \pref{thm:orthogonal_slow}, but first we
must bound the error term $\veps_n$ and relate the variance on the
right-hand side above to the true variance $\var\brk*{\ell(\besttwo(\cdot),
      \gtone(\cdot);
      \cdot)}$. We bound the error $\veps_n$ for the estimate $\muhat$ using vanilla (non-localized) Rademacher complexity and two-sided uniform convergence arguments over the function class $\cF$. In particular, using standard arguments, we can guarantee that with probability at least $1-\delta$ over the draw of $S_3$, we have
\[
\abs[\big]{\muhat - \inf_\theta L_{\cD}(\theta, \estone)} \leq{} \bigoh\prn[\bigg]{\Radexp(\ls\circ\target) + R\sqrt{\frac{\log(1/\delta)}{n}}}.
\]
Finally, we observe that since
$\gamma\mapsto{}\ls(\theta(z),\gamma;z)$ is $L$-Lipschitz for all
$\theta$ and $z$, we have $|L_{\cD}(\theta, \estone)-L_{\cD}(\theta,
\gtone)| =  O(L\|\gtone-\estone\|_{\nuisance})$ for all $\theta\in
\target$. Hence, we have $|L_{\cD}(\besttwo, \gtone) - \inf_{\theta}
L_{\cD}(\theta, \hat{g})| =  O(L\|\gtone-\estone\|_{\nuisance})$, and
\[
\veps_n = \bigoh(L\nrm{\gtone-\estone}_{\nuisance} + \Radexp(\ls\circ\target) + R\sqrt{\frac{\log(1/\delta)}{n}}.
\]
Finally, we observe that
\[
  \var\brk*{\ell(\besttwo(\cdot),
      \estone(\cdot); \cdot)}\leq 2\var\brk*{\ell(\besttwo(\cdot),
      \gtone(\cdot);
      \cdot)} + \bigoh(L^2\|\gtone-\estone\|^2_{\nuisance}).
  \] After another application of the AM-GM inequality, this yields
  \begin{align}
  &L_{\cD}(\esttwo, \estone) - L_{\cD}(\besttwo, \estone) \notag\\&= O\prn*{ \sqrt{V^{\star}}\left(\frac{\delta_n}{R}  + \sqrt{\frac{\log(1/\delta)}{n}}\right) + \frac{1}{R}\prn*{\delta_n^{2} + \cR_n^{2}(\ls\circ\target) + L^2\nrm*{\estone-\gtone}_{\nuisance}^{2}} + R\frac{\log(1/\delta)}{n}
    }.
    \label{eq:variance_penalized_pseudo}
  \end{align}
  Note that we have not used orthogonality up to this point. Applying
  \pref{thm:orthogonal_slow} (using \pref{ass:universal_orthogonality}
  and \pref{ass:smooth_loss_slow}) gives the final result.
\end{proof}

\section{Proofs from Section
  \journal{\ref{sec:oracle}}\preprint{\ref*{sec:oracle}}
and Appendix \journal{\ref{sec:oracle_lipschitz}}\preprint{\ref*{sec:oracle_lipschitz}}}
\label{app:rates}

\subsection{Notation}
We state the guarantees in this section using slightly more refined notation than in \pref{sec:algsandrates} and \pref{sec:oracle_lipschitz}. In particular, we consider both the case where the target and nuisance classes are parametric and where they are nonparametric. For the nuisance parameter class $\nuisance$, the two cases we consider are:
\begin{enumerate}[label=\textnormal{\textbf{\alph*)}}]
\item \emph{Parametric case.} There exists a constant $\done$ such that 
\[
\cH_{2}(\cG\rbar_t,\veps,n) \leq \bigoh(\done\log(1/\veps))\quad\forall{}t.%
\]
\item \emph{Nonparametric case.} There exists a constant $\pone$ such that
\[
\cH_{2}(\cG\rbar_t,\veps,n) \leq \bigoh(\veps^{-\pone})\quad\forall{}t.%
\]
\end{enumerate}
Likewise, for the target class, the two cases we consider are:
 \begin{enumerate}[label=\textnormal{\textbf{\alph*)}}]
 \item \emph{Parametric case.} There exists a constant $\dtwo$ such that 
 \[
 \cH_{2}(\target,\veps,n) \leq{} \bigoh(\dtwo\log(1/\veps)).%
 \]
 \item \emph{Nonparametric case.} There exists a constant $\ptwo$ such that
 \[
 \cH_{2}(\target,\veps,n) \leq{} \bigoh(\veps^{-\ptwo}).%
 \]
 \end{enumerate}

\subsection{Preliminaries}
\begin{definition}[Empirical Rademacher Complexity]
For a real-valued function class $\cF$ and sample set $S=z_1,\ldots,z_n$, the Rademacher complexity is defined via 
\begin{equation}
\Rad(\cF,S) = \En_{\eps}\sup_{f\in\cF}\brk*{\frac{1}{n}\sum_{i=1}^{n}\eps_if(z_i)},
\end{equation}
where $\eps=\eps_1,\ldots,\eps_n$ are i.i.d. Rademacher random variables. We define $\Rad(\cF)=\sup_{S\in\cZ^{n}}\Rad(\cF,S)$.
\end{definition}
We require the following technical lemmas. First is the Dudley entropy integral bound; we use the following form from \cite{srebro2010smoothness} (Lemma A.3). In all results that follow, we use $C>0$ to denote an absolute constant whose value may change from line to line.
\begin{lemma}
\label{lem:dudley}
For any real-valued function class $\cF\subseteq{}(\cZ\to\bbR)$, we have
\begin{equation}
\label{eq:dudley}
\Rad(\cF,S) \leq{} \inf_{\alpha>0}\crl*{
4\alpha + 10\int_{\alpha}^{\sup_{f\in\cF}\sqrt{\En_{n}f^{2}(z)}}
\sqrt{\frac{\cH_2(\cF,\veps,S)}{n}}d\veps
}.
\end{equation}
As a consequence, whenever $\cF$ takes values in $\unitrange$ the following bounds hold:
\begin{itemize}
\item If $\cH_2(\cF,\veps,S)\leq{}\bigoh(\veps^{-p})$, then $\Rad(\cF,S)\leq{}r_{n,p}$, where $r_{n,p}$ satisfies
\[
r_{n,p} \leq{} C_{p}\cdot\left\{\begin{array}{ll}
n^{-\frac{1}{2}},\quad&p<2.\\
n^{-\frac{1}{2}}\cdot\log{}n,\quad&p=2.\\
n^{-\frac{1}{p}},\quad&p>2.
\end{array}\right.
\]
\item If $\cH_2(\cF,\veps,S)\leq{}\bigoh(d\log(1/\veps))$, then $\Rad(\cF,S)\leq{}C\cdot\sqrt{d/n}$.
\end{itemize}
\end{lemma}We also require the following lemma, which controls the rate at which the empirical $L_2$ metric converges to the population $L_2$ metric in terms of metric entropy behavior.

\begin{lemma}
\label{lem:l2_concentration}
Let $\cF\subseteq(\cZ\to\unitrange)$, and let $S=\zr[n]$ be a collection of samples in $\cZ$ drawn i.i.d. from $\cD$.
\begin{itemize}
\item If $\cH_2(\cF,\veps,n)\leq\bigoh(\veps^{-p})$ for some $p$, then with probability at least $1-\delta$, for all $f,f'\in\cF$,
\[
\nrm*{f-f'}_{L_2(\cD)}^{2}
\leq{} 2d_{S}^{2}(f,f') + R_{n,p} + C\frac{\log(\log{}n/\delta)}{n}, 
\]
where
\[
R_{n,p} \leq{} C_{p}\cdot\left\{\begin{array}{ll}
n^{-1}\log^{3}n,\quad&p<2.\\
n^{-1}\log^{4}n,\quad&p=2.\\
n^{-\frac{2}{p}}\log^{3}n,\quad&p>2.
\end{array}\right.
\]
\item
If $\cH_2(\cF,\veps,n)\leq{}\bigoh(d\log(1/\veps))$, then with probability at least $1-\delta$, for all $f,f'\in\cF$,
\[
\nrm*{f-f'}_{L_2(\cD)}^{2}
\leq{} 2d_{S}^{2}(f,f') + C\frac{d\log(en/d) + \log(\log{}n/\delta)}{n}.
\]
\end{itemize}
\end{lemma}
\begin{proof}[\pfref{lem:l2_concentration}]
Using Lemma 8 and Lemma 9 from \cite{rakhlin2017empirical}, it also holds that with probability at least $1-4\delta$, for all $f,f'\in\cF$, we have
\[
\nrm*{f-f'}_{L_2(\cD)}^{2}
\leq{} 2d_{S}^{2}(f,f') + C(r^{\star} + \beta), 
\]
where $\beta=(\log(1/\delta)+\log\log{}n)/n$, and where $r^{\star}\leq{}\dtwo\log(en/\dtwo)$ in the parametric case, and $r^{\star}\leq{}\Rad^{2}(\cF)\log^{3}(n)$ in the general case.  The final result follows by applying the Rademacher complexity bounds from \pref{lem:dudley}.

\end{proof}
\begin{remark}
\label{rem:shifted}
Technically, the result in \cite{rakhlin2017empirical} we appeal to in the proof above is stated for $\brk*{0,1}$-valued classes, but it may be applied to our $\brk*{-1,+1}$-valued setting by shifting and rescaling the class $\cF$ (i.e., invoking with $\cF'\ldef(\cF+1)/2$). We appeal to the same reasoning throughout this section, shifting regression targets in the same fashion when necessary.
\end{remark}

\subsection{Overview of Proofs}
We now sketch the high-level approach behind the main results in \pref{sec:algsandrates} and \pref{sec:oracle_lipschitz}. The idea is to use out-of-the-box learning algorithms for both the nuisance and target stage. However, which algorithm gives an optimal rate will depend on the complexity of $\nuisance$ and $\target$. Moreover, some of the algorithms we employ for the target class require new analyses based on orthogonality to bound the error due to nuisance parameter estimation.

\paragraph{First stage}
Base on our assumptions on the metric entropy, we can obtain rates by appealing to the following generic algorithms.
\begin{itemize}
\item \emph{Global ERM}: For each $t$, select
\[
\estone_{t}\in\argmin_{g\in\nuisance\rbar_t}\sum_{i=1}^{n}((u_t)_{i}-g(\varone_i))^{2}.
\]
\item \emph{Skeleton Aggregation/Aggregation of $\veps$-nets} (\cite{yang1999information,rakhlin2017empirical}; see \pref{sec:skeleton} for a formal description): For each $t$, run the Skeleton Aggregation algorithm with the class $\nuisance\rbar_t$ on the dataset of instance-target pairs $(\varone_1,(u_t)_1),\ldots,(\varone_n,(u_t)_n)$. Let $(\estone)_{t}$ be the result.
\end{itemize}
\vspace{.5em}
\begin{proposition}[Rates for first stage, informal]
\label{prop:stageone_rates}
Suppose that \pref{ass:square_oracle} or \pref{ass:oracle_lipschitz} holds. Then 
Global ERM guarantees that with probability at least $1-\delta$,
\begin{align*}
\nrm*{\estone-\gtone}_{L_{2}(\ls_2,\cD)}^{2} \leq \left\{
\begin{array}{ll}
\Ot\prn*{\dimone\done\log(en/\done)\cdot{}n^{-1}},\quad&\text{Parametric case.}\\~\\
\Ot\prn[\big]{\dimone{}n^{-\frac{2}{2+\pone}\wedge\frac{1}{\pone}}},\quad&\text{Nonparametric case.}
\end{array}
\right.
\end{align*}
Skeleton Aggregation guarantees that with probability at least $1-\delta$,
\begin{align*}
\nrm*{\estone-\gtone}_{L_{2}(\ls_2,\cD)}^{2} \leq \left\{
\begin{array}{ll}
\Ot\prn*{\dimone\done\log(en/\done)\cdot{}n^{-1}},\quad&\text{Parametric case.}\\~\\
\Ot\prn[\big]{\dimone{}n^{-\frac{2}{2+\pone}}},\quad&\text{Nonparametric case.}
\end{array}
\right.
\end{align*}
Here the $\Ot$ notation suppresses $\log{}n$, $\log(\dimone)$, and $\log(\delta^{-1})$ factors.
\end{proposition}
A precise version of \pref{prop:stageone_rates} and detailed description of the algorithms are given in \pref{app:specific_rates}. Note that the minimax rate is $\Omega(n^{-\frac{2}{2+\pone}})$ \citep{yang1999information}, and so Skeleton Aggregation is optimal for all values of $\pone$, while Global ERM is optimal only for $\pone\leq{}2$. While these are not the only algorithms in the literature for which we have generic guarantees based on metric entropy (other choices include Star Aggregation \citep{liang2015learning} and Aggregation-of-Leaders \citep{rakhlin2017empirical}), they suffice for our goal in this section, which is to characterize the spectrum of admissible rates.

In all applications we study, the dimension $\dimone$ is constant. Nevertheless, studying procedures that jointly learn all output dimensions of $\nuisance$ and, in particular, deriving the correct statistical complexity when $\dimone$ is large is an interesting direction for future research and may be practically useful.

\paragraph{Second stage}
The idea behind the second-stage rates we provide is that the problem of obtaining a target predictor for the second stage can be solved by reducing to the classical square loss regression setting. We map our setting onto square loss regression by defining auxiliary variables $X = \varone$ and $Y=\Gamma(\gtone(\varone),\varone)$, and by defining auxiliary predictor classes
\begin{align*}
\gtF &= \crl*{X\mapsto{}\tri*{\Lambda(\gtone(\varone),\varone),\theta(\vartwo)}\mid{}\theta\in\target},\\
\estF &= \crl*{X\mapsto{}\tri*{\Lambda(\estone(\varone),\varone),\theta(\vartwo)}\mid{}\theta\in\target}.
\end{align*}
With these definitions, our goal to bound the excess risk in, e.g., \pref{thm:oracle_well_specified} can be equivalently stated as producing a predictor $\wh{f}\in\cF_0$ that enjoys a bound on
\begin{equation}
\label{eq:excess_risk_square}
\En\prn[\big]{\wh{f}(X)-Y}^{2} -\inf_{f\in\gtF}\En\prn*{f(X)-Y}^{2},
\end{equation}
which is the standard notion of square loss excess risk used in, e.g., \cite{liang2015learning,rakhlin2017empirical}. Defining, $\wt{Y}=\Gamma(\estone(\varone),\varone)$, we can apply any standard algorithm for the class $\cF$ to the dataset $(X_1,\wt{Y}_1),\ldots,(X_n,\wt{Y}_n)$. Note however that, due to the use of $\estone$ as a plug-in estimate, predictors produced via \pref{alg:sample_splitting} will---invoking \pref{def:algorithms}---give a guarantee of the form
\begin{equation}
\label{eq:bad_excess_risk_square}
\En\prn[\big]{\wh{f}(X)-\wt{Y}}^{2} -\inf_{f\in\estF}\En\prn[\big]{f(X)-\wt{Y}}^{2} \leq{} \Rate(\target,\sampletwo,\delta/2\midsem{}\esttwo,\estone),
\end{equation}
where $\wh{f}$ and the benchmark $f$ belong to $\cF$ instead of
$\cF_0$. The machinery developed in \pref{sec:orthogonal} relates the
left-hand-side of this expression to the oracle excess risk
\pref{eq:excess_risk_square}. Depending on the setting, more work is
required to show that the right-hand-side of
\pref{eq:bad_excess_risk_square} is controlled. This challenge is only
present in the well-specified setting. The difficulty is that while
the original problem \pref{eq:excess_risk_square} is well-specified in
this case, the presence of the plug-in estimator $\estone$ in
\pref{eq:bad_excess_risk_square} introduces additional
``misspecification''. We show for global ERM and Skeleton Aggregation
the right-hand-side of the expression is controlled as well, meaning
that $\Rate(\target,\sampletwo,\delta/2\midsem{}\esttwo,\estone)$ is
not much larger than the rate
$\Rate(\target,\sampletwo,\delta/2\midsem{}\esttwo,\gtone)$ that would
have been achieved if the true value for the nuisance parameter was
known. This achieved by exploiting orthogonality once again.

In the misspecified setting, we can simply upper bound the right-hand
side of \pref{eq:bad_excess_risk_square} by the worst-case bound
$\sup_{g\in\nuisance}\Rate(\target,\sampletwo,\delta/2\midsem{}\esttwo,
g)$ and get the desired growth. Since the model is misspecified to begin with, any extra misspecification introduced by using the plugin estimate here is irrelevant.
To be precise, the algorithm configuration is as follows.
\begin{itemize}
\item For stage one, use Skeleton Aggregation \citep{yang1999information,rakhlin2017empirical}. If $\pone\leq{}2$, global ERM can be used instead.
\item For stage two, in the misspecified setting with $\target$ convex, use global ERM.
\item For stage two, in the well-specified setting, we use Skeleton Aggregation, with a new analysis to account for the small amount of ``model misspecification'' introduced by the plug-in nuisance estimate $\estone$. If $\pone\leq{}2$, global ERM can be used instead; this is because skeleton ERM and global ERM are both optimal for $\pone\leq{}2$, even in the presence of nuisance parameters. 
\end{itemize}

\subsection{Skeleton Aggregation}
\label{sec:skeleton}
Here we briefly describe the Skeleton Aggregation meta-algorithm for real-valued regression \citep{yang1999information,rakhlin2017empirical}. The setting is as follows: we receive $n$ examples $S=(X_1,Y_1),\ldots,(X_n,Y_n)\in(\cX\times{}\bbR)^{n}$  i.i.d. from a distribution $\cD$. For a function class $\cF\subseteq{}(\cX\to\bbR)$, we define $L_{\cD}(f) = \En_{\cD}\prn*{f(X)-Y}^{2}$. Our goal is to produce a predictor $\wh{f}_S$ for which the excess risk $L_{\cD}(\wh{f}) - \inf_{f\in\cF}L_{\cD}(f)$
is small.

We call a \emph{sharp model selection aggregate} any algorithm that, given a finite collection of $M$ functions $f_1,\ldots,f_M$ and $n$ i.i.d. samples, returns a convex combination $\wh{f}=\sum_{i=1}^{M}\nu_if_i$ for which
\begin{equation}
\label{eq:ms_aggregate}
L(\wh{f})\leq{} \min_{i\in\brk*{M}}L(f_i) + C\frac{\log(M/\delta)}{n}
\end{equation}
with probability at least $1-\delta$. One such model selection aggregate is the \emph{star aggregation} algorithm of \cite{audibert2008progressive}, which produces a $2$-sparse convex combination $\wh{f}$ with the property \pref{eq:ms_aggregate} whenever $\abs*{Y}\leq1$ almost surely and the functions in $f_1,\ldots,f_M$ take values in $\unitrange$.

We use the following variant of skeleton aggregation, following \cite{rakhlin2017empirical}. Given a dataset $S=(X_1,Y_1),\ldots,(X_n,Y_n)$, we split it into two equal-sized parts $S'$ and $S''$.
\begin{itemize}
\item Fix a scale $\veps>0$, and let $N=\cH(\cF,\veps,S')$. Let $\crl[\big]{\wh{f}_i}_{i\in\brk*{N}}$ be a collection of functions that realize the cover, and assume the cover is proper without loss of generality.\footnoteorhide{Any improper $\veps$-cover can be made into a proper $2\veps$-cover.} 
\item Let $\wh{f}$ be the output of the star aggregation algorithm run with the collection $\crl[\big]{\wh{f}_i}_{i\in\brk*{N}}$ on the dataset $S''$.
\end{itemize}
For a simple analysis of this algorithm, see Section 6 of \cite{rakhlin2017empirical}. In general, the algorithm is optimal only in the well-specified setting in which $\En\brk*{Y\mid{}X}=f^{\star}(X)$ for some $f^{\star}\in\cF$. We give a more refined analysis in the presence of nuisance parameters in the sequel. As final remark, note that since we use a proper cover and the star aggregate is $2$-sparse, the final predictor $\wh{f}$ lies in the class $\cF'\ldef\cF + \starhull(\cF-\cF,0)$

\subsection{Rates for Specific Algorithms}
\label{app:specific_rates}
Given an example $\varall\in\cZ$, we define an auxiliary example $(\wt{X},\wt{Y})$ via $\wt{X}(z)=\varone$, $\wt{Y}(z)=
\Gamma(\estone(w), z)$. For the remainder of this section we make use of the auxiliary second-stage dataset $\wt{S}$ defined via
\begin{equation}
\label{eq:auxS}
\wt{S} = \crl[\big]{(\wt{X}(z),\wt{Y}(z))}_{z\in\sampletwo}.
\end{equation}
We make use of the following auxiliary predictor classes:
\begin{align}
\gtF &= \crl*{\wt{X}\mapsto{}\tri*{\Lambda(\gtone(\varone),\varone),\theta(\vartwo)}\mid{}\theta\in\target}\label{eq:gtF},\\
\estF &= \crl*{\wt{X}\mapsto{}\tri*{\Lambda(\estone(\varone),\varone),\theta(\vartwo)}\mid{}\theta\in\target}\label{eq:empF}.%
\end{align}
Finally, we define $\wt{\ls}(\yh,y) = \prn*{\yh-y}^{2}$ and $\wt{L}(f) = \En_{\wt{X},\wt{Y}}\wt{\ls}(f(\wt{X}),\wt{Y})$, where $(\wt{X},\wt{Y})$ are sampled from the distribution introduced by drawing $z\sim{}\cD$, and taking $(\auxX(z),\auxY(z))$. With these definitions, observe that for any $f\in\wh{F}$ of the form $\wt{X}\mapsto{}\tri*{\Lambda(\estone(\varone),\varone),\theta(\vartwo)}$, we have
\begin{equation}
\label{eq:risk_relation}
\auxL(f) = \poprisk(\theta,\estone)\text{~~~~~and~~~~~}\inf_{f\in\cF}\auxL(f) \leq{} \poprisk(\besttwo,\estone).
\end{equation}
We relate the metric entropy of the auxiliary class $\wh{\cF}$ to that of $\target$ as follows.
\begin{proposition}
\label{prop:entropy_aux}
Under \pref{ass:square_oracle}, it holds that
\begin{equation}
\label{eq:entropy_aux}
\cH_2(\wh{\cF},\veps,n)\leq{}\cH_{2}(\targetemp,\veps,n).
\end{equation}

\end{proposition}

\begin{lemma}[Rates for first stage]
\label{lem:stageone_rates}
Global ERM guarantees that with probability at least $1-\delta$,

\begin{align*}
\nrm*{\estone-\gtone}_{L_{2}(\ls_2,\cD)}^{2} \leq \dimone\cdot\left\{
\begin{array}{ll}
C\cdot{}\prn*{\done\log(en/\done) + \log(\delta^{-1})}n^{-1},\quad&\text{Parametric case,}\\
C_{\pone}\cdot{}n^{-\frac{2}{2+\pone}} + \log(\delta^{-1})n^{-1},\quad&\text{Nonparametric case, $\pone<2$,}\\
C\sqrt{(\log{}n + \log(\delta^{-1}))/n},\quad&\text{Nonparametric case, $\pone=2$,}\\
C_{\pone}\cdot{}n^{-\frac{1}{\pone}} + \sqrt{\log(\delta^{-1})/n},\quad&\text{Nonparametric case, $\pone>2$,}
\end{array}
\right.
\end{align*}
where $C_{\pone}$ is a constant that depends only on $\pone$. Skeleton Aggregation guarantees that with probability at least $1-\delta$,
\begin{align*}
\nrm*{\estone-\gtone}_{L_{2}(\ls_2,\cD)}^{2} \leq \dimone\cdot\left\{
\begin{array}{ll}
C\cdot{}\prn*{\done\log(en/\done) + \log(\delta^{-1})}n^{-1},\quad&\text{Parametric case.}\\
C_{\pone}\cdot{}n^{-\frac{2}{2+\pone}} + \log(\delta^{-1})n^{-1},\quad&\text{Nonparametric case.}
\end{array}
\right.
\end{align*}
\end{lemma}
\begin{proof}[\pfref{lem:stageone_rates}] 
In what follows we analyze the algorithms under consideration for the class $\nuisance|_i$ for a fixed coordinate $i$. The final result follows by union bounding over coordinates and summing the coordinate-wise error bounds we establish.\\
\noindent\emph{Global ERM.}~~ When we are either in the parametric case or the nonparametric case with $\pone<2$, the result is given by Theorem 5.2 of \cite{koltchinskii2011introduction}. See Example 3 and Example 4 that follow the theorem for precise calculations under these assumtions. See also \pref{rem:shifted}.

On the other hand, when $\pone\geq{}2$ we apply to the standard Rademacher complexity bound for ERM (e.g. \cite{shalev2014understanding}), which states that with probability at least $1-\delta$,
\[
\En\prn[\big]{\estone_i(w)-(g_0)_i(w)}^{2} \leq{} 2\cdot\mathfrak{R}_{n/2}(\ell_{\mathrm{square}}\circ\nuisance|_i,\sampleone) + \sqrt{\frac{\log(1/\delta)}{n}},
\]
where $\ell_{\mathrm{square}}(g_i, u_i) = \prn*{g_i(w) - u_i}^{2}$. The result follows by applying Lipschitz contraction to the Rademacher complexity (using that the class is bounded) and appealing to the Rademacher complexity bound from \pref{lem:dudley}.

\noindent\emph{Skeleton Aggregation.}~~We appeal to Section 6 of \cite{rakhlin2017empirical}. See \pref{rem:shifted}.

\end{proof}

\begin{lemma}
\label{lem:erm}
Consider the plug-in ERM algorithm for the setting in \pref{sec:square_oracle}, i.e.
\[
\esttwo \in \argmin_{\theta\in\target}\sum_{z\in\sampletwo}\ls(\theta,\estone\midsem{}z).
\]
Under the assumptions of \pref{thm:oracle_well_specified} and \pref{thm:oracle_misspecified}, global ERM guarantees
\[
\poprisk(\esttwo,\estone)-\poprisk(\besttwo,\estone) \leq C\cdot\left\{
\begin{array}{ll}
\prn*{\dtwo\log(en/\dtwo) + \log(\delta^{-1})}n^{-1},\quad&\text{Parametric case.}\\
C_{\ptwo}\cdot{}n^{-\frac{2}{2+\ptwo}} + \log(\delta^{-1})n^{-1},\quad&\text{Nonparametric case, $\ptwo<2$.}\\
\sqrt{(\log{}n + \log(\delta^{-1}))/n},\quad&\text{Nonparametric case, $\ptwo=2$.}\\
C_{\ptwo}\cdot{}n^{-\frac{1}{\ptwo}} + \sqrt{\log(\delta^{-1})/n},\quad&\text{Nonparametric case, $\ptwo>2$.}
\end{array}
\right.
\]

\end{lemma}
\begin{proof}[\pfref{lem:erm}] To begin, let
  $\wh{f}(\auxX)\ldef{}\tri[\big]{\Lambda(\estone(w),w),\esttwo(x)}$
  and observe that we can write $\wh{f}$ as the global ERM for the auxiliary dataset $\wt{S}$:
\[
\wt{f}\in\argmin_{f\in\wh{F}}\sum_{(\auxX,\auxY)\in\auxS}\wt{\ls}(f(\auxX),\auxY).
\]
\noindent\textbf{Case $\ptwo<2$.}~~ In the misspecified case we appeal to Theorem 5.1 in \cite{koltchinskii2011introduction}, using that $\wh{f}$ is the global ERM for the class $\wh{\cF}$. To invoke the theorem, we verify that a) $\wh{\cF}$ takes values in $\unitrange$ under \pref{ass:square_oracle} (see \pref{rem:shifted}), b) $\wh{\cF}$ inherits convexity from $\target$, and c) $\cH_2(\wh{\cF},\veps,n)\leq{}\cH_{2}(\target,\veps,n)$, following \pref{prop:entropy_aux}. The theorem (see also the following discussion in Example 3 and Example 4) therefore guarantees that with probability at least $1-\delta$,
\begin{align*}
\auxL(\wh{f}) -\inf_{f\in\wh{\cF}}\auxL(f) \leq C\cdot\left\{
\begin{array}{ll}
n^{-1}\prn*{\dtwo\log(en/\dtwo) + \log(\delta^{-1})},\quad&\text{Parametric case.}\\
C_{\ptwo}\cdot{}n^{-\frac{2}{2+\ptwo}} + n^{-1}\log(\delta^{-1}),\quad&\text{Nonparametric case, $\pone<2$.}
\end{array}
\right.
\end{align*}
The result now follows from \pref{eq:risk_relation}, in particular that $\auxL(\wh{f})=\poprisk(\esttwo,\estone)$.\\

\noindent\textbf{Case $\ptwo\geq2$.}~~We apply the standard Rademacher complexity bound for ERM (e.g. \cite{shalev2014understanding}), which states that with probability at least $1-\delta$,
\begin{align*}
\auxL(\wh{f}) -\inf_{f\in\wh{\cF}}\auxL(f) &\leq{} 2\cdot\mathfrak{R}_{n/2}(\wt{\ls}\circ\wh{\cF}) + C\sqrt{\frac{\log(1/\delta)}{n}}\\
&\leq{} C'\cdot\mathfrak{R}_{n/2}(\wh{\cF}) + C\sqrt{\frac{\log(1/\delta)}{n}},
\end{align*}
where we have applied Lipschitz contraction to the Rademacher complexity (using that the class is bounded). To complete the result, we use that $\cH_2(\wh{\cF},\veps,n)\leq{}\cH_{2}(\target,\veps,n)$ and appeal to the Rademacher complexity bound from \pref{lem:dudley}.
\end{proof}
\begin{lemma}
\label{lem:skeleton_aggregation}
Consider the following variant of the Skeleton Aggregation algorithm:\footnoteorhide{See \pref{sec:skeleton} for background.}
\begin{itemize}
\item Split $\sampletwo$ into equal-sized subsets $S'$ and $S''$.
\item Fix a scale $\veps>0$, and let $N=\cH_{2}(\target,\veps,S')$. Let $\crl*{\theta_i}_{i\in\brk*{N}}$ be a collection of functions that realize the cover, and assume the cover is proper without loss of generality. Define $f_i = \wt{X}\mapsto{}\tri*{\Lambda(\estone(\varone),\varone),\theta_i(\vartwo)}$ for each $i\in\brk*{N}$.
\item Let $\esttwo\in\target + \starhull(\target-\target,0)\rdef\targetemp$ realize the output of the star aggregation algorithm using the function class $\crl[\big]{\wh{f}_i}_{i\in\brk*{N}}$ on the dataset $\crl[\big]{(\auxX(z),\auxY(z))}_{z\in{}S''}$.%
\end{itemize}
Under the assumptions of \pref{thm:oracle_well_specified}, when the model is well-specified, Skeleton Aggregation guarantees that with probability at least $1-\delta$,
\begin{align*}
&\poprisk(\esttwo,\estone) - \poprisk(\besttwo,\estone) \\
&\leq C\cdot\nrm*{\estone-\gtone}_{L_4(\ls_2,\cD)}^{4} + C'\cdot{}\left\{
\begin{array}{ll}
\prn*{\dtwo\log(en/\dtwo) + \dimtwo\log(\dimtwo\log{}n/\delta)}n^{-1},\quad&\text{Parametric case,}\\
C_{\ptwo}\cdot{}n^{-\frac{2}{2+\ptwo}} + \dimtwo\log(\dimtwo\log{}n/\delta)n^{-1},\quad&\text{Nonparametric case,}
\end{array}
\right.
\end{align*}
so long as $\dimtwo=o(n^{\frac{\ptwo}{2+\ptwo}\wedge\frac{4}{\ptwo(2+\ptwo)}})$.

\end{lemma}
\begin{proof}[\pfref{lem:skeleton_aggregation}]

Let $\wh{\cF} = \crl*{f_i}_{i\in\brk*{N}}$. The Skeleton Aggregation algorithm as described outputs a predictor $\wh{f}\in\wh{\cF} + \starhull(\wh{\cF}-\wh{\cF},0)$ (see \pref{sec:skeleton}) such that
\[
\auxL(\wh{f})\leq{}\min_{i\in\brk*{N}}\auxL(\wh{f}_i) + C\prn*{
\frac{\log{}N}{n} + \frac{\log(1/\delta)}{n}
} \leq{}
\min_{i\in\brk*{N}}\auxL(\wh{f}_i) + C\prn*{
\frac{\cH_{2}(\target,\veps,S)}{n} + \frac{\log(1/\delta)}{n}
}.
\]
Translating back into the language of the lemma statement, recall that we can express each $f_i$ via $f_i=X\mapsto{}\tri*{\Lambda(\estone(\varone),\varone),\theta_i(\vartwo)}$, with $\crl*{\theta_i}_{i\in\brk*{N}}\subset\target$ since we have assumed a proper cover. Since this parameterization is linear in $\theta$, there must be some $\esttwo\in\target + \starhull(\target-\target,0)$ that realizes $\wh{f}$.
Using the expression for the risk in \pref{eq:risk_relation}, this implies
\[
\poprisk(\esttwo,\estone) \leq{}\min_{i\in\brk*{N}}\poprisk(\theta_i,\estone)  + C\prn*{
\frac{\cH_{2}(\target,\veps,S)}{n} + \frac{\log(1/\delta)}{n}
}.
\]
Adding and subtracting from both sides, we rewrite the inequality as
\begin{equation}
\poprisk(\esttwo,\estone) - \poprisk(\gttwo,\estone) \leq{}\min_{i\in\brk*{N}}\poprisk(\theta_i,\estone) - \poprisk(\gttwo,\estone) + C\prn*{
\frac{\cH_{2}(\target,\veps,S)}{n} + \frac{\log(1/\delta)}{n}
}.\label{eq:skeleton_basic_risk}
\end{equation}

Let $\nrm*{\theta -\theta'}_{\target}^2 =
\En\brk*{\tri*{\Lambda(\gtone(\varone),\varone), \theta(\vartwo) -
    \theta'(\vartwo)}^2}$. Observe that by
\pref{lem:single_index_smooth}, we have that for all $\theta,\bar{\theta}$,
\[
D^{2}_{\theta}\poprisk(\bar{\theta},\estone)[\theta-\gttwo,\theta-\gttwo] \leq C\cdot{}\prn*{\nrm*{\theta-\gttwo}_{\target}^{2}
+ \cdot\nrm*{\estone-\gtone}_{L_4(\ls_2,\cD)}^{4}},
\]
and by \pref{lem:suff_single_index}, we have that for all $\theta$ and $\bar{g}$,
\begin{align*}
 D^{2}_{g}D_{\theta}\poprisk(\gttwo,\bar{g})[\theta-\gttwo,\estone-\gtone,\estone-\gtone] &\leq{} 
 C\cdot\nrm*{\theta-\gttwo}_{\target}\cdot\nrm*{\estone-\gtone}_{L_4(\ls_2,\cD)}^{2}\\
&\leq{} 
 C\cdot\prn*{\nrm*{\theta-\gttwo}_{\target}^{2} + \nrm*{\estone-\gtone}_{L_4(\ls_2,\cD)}^{4}}. 
\end{align*}
Hence, since we have assumed orthogonality, \pref{lem:pseudo_risk}
implies that for all $i$,
\[
  \poprisk(\theta_i,\estone) - \poprisk(\gttwo,\estone)
  \leq{} C\cdot{}
    \prn*{\nrm*{\theta_i-\gttwo}_{\target}^{2} + \nrm*{\estone-\gtone}_{L_4(\ls_2,\cD)}^{4}}.
\]
Furthermore, \pref{ass:square_oracle} implies that
$\nrm*{\theta_i-\gttwo}_{\target}\leq{}\nrm*{\theta_i-\gttwo}_{L_2(\ls_2,\cD)}$. Plugging
this bound back into \pref{eq:skeleton_basic_risk}, we have that
\begin{align*}
\poprisk(\esttwo,\estone) - \poprisk(\gttwo,\estone) &\leq{}C\cdot{}\min_{i\in\brk*{N}}\nrm*{\theta_i-\gttwo}_{L_2(\ls_2,\cD)}^{2} + C'\prn*{
\frac{\cH_2(\cF,\veps,S)}{n} + \frac{\log(1/\delta)}{n}
} \\&~~~~+ C''\cdot\nrm*{\estone-\gtone}_{L_4(\ls_2,\cD)}^{4}
\end{align*}
for constants $C,C',C''>0$. We now invoke \pref{lem:l2_concentration} for each of the $\dimtwo$ output coordinates of the target space separately and union bound, which implies that with probability at least $1-\delta$ over the draw of $S'$,
\[
\nrm*{\theta_i-\gttwo}_{L_2(\ls_2,\cD)}^{2}
\leq{} 2d_{2,S'}^{2}(\theta_i,\gttwo) + U,
\]
where \[U\leq{} \dimtwo{}R_{n,\ptwo} + C\frac{\dimtwo\log(\dimtwo\log{}n/\delta)}{n}\]
in the nonparametric case, and \[U\leq{}C\dimtwo{}\frac{\dtwo\log(en/\dtwo)}{n} + C'\frac{\dimtwo\log(\dimtwo\log{}n/\delta)}{n}\] in the parametric case.
Returning to the excess risk, this implies
\begin{align*}
&\poprisk(\esttwo,\estone) - \poprisk(\gttwo,\estone) \\&\leq{}C\cdot{}\min_{i\in\brk*{N}}d_{S'}^{2}(\theta_i,\gttwo) + C'\prn*{
\frac{\cH_2(\cF,\veps,S)}{n}+ \frac{\log(1/\delta)}{n}
} + U + C''\cdot\nrm*{\estone-\gtone}_{L_4(\ls_2,\cD)}^{4}.
\end{align*}
The cover property implies that $\min_{i\in\brk*{N}}d_{2,S'}^{2}(\theta_i,\gttwo) \leq{} \veps^{2}$. So we are left with
\[
\poprisk(\esttwo,\estone) - \poprisk(\gttwo,\estone) \leq{}C\veps^{2} + C'\prn*{
\frac{\cH_2(\cF,\veps,S)}{n} + \frac{\log(1/\delta)}{n}
} + U + C''\cdot\nrm*{\estone-\gtone}_{L_4(\ls_2,\cD)}^{4}.
\]
Solving for the balance $\veps^{2}\asymp\frac{\cH_2(\cF,\veps,S)}{n}$, leads the first two terms to be of order $\dtwo\log(en/\dtwo)$ in the parametric case and $n^{-\frac{2}{2+\ptwo}}$ in the nonparametric case. Thus, in the parametric case, the term $U$ dominates and the final bound is $C\dimtwo{}\frac{\dtwo\log(en/\dtwo)}{n} + C'\frac{\dimtwo\log(\dimtwo\log{}n/\delta)}{n}$. In the nonparametric case, our assumption on the growth of $\dimtwo$ implies that $U$ is of lower order.

\end{proof}

\subsection{Proofs for Oracle Rates}

\begin{proof}[\pfref{thm:oracle_well_specified}]
First we invoke \pref{lem:stageone_rates}, which along with \pref{ass:moment} implies that depending on whether $\pone>2$, one of either global ERM or skeleton aggregation guarantees that with probability at least $1-\delta$,
\[
 \nrm*{\estone-\gtone}_{L_{4}(\ls_2,\cD)}^{2}\leq \prn*{\Rate(\nuisance,\sampleone,\delta)}^{2}
\leq{} \Ot\prn*{
\momentconstant^{2}\cdot{}n^{-\frac{2}{2+\pone}}
}.
\]
Observe that the assumption that the model is well-specified at
$(\gtone,\gttwo)$ implies that \pref{ass:well_specified} is
satisfied. We invoke \pref{ass:single_index} (implied by
\pref{ass:square_oracle}) and \pref{cor:single_index} to get
\[
\poprisk(\esttwo,\gtone) - \poprisk(\besttwo,\gtone) \leq
C\cdot\Rate(\target,\sampletwo,\delta\midsem{}\esttwo,\estone)
+ C'\prn*{\Rate(\nuisance,\sampleone,\delta)}^{4}.
\]
We now invoke \pref{lem:skeleton_aggregation} using that the model is assumed to be well-specified, which implies that with probability at least $1-\delta$, Skeleton Aggregation enjoys
\[
\Rate(\target,\sampletwo,\delta\midsem{}\esttwo,\estone)
\leq{} \Ot\prn*{n^{-\frac{2}{2+\ptwo}}} + C''\prn*{\Rate(\nuisance,\sampleone,\delta)}^{4}.
\]
Combining these results, we get
\[
\poprisk(\esttwo,\gtone) - \poprisk(\besttwo,\gtone) \leq \Ot\prn*{
n^{-\frac{2}{2+\ptwo}} + \momentconstant^{4}\cdot{}n^{-\frac{4}{2+\pone}}
}.
\]
The final result follows by setting $\pone$ to guarantee that the first term dominates this expression.

We mention in passing that to show that global ERM achieves the desired rate for stage two when $\ptwo\leq{}2$, one can appeal to the rates in \pref{app:constrained}.
\end{proof}

\begin{proof}[\pfref{thm:oracle_misspecified}]
As in the proof of \pref{thm:oracle_well_specified}, we invoke \pref{lem:stageone_rates}, which implies that Skeleton Aggregation\footnoteorhide{Alternatively, global ERM can be applied when $\pone\leq{}2$.} guarantees that with probability at least $1-\delta$,
\[
 \nrm*{\estone-\gtone}_{L_{2}(\ls_2,\cD)}^{2}\leq \momentconstant^{2}\prn*{\Rate(\nuisance,\sampleone,\delta) }^{2}
\leq{} \Ot\prn*{\momentconstant^{2}\cdot
n^{-\frac{2}{2+\pone}}
}.
\]
\journal{We remark in passing that global ERM can also be applied when $\pone\leq{}2$.} Observe that since $\target$ is convex, \pref{ass:well_specified} is
satisfied, and we can invoke \pref{ass:single_index} (implied by
\pref{ass:square_oracle}) and \pref{cor:single_index} to get
\[
\poprisk(\esttwo,\gtone) - \poprisk(\besttwo,\gtone) \leq
C\cdot\Rate(\target,\sampletwo,\delta\midsem{}\esttwo,\estone)
+ C'\prn*{\Rate(\nuisance,\sampleone,\delta)}^{4}.
\]
We use global ERM for the second stage. \pref{lem:erm} implies that since the class $\target$ is convex, with probability at least $1-\delta$, global ERM guarantees
\[
\Rate(\target,\sampletwo,\delta\midsem{}\esttwo,\estone)
\leq{} \Ot\prn*{n^{-\frac{2}{2+\ptwo}\wedge\frac{1}{\ptwo}}}.
\]
This leads to a final guarantee of
\[
\poprisk(\esttwo,\gtone) - \poprisk(\besttwo,\gtone) \leq \Ot\prn*{
n^{-\frac{2}{2+\ptwo}\wedge\frac{1}{\ptwo}} + \momentconstant^{4}\cdot{}n^{-\frac{4}{2+\pone}}
},
\]
and the stated result follows by setting $\pone$ to guarantee that the first term dominates this expression.

\end{proof}

\begin{proof}[\pfref{thm:oracle_slow}]
\pref{lem:stageone_rates} implies that either Skeleton Aggregation or global ERM (for $\pone\leq{}2$) guarantees that with probability at least $1-\delta$,
\[
 \nrm*{\estone-\gtone}_{L_{2}(\ls_2,\cD)}^{2}= \prn*{\Rate(\nuisance,\sampleone,\delta) }^{2}
\leq{} \Ot\prn*{
n^{-\frac{2}{2+\pone}}
}.
\]
\pref{thm:orthogonal_slow} guarantees
\[
\poprisk(\esttwo,\gtone) - \poprisk(\besttwo,\gtone) \leq
C\cdot\Rate(\target,\sampletwo,\delta\midsem{}\esttwo,\estone)
+ C'\prn*{\Rate(\nuisance,\sampleone,\delta)}^{2}.
\]
We use global ERM for the second stage. The standard Rademacher complexity bound for ERM (e.g. \cite{shalev2014understanding}), states that with probability at least $1-\delta$, the excess risk is bounded by the Rademacher complexity of the target class $\target$ composed with the loss class as follows:
\begin{align*}
\poprisk(\esttwo,\estone) - \poprisk(\besttwo,\estone) &\leq{} 2\cdot\underbrace{\En_{\eps}\sup_{\theta\in\target}\frac{2}{n}\sum_{i=1}^{n/2}\eps_{t}\ls(\theta(x_i),\estone(w_i);z_i)}_{\textstyle\rdef{} \mathfrak{R}_{n/2}(\ls\circ\target,\sampletwo)} + C\sqrt{\frac{\log(1/\delta)}{n}}.
\end{align*}
Using \pref{lem:dudley} and boundedness of the loss, we have
\[
 \mathfrak{R}_{n/2}(\ls\circ\target,\sampletwo) \leq{} C\cdot\inf_{\alpha>0}\crl*{
\alpha + \int_{\alpha}^{1}
\sqrt{\frac{\cH_2(\ls\circ\target,\veps,\sampletwo)}{n}}d\veps
}.
\]
Since the loss is $1$-Lipschitz with respect to $\ls_{2}$, we have $\cH_2(\ls\circ\target,\veps,\sampletwo)\leq{}\cH_{2}(\target,\veps,\sampletwo)$. Under the assumed growth of $\cH_{2}(\target,\veps,\sampletwo)\propto\veps^{-\ptwo}$, this gives
\[
\Rate(\target,\sampletwo,\delta\midsem{}\esttwo,\estone)
\leq{} \Ot\prn*{n^{-\frac{1}{2}\wedge\frac{1}{\ptwo}}}.
\]
This leads to a final guarantee of
\[
\poprisk(\esttwo,\gtone) - \poprisk(\besttwo,\gtone) \leq \Ot\prn*{
n^{-\frac{1}{2}\wedge\frac{1}{\ptwo}} + n^{-\frac{2}{2+\pone}}.
},
\]
The theorem statement follows by setting $\pone$ so that the first term dominates.

\end{proof}

\iftoggle{supp}{
\journal{\bibliographystyle{imsart-nameyear}}
\bibliography{refs}
}{}

\end{document}